\documentclass{amsart}

\usepackage{amssymb, amscd, amsthm, mathrsfs, bm}

\usepackage{color}
\usepackage[colorlinks=true]{hyperref}

\newtheorem{Thm}{Theorem}[subsection]
\newtheorem{Prop}[Thm]{Proposition}

\newtheorem{Lem}[Thm]{Lemma}
\newtheorem{Def}[Thm]{Definition}

\theoremstyle{remark}
\newtheorem{Rmk}[Thm]{Remark}
\newtheorem{Ex}[Thm]{Example}

\newtheorem*{Ack}{Acknowledgments}

\numberwithin{equation}{subsection}

\newcommand{\Order}{\mathcal{O}}
\newcommand{\into}{\hookrightarrow}
\newcommand{\onto}{\twoheadrightarrow}
\newcommand{\isomto}{\overset{\sim}{\to}}

\newcommand{\compose}{\mathbin{\circ}}
\newcommand{\tensor}{\mathbin{\otimes}}
\newcommand{\closure}[1]{\overline{#1}}
\newcommand{\N}{\mathbb{N}}
\newcommand{\Z}{\mathbb{Z}}
\newcommand{\Q}{\mathbb{Q}}

\newcommand{\F}{\mathbb{F}}
\newcommand{\et}{\mathrm{et}}
\newcommand{\Et}{\mathrm{Et}}
\newcommand{\fppf}{\mathrm{fppf}}
\newcommand{\zar}{\mathrm{zar}}

\newcommand{\id}{\mathrm{id}}
\newcommand{\Gm}{\mathbf{G}_{m}}
\newcommand{\Ga}{\mathbf{G}_{a}}
\newcommand{\tor}{\mathrm{tor}}

\newcommand{\dirlim}{\varinjlim}
\newcommand{\invlim}{\varprojlim}

\newcommand{\sep}{\mathrm{sep}}
\newcommand{\ur}{\mathrm{ur}}
\mathchardef\mhyphen="2D
\newcommand{\Mod}[1]{#1\mhyphen\mathrm{Mod}}
\newcommand{\rat}{\mathrm{rat}}
\newcommand{\ind}{\mathrm{ind}}
\newcommand{\pro}{\mathrm{pro}}
\newcommand{\perf}{\mathrm{perf}}
\newcommand{\perar}{\mathrm{perar}}
\newcommand{\set}{\mathrm{set}}
\newcommand{\Alg}{\mathrm{Alg}}
\newcommand{\Ind}{\mathrm{I}}
\newcommand{\Pro}{\mathrm{P}}

\newcommand{\alg}[1]{\mathbf{#1}}
\newcommand{\var}{\;\cdot\;}
\newcommand{\Ch}{\mathrm{Ch}}
\newcommand{\ideal}[1]{\mathfrak{#1}}
\newcommand{\ctensor}{\mathbin{\Hat{\otimes}}}

\newcommand{\algebrize}{\Acute{\mathsf{a}}}
\newcommand{\Frob}{\mathrm{Fr}}

\newcommand{\nis}{\mathrm{nis}}
\newcommand{\algfrak}[1]{\bm{\mathfrak{#1}}}

\newcommand{\boundary}{\partial}
\newcommand{\Zar}{\mathrm{Zar}}
\newcommand{\genby}[1]{\langle #1 \rangle}

\newcommand{\Sch}{\mathrm{Sch}}

\newcommand{\Affine}{\mathbb{A}}
\newcommand{\Weil}{\mathfrak{R}}
\DeclareMathOperator{\Gal}{Gal}
\DeclareMathOperator{\Hom}{Hom}
\DeclareMathOperator{\Aut}{Aut}
\DeclareMathOperator{\Ker}{Ker}
\DeclareMathOperator{\Coker}{Coker}
\DeclareMathOperator{\Ext}{Ext}
\DeclareMathOperator{\Tor}{Tor}
\DeclareMathOperator{\Spec}{Spec}
\DeclareMathOperator{\Ab}{Ab}
\DeclareMathOperator{\Set}{Set}
\DeclareMathOperator{\Pic}{Pic}
\DeclareMathOperator{\Tr}{Tr}

\let\Im\relax

\DeclareMathOperator{\Im}{Im}
\DeclareMathOperator{\sheafhom}{\alg{Hom}}
\DeclareMathOperator{\sheafext}{\alg{Ext}}

\DeclareMathOperator{\Res}{Res}
\DeclareMathOperator{\gr}{gr}
\DeclareMathOperator{\dlog}{dlog}
\DeclareMathOperator{\Br}{Br}
\DeclareMathOperator{\For}{For}

\title[Arithmetic duality in two-dimension]
	{Arithmetic duality for two-dimensional local rings with perfect residue field}
\author{Takashi Suzuki}
\address{
	Department of Mathematics, Chuo University,
	1-13-27 Kasuga, Bunkyo-ku, Tokyo 112-8551, Japan
}
\email{tsuzuki@gug.math.chuo-u.ac.jp}
\thanks{Partially supported by JSPS Grant-in-Aid 18J00415.}
\date{May 11, 2023}

\begin{document}

\begin{abstract}
We give a refinement of Saito's arithmetic duality for two-dimensional local rings
by giving algebraic group structures for arithmetic cohomology groups.
\end{abstract}

\maketitle

\tableofcontents


\section{Introduction}
\label{0125}


\subsection{Aim of the paper}
\label{0472}

Let $A$ be a complete noetherian normal two-dimensional local ring
with perfect residue field $F$ of characteristic $p > 0$.
When $F$ is finite, Saito \cite{Sai86,Sai87} gives
arithmetic duality and class field theory for $A$
by combining Kato's two-dimensional local class field theories \cite{Kat79} for
local fields at all height one prime ideals of $A$.
On the other hand, Serre \cite{Ser61} and Hazewinkel \cite[Appendix]{DG70b} refine
local class field theory for one-dimensional local fields
by allowing arbitrary perfect residue fields
and giving pro-algebraic group structures for unit groups.
In \cite{Suz22Duality}, we give a functorial reformulation of Serre-Hazewinkel's theory
by using the so-called ``rational \'etale site''.

In this paper, we refine Saito's theory in the style of Serre and Hazewinkel,
by allowing arbitrary perfect $F$
and giving ind-pro-algebraic group structures for arithmetic cohomology groups attached to $A$,
in the case where $A$ has mixed characteristic.
We use the formalism of the ind-rational pro-\'etale site \cite{Suz20}
and the perfect artinian \'etale site \cite{Suz21},
which are improvements of the rational \'etale site.

Our results are arithmetic in nature but, at the same time, also purely algebro-geometric,
since the residue field is a completely general perfect field.
They also reflect the intricacy of the singularity of $A$ and its resolutions.
For these reasons, we hope that our work will be of interest for not only number theorists
but also pure algebraic geometers.

However, we have to admit that the results of this paper are not in the best possible form.
Saito's theory treats the divisor class group of $A$,
the Brauer group of the punctured spectrum,
$K_{2}$-id\`ele class groups, unramified or ramified abelian extensions of the fraction field,
homology of graphs of resolutions and Hasse principles.
Compared to this rich theory, our theory is more abstract and not so complete,
since it is still on the level of cohomology groups.
Also, we do not give finiteness statements in the ideal strongest form.
Deducing more concrete consequences in Saito's style from our theory
as well as appropriate finiteness statements
is the subject of our next papers \cite{Suz22Nilp} and \cite{Suz22CFT}.


\subsection{Notation}
\label{0126}

We need some notation to state our results.
Let $F$ a perfect field of characteristic $p > 0$.
Let $\Alg_{u} / F$ be the category of perfections (inverse limits along Frobenius morphisms)
of commutative unipotent algebraic groups over $F$
with group scheme morphisms (\cite{Ser60}).
It is an abelian category.
Let $\Pro \Alg_{u} / F$ and $\Ind \Alg_{u} / F$ be its pro-category and ind-category, respectively,
and $\Ind \Pro \Alg_{u} / F$ the ind-category of $\Pro \Alg_{u} / F$.
Let $D^{b}(\Ind \Pro \Alg_{u} / F)$ be the bounded derived category of $\Ind \Pro \Alg_{u} / F$.

An $F$-algebra is said to be \emph{rational} if
it can be written as a finite product $F_{1}' \times \dots \times F_{n}'$,
where each $F'_{i}$ is the perfection of a finitely generated field over $F$ (\cite{Suz22Duality}).
Let $F^{\rat}$ be the category of rational $F$-algebras
with $F$-algebra homomorphisms.
Let $F^{\ind\rat}$ be the ind-category of $F^{\rat}$
(viewed as a full subcategory of the category of $F$-algebras).
Let $\Spec F^{\ind\rat}_{\pro\et}$ be the category $F^{\ind\rat}$
endowed with the pro-\'etale topology (\cite{Suz20}).
Let $\Ab(F^{\ind\rat}_{\pro\et})$ be the category of sheaves
of abelian groups on $\Spec F^{\ind\rat}_{\pro\et}$
and $D(F^{\ind\rat}_{\pro\et})$ its derived category.
Let $R \Gamma(F, \var) \colon D(F^{\ind\rat}_{\pro\et}) \to D(\Ab)$
be the cohomology functor for $\Spec F^{\ind\rat}_{\pro\et}$.
Let $\tensor^{L} = \tensor_{\Z}^{L}$ be the derived tensor product functor
for $\Spec F^{\ind\rat}_{\pro\et}$.
Let $R \sheafhom_{F^{\ind\rat}_{\pro\et}}$ be the derived sheaf-Hom functor
for $\Spec F^{\ind\rat}_{\pro\et}$.
We denote $(\var)^{\vee} = R \sheafhom_{F^{\ind\rat}_{\pro\et}}(\var, \Q_{p} / \Z_{p})$.

The Yoneda functor $\Ind \Pro \Alg_{u} / F \to \Ab(F^{\ind\rat}_{\pro\et})$
induces a fully faithful embedding
$D^{b}(\Ind \Pro \Alg_{u} / F) \into D(F^{\ind\rat}_{\pro\et})$
(\cite[Proposition 2.3.4]{Suz20}).
The functor $(\var)^{\vee}$ restricts to a contravariant equivalence between
$D^{b}(\Pro \Alg_{u} / F)$ and $D^{b}(\Ind \Alg_{u} / F)$
(\cite[Proposition 2.4.1 (b)]{Suz20}).
For the perfection of the additive group $\Ga$,
we have $\Ga^{\vee} \cong \Ga[-1]$.

Let $A$ be a complete noetherian normal two-dimensional local ring with $p \ne 0$ in $A$
whose residue field is the above $F$.
Let $\ideal{m}$ be its maximal ideal.
Set $X = \Spec A \setminus \{\ideal{m}\}$.
Let $U$ be a dense open subscheme of $X$.
Let $R \Gamma(U, \var) \colon D(U_{\et}) \to D(\Ab)$ be
the cohomology functor for the small \'etale site $U_{\et}$ of $U$.
Define the compactly supported cohomology $R \Gamma_{c}(U, \var)$ to be
the composite $R \Gamma(X, j_{!}(\var))$,
where $j_{!} \colon D(U_{\et}) \to D(X_{\et})$
is the extension-by-zero functor.
For any $q \in \Z$, set $H^{q}_{c}(U, \var) = H^{q} R \Gamma_{c}(U, \var)$.

Let $\closure{F}$ be an algebraic closure of $F$.
Let $\closure{A}$ be the corresponding completed unramified extension of $A$.
Set $\closure{U} = U \times_{\Spec A} \Spec \closure{A}$.
The above construction applied to $\closure{A}$ and $\closure{U}$ gives a functor
$R \Gamma_{c}(\closure{U}, \var) \colon D(\closure{U}_{\et}) \to D(\Ab)$.

Let $n \ge 1$ and $r$ be integers.
For $r \ge 0$,
let $\Z / p^{n} \Z(r) \in D(U_{\et})$ (or $D(\closure{U}_{\et})$) be
the Bloch cycle complex mod $p^{n}$
or, equivalently up to quasi-isomorphism,
the $p$-adic \'etale Tate twist $\mathfrak{T}_{n}(r)$
(\cite{Sch94}, \cite{Gei04Ded}, \cite{Sat07}, \cite{Sat13}).
Its restriction to $U' = U \cap \Spec A[1 / p]$ is
the usual Tate twist $\Z / p^{n} \Z(r)$ of $\Z / p^{n} \Z$.
For $r < 0$, we understand $\Z / p^{n} \Z(r) \in D(U_{\et})$ to be
$j_{!}'(\Z / p^{n} \Z(r))$,
where $j' \colon U' \into U$ is the inclusion.


\subsection{Main theorems}
\label{0127}

We will prove that the cohomology and the compactly supported cohomology of $U$ have
canonical structures as objects of $\Ind \Pro \Alg_{u} / F$:

\begin{Thm} \label{0128} \mbox{}
	\begin{enumerate}
		\item \label{0129}
			There exist canonical objects
				\[
						R \alg{\Gamma}(\alg{U}, \Z / p^{n} \Z(r)),
						R \alg{\Gamma}_{c}(\alg{U}, \Z / p^{n} \Z(r))
					\in
						D^{b}(\Ind \Pro \Alg_{u} / F)
				\]
			together with canonical isomorphisms
				\begin{gather*}
							R \Gamma \bigl(
								F,
								R \alg{\Gamma}(\alg{U}, \Z / p^{n} \Z(r))
							\bigr)
						\cong
							R \Gamma(U, \Z / p^{n} \Z(r)),
					\\
							R \Gamma \bigl(
								F,
								R \alg{\Gamma}_{c}(\alg{U}, \Z / p^{n} \Z(r))
							\bigr)
						\cong
							R \Gamma_{c}(U, \Z / p^{n} \Z(r))
				\end{gather*}
			in $D(\Ab)$.
		\item \label{0130}
			For any $q \in \Z$,
			let $\alg{H}^{q}(\alg{U}, \Z / p^{n} \Z(r))$ and $\alg{H}^{q}_{c}(\alg{U}, \Z / p^{n} \Z(r))$
			be the $q$-th cohomology objects
			of $R \alg{\Gamma}(\alg{U}, \Z / p^{n} \Z(r))$ and $R \alg{\Gamma}_{c}(\alg{U}, \Z / p^{n} \Z(r))$,
			respectively.
			
			Then the base changes of $\alg{H}^{q}(\alg{U}, \Z / p^{n} \Z(r))$
			and $\alg{H}_{c}^{q}(\alg{U}, \Z / p^{n} \Z(r))$ to $\closure{F}$
			with the natural $\Gal(\closure{F} / F)$-equivariant structures
			are canonically isomorphic to
			$\alg{H}^{q}(\closure{\alg{U}}, \Z / p^{n} \Z(r))$
			and $\alg{H}_{c}^{q}(\closure{\alg{U}}, \Z / p^{n} \Z(r))$,
			respectively.
		\item \label{0131}
			In particular, for any $q \in \Z$, we have
				\begin{gather*}
							\alg{H}^{q}(\alg{U}, \Z / p^{n} \Z(r))(\closure{F})
						\cong
							H^{q}(\closure{U}, \Z / p^{n} \Z(r)),
					\\
							\alg{H}^{q}_{c}(\alg{U}, \Z / p^{n} \Z(r))(\closure{F})
						\cong
							H^{q}_{c}(\closure{U}, \Z / p^{n} \Z(r))
				\end{gather*}
			as $\Gal(\closure{F} / F)$-modules,
	\end{enumerate}
\end{Thm}

We will prove the following refinement of Saito's arithmetic duality
\cite{Sai86,Sai87}:

\begin{Thm} \label{0132}
	Set $r' = 2 - r$.
	Then there exists a canonical morphism
		\[
					R \alg{\Gamma}(\alg{U}, \Z / p^{n} \Z(r))
				\tensor^{L}
					R \alg{\Gamma}_{c}(\alg{U}, \Z / p^{n} \Z(r'))
			\to
				\Q_{p} / \Z_{p}[-3]
		\]
	in $D(F^{\ind\rat}_{\pro\et})$ such that the induced morphisms
		\begin{gather*}
					R \alg{\Gamma}(\alg{U}, \Z / p^{n} \Z(r))
				\to
					R \alg{\Gamma}_{c}(\alg{U}, \Z / p^{n} \Z(r'))^{\vee}[-3],
			\\
					R \alg{\Gamma}_{c}(\alg{U}, \Z / p^{n} \Z(r'))
				\to
					R \alg{\Gamma}(\alg{U}, \Z / p^{n} \Z(r))^{\vee}[-3]
		\end{gather*}
	are isomorphisms.
\end{Thm}

The group $H^{1}(U, \Z / p^{n} \Z)$ classifies abelian $p$-coverings of $U$.
We have an exact sequence
	\[
			0
		\to
			\Pic(U) / p^{n} \Pic(U)
		\to
			H^{2}(U, \Z / p^{n} \Z(1))
		\to
			\Br(U)[p^{n}]
		\to
			0.
	\]
The group $H^{3}_{c}(U, \Z / p^{n} \Z(2))$ contains
the $K_{2}$-id\`ele class group of $U$ (\cite[\S 1]{Sai87}) mod $p^{n}$.
The above theorems thus give ind-pro-algebraic group structures for these important groups
and show their duality.


\subsection{Ideas of proof}
\label{0476}

The basic idea lies in Saito's computations \cite[(4.11) (A), (B)]{Sai86}
of his duality pairing when $F$ is finite.
He introduces a filtration on $K_{1}$ and $K_{2}$ by symbols
and writes the induced pairings on the graded pieces as the pairing
	\[
			F[[t]] \times (\Omega_{F((t))}^{1} / \Omega_{F[[t]]}^{1})
		\stackrel{\mathrm{residue}}{\longrightarrow}
			F
		\stackrel{\mathrm{Tr}_{F / \F_{p}}}{\longrightarrow}
			\Z / p \Z
	\]
and similar pairings with $F[[t]] / F[[t]]^{p}$,
$F((t))^{\times} / F((t))^{\times p}$,
$F((t)) / (\mathrm{Frob} - 1) F((t))$ and so on.
Now we view $F[[t]]$ as a pro-algebraic group $\prod_{n \ge 0} \Ga t^{n}$
and $\Omega_{F((t))}^{1} / \Omega_{F[[t]]}^{1}$ as an ind-algebraic group
$\bigoplus_{n \ge 1} \Ga t^{-n} dt$.
We replace the trace map $F \to \Z / p \Z$ by the morphism
$\Ga \to \Z / p \Z[1]$ to the shift of $\Z / p \Z$ in the derived category
coming from the Artin-Schreier extension class
$0 \to \Z / p \Z \to \Ga \to \Ga \to 0$.
The multiplication morphism $\Ga \times \Ga \to \Ga$ followed by
the morphism $\Ga \to \Z / p \Z[1]$ is a perfect pairing in a suitable sense.

How do we perform such replacement for the entire $K_{1}$ and $K_{2}$ and cohomology groups
in a canonical manner?
For a perfect (possibly transcendental) field extension $F' / F$,
define the ``base change'', ``$F' \tensor_{F} A$'', by
	\[
			\alg{A}(F')
		=
			W(F') \ctensor_{W(F)} A
		=
			\invlim_{n \ge 1}
				(W(F') \tensor_{W(F)} A / \ideal{m}^{n}),
	\]
where $W$ denotes the $p$-typical Witt vectors of infinite length
and $\ideal{m}$ the maximal ideal of $A$.
This ring $\alg{A}(F')$ is the same kind of object as $A$,
except that the residue field is now $F'$.
Now we can run Saito's theory for the ring $\alg{A}(F')$ instead of $A$.
Then everything becomes a functor in $F'$.
The group $F[[t]]$ above is now the functor $F' \mapsto F'[[t]]$.
By the theory of the perfect artinian \'etale site,
such functors in perfect field extensions uniquely determine ind-pro-algebraic groups of interest
together with their derived categorical information.

We need to treat functors in $F'$ \'etale locally in derived categories.
This treatment is best done
with Artin-Milne's relative site constructions
\cite[\S3, ``an auxiliary site $X / S_{\mathrm{perf}}$'']{AM76}.
Roughly speaking, we bundle the \'etale sites of $\alg{A}(F')$ for all $F'$ together into a single site,
so that it admits a morphism of sites to the perfect artinian \'etale site.
The cohomology groups as functors in $F'$ are nothing but derived pushforward sheaves.
In Section \ref{0029}, we develop a general theory of relative sites
and its compact support cohomology and cup product formalism.

With this machinery, we can mostly follow Saito's theory.
In Sections \ref{0169}, \ref{0174} and \ref{0190},
we first develop the local theory,
refining Kato's two-dimensional local class field theory
with ind-pro-algebraic group structures.
With Saito's computations \cite[(4.11) (A), (B)]{Sai86} recalled in
Section \ref{0060},
the case of ``enough regular'' $A$ is done in Section \ref{0113}.
For the general $A$, we take a resolution of singularities of $A$ just as Saito does.
An additional local theory shows up from generic points
of the special fiber (or the reduced exceptional divisor) of the resolution.
This local theory is treated in Section \ref{0217}.


\subsection{Technical difficulties}
\label{0477}

Our duality pairings are hard to calculate
since they are genuinely derived (as can be seen from the shift $\Z / p \Z[1]$ above)
and we have to take care of higher Ext groups.
To overcome this, the key idea is to change the topology.
If we do not work \'etale locally but instead Zariski locally,
then higher Ext groups vanish in the situations at hand,
and we can explicitly calculate duality pairings using Galois symbols and differential forms.
We still need to ensure that
passage from the perfect artinian \'etale site to its Zariski topology version
does not lose much information.
This requires a detailed analysis of Eilenberg-Mac Lane's cubical construction,
which is done in Section \ref{0133}.

In Saito's theory, the additional local theory
coming from generic points of the special fiber of the resolution
is Kato's class field theory \cite{Kat82} for local fields
with residue field a one-dimensional function field.
Unfortunately, this does not translate nicely to our setting.
The reason is that a function field over $F$ is not an object
functorial in (transcendental) extensions $F'$ of $F$:
if $k$ is such a function field, then $k \tensor_{F} F'$ is not necessarily a field.
What works instead is that for a smooth affine curve $V$ with function field $k$,
the object $V \times_{F} F'$ is again a smooth affine curve (over $F'$).
This means that we need to develop a duality theory
for $p$-adic tubular neighborhoods of smooth affine curves.
This is the most precise description of Section \ref{0217}.

This local theory has two difficulties.
First, we use $p$-adic nearby cycle constructions
to reduce the study of duality for $p$-adic tubular neighborhoods
to the classical duality theory for smooth affine curves.
But in the latter duality theory, we have to use the version of compact support cohomology
with local components given by cohomology of \emph{complete} local fields
rather than henselian local fields
(since henselian local fields of positive characteristic do not satisfy a duality).
We have already encountered this type of compact support cohomology in \cite{SuzCurve},
and it is very complicated and does not work well with relative site constructions.
For example, if one considers the additive algebraic group $\Ga$
as a sheaf on the (small!) \'etale site of $V$
and pulls it back to a complete local field $\Hat{k}_{x}$,
the result is no longer $\Ga$.
We deal with this problem by considering the morphism of sites
	\[
			\bigsqcup_{x \text{ a boundary point of } V}
				\Spec \Hat{k}_{x, \et}
		\to
			V_{\et}
	\]
(or its relative site version) as a \emph{fibered site} over the poset $\{\bullet \le \bullet\}$
and taking its total site as the main setting of compact support cohomology and cup product
for tubular neighborhoods.
The generalities on this setting are explained in Section \ref{0325}.
This difficulty also affects the step on resolution of singularities,
where we need fibered sites over more complicated posets.
Hence we need more on fibered sites in Section \ref{0322}
and more on completions in Section \ref{0303}.

The second difficulty is that the calculations of the duality pairing
for cohomology of $p$-adic tubular neighborhoods of affine curves is hard and necessarily indirect
since affine curves are not local objects.
We first need to take the nearby cycle functor to the affine curve,
work Nisnevich locally over the curve to write the pairing,
and then push it down to the base field.
This step is explained in Sections \ref{0227} and \ref{0250}.

\begin{Ack}
	The author would like to thank Kazuya Kato for his suggestion of this project.
\end{Ack}


\section{Notation and general constructions}
\label{0475}

Throughout this paper, let $F$ be a perfect field of characteristic $p > 0$.
All algebraic groups and group schemes are assumed commutative.
The symbol $\subset$ includes the case of equality.


\subsection{Categories}

The categories of sets and abelian groups are denoted by $\Set$ and $\Ab$, respectively.
For an abelian category $\mathcal{A}$,
the category of complexes in $\mathcal{A}$ in cohomological grading is denoted by $\Ch(\mathcal{A})$.
If $A \to B$ is a morphism in $\Ch(\mathcal{A})$,
then its mapping cone is denoted by $[A \to B]$.
The homotopy category of $\Ch(\mathcal{A})$ is denoted by $K(\mathcal{A})$
with derived category $D(\mathcal{A})$.
The full subcategory of $\Ch(\mathcal{A})$ consisting of
bounded below, bounded above and bounded complexes are denoted by
$\Ch^{+}(\mathcal{A})$, $\Ch^{-}(\mathcal{A})$ and $\Ch^{b}(\mathcal{A})$, respectively.
Similar notation applies to $K(\mathcal{A})$ and $D(\mathcal{A})$.
If we say $A \to B \to C$ is a distinguished triangle in a triangulated category,
we implicitly assume that a morphism $C \to A[1]$ to the shift of $A$ is given,
and the triangle $A \to B \to C \to A[1]$ is distinguished.


\subsection{Sites and derived categories}

A good reference for these is \cite{KS06}.
All sites in this paper are defined by given pretopologies.
For a site $S$, we mean by $X \in S$ an object of the underlying category of $S$
and by $Y \to X$ in $S$ a morphism in the underlying category of $S$.
The categories of sheaves of sets and abelian groups on $S$ are denoted by
$\Set(S)$ and $\Ab(S)$, respectively.
Denote $\Ch(\Ab(S))$ by $\Ch(S)$.
Similar notation defines $K(S)$ and $D(S)$.
Let $D_{\tor}^{+}(S) \subset D^{+}(S)$ be the full subcategory
consisting of objects whose cohomology sheaves are torsion.
The sheaf-Hom functor for $\Ab(S)$ is denoted by $\sheafhom_{S}$
and its $n$-th right derived functor by $\sheafext_{S}^{n}$.
Denote $\tensor = \tensor_{\Z}$ and $\tensor^{L} = \tensor_{\Z}^{L}$.
For objects $G, H, K \in D(S)$,
we say that a morphism $G \tensor^{L} H \to K$ is a perfect pairing
if the two induced morphisms $G \to R \sheafhom_{S}(H, K)$ and
$H \to R \sheafhom_{S}(G, K)$ are both isomorphisms.

A premorphism of sites $f \colon S' \to S$ between sites defined by pretopologies
is a functor $f^{-1}$ from the underlying category of $S$ to the underlying category of $S'$
that sends covering families to covering families
such that the natural morphism $f^{-1}(Y \times_{X} Z) \to f^{-1} Y \times_{f^{-1} X} f^{-1} Z$
is an isomorphism whenever $Y \to X$ appears in a covering family.
Its pushforward functor $f_{\ast} \colon \Ab(S') \to \Ab(S)$ sends acyclic sheaves to acyclic sheaves,
where a sheaf $G \in \Ab(S)$ is said to be acyclic
if $H^{n}(X, G) = 0$ for all $X \in S$ and $n \ge 1$.
Acyclic sheaves calculate $R f_{\ast}$ (\cite[Proposition 2.4.2]{SuzCurve}).
The pullback functors for sheaves of sets and abelian groups are
denoted by $f^{\ast \set}$ and $f^{\ast}$, respectively.
Note that $f^{\ast \set}$ and $f^{\ast}$ are not necessarily exact functors.
The left derived functor $L f^{\ast} \colon D(S) \to D(S')$ of $f^{\ast}$ exists
and is left adjoint to $R f_{\ast} \colon D(S') \to D(S)$
(\cite[the paragraph after Definition 2.3]{Suz21}).
An object $G \in D(S)$ is said to be $f$-compatible (\cite[Definition 2.5 (a)]{Suz21})
if the natural morphism $L (f|_{X})^{\ast}(G|_{X}) \to (L f^{\ast} G)|_{f^{-1} X}$
is an isomorphism for all $X \in S$,
where $G|_{X} \in D(S / X)$ is the restriction to the localization $S / X$
and $f|_{X} \colon S' / f^{-1} X \to S / X$ is the premorphism induced on the localizations.
The object $G$ is said to be $f$-acyclic (\cite[Definition 2.5 (b)]{Suz21})
if the natural morphism $G \to R f_{\ast} L f^{\ast} G$ is an isomorphism.
If $f^{\ast \set}$ is exact (that is, commutes with finite inverse limits),
we say that $f$ is a morphism of sites.
In this case, we simply write $f^{\ast \set} = f^{\ast}$.

Let $f \colon T \to S$ be a morphism of sites
defined by a functor $f^{-1}$ on the underlying categories.
Recall from \cite[Proposition 2.4]{Suz21} (or \cite[Tags 0B6C, 0B6D]{Sta21})
that there exist canonical morphisms
	\begin{gather}
			\notag
				R f_{\ast} R \sheafhom_{T}(G, H)
			\to
				R \sheafhom_{S}(R f_{\ast} G, R f_{\ast} H),
		\\ \label{0327}
				R f_{\ast} G \tensor^{L} R f_{\ast} H
			\to
				R f_{\ast}(G \tensor^{L} H)
	\end{gather}
in $D(S)$ functorial in $G, H \in D(T)$.
The second morphism results from the first,
and the first morphism is the derived functor of the morphism
	\[
			f_{\ast} \sheafhom_{T}(G, H)
		\to
			\sheafhom_{S}(f_{\ast} G, f_{\ast} H)
	\]
of functoriality of $f_{\ast}$,
which is a morphism in $\Ch(S)$ functorial in $G, H \in \Ch(T)$.
Here we used the following fact (proved during \cite[Proposition 2.4]{Suz21}):

\begin{Prop} \label{0326}
	Let $G, H \in \Ch(T)$.
	Assume that $H$ is K-injective.
	Then the complex $f_{\ast} \sheafhom_{T}(G, H)$ represents
	$R f_{\ast} R \sheafhom_{T}(G, H)$.
\end{Prop}


\subsection{The ind-rational and perfect aritinian \'etale sites}

The perfection of an $F$-algebra (resp.\ an $F$-scheme) is
the direct (resp.\ inverse) limit along Frobenius morphisms on it.
A scheme $Y'$ over an $\F_{p}$-scheme $Y$ is said to be relatively perfect over $Y$
if its relative Frobenius morphism $Y' \to Y'^{(p)}$ over $Y$ is an isomorphism.
For a ring $A$ complete with respect to an ideal $I$ containing $p$
and a relatively perfect $A / I$-algebra $B'$,
its Kato canonical lifting over $A$ (\cite[\S 1, Definition 1]{Kat82}) is
a unique complete $A$-algebra $A'$
together with an isomorphism $A' / I A' \cong B'$ over $A / I$
such that $A' / I^{n} A'$ is formally \'etale over $A / I^{n}$ for all $n$.
If $B'$ is flat over $A / I$ and $A$ is noetherian,
then it is characterized as a unique complete flat $A$-algebra
together with an isomorphism $A' / I A' \cong B'$ over $A / I$
(\cite[\S 1, Lemma 1]{Kat82} plus \cite[Tag 0912]{Sta21}).

We recall the ind-rational pro-\'etale site $\Spec F^{\ind\rat}_{\pro\et}$ from \cite{Suz20}.
An $F$-algebra is said to be rational if it is a finite direct product of
perfections of finitely generated fields over $F$.
An ind-rational $F$-algebra is a direct limit of a filtered direct system consisting of
rational $F$-algebras.
Let $F^{\ind\rat}$ be the category of ind-rational $F$-algebras
with $F$-algebra homomorphisms.
Then $\Spec F^{\ind\rat}_{\pro\et}$ is (the opposite category of) the category $F^{\ind\rat}$
equipped with the pro-\'etale topology.
We denote $\Ab(F^{\ind\rat}_{\pro\et}) = \Ab(\Spec F^{\ind\rat}_{\pro\et})$,
and use similar such notation as $\sheafhom_{F^{\ind\rat}_{\pro\et}}$.

We recall the perfect artinian \'etale site $\Spec F^{\perar}_{\et}$
from \cite{Suz21}.
Let $F^{\perar}$ be the category of perfect artinian $F$-algebras
(or, equivalently, finite direct products of perfect field extensions of $F$)
with $F$-algebra homomorphisms.
Then $\Spec F^{\perar}_{\et}$ is the category $F^{\perar}$ endowed with the \'etale topology.

We also use the Zariski topology on $F^{\perar}$.
Define $\Spec F^{\perar}_{\zar}$ to be the category $F^{\perar}$
endowed with the topology
where a covering of an object $F' \in F^{\perar}$
is a finite family $\{F' \to F'_{i}\}$ of objects over $F'$
such that each $F' \to F'_{i}$ is a projection onto a direct factor
and $\prod_{i} F'_{i}$ is faithfully flat over $F'$.
We call it the perfect artinian Zariski site.

We recall the perfect pro-fppf site $\Spec F^{\perf}_{\pro\fppf}$ \cite{Suz22Duality}.
A perfect $F$-algebra homomorphism $R \to S$ is said to be
flat of finite presentation (in the perfect algebra sense)
if $S$ is the perfection of a flat $R$-algebra of finite presentation.
It is said to be flat of ind-finite presentation
if $S$ is the direct limit of a filtered direct system of perfect $R$-algebras
flat of finite presentation.
Define $\Spec F^{\perf}_{\pro\fppf}$ to be the category of perfect $F$-algebras
where a covering of an object $R$ is a finite family $\{R_{i}\}$ of $R$-algebras
flat of ind-finite presentation such that $\prod_{i} R_{i}$ is faithfully flat over $R$.

We recall the functor
$\algebrize \colon D(F^{\perar}_{\et}) \to D(F^{\ind\rat}_{\pro\et})$
\cite{Suz21}.
Let
	\begin{equation} \label{0460}
			\Spec F^{\perf}_{\pro\fppf}
		\stackrel{f}{\to}
			\Spec F^{\ind\rat}_{\pro\et}
		\stackrel{g}{\to}
			\Spec F^{\perar}_{\et}
	\end{equation}
be the premorphisms defined by the inclusion functors on the underlying categories.
Let $h = g \compose f$.
Then $\algebrize = R f_{\ast} L h^{\ast}$.
For objects $G, H \in D(F^{\perar}_{\et})$ such that $G$ or $H$ is $h$-compatible,
we have a canonical morphism
	\begin{equation} \label{0453}
			\algebrize G \tensor^{L} \algebrize H
		\to
			\algebrize(G \tensor^{L} H)
	\end{equation}
in $D(F^{\ind\rat}_{\pro\et})$ functorial in such $G$ and $H$
(\cite[Proposition 7.8]{Suz21}).


\subsection{Quasi-algebraic groups}

A (commutative) quasi-algebraic group over $F$ is
the perfection of a group scheme of finite type over $F$
(\cite{Ser60}).
The category of quasi-algebraic groups over $F$ with $F$-group scheme morphisms
is denoted by $\Alg / F$.
The identity component and the group of components of $G \in \Alg / F$ are denoted
by $G^{0}$ and $\pi_{0}(G)$, respectively.
By abuse of notation, $\Ga$, $\Gm$ and $W_{n}$ ($p$-typical Witt vectors of length $n$)
are the perfections of the corresponding group schemes of finite type.
We say that a quasi-algebraic group is unipotent if it is the perfection of such a group.
Let $\Alg_{u} / F \subset \Alg / F$ be the full subcategory of unipotent groups.

The procategory of $\Alg / F$ is denoted by $\Pro \Alg / F$.
We consider its full subcategory $\Pro' \Alg / F$ consisting of objects isomorphic to
filtered inverse systems with affine transition morphisms.
Its objects are schemes.
The ind-category of $\Pro' \Alg / F$ is denoted by $\Ind \Pro' \Alg / F$.
The ind-category $\Ind \Pro \Alg_{u} / F$ of the pro-category of $\Alg_{u} / F$
is a full subcategory of $\Ind \Pro' \Alg / F$.
The functors $G \mapsto G^{0}, \pi_{0}(G)$ naturally extend to $\Ind \Pro' \Alg / F$.
An object $G \in \Ind \Pro' \Alg / F$ is said to be connected
if $G^{0} \isomto G$.
The Yoneda functor gives a functor from $\Ind \Pro' \Alg / F$
to $\Ab(F^{\ind\rat}_{\pro\et})$, $\Ab(F^{\perar}_{\et})$ or $\Ab(F^{\perar}_{\zar})$.
The image of $G \in \Ind \Pro' \Alg / F$ by this functor is denoted by the same symbol $G$
by abuse of notation.


\subsection{\'Etale Tate twists}

For each $n \ge 1$, we set
	\[
			\Lambda_{n}
		=
			\Z / p^{n} \Z,
		\quad
			\Lambda
		=
			\Lambda_{1}
		=
			\Z / p \Z,
		\quad
			\Lambda_{\infty}
		=
			\Q_{p} / \Z_{p}.
	\]
They are abelian groups, but also viewed as constant sheaves on the sites at hand.

For a regular $\F_{p}$-scheme $Y$, $n \ge 1$ and $r \in \Z$,
we denote the logarithmic Hodge-Witt sheaf on $Y_{\et}$ by
$\nu_{n}(r) = W_{n} \Omega_{Y, \log}^{r}$
(\cite[Definition 2.6]{Shi07}) if $r \ge 0$
and set $\nu_{n}(r) = 0$ otherwise.
We set $\nu(r) = \nu_{1}(r)$ and $\nu_{\infty}(r) = \dirlim_{n} \nu_{n}(r)$.

Let $X$ be a regular scheme of dimension $\le 1$ such that
$U := X \times_{\Spec \Z} \Spec \Z[1 / p]$ is dense.
Let $j \colon U \into X$ be the inclusion.
For any $x \in X \setminus U$, $n \ge 1$ and $r \in \Z$,
we have the Bloch-Kato boundary morphism
$i_{x}^{\ast} R^{r} j_{\ast} \Lambda_{n}(r) \to \nu_{n}(r - 1)$
in $\Ab(x_{\et})$,
where $i_{x} \colon x \into X$ is the inclusion
(\cite[(1.3) (ii)]{Kat86Hasse}, \cite[Theorem (1.4) (i)]{BK86}).
Define $\mathfrak{T}_{n}(r)$ to be the canonical mapping fiber of the resulting morphism
	\begin{equation} \label{0000}
				\tau_{\le r}
				R j_{\ast} \Lambda_{n}(r)
			\to
				\bigoplus_{x \in X \setminus U}
					i_{x, \ast} \nu_{n}(r - 1)[-r]
	\end{equation}
in $D(X_{\et})$ if $r \ge 0$
and define $\mathfrak{T}_{n}(r) = j_{!} \Lambda_{n}(r)$ otherwise.
Set $\mathfrak{T}(r) = \mathfrak{T}_{1}(r)$.


\subsection{Two-dimensional local rings}
\label{0463}

When given a two-dimensional noetherian henselian normal local ring $A$ with residue field $F$,
we use the following notation.
Let $\ideal{m}$ be the maximal ideal of $A$.
Let $X = \Spec A \setminus \{\ideal{m}\}$.
Let $P$ be the set of height one prime ideals of $A$.
For $\ideal{p} \in P$,
let $A_{\ideal{p}}^{h}$ be the henselian local ring of $A$ at $\ideal{p}$,
$K_{\ideal{p}}^{h}$ its fraction field,
$\kappa(\ideal{p})$ the residue field of $A_{\ideal{p}}^{h}$,
and $F_{\ideal{p}}$ the residue field of $\kappa(\ideal{p})$ with respect to the natural valuation.


\section{Derived sheaf-Hom for algebraic groups}
\label{0133}

In \cite[\S 2]{Suz20}, we studied ind-pro-algebraic groups over $F$
as sheaves over $\Spec F^{\ind\rat}_{\pro\et}$
and a duality operation for them.
In \cite{Suz21}, we studied a similar theory over $\Spec F^{\perar}_{\et}$.
In this paper, in Section \ref{0309},
we provide some more about a duality operation over $\Spec F^{\perar}_{\et}$.
After proving a technical result in Sections \ref{0154} and \ref{0155},
we give a similar theory (including a duality operation) over $\Spec F^{\perar}_{\zar}$.
In later sections of this paper, most (not all) of the duality results will be proved
first over $\Spec F^{\perar}_{\zar}$.
They will then be brought to $\Spec F^{\perar}_{\et}$
and finally to $\Spec F^{\ind\rat}_{\pro\et}$.
In Section \ref{0518},
we study how the functor
$\algebrize \colon D(F^{\perar}_{\et}) \to D(F^{\ind\rat}_{\pro\et})$
recalled at \eqref{0460} behaves under base change
and commutes with the associated equivariant structures (or descent data).


\subsection{Derived Hom in the perfect artinian \'etale site}
\label{0309}

Here is a class of ind-pro-algebraic groups over $F$
that can be treated nicely with $\Spec F^{\perar}_{\et}$:

\begin{Def} \label{0149}
	Consider the following two types of objects of $\Ind \Pro \Alg_{u} / F$:
	\begin{enumerate}
		\item \label{0005}
			Objects isomorphic to $\invlim_{n \ge 0} G_{n}$,
			where each $G_{n} \in \Alg / F$ is connected unipotent
			and the transition morphisms $G_{n + 1} \to G_{n}$ are surjective
			with connected kernel.
		\item \label{0006}
			Objects isomorphic to $\dirlim_{n \ge 0} G_{n}$,
			where each $G_{n} \in \Alg / F$ is connected unipotent
			and the transition morphisms $G_{n} \to G_{n + 1}$ are injective.
	\end{enumerate}
	Define $\mathcal{W}_{F}$ to be the full subcategory of $\Ind \Pro \Alg_{u} / F$
	consisting of objects $G$ that admit a subobject $G'$ of Type \eqref{0005}
	such that $(G / G')^{0}$ is of Type \eqref{0006}
	and $\pi_{0}(G / G')$ is finite \'etale $p$-primary.
\end{Def}

These conditions mean that the connected part of an object of $\mathcal{W}_{F}$ is
built up from $\Ga$ by countable successive extensions in an ``ind-pro'' manner.
The category $\mathcal{W}_{F}$ is an exact subcategory of $\Ind \Pro \Alg_{u} / F$.
It is also a full subcategory of $\Ab(F^{\perar}_{\et})$
via the Yoneda functor $\mathcal{W}_{F} \to \Ab(F^{\perar}_{\et})$
by \cite[Proposition 7.1]{Suz21}.

In \cite[\S 2]{Suz20}, we have established a fully faithful embedding
$D^{b}(\Ind \Pro \Alg / F) \into D(F^{\ind\rat}_{\pro\et})$
and a duality operation $R \sheafhom_{F^{\ind\rat}_{\pro\et}}(\var, \Lambda_{\infty})$
(Serre duality) for several kinds of objects of $D^{b}(\Ind \Pro \Alg / F)$.
In this subsection, we will bring these results to $\mathcal{W}_{F} \subset \Ab(F^{\perar}_{\et})$
and its triangulated version.
This is a continuation of and complements to \cite{Suz21}.

Let $f, g$ and $h = g \compose f$ be the premorphisms given in \eqref{0460}.
With regards to objects of $\mathcal{W}_{F}$,
we can freely pass between $\Spec F^{\perar}_{\et}$ and $\Spec F^{\ind\rat}_{\pro\et}$:

\begin{Prop} \label{0492}
	Let $G \in \mathcal{W}_{F}$.
	Then $\algebrize G \cong G$ in $D(F^{\ind\rat}_{\pro\et})$
	and $R g_{\ast} G \cong G$ in $D(F^{\perar}_{\et})$.
\end{Prop}

\begin{proof}
	These are \cite[Proposition 7.9 (b)]{Suz21} and \cite[Proposition (2.4.2)]{Suz20}.
\end{proof}

We will treat a triangulated version of this statement.
For this, we need:

\begin{Def} \label{0151}
	Define $\genby{\mathcal{W}_{F}}_{F^{\perar}_{\et}}$
	(resp.\ $\genby{\mathcal{W}_{F}}_{F^{\ind\rat}_{\pro\et}}$)
	to be the smallest full triangulated subcategory of
	$D(F^{\perar}_{\et})$ (resp.\ $D(F^{\ind\rat}_{\pro\et})$)
	closed under direct summands containing objects of $\mathcal{W}_{F}$ placed in degree zero.
\end{Def}

Unfortunately, cohomology objects of an object of $\genby{\mathcal{W}_{F}}_{F^{\perar}_{\et}}$
in $D(F^{\perar}_{\et})$
is not necessarily representable by an ind-pro-algebraic group.
That is why we will need to eventually pass to $\Spec F^{\ind\rat}_{\pro\et}$
where we have full control of $\Ind \Pro \Alg / F$.

The functor $\algebrize$ applied to objects of this category preserves
(\'etale or pro-\'etale) cohomological information:

\begin{Prop} \label{0153}
	Any $G \in \genby{\mathcal{W}_{F}}_{F^{\perar}_{\et}}$ is
	$h$-compatible and $h$-acyclic.
	In particular, we have
	$R \Gamma(F', \algebrize(G)) \cong R \Gamma(F', G)$
	for any $F' \in F^{\perar}$.
\end{Prop}

\begin{proof}
	This follows from \cite[Proposition 7.2 and Proposition 7.10]{Suz21}.
\end{proof}

In particular, the product structure \eqref{0453} is available
for objects of $\genby{\mathcal{W}_{F}}_{F^{\perar}_{\et}}$.
Here is the triangulated version of Proposition \ref{0492}:

\begin{Prop} \label{0152}
	The functors $\algebrize \colon D(F^{\perar}_{\et}) \to D(F^{\ind\rat}_{\pro\et})$ and
	$R g_{\ast} \colon D(F^{\ind\rat}_{\pro\et}) \to D(F^{\perar}_{\et})$
	restrict to equivalences between
	$\genby{\mathcal{W}_{F}}_{F^{\perar}_{\et}}$
	and $\genby{\mathcal{W}_{F}}_{F^{\ind\rat}_{\pro\et}}$
	inverse to each other.
	In particular, $\algebrize$ maps
	$\genby{\mathcal{W}_{F}}_{F^{\perar}_{\et}}$
	into $D^{b}(\Ind \Pro \Alg_{u} / F)$.
\end{Prop}

\begin{proof}
	We have $R h_{\ast} \cong R g_{\ast} R f_{\ast} $
	and $L h^{\ast} \cong L f^{\ast} L g^{\ast}$
	by \cite[Proposition 2.6]{Suz21}.
	Hence we have natural transformations
		\begin{gather*}
					\id
				\to
					R h_{\ast} L h^{\ast}
				\cong
					R g_{\ast} R f_{\ast} L h^{\ast}
				=
					R g_{\ast} \algebrize,
			\\
					\id
				\to
					R f_{\ast} L f^{\ast}
				\gets
					R f_{\ast} L f^{\ast} L g^{\ast} R g_{\ast}
				\cong
					R f_{\ast} L h^{\ast} R g_{\ast}
				=
					\algebrize R g_{\ast}.
		\end{gather*}
	For $G \in \mathcal{W}_{F}$, we have $L f^{\ast} G \cong G$ and $R f_{\ast} G \cong G$
	by \cite[Proposition 7.6]{Suz21}.
	Hence Proposition \ref{0492} shows that
	the above natural transformations induce isomorphisms for $G \in \mathcal{W}_{F}$.
	Therefore they induce isomorphisms for the mentioned generated subcategories.
\end{proof}

Cohomology objects of objects of $\genby{\mathcal{W}_{F}}_{F^{\perar}_{\et}}$
do not necessarily belong to $\mathcal{W}_{F}$.
When they do, applying $\algebrize$ preserves cohomology objects:

\begin{Prop} \label{0501}
	Let $G \in D^{b}(F^{\perar}_{\et})$.
	Assume that $H^{q} G \in \mathcal{W}_{F}$ for all $q$
	(which implies $G \in \genby{\mathcal{W}_{F}}_{F^{\perar}_{\et}}$).
	Then $H^{q} \algebrize G \cong H^{q} G$ in $\Ind \Pro \Alg_{u} / F$.
\end{Prop}

\begin{proof}
	This is \cite[Proposition 7.9 (b)]{Suz21}.
\end{proof}

We recall Serre duality in the particular case of $\genby{\mathcal{W}_{F}}_{F^{\ind\rat}_{\pro\et}}$:

\begin{Prop} \mbox{} \label{0008}
	\begin{enumerate}
		\item
			The functor $R \sheafhom_{F^{\ind\rat}_{\pro\et}}(\var, \Lambda_{\infty})$
			gives a contravariant autoequivalence on
			$\genby{\mathcal{W}_{F}}_{F^{\ind\rat}_{\pro\et}}$
			with inverse itself.
		\item
			If $G \in \mathcal{W}_{F}$ is connected,
			then $R \sheafhom_{F^{\ind\rat}_{\pro\et}}(G, \Lambda_{\infty})$ is concentrated in degree $1$
			whose cohomology is a connected group in $\mathcal{W}_{F}$.
			If $G$ is of Type \eqref{0005},
			then it is of Type \eqref{0006},
			and vice versa.
		\item
			If $G \in \mathcal{W}_{F}$ is finite,
			then $R \sheafhom_{F^{\ind\rat}_{\pro\et}}(G, \Lambda_{\infty})$ is concentrated in degree $0$
			whose cohomology is the Pontryagin dual of $G$.
	\end{enumerate}
\end{Prop}

\begin{proof}
	This follows from \cite[Proposition 2.4.1 (a), (b)]{Suz20}.
\end{proof}

Now the functor $R \sheafhom(\var, \Lambda_{\infty})$ can be considered
in either $\Spec F^{\ind\rat}_{\pro\et}$ or $\Spec F^{\perar}_{\et}$
as long as it is applied to objects coming from $\mathcal{W}_{F}$:

\begin{Prop} \label{0496}
	Let $G \in \genby{\mathcal{W}_{F}}_{F^{\perar}_{\et}}$
	and set $G' = \algebrize G$.
	Then
		\begin{gather} \label{0493}
					R \sheafhom_{F^{\perar}_{\et}}(G, \Lambda_{\infty})
				\cong
					R g_{\ast}
					R \sheafhom_{F^{\ind\rat}_{\pro\et}}(G', \Lambda_{\infty}),
			\\ \label{0494}
					\algebrize
					R \sheafhom_{F^{\perar}_{\et}}(G, \Lambda_{\infty})
				\cong
					R \sheafhom_{F^{\ind\rat}_{\pro\et}}(G', \Lambda_{\infty}).
		\end{gather}
\end{Prop}

\begin{proof}
	We may assume $G \in \mathcal{W}_{F}$ (so we identify $G' = G$).
	By the $h$-compatibility of $G$ (Proposition \ref{0153}),
	$L h^{\ast} G \cong G$ (\cite[Theorem 3.15]{Suz21}) and
	$R h_{\ast} \Lambda_{\infty} \cong \Lambda_{\infty}$,
	we have
		\[
				R \sheafhom_{F^{\perar}_{\et}}(G, \Lambda_{\infty})
			\cong
				R h_{\ast}
				R \sheafhom_{F^{\perf}_{\pro\fppf}}(G, \Lambda_{\infty})
		\]
	by \cite[Proposition 2.8]{Suz21}.
	We have
		\[
				R f_{\ast}
				R \sheafhom_{F^{\perf}_{\pro\fppf}}(G, \Lambda_{\infty})
			\cong
				R \sheafhom_{F^{\ind\rat}_{\pro\et}}(G, \Lambda_{\infty})
		\]
	by \cite[Theorem (2.3.1), Proposition (2.3.2)]{Suz20}.
	Taking $R g_{\ast}$ of both sides, we obtain \eqref{0493}.
	
	This implies that
	$R \sheafhom_{F^{\perar}_{\et}}(G, \Lambda_{\infty})$
	is in $\genby{\mathcal{W}_{F}}_{F^{\perar}_{\et}}$
	by Propositions \ref{0008} and \ref{0152}.
	Hence \eqref{0494} follows from \eqref{0493} by
	Proposition \ref{0152}.
\end{proof}

Therefore Serre duality is valid also in $\Spec F^{\perar}_{\et}$:

\begin{Prop} \label{0497}
	Proposition \ref{0008} also holds
	with all the instances of ``$F^{\ind\rat}_{\pro\et}$''
	replaced by ``$F^{\perar}_{\et}$''.
\end{Prop}

\begin{proof}
	This follows from Propositions \ref{0496}, \ref{0152} and \ref{0492}.
\end{proof}

In this paper, we will prove duality statements first over $\Spec F^{\perar}_{\et}$
and then bring them to $\Spec F^{\ind\rat}_{\pro\et}$.
This is possible because of the following:

\begin{Prop} \label{0010}
	Let $G, G' \in \genby{\mathcal{W}_{F}}_{F^{\perar}_{\et}}$.
	Let
		\[
				G \tensor^{L} G' \to \Lambda_{\infty}
		\]
	be a morphism in $D(F^{\perar}_{\et})$.
	Then it is a perfect pairing if and only if the induced morphism
		\[
				\algebrize G \tensor^{L} \algebrize G'
			\to
				\Lambda_{\infty}
		\]
	by \eqref{0453} is a perfect pairing in $D(F^{\ind\rat}_{\pro\et})$.
\end{Prop}

\begin{proof}
	This follows from Proposition \ref{0496}.
\end{proof}

In some situations, the base field is not $F$ but a finite extension $F'$ of $F$.
For example, consider residue fields of a curve over $F$ at closed points.
We can bring duality statements over $F'$ to $F$.
To state this, let $\alpha \colon \Spec F'^{\perar}_{\et} \to \Spec F^{\perar}_{\et}$ be the natural morphism.
Set $\Weil_{F' / F} = \alpha_{\ast}$,
which is the Weil restriction functor (\cite[\S 7.6]{BLR90}).

\begin{Prop} \label{0452}
	Let $F'$ be a finite extension of $F$.
	Let $G \tensor^{L} G' \to \Lambda_{\infty}$ be a perfect pairing in $D(F'^{\perar}_{\et})$.
	Then the composite morphism
		\[
					\Weil_{F' / F} G
				\tensor^{L}
					\Weil_{F' / F} G'
			\to
				\Weil_{F' / F} \Lambda_{\infty}
			\to
				\Lambda_{\infty}
		\]
	in $D(F^{\perar}_{\et})$ is a perfect pairing,
	where the last morphism is the norm map.
\end{Prop}

\begin{proof}
	This follows from the duality for finite \'etale morphisms
		\[
				\alpha_{\ast}
				R \sheafhom_{F^{\perar}_{1, \et}}(G, \alpha^{\ast} H)
			\cong
				R \sheafhom_{F^{\perar}_{\et}}(\alpha_{\ast} G, H)
		\]
	(where $H \in D(F^{\perar}_{\et})$),
	which can be proven in the same way as usual
	(\cite[Chapter V, Proposition 1.13]{Mil80}).
\end{proof}


\subsection{The cubical construction and Mac Lane's resolution}
\label{0154}

We recall the cubical construction $Q(G)$ and Mac Lane's resolution $M(G)$
of an abelian group $G$ (\cite{ML57}).
See also \cite[Section 13.2 and Exercise E.13.2.1]{Lod98}.
These constructions are useful for explicitly describing $\Ext$ groups over sites
in terms of cohomology groups.
The goal is to show that,
in order to prove duality statements over $\Spec F^{\perar}_{\et}$,
it is enough to prove them over $\Spec F^{\perar}_{\zar}$ (Proposition \ref{0026}).

For $n \ge 0$, set $2^{n} = \{0, 1\}^{n}$,
whose elements are $n$-tuples $(\varepsilon(1), \dots, \varepsilon(n))$
of numbers $\varepsilon(1), \dots, \varepsilon(n) \in 2 = \{0, 1\}$.
Let $Q'(G)$ be the graded abelian group in non-negative degrees
whose $n$-th term is $Q'_{n}(G) = \Z[G^{2^{n}}]$,
which is freely generated by functions $t \colon 2^{n} \to G$.
For each $1 \le i \le n + 1$ and $t \in G^{2^{n + 1}}$,
define functions $R_{i} t, S_{i} t, P_{i} t \in G^{2^{n}}$ by
	\begin{align*}
				(R_{i} t)(\varepsilon(1), \dots, \varepsilon(n))
		&	=
				t(\varepsilon(1), \dots, \varepsilon(i - 1), 0, \varepsilon(i), \dots, \varepsilon(n)),
		\\
				(S_{i} t)(\varepsilon(1), \dots, \varepsilon(n))
		&	=
				t(\varepsilon(1), \dots, \varepsilon(i - 1), 1, \varepsilon(i), \dots, \varepsilon(n)),
		\\
				(P_{i} t)(\varepsilon(1), \dots, \varepsilon(n))
		&	=
					t(\varepsilon(1), \dots, \varepsilon(i - 1), 0, \varepsilon(i), \dots, \varepsilon(n))
		\\
		&	\qquad
				+
					t(\varepsilon(1), \dots, \varepsilon(i - 1), 1, \varepsilon(i), \dots, \varepsilon(n)),
	\end{align*}
where the sum is the group operation in $G$.
Define $\boundary t \in \Z[G^{2^{n}}]$ by
	\[
			\boundary t
		=
			\sum_{i = 1}^{n + 1}
				(-1)^{i} (P_{i} t - R_{i} t - S_{i} t),
	\]
where the right-hand side is a (formal) sum in $\Z[G^{2^{n}}]$ (not sums in $G^{2^{n}}$).
Then $\boundary \compose \boundary = 0$.
We consider $Q'(G)$ as a complex (in the homological grading) with differential $\boundary$.

For each $1 \le i \le n$, a function $t \in G^{2^{n}}$ is said to be an $i$-slab
if $t(\varepsilon(1), \dots, \varepsilon(n)) = 0$ whenever $\varepsilon(i) = 0$
or if $t(\varepsilon(1), \dots, \varepsilon(n)) = 0$ whenever $\varepsilon(i) = 1$.
For each $1 \le i \le n - 1$ (where $n \ge 1$),
a function $t \in G^{2^{n}}$ is said to be an $i$-diagonal
if $t(\varepsilon(1), \dots, \varepsilon(n)) = 0$
whenever $\varepsilon(i) \ne \varepsilon(i + 1)$.
Let $N_{G} \subset Q'(G)$ be the graded subgroup generated by all slabs and diagonals.
It is a subcomplex.
Each term $N_{G, n} \subset Q'_{n}(G)$ is a direct summand
and the splitting $Q'_{n}(G) \cong N_{G, n} \oplus Q_{n}(G)$ can be taken functorially in $G$
(\cite[Section 5]{Pir96}, \cite[Proposition 2.6]{JP91}).
This splitting does not respect $\partial$, though.

Now define a complex by $Q(G) = Q'(G) / N_{G}$.
It has a $G$-augmentation given by $Q_{0}(G) = \Z[G] / {\Z \cdot (0)} \onto G$, $(g) \mapsto g$.
In particular, we have a $\Z$-augmented complex $Q(\Z)$.
For $t \in G^{2^{m}}$ and $u \in \Z^{2^{n}}$,
define a function $t u \in G^{2^{m + n}}$ by
	\[
			(t u)(\varepsilon(1), \dots, \varepsilon(m + n))
		=
				t(\varepsilon(1), \dots, \varepsilon(m))
			\cdot
				u(\varepsilon(m + 1), \dots, \varepsilon(m + n)),
	\]
where the dot on the right-hand side is the $\Z$-action on the abelian group $G$.
The linear extension of this operation defines a morphism
$Q(G) \tensor_{\Z} Q(\Z) \to Q(G)$ of graded abelian groups.
When $G = \Z$, this gives a differential graded ring structure on $Q(\Z)$,
and when $G$ is arbitrary,
it gives a differential graded right $Q(\Z)$-module structure on $Q(G)$.
The augmentation $Q(\Z) \onto \Z$ is a morphism of differential graded rings
when $\Z$ is viewed as concentrated in degree zero.
In particular, $\Z$ can be viewed as a differential graded left $Q(\Z)$-module.

Define $M(G) = (\dots \to M_{1}(G) \to M_{0}(G))$ to be the two-sided bar construction $B(Q(G), Q(\Z), \Z)$
(\cite[Appendix A]{GM74};
the notation in \cite[Section 7]{ML57} is
$Q(G) \tensor_{Q(\Z)} B(Q(\Z), \eta_{Q})$).
We do not review the definition of two-sided bar constructions.
It is a $G$-augmented complex.
As a graded abelian group (forgetting the differential),
the complex $B(Q(G), Q(\Z), \Z)$ as a functor in $G$ is given by
$Q(G) \tensor_{\Z} \mathcal{B}$
for some graded abelian group $\mathcal{B}$ that does not depend on $G$.
(The group $\mathcal{B}$ is given by $\Bar{B}(0, Q(\Z), \eta_{Q})$
in the notation of \cite[(7-5)]{ML57}
and $B(\Z, Q(\Z), \Z)$ as a two-sided bar construction.)
Each term $\mathcal{B}_{n}$ of $\mathcal{B}$ is free.
The augmentation $M(G) \to G$ gives a resolution
$\dots \to M_{1}(G) \to M_{0}(G) \to G \to 0$ of $G$.

The above constructions are functorial in $G$.
Hence they extend to any sheaves over any sites.
That is, for any sheaf $G \in \Ab(S)$ on a site $S$,
we have sheaves and complexes of sheaves $\Z[G]$, $Q'(G)$, $Q(G)$, $M(G)$ in $\Ab(S)$
by the sheafifications of the presheaves $X \mapsto \Z[G(X)], Q'(G(X)), Q(G(X)), M(G(X))$, respectively.
The complex $M(G)$ is a resolution of $G$ since sheafification is exact.


\subsection{The underlying set complexes}
\label{0155}

We need to describe $R \Hom_{F^{\perar}_{\Zar}}$ from a connected group $G$ to a constant group $H$
in order to prove Proposition \ref{0026} in the next subsection.
We take Mac Lane's resolution of $G$.
When describing $R \Hom_{F^{\perar}_{\Zar}}$ from $M(G)$ to $H$,
what matter are only the underlying sets of self products $G^{2^{n}}$ of $G$
and various maps between them.
Thus the problem is to describe the ``underlying set version'' of $M(G)$
or, better, of $Q(G)$.

Let $\Spec F_{\Zar}$ be the category of $F$-algebras (or $F$-schemes)
endowed with the Zariski topology.
Let $\Set$ be the topos of sets.
The functor sending an $F$-scheme $X$ to its underlying set $|X|$
defines a premorphism of sites $f \colon \Set \to \Spec F_{\Zar}$.
The pullback $f^{\ast} \Z[X]$ of the free abelian sheaf generated by a representable sheaf $X$
is given by the free abelian group $\Z[|X|]$ generated by $|X|$.
For a (commutative) group scheme $G$ over $F$,
define $|Q'|(G) = f^{\ast} Q'(G)$, $|Q|(G) = f^{\ast} Q(G)$ and $|M|(G) = f^{\ast} M(G)$.
They are complexes of abelian groups.
The $n$-th term $|Q'|_{n}(G)$ of the complex $|Q'|(G)$ is given by $\Z[|G^{2^{n}}|]$.
Here is the key technical result:

\begin{Prop} \label{0020}
	Let $G$ be an integral group scheme over $F$.
	Then $|Q|(G)$ is an exact complex.
\end{Prop}

This in particular claims that
$\partial \colon |Q|_{1}(G) \to |Q|_{0}(G)$ is surjective
(no augmentation considered).

Below we will prove the proposition.
Here is the idea of the proof.
Below we identify a point of a scheme and the corresponding irreducible closed subset
or the integral closed subscheme.
For an irreducible closed $X \subset G$,
consider $V'(X) := G \times_{F} X \subset G^{2}$ as an element of $|Q'|_{1}(G)$.
Since the group operation map $G \times G \to G$ restricts to a surjection
$G \times X \onto G$, we have $P_{1}(X) = G$.
Therefore
	\[
			\partial V'(X)
		=
			R_{1}(X) + S_{1}(X) - P_{1}(X)
		=
			G + X - G
		=
			X.
	\]
Thus any element of $|Q'_{0}|(G)$ is a boundary.
For any $n \ge 0$ and irreducible closed $X \subset G^{2^{n}}$,
we can similarly define $V'(X) = G^{2^{n}} \times X \subset G^{2^{n + 1}}$
and show $\partial V'(X) + V' \partial(X) = X$,
thus proving $|Q'|(G)$ is an exact complex.
The problem is that $V'$ does not preserve slabs and diagonals
and hence does not induce a map $|Q_{n}|(G) \to |Q_{n + 1}|(G)$.
For example, if $X = G \times 0 \subset G^{2}$ is a $1$-slab,
then
	\[
			V'(X)
		=
			\begin{pmatrix}
				G & G \\
				G & 0
			\end{pmatrix}
		\subset
			G^{4},
	\]
which is neither a slab or a diagonal.
To overcome this, we modify $V'(X) = G^{2^{n}} \times X$
so that if $X$ has a zero entry somewhere
(meaning the projection to that entry is the set $\{0\}$),
then we replace the corresponding entry of $G^{2^{n}}$ by zero (see \eqref{0479}).
This new map $V$ gives a well-defined map $|Q_{n}|(G) \to |Q_{n + 1}|(G)$.
Unfortunately, it no longer satisfies $\partial V(X) + V \partial(X) = X$.
Nonetheless, we can prove a slightly weaker property,
namely, $\varphi := \partial V + V \partial - \id$
is nilpotent in each degree (Proposition \ref{0019}),
since applying $\varphi$ more and more makes $X$ closer and closer to the regular shape
where entries are either zero or $G$ (Proposition \ref{0018}).
This nilpotence is enough to conclude the exactness of $|Q|(G)$.

Now we start proving the proposition.
For any $n \ge 0$ and any $(\varepsilon(1), \dots, \varepsilon(n)) \in 2^{n}$,
let $p_{(\varepsilon(1), \dots, \varepsilon(n))} \colon G^{2^{n}} \onto G$
be the projection onto the product factor
corresponding to $(\varepsilon(1), \dots, \varepsilon(n))$.
For an irreducible closed $X \subset G^{2^{n}}$,
define an integer $r_{X}(\varepsilon(1), \dots, \varepsilon(n)) \in \{0, 1\}$ to be
$0$ if $p_{(\varepsilon(1), \dots, \varepsilon(n))}(X) = \{0\}$
and $1$ otherwise.
Define
	\begin{equation} \label{0479}
			V(X)
		=
				\biggl(
					\prod_{(\varepsilon(1), \dots, \varepsilon(n))}
						G^{r_{X}(\varepsilon(1), \dots, \varepsilon(n))}
				\biggr)
			\times
				X
		\subset
			G^{2^{n}} \times G^{2^{n}}
	\end{equation}
where products are fiber products over $F$.
We view $V(X)$ as a closed subscheme of $G^{2^{n + 1}}$ via the assignment
	\begin{align*}
		&
					\bigl(
						t(0, \varepsilon(1), \dots, \varepsilon(n)),
						t(1, \varepsilon(1), \dots, \varepsilon(n))
					\bigr)_{\varepsilon(1), \dots, \varepsilon(n)}
				\in
					V(X)
		\\
		&	\leftrightarrow
					\bigl(
						t(\varepsilon(1), \dots, \varepsilon(n + 1))
					\bigr)_{\varepsilon(1), \dots, \varepsilon(n + 1)}
				\in
					G^{2^{n + 1}}.
	\end{align*}
This defines a map $V \colon |G^{2^{n}}| \to |G^{2^{n + 1}}|$,
which linearly extends to a homomorphism
$V \colon |Q'|_{n}(G) \to |Q'|_{n + 1}(G)$.
If $X$ is a slab or a diagonal, then so is $V(X)$.
Hence it induces a homomorphism
$V \colon |Q|_{n}(G) \to |Q|_{n + 1}(G)$.
Define an endomorphism $\varphi$ of the chain complex $|Q|(G)$ by
$\varphi = \partial V + V \partial - \id$.
The terms of $\varphi(X)$ are more ``regular'' than $X$
in the following sense:

\begin{Prop} \label{0011}
	Let $X \subset G^{2^{n}}$ be irreducible closed.
	Then $\varphi(X)$ is a $\Z$-linear combination
	of the following types of irreducible subsets of $G^{2^{n}}$:
		\begin{equation} \label{0012}
				\prod_{\varepsilon(1), \dots, \varepsilon(n - 1)}
					G^{r(\varepsilon(1), \dots, \varepsilon(n - 1))}
			\times
				\closure{P_{j}(X)},
		\end{equation}
	where the integers $r(\varepsilon(1), \dots, \varepsilon(n - 1))$ are either $0$ or $1$,
	the integer $j$ satisfies $1 \le j \le n$,
	and the scheme $\closure{P_{j}(X)}$ is the closure of $P_{j}(X)$ in $G^{2^{n - 1}}$,
	satisfying the following property:
	\begin{equation} \label{0481}
		\begin{minipage}[c]{300pt}
			if $r(\varepsilon(1), \dots, \varepsilon(n - 1)) = 0$,
			then $p_{\varepsilon(1), \dots, \varepsilon(n - 1)}(P_{j}(X)) = \{0\}$.
		\end{minipage}
	\end{equation}
\end{Prop}

\begin{proof}
	We have
		$
				X
			\subset
				\prod_{\varepsilon(1), \dots, \varepsilon(n)}
					G^{r_{X}(\varepsilon(1), \dots, \varepsilon(n))}
		$
	as a closed subscheme of $G^{2^{n}}$
	by the definition of $r_{X}(\varepsilon(1), \dots, \varepsilon(n))$.
	Hence the group operation map
	$G^{2^{n}} \times G^{2^{n}} \to G^{2^{n}}$
	restricts to a surjection
		$
				V(X)
			\onto
				\prod_{\varepsilon(1), \dots, \varepsilon(n)}
					G^{r_{X}(\varepsilon(1), \dots, \varepsilon(n))}
		$.
	This means that
		$
				P_{1}(V(X))
			=
				\prod_{\varepsilon(1), \dots, \varepsilon(n)}
					G^{r_{X}(\varepsilon(1), \dots, \varepsilon(n))}
		$.
	Therefore
		\begin{align*}
		&		\closure{P_{1}(V(X))} - \closure{R_{1}(V(X))} - \closure{S_{1}(V(X))}
		\\
		&	=
					\prod_{\varepsilon(1), \dots, \varepsilon(n)}
						G^{r_{X}(\varepsilon(1), \dots, \varepsilon(n))}
				-
					\prod_{\varepsilon(1), \dots, \varepsilon(n)}
						G^{r_{X}(\varepsilon(1), \dots, \varepsilon(n))}
				+
					X
		\\
		&	=
				X.
		\end{align*}
	
	Let $1 \le j \le n$.
	For each $\varepsilon(1), \dots, \varepsilon(n - 1) \in 2$,
	define $r_{j}(\varepsilon(1), \dots, \varepsilon(n - 1)))$ to be
	the maximum of the set
		\[
			\bigl\{
				r_{X}(\varepsilon(1), \dots, \varepsilon(j - 1), 0, \varepsilon(j), \dots, \varepsilon(n - 1)),\,
				r_{X}(\varepsilon(1), \dots, \varepsilon(j - 1), 1, \varepsilon(j), \dots, \varepsilon(n - 1))
			\bigr\}.
		\]
	Then the the group operation $G \times G \to G$
	restricts to a surjection
		\[
					G^{r_{X}(\varepsilon(1), \dots, \varepsilon(j - 1), 0, \varepsilon(j), \varepsilon(n - 1))}
				\times
					G^{r_{X}(\varepsilon(1), \dots, \varepsilon(j - 1), 1, \varepsilon(j), \varepsilon(n - 1))}
			\onto
				G^{r_{j}(\varepsilon(1), \dots, \varepsilon(n - 1))}.
		\]
	Hence
		\[
				\closure{P_{j + 1}(V(X))}
			=
					\prod_{\varepsilon(1), \dots, \varepsilon(n - 1)}
						G^{r_{j}(\varepsilon(1), \dots, \varepsilon(n - 1))}
				\times
					\closure{P_{j}(X)},
		\]
	which is of the form \eqref{0012} satisfying Property \eqref{0481}.
	We have
		\begin{gather*}
					\closure{R_{j + 1}(V(X))}
				=
						\prod_{\varepsilon(1), \dots, \varepsilon(n - 1)}
							G^{r_{X}(\varepsilon(1), \dots, \varepsilon(j - 1), 0, \varepsilon(j), \dots, \varepsilon(n - 1))}
					\times
						\closure{R_{j}(X)},
			\\
					\closure{S_{j + 1}(V(X))}
				=
						\prod_{\varepsilon(1), \dots, \varepsilon(n - 1)}
							G^{r_{X}(\varepsilon(1), \dots, \varepsilon(j - 1), 1, \varepsilon(j), \dots, \varepsilon(n - 1))}
					\times
						\closure{S_{j}(X)}.
		\end{gather*}
	
	On the other hand, we have
		\[
				V(\closure{P_{j}(X)})
			=
				\biggl(
					\prod_{\varepsilon(1), \dots, \varepsilon(n - 1)}
						G^{r_{\closure{P_{j}(X)}}(\varepsilon(1), \dots, \varepsilon(n - 1))}
				\biggr)
			\times
				\closure{P_{j}(X)}
		\]
	by definition, which is of the form \eqref{0012} satisfying Property \eqref{0481}.
	Also,
		\begin{gather*}
					r_{\closure{R_{j}(X)}}(\varepsilon(1), \dots, \varepsilon(n - 1))
				=
					r_{X}(\varepsilon(1), \dots, \varepsilon(j - 1), 0, \varepsilon(j), \dots, \varepsilon(n - 1)),
			\\
					r_{\closure{S_{j}(X)}}(\varepsilon(1), \dots, \varepsilon(n - 1))
				=
					r_{X}(\varepsilon(1), \dots, \varepsilon(j - 1), 1, \varepsilon(j), \dots, \varepsilon(n - 1)).
		\end{gather*}
	Combining these with the above presentation of
	$\closure{R_{j + 1}(V(X))}$ and $\closure{S_{j + 1}(V(X))}$, we have
		\[
				V(\closure{R_{j}(X)})
			=
				\closure{R_{j + 1}(V(X))},
			\quad
				V(\closure{S_{j}(X)})
			=
				\closure{S_{j + 1}(V(X))}.
		\]
	Now the result follows by the definition of $\partial$ for $Q'$.
\end{proof}

The map $\varphi$ is zero on a most regular $X$:

\begin{Prop} \label{0013}
	Let $r \in 2^{2^{n}}$.
	Then
		\[
				\varphi \biggl(
					\prod_{\varepsilon(1), \dots, \varepsilon(n)}
						G^{r(\varepsilon(1), \dots, \varepsilon(n))}
				\biggr)
			=
				0.
		\]
\end{Prop}

\begin{proof}
	Let $X = \prod_{\varepsilon(1), \dots, \varepsilon(n)}
	G^{r(\varepsilon(1), \dots, \varepsilon(n))}$.
	Then $r_{X} = r$.
	From the proof of Proposition \ref{0011},
	we need to show that
		\begin{align*}
			&		\max
					\bigl\{
						r(\varepsilon(1), \dots, \varepsilon(j - 1), 0, \varepsilon(j), \dots, \varepsilon(n - 1)),\,
						r(\varepsilon(1), \dots, \varepsilon(j - 1), 1, \varepsilon(j), \dots, \varepsilon(n - 1))
					\bigr\}
			\\
			&	=
					r_{\closure{P_{j}(X)}}(\varepsilon(1), \dots, \varepsilon(n - 1)).
		\end{align*}
	This follows from the definition of $r_{\closure{P_{j}(X)}}$.
\end{proof}

Now comes the induction step.
Let $n \ge 0$ and $1 \le i \le n$.
Consider the closed subschemes of $G^{2^{n} - 2^{n - i}} \times G^{2^{n - i}}$ of the form:
	\begin{equation} \label{0014}
			\Biggl(
				\prod_{\substack{
					\varepsilon(1), \dots, \varepsilon(n) \text{ s.t.} \\
					(\varepsilon(1), \dots, \varepsilon(i)) \ne (1, \dots, 1)
				}}
					G^{s(\varepsilon(1), \dots, \varepsilon(n))}
			\Biggr)
		\times
			Y,
	\end{equation}
where the integers $s(\varepsilon(1), \dots, \varepsilon(n))$ for $\varepsilon(1), \dots, \varepsilon(n) \in 2$
with $(\varepsilon(1), \dots, \varepsilon(i)) \ne (1, \dots, 1)$ are either $0$ or $1$,
the scheme $Y$ is an integral closed subscheme of $G^{2^{n - i}}$,
satisfying the following property:
	\begin{equation} \label{0480}
		\begin{minipage}[c]{300pt}
			For any $\varepsilon(1), \dots, \varepsilon(n) \in 2$,
			if exactly only one of $\varepsilon(1), \dots, \varepsilon(i)$ is $0$
			and $s(\varepsilon(1), \dots, \varepsilon(n)) = 0$,
			then $p_{\varepsilon(i + 1), \dots, \varepsilon(n)}(Y) = \{0\}$.
		\end{minipage}
	\end{equation}
We view them as closed subschemes of $G^{2^{n}}$ via the restriction of the assignment
	\[
			G^{2^{n} - 2^{n - i}} \times G^{2^{n - i}}
		\leftrightarrow
			G^{2^{n}}
	\]
given by
	\begin{align*}
		&		\Bigl(
					\bigl(
						t(\varepsilon(1), \dots, \varepsilon(n))
					\bigr)_{\substack{
						\varepsilon(1), \dots, \varepsilon(n) \text{ s.t.} \\
						(\varepsilon(1), \dots, \varepsilon(i)) \ne (1, \dots, 1)
					}},
					\bigl(
						t(1, \dots, 1, \varepsilon(i + 1), \dots, \varepsilon(n))
					\bigr)_{\varepsilon(i + 1), \dots, \varepsilon(n)}
				\Bigr)
		\\
		&	\leftrightarrow
				\bigl(
					t(\varepsilon(1), \dots, \varepsilon(n))
				\bigr)_{\varepsilon(i + 1), \dots, \varepsilon(n)}
	\end{align*}
Also consider the closed subschemes of $G^{2^{n}}$ of the form:
	\begin{equation} \label{0017}
		\prod_{\varepsilon(1), \dots, \varepsilon(n)}
			G^{s(\varepsilon(1), \dots, \varepsilon(n))},
	\end{equation}
where $s \in 2^{2^{n}}$.
Define $C_{n, i}$ to be the subgroup of $|Q|_{n}(G)$ generated by
the closed subschemes of $G^{2^{n}}$ of the form
\eqref{0014} and
\eqref{0017}.
Also define $C_{n, 0} = |Q|_{n}(G)$,
and define $C_{n, n + 1}$ to be the subgroup of $|Q|_{n}(G)$ generated by
the closed subschemes of $G^{2^{n}}$ of the form
\eqref{0017}.

\begin{Prop} \label{0018}
	Let $n \ge 0$ and $0 \le i \le n$.
	Then $\varphi \colon |Q|_{n}(G) \to |Q|_{n}(G)$ maps
	$C_{n, i}$ into $C_{n, i + 1}$.
\end{Prop}

\begin{proof}
	The statement for $i = 0$ follows from
	Proposition \ref{0011}.
	Assume $1 \le i \le n$.
	The subschemes of the form \eqref{0017} are killed by $\varphi$
	by Proposition \ref{0013}.
	Let $X \in C_{n, i}$ be a subscheme of $G^{2^{n}}$ of the form
	\eqref{0014}.
	Then $\varphi(X)$ is a $\Z$-linear combination of terms of the form
	\eqref{0012}
	by Proposition \ref{0011}.
	Let $X'$ be any such term of the form \eqref{0012}
	(in particular, $1 \le j \le n$).
	We want to show that $X' \in C_{n, i + 1}$.
	
	Assume first that $j \le i$.
	Define $s'(\varepsilon(1), \dots, \varepsilon(n - 1))$ to be the maximum of
		\[
			s(\varepsilon(1), \dots, \varepsilon(j - 1), 0, \varepsilon(j), \dots, \varepsilon(n - 1))
		\]
	and
		\[
			s(\varepsilon(1), \dots, \varepsilon(j - 1), 1, \varepsilon(j), \dots, \varepsilon(n - 1))
		\]
	if $(\varepsilon(1), \dots, \varepsilon(i - 1)) \ne (1, \dots, 1)$
	and
		\[
			s(1, \dots, 1, 0, 1, \dots, 1, \varepsilon(i), \dots, \varepsilon(n - 1))
		\]
	(where the $0$ is in the $j$-th position) if
	$(\varepsilon(1), \dots, \varepsilon(i - 1)) = (1, \dots, 1)$.
	Then
		\[
				\closure{P_{j}(X)}
			=
				\prod_{\varepsilon(1), \dots, \varepsilon(n - 1)}
					G^{s'(\varepsilon(1), \dots, \varepsilon(n - 1))}
		\]
	by Property \eqref{0480} for $s$.
	Therefore $X'$ is of the form \eqref{0017},
	thus $X' \in C_{n, i + 1}$.
	
	Assume next that $j \ge i + 1$.
	For $\varepsilon(2), \dots, \varepsilon(n) \in \{0, 1\}$ such that
	$(\varepsilon(2), \dots, \varepsilon(i + 1)) \ne (1, \dots, 1)$,
	define $s'(1, \varepsilon(2), \dots, \varepsilon(n))$ to be the maximum of
		\[
			s(\varepsilon(2), \dots, \varepsilon(j), 0, \varepsilon(j + 1), \dots, \varepsilon(n))
		\]
	and
		\[
			s(\varepsilon(2), \dots, \varepsilon(j), 1, \varepsilon(j + 1), \dots, \varepsilon(n)).
		\]
	Set $Y' = \closure{P_{j - i}(Y)}$.
	Then
		\begin{equation} \label{0458}
				\closure{P_{j}(X)}
			=
					\Biggl(
						\prod_{\substack{
							\varepsilon(2), \dots, \varepsilon(n) \text{ s.t.} \\
							(\varepsilon(2), \dots, \varepsilon(i + 1)) \ne (1, \dots, 1)
						}}
							G^{s'(1, \varepsilon(2), \dots, \varepsilon(n))}
					\Biggr)
				\times
					Y'.
		\end{equation}
	For arbitrary $\varepsilon(2), \dots, \varepsilon(n) \in \{0, 1\}$,
	set
		\[
				s'(0, \varepsilon(2), \dots, \varepsilon(n))
			=
				r(\varepsilon(2), \dots, \varepsilon(n)).
		\]
	Then
		\[
				X'
			=
					\Biggl(
						\prod_{\substack{
							\varepsilon(1), \dots, \varepsilon(n) \text{ s.t.} \\
							(\varepsilon(1), \dots, \varepsilon(i + 1)) \ne (1, \dots, 1)
						}}
							G^{s'(\varepsilon(1), \dots, \varepsilon(n))}
					\Biggr)
				\times
					Y',
		\]
	
	We want to show that $X'$ is of the form \eqref{0014}
	(with $i$ incremented by $1$) and hence $X' \in C_{n, i + 1}$.
	We will check that $s'$ satisfies Property \eqref{0480}.
	Let $\varepsilon(1), \dots, \varepsilon(n) \in 2$ be such that
	exactly only one of $\varepsilon(1), \dots, \varepsilon(i + 1)$ is $0$
	and $s'(\varepsilon(1), \dots, \varepsilon(n)) = 0$.
	First assume that $\varepsilon(1) = 0$
	(and hence $\varepsilon(2) = \dots = \varepsilon(i + 1) = 1$).
	Then $r(1, \dots, 1, \varepsilon(i + 2), \dots, \varepsilon(n)) = 0$.
	By Property \eqref{0481}, we have
	$p_{1, \dots, 1, \varepsilon(i + 2), \dots, \varepsilon(n)}(P_{j}(X)) = 0$.
	Hence by \eqref{0458}, we have
	$p_{\varepsilon(i + 2), \dots, \varepsilon(n)}(Y') = 0$.
	Thus Property \eqref{0480} for $s'$ is verified in this case.
	
	Next assume that $\varepsilon(1) = 1$
	(and hence exactly only one of $\varepsilon(2), \dots, \varepsilon(i + 1)$ is $0$). 
	Then
		\begin{align*}
			&
					s(\varepsilon(2), \dots, \varepsilon(j), 0, \varepsilon(j + 1), \dots, \varepsilon(n))
			\\
			&	=
					s(\varepsilon(2), \dots, \varepsilon(j), 1, \varepsilon(j + 1), \dots, \varepsilon(n))
				=
					0.
		\end{align*}
	Hence by Property \eqref{0480} for $s$, we have
		\begin{align*}
			&
					p_{\varepsilon(i + 2), \dots, \varepsilon(j), 0, \varepsilon(j + 1), \dots, \varepsilon(n)}(Y)
			\\
			&	=
					p_{\varepsilon(i + 2), \dots, \varepsilon(j), 1, \varepsilon(j + 1), \dots, \varepsilon(n)}(Y)
				=
					0.
		\end{align*}
	As $Y' = \closure{P_{j - i}(Y)}$, this implies
	$p_{\varepsilon(i + 2), \dots, \varepsilon(n)}(Y') = 0$.
	Thus Property \eqref{0480} for $s'$ is verified also in this case.
	
	Therefore $X'$ is of the form \eqref{0014} and hence $X' \in C_{n, i + 1}$,
	proving the proposition.
\end{proof}

\begin{Prop} \label{0019}
	For any $n \ge 0$, the $(n + 2)$-times iterate $\varphi^{n + 2}$ of $\varphi$
	on the degree $n$ part $|Q|_{n}(G)$ is zero.
\end{Prop}

\begin{proof}
	By Proposition \ref{0018},
	we have $\varphi^{n + 1}(|Q|_{n}(G)) \subset C_{n, n + 1}$.
	But $\varphi$ is zero on $C_{n, n + 1}$
	by Proposition \ref{0013}.
\end{proof}

\begin{proof}[Proof of Proposition \ref{0020}]
	This follows from Proposition \ref{0019}.
\end{proof}

\begin{Prop} \label{0021}
	The complex $|M|(G)$ is exact.
\end{Prop}

\begin{proof}
	Since $|M|(G) = B(|Q|(G), Q(\Z), \Z)$ and each term of $Q(\Z)$ is a free $\Z$-module,
	we have the (homologically graded) Eilenberg-Moore spectral sequence
		\[
				E_{i j}^{2}
			=
				\Tor_{i, j}^{HQ(\Z)}(H|Q|(G), \Z)
			\Longrightarrow
				H_{i + j} |M|(G)
		\]
	by \cite[Proposition 10.19 and the paragraph thereafter]{BMR14}.
	Since $H|Q|(G) = 0$ by Proposition \ref{0020},
	the result follows.
\end{proof}


\subsection{Derived Hom in the perfect artinian Zariski site}
\label{0157}

Here is a consequence of Proposition \ref{0021}:

\begin{Prop} \label{0022}
	Let $G \in \Ind \Pro' \Alg / F$ be connected.
	Let $H$ be a constant group over $F$.
	Then $R \sheafhom_{F^{\perar}_{\zar}}(G, H) = 0$.
\end{Prop}

\begin{proof}
	We may assume that $G$ is a connected pro-algebraic group over $F$.
	It is enough to show that
	$R \Hom_{F^{\perar}_{\zar}}(G, H) = 0$
	(since $F$ can be replaced by an arbitrary field in $F^{\perar}$).
	Let $M(G)$ be Mac Lane's resolution of $G$ in $\Ab(F^{\perar}_{\zar})$.
	Then $R \Hom_{F^{\perar}_{\zar}}(G, H) \cong R \Hom_{F^{\perar}_{\zar}}(M(G), H)$.
	For any $n \ge 0$, the sheaf $M_{n}(G)$ is a direct factor
	of a direct sum of sheaves of the form $\Z[G^{m}]$ for various $m \ge 0$.
	The sheaf $\Z[G^{m}] \in \Ab(F^{\perar}_{\zar})$ is a direct sum of
	sheaves of the form $\Z[\Spec F']$ for various fields $F' \in F^{\perar}$.
	We have
		\[
				\Ext_{F^{\perar}_{\zar}}^{j}(\Z[\Spec F'], H)
			\cong
				H^{j}(F'_{\zar}, H)
			=
				0
		\]
	for $j \ge 1$.
	Hence $R \Hom_{F^{\perar}_{\zar}}(M(G), H)$ is represented by
	the complex $\Hom_{F^{\perar}_{\zar}}(M(G), H)$.
	Since $H$ is constant, we have
		\[
				\Hom_{F^{\perar}_{\zar}}(\Z[\Spec F'], H)
			\cong
				\Hom_{\Ab}(\Z, H).
		\]
	Hence we have
		\[
				\Hom_{F^{\perar}_{\zar}}(\Z[G^{n}], H)
			\cong
				\Hom_{\Ab}(\Z[|G^{n}|], H).
		\]
	Hence
		\[
				\Hom_{F^{\perar}_{\zar}}(M(G), H)
			\cong
				\Hom_{\Ab}(|M|(G), H)
		\]
	as complexes in $\Ab$.
	As each term of $|M|(G)$ is free, this complex represents
	$R \Hom_{\Ab}(|M|(G), H)$,
	which is zero by Proposition \ref{0021}.
\end{proof}

Let 
	\[
			\varepsilon
		\colon
			\Spec F^{\perar}_{\et}
		\to
			\Spec F^{\perar}_{\zar}
	\]
be the morphism of sites defined by the identity functor.

\begin{Prop} \label{0023}
	Let $G \in \genby{\mathcal{W}_{F}}_{F^{\perar}_{\et}}$.
	Let $H$ be a constant group over $F$.
	Then
		\[
			R \sheafhom_{F^{\perar}_{\zar}}(R \varepsilon_{\ast} G, H) = 0.
		\]
\end{Prop}

\begin{proof}
	We may assume $G \in \mathcal{W}_{F}$.
	Let $G' \subset G$ be of Type \eqref{0005}
	such that $(G / G')^{0}$ is of Type \eqref{0006}
	and $\pi_{0}(G / G')$ is finite \'etale $p$-primary.
	Then for any $F' \in F^{\perar}$,
	we have $H^{m}(F'_{\et}, G') \cong H^{m}(F'_{\pro\et}, G') = 0$ for $m \ge 1$
	as in the proof of \cite[Proposition 2.4.2 (b)]{Suz20}.
	In particular, the exact sequence
	$0 \to G' \to G \to G / G' \to 0$ in $\Ind \Pro \Alg / F$
	remains exact in $\Ab(F^{\perar}_{\zar})$
	and $R \varepsilon_{\ast} G' \cong G'$.
	As the statement is true if $G = G'$ by
	Proposition \ref{0022},
	we may assume that $G' = 0$.
	The case where $G$ is of Type \eqref{0006}
	follows from Proposition \ref{0022}.
	Hence we may assume $G$ is finite \'etale $p$-primary.
	But then $G$ embeds into a connected unipotent quasi-algebraic group.
	Hence Proposition \ref{0022}
	again implies the result.
\end{proof}

The sheaf $\Lambda_{\infty}$ is a nice dualizing object over $\Spec F^{\perar}_{\et}$
by Proposition \ref{0497}.
It has the following Zariski counterpart.
Let $\Frob \colon W_{n} \to W_{n}$ be the Frobenius morphism
on the group scheme of $p$-typical Witt vectors of length $n$ over $F$.
Set $\wp = \Frob - 1$.
Define $\xi_{n} = W_{n} / \wp W_{n}$ in $\Ab(F^{\perar}_{\zar})$
and set $\xi_{\infty} = \dirlim_{n} \xi_{n}$ and $\xi = \xi_{1}$.
The exact sequence $0 \to \Lambda_{n} \to W_{n} \stackrel{\wp}{\to} W_{n} \to 0$
gives $R^{n} \varepsilon_{\ast} \Lambda_{n} = 0$ for $n \ge 2$
and $R^{1} \varepsilon_{\ast} \Lambda_{n} \cong \xi_{n}$.
Hence we have a canonical morphism
	\begin{equation} \label{0024}
			R \varepsilon_{\ast} \Lambda_{\infty}
		\to
			\xi_{\infty}[-1]
	\end{equation}
in $D(F^{\perar}_{\zar})$.
The sheaf $\xi_{\infty}$ is $p$-divisible with $p^{n}$-torsion part $\xi_{n}$ for any $n$.
The functor $R \sheafhom_{F^{\perar}_{\zar}}(\var, \xi_{\infty})$ is a Zariski topology version
of $R \sheafhom_{F^{\perar}_{\et}}(\var, \Lambda_{\infty})$
in the following sense:

\begin{Prop} \label{0025}
	Let $G \in \genby{\mathcal{W}_{F}}_{F^{\perar}_{\et}}$.
	Then the morphism
		\begin{align*}
					R \varepsilon_{\ast}
					R \sheafhom_{F^{\perar}_{\et}}(G, \Lambda_{\infty})
			&	\cong
					R \sheafhom_{F^{\perar}_{\zar}}(R \varepsilon_{\ast} G, R \varepsilon_{\ast} \Lambda_{\infty})
			\\
			&	\to
					R \sheafhom_{F^{\perar}_{\zar}}(R \varepsilon_{\ast} G, \xi_{\infty})[-1]
		\end{align*}
	is an isomorphism.
\end{Prop}

\begin{proof}
	This follows from
	Proposition \ref{0023}.
\end{proof}

Here is the promised result in this section:

\begin{Prop} \label{0026}
	Let $G, G' \in \genby{\mathcal{W}_{F}}_{F^{\perar}_{\et}}$.
	Let
		\[
				G \tensor^{L} G' \to \Lambda_{\infty}
		\]
	be a morphism in $D(F^{\perar}_{\et})$.
	Then it is a perfect pairing if and only if the induced morphism
		\[
				R \varepsilon_{\ast} G \tensor^{L} R \varepsilon_{\ast} G'
			\to
				R \varepsilon_{\ast} \Lambda_{\infty}
			\to
				\xi_{\infty}[-1]
		\]
	is a perfect pairing in $D(F^{\perar}_{\zar})$.
\end{Prop}

\begin{proof}
	By Proposition \ref{0025},
	we have
		\[
				R \sheafhom_{F^{\perar}_{\zar}}(R \varepsilon_{\ast} G, \xi_{\infty})[-1]
			\cong
				R \varepsilon_{\ast}
				R \sheafhom_{F^{\perar}_{\et}}(G, \Lambda_{\infty}).
		\]
	As $\varepsilon^{\ast} R \varepsilon_{\ast} \cong \id$, the result follows.
\end{proof}

We give two concrete duality statements in $\Spec F^{\perar}_{\zar}$.
In this paper, a Tate vector space over $F$ is
an ind-pro-object in finite-dimensional $F$-vector spaces
isomorphic to $F^{I} \oplus F^{\oplus J}$, where $I$ and $J$ are at most countable sets.
It can be functorially viewed as an ind-pro-algebraic group over $F$,
namely $\Ga^{I} \oplus \Ga^{\oplus J}$,
which we call a Tate vector group.

\begin{Prop} \label{0027}
	Let $V$ be a Tate vector space over $F$
	and $V'$ its dual.
	Let $G$ and $G'$ be the ind-pro-algebraic groups over $F$
	associated with $V$ and $V'$, respectively.
	Let $G \times G' \to \Ga$ be the evaluation pairing.
	Consider the natural pairing
		\[
			G \times G' \to \Ga \onto \xi \into \xi_{\infty}
		\]
	in $\Ab(F^{\perar}_{\zar})$.
	The induced morphism
		\[
			G \tensor^{L} G' \to \xi_{\infty}
		\]
	is a perfect pairing in $D(F^{\perar}_{\zar})$.
\end{Prop}

\begin{proof}
	The composite morphism
		\[
			G \otimes^{L} G' \to \Ga \to \Lambda[1] \to \Lambda_{\infty}[1]
		\]
	in $D(F^{\ind\rat}_{\pro\et})$ is a perfect pairing by Serre duality
	(\cite[Proposition 2.4.1]{Suz20}, \cite[Chapter III, Lemma 0.13 (c)]{Mil06}).
	By Proposition \ref{0010},
	this implies that the same morphism considered in $D(F^{\perar}_{\et})$
	is a perfect pairing.
	By Proposition \ref{0026}, we get the result.
\end{proof}

\begin{Prop} \label{0028}
	The multiplication pairing
		\[
				\Lambda \tensor^{L} \xi
			\to
				\xi
			\to
				\xi_{\infty}
		\]
	is a perfect pairing in $D(F^{\perar}_{\zar})$.
\end{Prop}

\begin{proof}
	The isomorphism
		\[
				\xi
			\isomto
				R \sheafhom_{F^{\perar}_{\zar}}(\Lambda, \xi_{\infty})
		\]
	is obvious.
	We show
		\[
				\Lambda
			\isomto
				R \sheafhom_{F^{\perar}_{\zar}}(\xi, \xi_{\infty}).
		\]
	We have isomorphisms
		\[
				R \varepsilon_{\ast} \Lambda
			\isomto
				R \sheafhom_{F^{\perar}_{\zar}}(
					R \varepsilon_{\ast} \Lambda,
					R \varepsilon_{\ast} \Lambda_{\infty}
				)
			\cong
				R \sheafhom_{F^{\perar}_{\zar}}(
					R \varepsilon_{\ast} \Lambda,
					\xi_{\infty}
				)[-1]
		\]
	by Proposition \ref{0025}.
	This isomorphism followed by the morphism
		\[
				R \sheafhom_{F^{\perar}_{\zar}}(R \varepsilon_{\ast} \Lambda, \xi_{\infty})[-1]
			\to
				R \sheafhom_{F^{\perar}_{\zar}}(\Lambda, \xi_{\infty})[-1]
			=
				\xi[-1]
		\]
	is the natural  morphism $R \varepsilon_{\ast} \Lambda \to \xi[-1]$.
	Hence $\Lambda \cong R \sheafhom_{F^{\perar}_{\zar}}(\xi, \xi_{\infty})$.
\end{proof}


\subsection{Equivariant structures with respect to base change}
\label{0518}

When a sheaf on a site is base-changed to the localization at some object,
it naturally acquires an equivariant structure (or a descent datum).
In this subsection, we give some basics about this topic,
paying attention to functoriality with respect to
premorphisms of sites without exact pullback functors.
This is needed to understand the behavior of the functor $\algebrize$ under base change
and the associated equivariant structures.
The proof of Statement \eqref{0130} of Theorem \ref{0128}
will be a straightforward consequence of the results of this subsection.

Let $S$ be a site and $\Sigma$ a group.
Assume an action of $\Sigma$ on the underlying category of $S$ is given.
This means that for any $\sigma \in \Sigma$,
we are given an equivalence of categories
$\sigma^{-1} \colon S \isomto S$, $X \mapsto \sigma^{-1} X$,
such that the composite $\tau^{-1} \sigma^{-1} X \colon S \isomto S$ is
naturally $(\sigma \tau)^{-1}$.
Assume moreover that $\sigma^{-1}$ for any $\sigma$ preserves covering families.
The equivalence of sites $S \isomto S$ induced by the functor $\sigma^{-1}$
is denoted by $\sigma$
(so $(\sigma_{\ast} C)(X) = C(\sigma^{-1} X)$ for a sheaf $C$ and $X \in S$).

Let $S_{\Sigma}$ be the category where the objects are the objects of $S$
and the set of morphisms $\Hom_{S_{\Sigma}}(X, Y)$ for objects $X, Y$
is $\bigsqcup_{\sigma \in \Sigma} \Hom_{S}(X, \sigma^{-1} Y)$.
A morphism is denoted by $(f, \sigma)$,
where $\sigma \in \Sigma$ and $f \colon X \to \sigma^{-1} Y$ in $S$.
The composite of $(f, \sigma) \colon X \to Y$ and $(g, \tau) \colon Y \to Z$
is $((\sigma^{-1} g) \compose f, \tau \sigma)$.
A family of morphisms $\{(f_{\lambda}, \sigma_{\lambda}) \colon X_{\lambda} \to X\}$
with fixed target is said to be a covering
if $\{\sigma_{\lambda} f_{\lambda} \colon \sigma_{\lambda} X_{\lambda} \to X\}$
is a covering in $S$ (where $\sigma_{\lambda}$ is the inverse of $\sigma_{\lambda}^{-1}$).
This defines a pretopology on $S_{\Sigma}$,
and we consider $S_{\Sigma}$ as a site with this pretopology.

The functor from the underlying category of $S$ to the underlying category of $S_{\Sigma}$
given by sending $X$ to $X$ and $f \colon X \to Y$ to $(f, \id) \colon X \to Y$
defines a premorphism of sites
	\begin{equation} \label{0519}
		\eta^{\Sigma} \colon S_{\Sigma} \to S.
	\end{equation}
The pushforward functor $\eta^{\Sigma}_{\ast}$ is called the forgetful functor
and denoted by $\For^{\Sigma}$.
Any sheaf on $S_{\Sigma}$ thus defines a sheaf on $S$.
Note that $(\id, \sigma) \colon \sigma^{-1} X \to X$ is an isomorphism
for any $X \in S$ and $\sigma \in \Sigma$.
Hence any sheaf on $S_{\Sigma}$ can be identified with its forgetful image $C \in \Set(S)$
together with a $\Sigma$-equivariant structure
$\varphi_{\sigma} \colon C \isomto \sigma_{\ast} C$
(satisfying $(\sigma_{\ast} \varphi_{\tau}) \compose \varphi_{\sigma} = \varphi_{\sigma \tau}$).
Thus $\Set(S_{\Sigma})$ can be identified with
the category of $\Sigma$-equivariant sheaves on $S$.
Here are basic properties of the forgetful functor:

\begin{Prop} \label{0520}
	The forgetful functor $\For^{\Sigma} \colon \Ab(S_{\Sigma}) \to \Ab(S)$ on abelian sheaves
	is exact and conservative.
	It sends the free abelian sheaf $\Z[X]$ on an object $X$
	to $\bigoplus_{\sigma \in \Sigma} \Z[\sigma^{-1} X]$.
\end{Prop}

\begin{proof}
	The exactness and the conservativity are obvious.
	For a representable presheaf $X$ of sets,
	the presheaf $\For^{\Sigma} X$ is given by
	$\bigsqcup_{\sigma \in \Sigma} \sigma^{-1} X$.
	The forgetful functor commutes with sheafification
	since the underlying functor of $\eta^{\Sigma}$ is cocontinuous
	and by \cite[Tag 00XM]{Sta21}.
	This gives the statement about $\Z[X]$.
\end{proof}

Let $S'$ be another site with $\Sigma$-action.
Let $f \colon S' \to S$ be a premorphism defined by a functor $f^{-1}$ on the underlying categories
that commutes with the $\Sigma$-actions
(that is, $f^{-1} \sigma^{-1} X \cong \sigma^{-1} f^{-1} X$
for all $\sigma \in \Sigma$ and $X \in S$ naturally).
It induces a premorphism $f^{\Sigma} \colon S'_{\Sigma} \to S_{\Sigma}$
with a commutative diagram
	\[
		\begin{CD}
				S'_{\Sigma}
			@> f^{\Sigma} >>
				S_{\Sigma}
			\\ @V \eta^{\Sigma} VV @VV \eta^{\Sigma} V \\
				S'
			@>> f >
				S,
		\end{CD}
	\]
where the vertical arrows are \eqref{0519} for $S'$ and for $S$.
In particular, we have a commutative diagram
	\begin{equation} \label{0523}
		\begin{CD}
				D(S'_{\Sigma})
			@> R f_{\ast}^{\Sigma} >>
				D(S_{\Sigma})
			\\ @V \For^{\Sigma} VV @VV \For^{\Sigma} V \\
				D(S')
			@>> R f_{\ast} >
				D(S).
		\end{CD}
	\end{equation}

A particular case of this situation is where $S$ is the terminal site
and $f^{-1}$ sends the unique object of $S$ to a terminal object of $S'$.
In this case, $\Set(S_{\Sigma})$ is just the category of $\Sigma$-sets.
Denote the functor $f_{\ast}^{\Sigma}$ by $\Gamma^{\Sigma}(S', \var)$.
It sends a $\Sigma$-equivariant sheaf $C'$ on $S'$
to its global section $\Gamma(S', C')$ with the natural $\Sigma$-action.
Hence $R f_{\ast}^{\Sigma} = R \Gamma^{\Sigma}(S', \var)$ is the functor
$R \Gamma(S', \var)$ upgraded as a functor valued in $D(\Mod{\Sigma})$,
where $\Mod{\Sigma}$ is the category of $\Sigma$-modules,
and $R^{q} f_{\ast}^{\Sigma} G' = H^{q} R \Gamma^{\Sigma}(S', G')$
for any $q$ and $G' \in D(S'_{\Sigma})$ is
$H^{q}(S', G')$ with the natural $\Sigma$-action
by \eqref{0523}.

The forgetful functor is also compatible with derived pullback:

\begin{Prop} \label{0522}
	The diagram
		\[
			\begin{CD}
					D(S_{\Sigma})
				@> L f^{\Sigma, \ast} >>
					D(S'_{\Sigma})
				\\ @V \For^{\Sigma} VV @VV \For^{\Sigma} V \\
					D(S)
				@>> L f^{\ast} >
					D(S')
			\end{CD}
		\]
	commutes (where $f^{\Sigma, \ast} = (f^{\Sigma})^{\ast}$).
\end{Prop}

\begin{proof}
	The diagram
		\[
			\begin{CD}
					\Ab(S_{\Sigma})
				@> f^{\Sigma, \ast} >>
					\Ab(S'_{\Sigma})
				\\ @V \For^{\Sigma} VV @VV \For^{\Sigma} V \\
					\Ab(S)
				@>> f^{\ast} >
					\Ab(S')
			\end{CD}
		\]
	commutes since
	the two functors $\Ab(S_{\Sigma}) \rightrightarrows \Ab(S')$ coming from the diagram
	both send $\Z[X]$ (with $X \in S$) to
		$
				\bigoplus_{\sigma \in \Sigma}
					\Z[\sigma^{-1} f^{-1} X]
			\cong
				\bigoplus_{\sigma \in \Sigma} \Z[f^{-1} \sigma^{-1} X]
		$
	by Proposition \ref{0520}.
	Since $L_{n} f^{\ast} \Z[X] = 0$ for all $n \ge 1$ by \cite[Lemma 3.7.2]{Suz22Duality},
	this derives.
\end{proof}

Next, let $S$ be a site and $\Sigma$ a group (but no action given).
Let $X_{0} \in S$.
Assume a group homomorphism $\Sigma \to \Aut_{S}(X_{0})$ is given
(so we have an isomorphism $\sigma \colon X_{0} \isomto X_{0}$ in $S$
for any $\sigma \in \Sigma$ in a compatible way).
For any $X / X_{0} \in S / X_{0}$ and $\sigma \in \Sigma$,
we have a new object $\sigma^{-1} X / X_{0} \in S / X_{0}$ defined by the commutative diagram
	\[
		\begin{CD}
				\sigma^{-1} X
			@> \sigma > \sim >
				X
			\\ @VVV @VVV \\
				X_{0}
			@> \sigma > \sim >
				X_{0}
		\end{CD}
	\]
in $S$.
Then $\Sigma$ acts on the underlying category of the localization $S / X_{0}$
by sending $X / X_{0}$ to $\sigma^{-1} X / X_{0}$ for $\sigma \in \Sigma$.
It preserves covering families.
Hence the previous paragraphs define a site $(S / X_{0})_{\Sigma}$
and a premorphism \eqref{0519}, which we now denote by
$\eta_{X_{0}}^{\Sigma} \colon (S / X_{0})_{\Sigma} \to S / X_{0}$.
Let $\For_{X_{0}}^{\Sigma} = \eta_{X_{0}, \ast}^{\Sigma}$ be the forgetful functor.

The functor from the underlying category of $(S / X_{0})_{\Sigma}$ to the underlying category of $S$
given by sending $X / X_{0}$ to $X$ and $(f, \sigma) \colon X \to Y$ to the composite
$X \stackrel{f}{\to} \sigma^{-1} Y \stackrel{\sigma}{\to} Y$
defines a premorphism
	\[
			\theta_{X_{0}}^{\Sigma}
		\colon
			S
		\to
			(S / X_{0})_{\Sigma}.
	\]
Its pushforward functor $\theta_{X_{0}, \ast}^{\Sigma}$
sends $C \in \Set(S)$ to its restriction $C|_{X_{0}} \in \Set(S / X_{0})$
with the natural $\Sigma$-equivariant structure.
Call it the $\Sigma$-restriction functor and denote it by $(\var)|_{X_{0}}^{\Sigma}$.
It is an exact functor.
Now we have two premorphisms
	\[
			S
		\stackrel{\theta_{X_{0}}^{\Sigma}}{\to}
			(S / X_{0})_{\Sigma}
		\stackrel{\eta_{X_{0}}^{\Sigma}}{\to}
			S / X_{0},
	\]
whose composite is the natural localization premorphism $S \to S / X_{0}$
(defined by the functor $X / X_{0} \mapsto X$).
In particular, we have $\For_{X_{0}}^{\Sigma} \compose |_{X_{0}}^{\Sigma} \cong |_{X_{0}}$.

Let $f \colon S' \to S$ be a premorphism from another site $S'$
defined by a functor $f^{-1}$ on the underlying categories.
Set $X_{0}' = f^{-1} X_{0}$.
We have the localization $f|_{X_{0}} \colon S' / X_{0}' \to S / X_{0}$ of $f$.
It induces a premorphism
$f|_{X_{0}}^{\Sigma} \colon (S' / X_{0}')_{\Sigma} \to (S / X_{0})_{\Sigma}$
as above.
The diagram
	\[
		\begin{CD}
				S'
			@> f >>
				S
			\\
			@V \theta_{X_{0}'}^{\Sigma} VV
			@VV \theta_{X_{0}}^{\Sigma} V
			\\
				(S' / X_{0}')_{\Sigma}
			@>> f|_{X_{0}}^{\Sigma} >
				(S / X_{0})_{\Sigma}
		\end{CD}
	\]
commutes.
In particular, we have a commutative diagram
	\begin{equation} \label{0525}
		\begin{CD}
				D(S')
			@> R f_{\ast} >>
				D(S)
			\\
			@V |_{X_{0}'}^{\Sigma} VV
			@VV |_{X_{0}}^{\Sigma} V
			\\
				D(S' / X_{0}')_{\Sigma}
			@>> R (f|_{X_{0}}^{\Sigma})_{\ast} >
				D(S / X_{0})_{\Sigma},
		\end{CD}
	\end{equation}
where $D(S / X_{0})_{\Sigma}$ means $D((S / X_{0})_{\Sigma})$.

Some restriction is needed for the compatibility with derived pullback:

\begin{Prop} \label{0524}
	The diagram
		\[
			\begin{CD}
					D(S)
				@> L f^{\ast} >>
					D(S')
				\\
				@V |_{X_{0}}^{\Sigma} VV
				@VV |_{X_{0}}^{\Sigma} V
				\\
					D(S' / X_{0}')_{\Sigma}
				@>> L (f|_{X_{0}}^{\Sigma})^{\ast} >
					D(S / X_{0})_{\Sigma}
			\end{CD}
		\]
	commutes on the full subcategory of $D(S)$
	consisting of $f$-compatible objects.
\end{Prop}

\begin{proof}
	Let $G \in D(S)$ be $f$-compatible.
	We want to show that the natural morphism
	$L (f|_{X_{0}}^{\Sigma})^{\ast}(G|_{X_{0}}^{\Sigma}) \to (L f^{\ast} G)|_{X_{0}}^{\Sigma}$
	is an isomorphism.
	Applying the forgetful functor $\For_{X_{0}}^{\Sigma}$ to this morphism
	yields the natural morphism
	$L (f|_{X_{0}})^{\ast}(G|_{X_{0}}) \to (L f^{\ast} G)|_{X_{0}}$
	by Proposition \ref{0522}.
	This latter morphism is an isomorphism by assumption.
	Since $\For_{X_{0}}^{\Sigma}$ is conservative
	by Proposition \ref{0520}, we get the result.
\end{proof}

Now we apply these results to the sites in \eqref{0460}.
For any field $F_{0} \in F^{\perar}$,
the localizations of the sites
	\[
			\Spec F^{\perf}_{\pro\fppf},
		\quad
			\Spec F^{\ind\rat}_{\pro\et},
		\quad
			\Spec F^{\perar}_{\et}
	\]
at $F_{0}$ are
	\[
			\Spec F^{\perf}_{0, \pro\fppf},
		\quad
			\Spec F^{\ind\rat}_{0, \pro\et},
		\quad
			\Spec F^{\perar}_{0, \et},
	\]
respectively.
See \cite[the paragraphs after Definition 2.1.3]{Suz22Duality} for related subtleties
around $\Spec F^{\ind\rat}_{\pro\et}$.
Let
	\[
			h_{F_{0}}
		\colon
			\Spec F^{\perf}_{0, \pro\fppf}
		\stackrel{f_{F_{0}}}{\to}
			\Spec F^{\ind\rat}_{0, \pro\et}
		\stackrel{g_{F_{0}}}{\to}
			\Spec F^{\perar}_{0, \et}
	\]
be the premorphisms defined by the inclusion functors
and set $\algebrize_{F_{0}} = R f_{F_{0}, \ast} L h_{F_{0}}^{\ast}$.
Let $\Sigma = \Aut(F_{0} / F)$.
Then we have the $\Sigma$-equivariant versions
	\[
			h_{F_{0}}^{\Sigma}
		\colon
			(\Spec F^{\perf}_{0, \pro\fppf})_{\Sigma}
		\stackrel{f_{F_{0}}^{\Sigma}}{\to}
			(\Spec F^{\ind\rat}_{0, \pro\et})_{\Sigma}
		\stackrel{g_{F_{0}}^{\Sigma}}{\to}
			(\Spec F^{\perar}_{0, \et})_{\Sigma}
	\]
and $\algebrize_{F_{0}}^{\Sigma} = R f_{F_{0}, \ast}^{\Sigma} L h_{F_{0}}^{\Sigma, \ast}$.
We have a commutative diagram
	\[
		\begin{CD}
				D(F^{\perar}_{0, \et})_{\Sigma}
			@> \algebrize_{F_{0}}^{\Sigma} >>
				D(F^{\perf}_{0, \pro\et})_{\Sigma}
			\\
			@V \For_{F_{0}}^{\Sigma} VV
			@VV \For_{F_{0}}^{\Sigma} V
			\\
				D(F^{\perar}_{0, \et})
			@>> \algebrize_{F_{0}} >
				D(F^{\perf}_{0, \pro\et})
		\end{CD}
	\]
by \eqref{0523} and Proposition \ref{0522}.

\begin{Prop} \label{0528}
	Let $G \in D(F^{\perar}_{\et})$ be $h$-compatible
	(for example, $G \in \genby{\mathcal{W}_{F}}_{F^{\perar}_{\et}}$).
	Then
		\begin{equation} \label{0526}
				(\algebrize G)|_{F_{0}}^{\Sigma}
			\cong
				\algebrize_{F_{0}}^{\Sigma}(G|_{F_{0}}^{\Sigma})
		\end{equation}
	in $D(F^{\ind\rat}_{0, \pro\et})_{\Sigma}$.
	In particular,
		\[
				H^{q} \bigl(
					(\algebrize G)|_{F_{0}}
				\bigr)
			\cong
				H^{q} \bigl(
					\algebrize_{F_{0}}(G|_{F_{0}})
				\bigr)
		\]
	as $\Sigma$-equivariant sheaves on $\Spec F^{\ind\rat}_{0, \pro\et}$ for all $q$,
		\[
				R \Gamma^{\Sigma} \bigl(
					F_{0},
					(\algebrize G)|_{F_{0}}^{\Sigma}
				\bigr)
			\cong
				R \Gamma^{\Sigma} \bigl(
					F_{0},
					\algebrize_{F_{0}}^{\Sigma}(G|_{F_{0}}^{\Sigma})
				\bigr)
		\]
	in $D(\Mod{\Sigma})$ and
		\[
				H^{q} \bigl(
					F_{0},
					\algebrize G)
				\bigr)
			\cong
				H^{q} \bigl(
					F_{0},
					\algebrize_{F_{0}}(G|_{F_{0}})
				\bigr)
		\]
	as $\Sigma$-modules for all $q$.
	
	If moreover the isomorphic objects \eqref{0526} are bounded below
	(for example, if $G \in \genby{\mathcal{W}_{F}}_{F^{\perar}_{\et}}$),
	then we have an isomorphism between the spectral sequence
		\[
				E_{2}^{i j}
			=
				H^{i} \bigl(
					F_{0},
					H^{j}(\algebrize G)
				\bigr)
			\Longrightarrow
				H^{i + j}(F_{0}, \algebrize G)
		\]
	of $\Sigma$-modules and the spectral sequence
		\[
				E_{2}^{i j}
			=
				H^{i} \Bigl(
					F_{0},
					H^{j} \bigl(
						\algebrize_{F_{0}}(G|_{F_{0}})
					\bigr)
				\Bigr)
			\Longrightarrow
				H^{i + j} \bigl(
					F_{0},
					\algebrize_{F_{0}}(G|_{F_{0}})
				\bigr)
		\]
	of $\Sigma$-modules compatible with the $E_{\infty}$-terms.
\end{Prop}

\begin{proof}
	Apply Proposition \ref{0524} for $h$ and \eqref{0525} for $f$.
\end{proof}


\section{Relative sites and cup product with support}
\label{0029}

Let $\Sch$ be the category of schemes.


\subsection{Relative sites}

The following defines something like ``$X \to \Spec F$''
when $X$ is not really an $F$-scheme
or, even if $X$ is an $F$-scheme,
when we want the base change $X \times_{F} F'$ to involve completions:

\begin{Def} \mbox{} \label{0312}
	\begin{enumerate}
		\item
			An \emph{$F^{\perar}$-scheme} is
			a functor $\alg{X}$ from the opposite of $F^{\perar}$ to $\Sch$
			satisfying the following two properties:
			\begin{enumerate}
				\item \label{0158}
					$\alg{X}$ commutes with finite coproducts.
				\item \label{0159}
					For any \'etale (resp.\ faithfully flat \'etale) morphism $F' \to F''$ in $F^{\perar}$,
					the morphism $\alg{X}(F'') \to \alg{X}(F')$ is
					finite \'etale (resp.\ finite faithfully flat \'etale),
					and we have $\alg{X}(F'' \tensor_{F'} F''') \isomto \alg{X}(F'') \times_{\alg{X}(F')} \alg{X}(F''')$
					for any other morphism $F' \to F'''$ in $F^{\perar}$.
			\end{enumerate}
			If $\alg{X}(F') = \Spec \alg{R}(F')$ is affine for all $F' \in F^{\perar}$,
			the functor $\alg{R}$ is called an \emph{$F^{\perar}$-algebra}.
		\item
			A morphism of $F^{\perar}$-schemes
			is a natural transformation $\alg{Y} \to \alg{X}$ of functors such that
			$\alg{Y}(F'') \isomto \alg{Y}(F') \times_{\alg{X}(F')} \alg{X}(F'')$
			for all \'etale morphisms $F' \to F''$ in $F^{\perar}$.
			A morphism of $F^{\perar}$-algebras $\alg{R} \to \alg{S}$ is similarly defined.
	\end{enumerate}
\end{Def}

The condition \eqref{0159} ensures that even though a scheme morphism
$\alg{X}(F) \to \Spec F$ does not necessarily exist or
$\alg{X}(F')$ is not necessarily $\alg{X}(F) \times_{F} F'$,
there is a certain morphism of sites
to $\Spec F^{\perar}_{\et}$ and $\Spec F^{\perar}_{\zar}$
(Proposition \ref{0034}).
We now define an ``\'etale site of $\alg{X}$''.

\begin{Def} \label{0160}
	An \emph{\'etale $\alg{X}$-scheme} is a pair $(X', F')$,
	where $F' \in F^{\perar}$ and $X'$ is an \'etale $\alg{X}(F')$-scheme.
	A morphism $(X'', F'') \to (X', F')$ of \'etale $\alg{X}$-schemes
	consists of an $F$-algebra homomorphism $F' \to F''$
	and a scheme morphism $X'' \to X'$ such that the diagram
		\[
			\begin{CD}
					X''
				@>>>
					X'
				\\ @VVV @VVV \\
					\alg{X}(F'')
				@>>>
					\alg{X}(F')
			\end{CD}
		\]
	commutes.
	Composition is defined in the obvious way.
\end{Def}

If $\alg{X}(F')$ is an affine scheme, $\Spec \alg{R}(F')$, for all $F' \in F^{\perar}$,
then we define the category of \'etale $\alg{R}$-algebras to be the opposite of
the full subcategory of the category of \'etale $\alg{X}$-schemes $(X', F')$ with $X'$ affine.
Its objects are pairs $(R', F')$ with $F' \in F^{\perar}$
and $R'$ an \'etale $\alg{R}(F')$-algebra.

For a scheme $X$,
let $X_{\tau}$ be either the small \'etale site
$X_{\et}$ (for $\tau = \et$) or
the small Nisnevich site $X_{\nis}$ (for $\tau = \nis$).
Let $\Spec F^{\perar}_{\tau}$ be $\Spec F^{\perar}_{\et}$ if $\tau = \et$
and $\Spec F^{\perar}_{\zar}$ if $\tau = \nis$.

\begin{Def} \mbox{}
	\begin{enumerate}
		\item
			A family of morphisms $\{(X'_{\lambda}, F'_{\lambda}) \to (X', F')\}$
			of \'etale $\alg{X}$-schemes with fixed target
			is said to be a $\tau$-covering
			if each $F' \to F'_{\lambda}$ is \'etale
			and $\{X'_{\lambda} \to X'\}$ is a covering for the site $X'_{\tau}$.
		\item
			This class of families defines a pretopology.
			Define $\alg{X}_{\tau}$ to be the category of \'etale $\alg{X}$-schemes
			endowed with the topology generated by this pretopology.
	\end{enumerate}
\end{Def}

Note that $\{F' \to F'_{\lambda}\}$ is not required to be a covering.
The topology of $\alg{X}_{\tau}$ is really about the $X'$ components
and not the $F'$ components.
This is needed for Propositions \ref{0030} and \ref{0031}.
Our default choice of the topology is $\et$.
The Nisnevich topology is used only in
Section \ref{0217} in an auxiliary manner.

The covering families for $\alg{X}_{\tau}$ are generated by
families $\{(X'_{\lambda}, F') \to (X', F')\}$ with $\{X'_{\lambda} \to X'\}$ a $\tau$-covering
and singleton families $\{(X'', F'') \to (X'', F')\}$ with $F' \to F''$ \'etale.

Define a morphism of sites
	\[
			\varepsilon
		\colon
			\alg{X}_{\et}
		\to
			\alg{X}_{\nis}
	\]
by the identity functor on the underlying categories.

\begin{Ex}
	We give some examples of sheaves on $\alg{X}_{\tau}$.
	Let $\Affine^{1}$ be the sheaf of rings on $\alg{X}_{\tau}$
	sending $(X', F') \mapsto \Gamma(X', \Order_{X'})$.
	Let $\Ga \in \Ab(\alg{X}_{\tau})$ be the underlying additive group of $\Affine^{1}$.
	Let $\Gm \in \Ab(\alg{X}_{\tau})$ be $(\Affine^{1})^{\times}$.
	For a quasi-coherent sheaf $M$ on $\alg{X}(F)$,
	the sheaf $(X', F') \mapsto \Gamma(X', M)$
	(where $M$ is pulled back to $X'$ as an $\Order_{X'}$-module)
	is	denoted the same symbol $M$.
	For $q \ge 0$, let $\Omega^{r} \in \Ab(\alg{X}_{\tau})$ be the sheaf
	$(X', F') \mapsto \Gamma(X', \Omega_{X' / \Z}^{r})$.
	The differential maps $d \colon \Omega_{X' / \Z}^{r} \to \Omega_{X' / \Z}^{r + 1}$
	form a morphism of sheaves $d \colon \Omega^{r} \to \Omega^{r + 1}$.
	Since \'etale and Nisnevich cohomology with coefficients in
	a quasi-coherent $M$ and $\Omega_{X' / \Z}$ agree
	(both isomorphic to Zariski cohomology),
	we have $R^{q} \varepsilon_{\ast} M = R^{q} \varepsilon_{\ast} \Omega^{r} = 0$
	for all $q \ge 1$.
	Let $B \Omega^{r} \subset Z \Omega^{r} \subset \Omega^{r}$ be
	the image and the kernel of $d$, respectively.
	
	Assume that $\alg{X}(F')$ is regular of characteristic $p$ for all $F' \in F^{\perar}$.
	Then $\nu_{n}(r)$ can be viewed as a sheaf
	$(X', F') \mapsto \Gamma(X', \nu_{n}(r))$
	on $\alg{X}_{\tau}$
	(where the right-hand side is the sheaf $\nu_{n}(r)$ on $X'_{\et}$).
	The Cartier operator gives a morphism of sheaves
	$C \colon Z \Omega^{r} \to \Omega^{r}$
	such that $\nu(r) := \nu_{1}(r) = \Ker(C - 1)$.
	If $\tau = \et$, the morphism $C - 1$ is surjective.
	If $\tau = \nis$, we set $\xi(r) = \Coker(C - 1)$ and $\xi = \xi(0)$.

	Assume that $\alg{X}(F')$ is regular of dimension $\le 1$
	such that $\alg{X}(F') \times_{\Spec \Z} \Spec \Z[1 / p]$ is dense for all $F'$.
	Then $\mathfrak{T}_{n}(r)$ can be viewed as a sheaf
	$(X', F') \mapsto \Gamma(X', \mathfrak{T}_{n}(r))$
	on $\alg{X}_{\tau}$.
\end{Ex}

\begin{Prop} \label{0030}
	Let $(X'', F'')$ be an \'etale $\alg{X}$-scheme.
	Let $F' \to F''$ be an \'etale morphism in $F^{\perar}$.
	Then the sheafification of the morphism $(X'', F'') \to (X'', F')$ of representable presheaves of sets
	in $\alg{X}_{\tau}$ is an isomorphism.
\end{Prop}

\begin{proof}
	Let $C \in \Set(\alg{X}_{\tau})$ be any sheaf.
	Since $(X'', F'') \to (X'', F')$ is a $\tau$-covering,
	the sequence
		\[
				C(X'', F')
			\to
				C(X'', F'')
			\rightrightarrows
				C(X'', F'' \tensor_{F'} F'')
		\]
	is an equalizer.
	Since $F'$ is \'etale over $F''$,
	the multiplication map $F'' \tensor_{F'} F'' \onto F''$ is a projection onto a direct factor.
	In particular, the morphism $(X'', F'') \to (X'', F'' \tensor_{F'} F'')$ is a $\tau$-covering.
	Hence the induced map
	$C(X'', F'' \tensor_{F'} F'') \to C(X'', F'')$ is injective.
	With this and the above sequence,
	we know that the map $C(X'', F') \to C(X'', F'')$ is bijective.
	As $C$ is arbitrary, this implies the result.
\end{proof}

The cohomology theory of $\alg{X}_{\tau}$ is the cohomology theories of $\alg{X}(F')_{\tau}$
bundled together:

\begin{Prop} \label{0031}
	Let $(X', F')$ be an \'etale $\alg{X}$-scheme.
	Let $\alg{X}_{\tau} / (X', F')$ be
	the localization of $\alg{X}_{\tau}$ at $(X', F')$.
	Consider the functor $X'_{1} \mapsto (X_{1}', F')$
	from the category of \'etale $X'$-schemes
	to the category of \'etale $\alg{X}$-schemes over $(X', F')$.
	This defines a morphism of sites
		\[
				h_{(X', F')}
			\colon
				\alg{X}_{\tau} / (X', F')
			\to
				X'_{\tau}.
		\]
	Its pushforward functor is exact.
\end{Prop}

\begin{proof}
	First, the category of \'etale $X'$-schemes has finite inverse limits
	and the functor $X'_{1} \mapsto (X'_{1}, F')$ commutes with these limits.
	Hence this functor defines a morphism of sites
	(that is, its pullback functor is exact).
	Let $C \to D$ be an epimorphism of sheaves of sets
	on $\alg{X}_{\tau} / (X', F')$.
	Let $X'_{1}$ be an \'etale $X'$-scheme
	and let $x \in (h_{(X', F'), \ast} D)(X'_{1}) = D(X'_{1}, F')$ be an arbitrary element.
	There exist a $\tau$-covering $\{(X'_{\lambda}, F'_{\lambda}) \to (X'_{1}, F')\}_{\lambda}$
	and elements $x_{\lambda} \in C(X'_{\lambda}, F'_{\lambda})$ such that
	the images of $x$ and $x_{\lambda}$ in $D(X'_{\lambda}, F'_{\lambda})$ agree.
	By Proposition \ref{0030},
	we may take $F'_{\lambda} = F'$ for all $\lambda$.
	Then $x_{\lambda} \in (h_{(X', F'), \ast} D)(X'_{\lambda})$.
	As $\{X'_{\lambda} \to X'_{1}\}_{\lambda}$ is a $\tau$-covering,
	this shows that $f_{\ast} C \to h_{(X', F'), \ast} D$ is an epimorphism.
\end{proof}

\begin{Def} \label{0161}
	In the notation of Proposition \ref{0031},
	for a sheaf $C$ on $\alg{X}_{\tau}$
	(of sets or abelian groups),
	we call $h_{(X', F'), \ast}(C|_{(X', F')})$ the \emph{restriction of $C$ to $X'_{\tau}$}
	and denote it by $C|_{X'_{\tau}}$.
\end{Def}

If $\alg{X}(F') = \Spec \alg{R}(F')$ is affine for all $F'$
and $X' = \Spec R'$ is affine,
then $C|_{X'_{\tau}}$ is also denoted by $C|_{R'_{\tau}}$.
The restrictions as defined above determine the sheaf $C$:

\begin{Prop} \label{0032}
	Consider the following set of data:
		\begin{enumerate}
			\item \label{0162}
				$C_{F'} \in \Set(\alg{X}(F')_{\tau})$ for each $F' \in F^{\perar}$.
			\item \label{0163}
				$\varphi_{F'' / F'} \in \Hom_{\alg{X}(F')_{\tau}}(C_{F'}, f_{F'' / F', \ast} C_{F''})$
				for each morphism $F' \to F''$ in $F^{\perar}$,
				where $f_{F'' / F'} \colon \alg{X}(F'')_{\tau} \to \alg{X}(F')_{\tau}$
				is the natural morphism.
		\end{enumerate}
	Assume the following conditions:
		\begin{enumerate}
			\item \label{0164}
				For any morphisms $F' \to F'' \to F'''$ in $F^{\perar}$, the composite
					\[
							C_{F'}
						\xrightarrow{\varphi_{F'' / F'}}
							f_{F'' / F', \ast} C_{F''}
						\xrightarrow{f_{F'' / F', \ast} \varphi_{F''' / F''}}
							f_{F''' / F', \ast} C_{F'''}
					\]
				is $\varphi_{F''' / F'}$.
			\item \label{0165}
				For any \'etale morphism $F' \to F''$ in $F^{\perar}$,
				the morphism
					\[
						f_{F'' / F'}^{\ast} C_{F'} \to C_{F''}
					\]
				induced by $\varphi_{F'' / F'}$ by adjunction is an isomorphism.
		\end{enumerate}
	Then there exists a unique sheaf $C \in \Set(\alg{X}_{\tau})$
	together with an isomorphism $C|_{\alg{X}(F')_{\tau}} \cong C_{F'}$ for each $F' \in F^{\perar}$
	such that $\varphi_{F'' / F'}$ corresponds to the morphism
		\[
				C|_{\alg{X}(F')_{\tau}}
			\to
				f_{F'' / F', \ast} (C|_{\alg{X}(F'')_{\tau}})
		\]
	given by the natural map $C(X', F') \to C(X' \times_{\alg{X}(F')} \alg{X}(F''), F'')$
	for any \'etale $\alg{X}(F')$-scheme $X'$.
	Conversely, any $C \in \Set(\alg{X}_{\tau})$ gives rise to
	a data $\{C_{F'}, \varphi_{F'' / F'}\}$ as above by setting
	$C_{F'} = C|_{\alg{X}(F')_{\tau}}$ and
	$\varphi_{F'' / F'}$ the natural morphism
	satisfying the above conditions.
\end{Prop}

\begin{proof}
	To define $C$, set $C(X', F') = C_{F'}(X')$ for any \'etale $\alg{X}$-scheme $(X', F')$.
	For a morphism $(X'', F'') \to (X', F')$,
	define a map $C(X', F') \to C(X'', F'')$ by the composite
		\[
				C_{F'}(X')
			\xrightarrow{\varphi_{F'' / F', X'}}
				C_{F''}(X' \times_{\alg{X}(F')} \alg{X}(F''))
			\to
				C_{F''}(X'').
		\]
	Then $C \in \Set(\alg{X}_{\tau})$
	since a covering $\{(X'_{i}, F'_{i})\}$ of an object $(X', F')$ decomposed as a covering
	$\{(X'_{i}, F')\}$ of $(X', F')$ and a covering
	$(X'_{i}, F'_{i})$ of $(X'_{i}, F')$ for each $i$.
	The converse direction is easy.
\end{proof}

The topos-theoretic points of $\alg{X}_{\tau}$ are described by
those of $\alg{X}(F')_{\tau}$:

\begin{Prop} \label{0166}
	Let $F' \in F^{\perar}$ be a field.
	Let $x'$ be a point of $\alg{X}(F')_{\tau}$
	(that is, the $\Spec$ of the henselian or strict henselian local ring of $\alg{X}(F')$ at a point
	if $\tau = \nis$ or $\tau = \et$, respectively).
	\begin{enumerate}
		\item \label{0167}
			For a sheaf $C \in \Set(\alg{X}_{\tau})$,
			consider the stalk $C_{x'} := (C|_{\alg{X}(F')_{\tau}})_{x'}$
			of $C|_{\alg{X}(F')_{\tau}} \in \Set(\alg{X}(F')_{\tau})$ at $x'$.
			The functor $C \mapsto C_{x'}$ gives a morphism of topoi
			$p_{(x', F')} \colon \Set \to \Set(\alg{X}_{\tau})$,
			which is a topos-theoretic point of $\Set(\alg{X}_{\tau})$.
		\item \label{0168}
			The family of points $\{p_{(x', F')}\}$ over all such pairs $(x', F')$ is conservative.
	\end{enumerate}
\end{Prop}

\begin{proof}
	The functor $C \mapsto C|_{\alg{X}(F')_{\tau}}$ commutes with all (small) direct and inverse limits.
	Hence the functor $C \mapsto C_{x'}$ commutes with all direct limits and all finite inverse limits.
	Therefore it gives a point.
	The conservativity follows from
	Proposition \ref{0032}.
\end{proof}

Now we define the ``structure morphism'' for $\alg{X}_{\tau}$.
Its relative duality theory, in the case of two-dimensional local rings,
will be the main theorem of this paper.

\begin{Prop} \label{0034}
	The (contravariant) functor $F' \mapsto (\alg{X}(F'), F')$ defines a morphism of sites
		\[
				\pi_{\alg{X}, \tau}
			\colon
				\alg{X}_{\tau}
			\to
				\Spec F^{\perar}_{\tau}.
		\]
\end{Prop}

\begin{proof}
	By assumption on $\alg{X}$, the stated functor defines a premorphism of sites.
	The functor has a left adjoint given by
	$(X', F') \mapsto F'$
	(as a functor from the category of \'etale $\alg{X}$-schemes to the opposite of $F^{\perar}$).
	Hence its pullback functor is exact.
\end{proof}

In particular, we have its right derived pushforward functor
	\[
			R \pi_{\alg{X}, \tau, \ast}
		\colon
			D(\alg{X}_{\tau})
		\to
			D(F^{\perar}_{\tau}).
	\]
Our default choice of $\tau$ is $\et$,
so $\pi_{\alg{X}, \et}$ is also simply denoted by $\pi_{\alg{X}}$.
If $\alg{X}(F') = \Spec \alg{R}(F')$ is affine for all $F'$,
then $\pi_{\alg{X}, \tau}$ is also denoted by $\pi_{\alg{R}, \tau}$.

We will see some functoriality of the above constructions.

\begin{Prop} \label{0527}
	Let $F_{0} \in F^{\perar}$ be a field.
	Let $\alg{X}_{F_{0}}$ be the restriction of $\alg{X}$ to $F_{0}^{\perar}$.
	Then $\alg{X}_{F_{0}}$ is an $F_{0}^{\perar}$-scheme.
	We have
		\[
				\alg{X}_{\tau} / (\alg{X}(F_{0}), F_{0})
			\cong
				\alg{X}_{F_{0}, \tau},
		\]
	with which $\pi_{\alg{X}, \tau}|_{F_{0}}$ and $\pi_{\alg{X}_{F_{0}}}$
	(both targeting $\Spec F^{\perar}_{0, \et}$) are compatible.
\end{Prop}

\begin{proof}
	Obvious.
\end{proof}

Let $\alg{Y} \to \alg{X}$ be a morphism of $F^{\perar}$-schemes.
For an \'etale $\alg{X}$-scheme $(X', F')$, define
	\[
			\alg{Y}(X', F')
		=
			X' \times_{\alg{X}(F')} \alg{Y}(F').
	\]
If $\alg{X}(F') = \Spec \alg{R}(F')$ and $\alg{Y}(F') = \Spec \alg{S}(F')$ are
affine for all $F'$,
then we also define
	\[
			\alg{S}(R', F')
		=
			R' \tensor_{\alg{R}(F')} \alg{S}(F')
	\]
for $(R', F') \in R_{\tau} / F^{\perar}$,
so that $\alg{Y}(X', F') = \Spec \alg{S}(R', F')$.

\begin{Prop} \label{0036}
	The functor $(X', F') \mapsto (\alg{Y}(X', F'), F')$
	defines a morphism of sites
		\[
				\pi_{\alg{Y} / \alg{X}, \tau}
			\colon
				\alg{Y}_{\tau}
			\to
				\alg{X}_{\tau}
		\]
	We have $\pi_{\alg{X}, \tau} \compose \pi_{\alg{Y} / \alg{X}, \tau} = \pi_{\alg{Y}, \tau}$.
\end{Prop}

\begin{proof}
	We show that the functor defines a premorphism of sites.
	Let $(X'', F'')$ be an \'etale $\alg{X}$-scheme
	and $F' \to F''$ an \'etale morphism in $F^{\perar}$.
	By the definition of morphisms of \'etale $\alg{X}$-schemes,
	we have $\alg{Y}(X'', F'') \isomto \alg{Y}(X'', F')$.
	Hence $(\alg{Y}(X'', F''), F'') \to (\alg{Y}(X'', F'), F')$ is a covering.
	For any other morphism $F' \to F'''$ in $F^{\perar}$,
	we have $\alg{Y}(X''', F'' \tensor_{F'} F''') \isomto \alg{Y}(X''', F''')$
	similarly.
	Therefore the diagram
		\[
			\begin{CD}
					\bigl(
						\alg{Y}(X''', F'' \tensor_{F'} F'''), F'' \tensor_{F'} F'''
					\bigr)
				@>>>
					(\alg{Y}(X''', F'''), F''')
				\\ @VVV @VVV \\
					(\alg{Y}(X'', F''), F'')
				@>>>
					(\alg{Y}(X'', F'), F')
			\end{CD}
		\]
	is cartesian.
	The case of a covering of the form $\{(X'_{\lambda}, F') \to (X', F')\}$
	(where $\{X'_{\lambda} \to X'\}$ is a $\tau$-covering) is easier.
	Thus the functor defines a premorphism of sites.
	
	The pullback functor for this premorphism sends a sheaf $C \in \Set(\alg{X}_{\tau})$
	to the sheafification of the presheaf
		\[
				(Y', F')
			\mapsto
				\dirlim_{X'}
					C(X', F')
		\]
	on $\alg{Y}_{\tau}$,
	where $X'$ runs through the diagrams
		\[
			\begin{CD}
					Y' @>>> X'
				\\ @VVV @VVV \\
					\alg{Y}(F') @>> \pi_{\alg{Y}(F') / \alg{X}(F')} > \alg{X}(F')
			\end{CD}
		\]
	with $X' \to \alg{X}(F')$ \'etale.
	These diagrams form (the opposite of) a filtered category,
	so its direct limit is exact.
	The second statement is obvious.
\end{proof}

If $\alg{X}(F') = \Spec \alg{R}(F')$ and $\alg{Y}(F') = \Spec \alg{S}(F')$ are
affine for all $F'$,
then $\pi_{\alg{Y} / \alg{X}, \tau}$ is also denoted by $\pi_{\alg{S} / \alg{R}, \tau}$.

The pullback functor $\pi_{\alg{Y} / \alg{X}, \tau}^{\ast}$ is nothing but
the pullback functors for the scheme morphisms
$\pi_{\alg{Y}(F') / \alg{X}(F')} \colon \alg{Y}(F') \to \alg{X}(F')$
bundled together:

\begin{Prop} \label{0311}
	For any $C \in \Set(\alg{X}_{\tau})$ and any $F' \in F^{\perar}$,
	the natural morphism
		\[
				\pi_{\alg{Y}(F') / \alg{X}(F')}^{\ast}(C|_{\alg{X}(F')_{\tau}})
			\to
				(\pi_{\alg{Y} / \alg{X}, \tau}^{\ast} C)|_{\alg{Y}(F')_{\tau}}
		\]
	is an isomorphism.
\end{Prop}

\begin{proof}
	This follows from the description of $\pi_{\alg{Y} / \alg{X}, \tau}^{\ast} C$
	in the proof of Proposition \ref{0036}.
\end{proof}

Set $X = \alg{X}(F)$.
Let $Y$ be an $X$-scheme.
For any $F' \in F^{\perar}$,
define $\alg{Y}(F') = Y \times_{X} \alg{X}(F')$.
Then the functor $\alg{Y}$ is an $F^{\perar}$-scheme.
Hence the natural morphism $\alg{Y} \to \alg{X}$ of $F^{\perar}$-schemes
defines a morphism of sites $\alg{Y}_{\tau} \to \alg{X}_{\tau}$ as above.

From $\alg{X}$, we have its henselian local rings as $F^{\perar}$-schemes in the following manner.
Assume that the morphism $\alg{X}(F') \to X$ is affine for all $F' \in F^{\perar}$.
Let $x \in X$ be a point (of the underlying set).
Then the scheme $\Spec \Order_{X, x} \times_{X} \alg{X}(F')$ is affine
and $\Spec \kappa(x) \times_{X} \alg{X}(F')$ is a closed subscheme of it
(where $\kappa(x)$ is the residue field of $\Order_{X, x}$).
Define $\Spec \alg{O}_{X, x}^{h}(F')$ to be the henselization of the pair
	\[
		\bigl(
			\Spec \Order_{X, x} \times_{X} \alg{X}(F'),
			\Spec \kappa(x) \times_{X} \alg{X}(F')
		\bigr).
	\]
Then the functor $\alg{O}_{X, x}^{h}$ is an $F^{\perar}$-algebra.
The natural morphism $\Spec \alg{O}_{X, x}^{h}(F') \to \alg{X}(F')$ defines
a morphism of $F^{\perar}$-schemes $\Spec \alg{O}_{X, x}^{h} \to \alg{X}$ and hence
a morphism of sites
	\[
			\pi_{\alg{O}_{X, x}^{h} / \alg{X}, \tau}
		\colon
			\Spec \alg{O}_{X, x, \tau}^{h}
		\to
			\alg{X}_{\tau}
	\]
as above.

We have a natural transformation
	\[
			R(
				\pi_{\alg{O}_{X, x}^{h}, \tau, \ast}
				\pi_{\alg{O}_{X, x}^{h} / \alg{X}, \tau}^{\ast}
			)
		\to
			R(\pi_{\alg{O}_{X, x}^{h}, \tau, \ast})
			\pi_{\alg{O}_{X, x}^{h} / \alg{X}, \tau}^{\ast}
		\colon
			D(\alg{X}_{\tau})
		\to
			D(F^{\perar}_{\tau}).
	\]
This is an isomorphism on $D^{+}(\alg{X}_{\tau})$ by the following:

\begin{Prop} \label{0314}
	The functor
		$
				\pi_{\alg{O}_{X, x}^{h} / \alg{X}, \tau}^{\ast}
			\colon
				\Ab(\alg{X}_{\tau})
			\to
				\Ab(\alg{O}_{X, x, \tau}^{h})
		$
	sends acyclic sheaves to acyclic sheaves.
\end{Prop}

\begin{proof}
	Let $G \in \Ab(\alg{X}_{\tau})$ be acyclic.
	Let $F' \in F^{\perar}$.
	By Proposition \ref{0031},
	it is enough to show that the restriction
		$
				(
					\pi_{\alg{O}_{X, x}^{h} / \alg{X}, \tau}^{\ast} G
				)|_{\alg{O}_{X, x}^{h}(F')_{\tau}}
			\in
				\Ab(\alg{O}_{X, x}^{h}(F')_{\tau})
		$
	is acyclic.
	Hence by Proposition \ref{0311},
	it is enough to show that
		$
				\pi_{\alg{O}_{X, x}^{h}(F') / \alg{X}(F')}^{\ast}(
					 G|_{\alg{X}(F')_{\tau}}
				)
			\in
				\Ab(\alg{O}_{X, x}^{h}(F')_{\tau})
		$
	is acyclic.
	The acyclicity of $G$ implies the acyclicity of
	$G|_{\alg{X}(F')_{\tau}} \in \Ab(\alg{X}(F')_{\tau})$.
	Hence
		$
			\pi_{\alg{O}_{X, x}^{h}(F') / \alg{X}(F')}^{\ast}(
				 G|_{\alg{X}(F')_{\tau}}
			)
		$
	is acyclic
	since $\Spec \alg{O}_{X, x}^{h}(F')$ is a filtered inverse limit of
	affine \'etale schemes over $\alg{X}(F')$
	and cohomology commutes with such limits \cite[Expos\'e VII, Corollaire 5.9]{AGV72b}.
\end{proof}


\subsection{Henselian neighborhoods}

We give two kinds of base change results
for closed immersions.
Let $\alg{A}$ be an $F^{\perar}$-algebra.
Assume that $\alg{A}(F')$ is a finite product of henselian local rings for all $F' \in F^{\perar}$
and that $\alg{A}(F') \to \alg{A}(F'')$ maps the Jacobson radical into the Jacobson radical
for all morphisms $F' \to F''$ in $F^{\perar}$.
Let $\alg{k}(F')$ be the quotient of $\alg{A}(F')$ by the Jacobson radical
(which is a finite product of fields).

\begin{Prop}
	The functor $\alg{k}$ is an $F^{\perar}$-algebra.
\end{Prop}

\begin{proof}
	It obviously preserves finite products.
	Let $F' \to F''$ be \'etale in $F^{\perar}$.
	Since $\alg{A}(F') \to \alg{A}(F'')$ is finite \'etale,
	the image of the Jacobson radical of $\alg{A}(F')$ generates the Jacobson radical of $\alg{A}(F'')$,
	and the induced map $\alg{k}(F') \to \alg{k}(F'')$ is finite \'etale.
	Let $F' \to F'''$ be another morphism in $F^{\perar}$.
	Then the natural map
	$\alg{A}(F'') \tensor_{\alg{A}(F')} \alg{A}(F''') \to \alg{A}(F'' \tensor_{F'} F''')$
	is an isomorphism.
	Since $\alg{A}(F') \to \alg{A}(F'')$ is \'etale,
	the ideal of $\alg{A}(F'') \tensor_{\alg{A}(F')} \alg{A}(F''')$
	generated by the image of the Jacobson radical of $\alg{A}(F''')$
	contains the image of the Jacobson radical of $\alg{A}(F'')$.
	On the other hand, the image of the Jacobson radical of $\alg{A}(F''')$ in $\alg{A}(F'' \tensor_{F'} F''')$
	generates the Jacobson radical
	since $F''' \to F'' \tensor_{F'} F'''$ is \'etale.
	Therefore, on the quotients, we have
	$\alg{k}(F'') \tensor_{\alg{k}(F')} \alg{k}(F''') \isomto \alg{k}(F'' \tensor_{F'} F''')$.
\end{proof}

The natural morphism $\alg{A} \onto \alg{k}$ of $F^{\perar}$-algebras then defines a morphism of sites
	\[
			i_{\alg{k}, \tau}
		:=
			\pi_{\alg{k} / \alg{A}, \tau}
		\colon
			\Spec \alg{k}_{\tau}
		\to
			\Spec \alg{A}_{\tau}.
	\]
On the other hand, for an \'etale $\alg{k}$-algebra $(k', F')$,
the $\alg{k}(F')$-algebra $k'$ is finite \'etale,
so it admits a canonical finite \'etale lifting to $\alg{A}(F')$.
Denote this lifting by $\alg{A}(k', F')$.

\begin{Prop}
	The functor $(k', F') \mapsto (\alg{A}(k', F'), F')$ defines a morphism of sites
		\[
				\pi_{\alg{A} / \alg{k}, \tau}
			\colon
				\Spec \alg{A}_{\tau}
			\to
				\Spec \alg{k}_{\tau}.
		\]
\end{Prop}

\begin{proof}
	Let $(k', F') \to (k'', F'')$ be a morphism of \'etale $\alg{k}$-algebras
	that appears in a covering family for the site $\Spec \alg{k}_{\tau}$,
	that is, that both $F' \to F''$ and $k' \to k''$ are \'etale.
	In the commutative diagram
		\[
			\begin{CD}
					\alg{A}(F')
				@>>>
					\alg{A}(F'')
				\\ @VVV @VVV \\
					\alg{A}(F', k')
				@>>>
					\alg{A}(F'', k''),
			\end{CD}
		\]
	all the morphisms but the lower horizontal one is finite \'etale.
	It follows that the lower horizontal map is finite \'etale.
	Hence $(\alg{A}(F', k'), F') \to (\alg{A}(F'', k''), F'')$ appears
	in a covering family for $\Spec \alg{A}_{\tau}$.
	
	The quotient of the map $\alg{A}(k', F') \to \alg{A}(k'', F'')$ by the Jacobson radicals
	is $k' \to k''$.
	It follows that if $k' \to k''$ is a $\tau$-covering,
	then so is $\alg{A}(k', F') \to \alg{A}(k'', F'')$.
	
	Let $(k', F') \to (k''', F''')$ be another morphism of \'etale $\alg{k}$-algebras.
	Then both $\alg{A}(k'', F'') \tensor_{\alg{A}(k', F')} \alg{A}(k''', F''')$
	and $\alg{A}(k'' \tensor_{k'} k''', F'' \tensor_{F'} F''')$
	are finite \'etale liftings of the finite \'etale $k'''$-algebra $k'' \tensor_{k'} k'''$
	to $\alg{A}(k''', F''')$.
	Hence the natural morphism between them is an isomorphism.
	Thus the stated functor defines a premorphism of sites.
	
	The pullback $\pi_{\alg{A} / \alg{k}, \tau}^{\ast \set} C$ of a sheaf $C \in \Set(\alg{k}_{\tau})$
	is given by the $\tau$-sheafification of the presheaf that sends
	an \'etale $\alg{A}$-algebra $(A', F')$ to the direct limit of the sets $C(k', F')$,
	where the index category consists of an \'etale $\alg{k}(F')$-algebra $k'$
	and an $\alg{A}(F')$-algebra homomorphism $\alg{A}(k', F') \to A'$.
	This index category is filtered.
	Since filtered direct limits are exact,
	we know that $\pi_{\alg{A} / \alg{k}, \tau}^{\ast \set}$ is exact.
	Thus $\pi_{\alg{A} / \alg{k}, \tau}$ is a morphism of sites.
\end{proof}

\begin{Prop} \label{0313}
	The right adjoint to the functor $(A', F') \mapsto (A' \tensor_{\alg{A}(F')} \alg{k}(F'), F')$
	from \'etale $\alg{A}$-algebras to \'etale $\alg{k}$-algebras
	is given by the functor $(k'', F'') \mapsto (\alg{A}(k'', F''), F'')$.
	In particular, we have $i_{\alg{k}, \tau}^{\ast} \cong \pi_{\alg{A} / \alg{k}, \tau, \ast}$,
	hence
		\[
					\pi_{\alg{k}, \tau, \ast} i_{\alg{k}, \tau}^{\ast}
				\cong
					\pi_{\alg{A}, \tau, \ast},
			\quad
					R \pi_{\alg{k}, \tau, \ast} i_{\alg{k}, \tau}^{\ast}
				\cong
					R \pi_{\alg{A}, \tau, \ast}.
		\]
\end{Prop}

\begin{proof}
	Morphisms $(A' \tensor_{\alg{A}(F')} \alg{k}(F'), F') \to (k'', F'')$
	bijectively corresponds to morphisms
	$A' \tensor_{\alg{A}(F')} \alg{k}(F'') \to k''$
	of finite \'etale $\alg{k}(F'')$-algebras.
	Morphisms $(A', F') \to (\alg{A}(k'', F''), F'')$
	bijectively corresponds to morphisms
	$A' \tensor_{\alg{A}(F')} \alg{A}(F'') \to \alg{A}(k'', F'')$
	of \'etale $\alg{A}(F'')$-algebras.
	The latter morphisms factor through
	the maximal quotient of $A' \tensor_{\alg{A}(F')} \alg{A}(F'')$
	finite over $\alg{A}(F'')$.
	Hence these morphisms correspond bijectively
	by the functoriality of finite \'etale liftings.
\end{proof}

In particular, $i_{\alg{k}, \tau}^{\ast}$ sends K-injectives to K-injectives.

A little more generally,
let $\alg{A}$ be an $F^{\perar}$-algebra
and $\alg{B}$ another $F^{\perar}$-algebra.
Let $\alg{A} \onto \alg{B}$ be a (section-wise) surjection of $F^{\perar}$-algebras
with kernel $\alg{I}$.
Assume that $(\alg{A}(F'), \alg{I}(F'))$ is a henselian pair for all $F' \in F^{\perar}$.
Set $i_{\alg{B}, \et} = \pi_{\alg{B} / \alg{A}, \et}$.
We have a natural morphism
	\[
			R \pi_{\alg{A}, \et, \ast}
		\to
			R \pi_{\alg{B}, \et, \ast} i_{\alg{B}, \et}^{\ast}
		\colon
			D(\alg{A}_{\et})
		\to
			D(F^{\perar}_{\et}).
	\]

\begin{Prop} \label{0341}
	For any $G \in D_{\tor}^{+}(\alg{A}_{\et})$,
	the above morphism applied to $G$,
		\[
				R \pi_{\alg{A}, \et, \ast} G
			\to
				R \pi_{\alg{B}, \et, \ast} i_{\alg{B}, \et}^{\ast} G,
		\]
	is an isomorphism.
\end{Prop}

\begin{proof}
	We may assume that $G$ is concentrated in degree zero.
	For $F' \in F^{\perar}$, we want to show that the morphism
		\[
				R \Gamma(\alg{A}(F'), G)
			\to
				R \Gamma(\alg{B}(F'), i_{\alg{B}, \et}^{\ast} G)
		\]
	is an isomorphism.
	But this is Gabber's affine analogue of proper base change \cite[Theorem 1]{Gab94}.
\end{proof}


\subsection{Shriek functors and cup product}
\label{0381}

We develop a theory of compact support cohomology and cup product
for relative sites.

Let $\alg{X}$ be an $F^{\perar}$-scheme.
Set $X = \alg{X}(F)$.
Let $i \colon Z \into X$ be a closed immersion.
Set $j \colon U = X \setminus Z \into X$.
Then $Z$ and $U$ define $F^{\perar}$-schemes $\alg{Z}$ and $\alg{U}$, respectively,
as seen after Proposition \ref{0311}.
The morphisms $i$ and $j$ define natural morphisms
$i \colon \alg{Z} \to \alg{X}$ and $j \colon \alg{U} \to \alg{X}$ of $F^{\perar}$-schemes,
which then define morphisms of sites
$i \colon \alg{Z}_{\et} \to \alg{X}_{\et}$ and $j \colon \alg{U}_{\et} \to \alg{X}_{\et}$.
For any $F' \in F^{\perar}$,
we have the usual shriek functors
$i_{F'}^{!} \colon \Ab(\alg{X}(F')_{\et}) \to \Ab(\alg{Z}(F')_{\et})$
and $j_{F', !} \colon \Ab(\alg{U}(F')_{\et}) \to \Ab(\alg{X}(F')_{\et})$
(see \cite[Tags 0F4Z and 0F59]{Sta21} for example).
For $G \in \Ab(\alg{X}(F')_{\et})$ and $H \in \Ab(\alg{U}(F')_{\et})$,
by Proposition \ref{0032},
the sheaves $i_{F'}^{!}(G|_{\alg{X}(F')})$ and $j_{F', !}(H|_{\alg{U}(F')})$
together with natural functoriality with respect to varying $F' \in F^{\perar}$
define objects of $\Ab(\alg{Z}_{\et})$ and $\Ab(\alg{X}_{\et})$, respectively.
We denote them by $i^{!} G$ and $j_{!} H$, respectively.
The obtained functors $i^{!} \colon \Ab(\alg{X}_{\et}) \to \Ab(\alg{Z}_{\et})$
and $j_{!} \colon \Ab(\alg{U}_{\et}) \to \Ab(\alg{X}_{\et})$
are right adjoint and left adjoint, respectively,
to $i_{\ast} \colon \Ab(\alg{Z}_{\et}) \to \Ab(\alg{X}_{\et})$
and $j^{\ast} \colon \Ab(\alg{X}_{\et}) \to \Ab(\alg{U}_{\et})$.
The functor $i^{!}$ is left exact and the functor $j_{!}$ is exact.
We have distinguished triangles
	\begin{equation} \label{0461}
				i_{\ast} R i^{!} G
			\to
				G
			\to
				R j_{\ast} j^{\ast} G,
		\quad
				j_{!} j^{\ast} G
			\to
				G
			\to
				i_{\ast} i^{\ast} G
	\end{equation}
in $D(\alg{X}_{\et})$ functorial in $G \in D(\alg{X}_{\et})$
(see \cite[Lemmas 6.1.11 and 6.1.16]{BS15} and \cite[0GKL]{Sta21} for example).
Hence the functors $R i^{!}$, $j_{!}$ are the right derived functors
of the functors of additive categories with translation (\cite[Definition 10.1.1]{KS06})
	\begin{gather*}
				[i^{\ast} \to i^{\ast} j_{\ast} j^{\ast}][-1]
			\colon
				\Ch(\alg{X}_{\et})
			\to
				\Ch(\alg{Z}_{\et}),
		\\
				[j_{\ast} \to i_{\ast} i^{\ast} j_{\ast}][-1]
			\colon
				\Ch(\alg{U}_{\et})
			\to
				\Ch(\alg{X}_{\et}),
	\end{gather*}
respectively.
See \cite[Sections 13.3, 14.3]{KS06} for derived functors of functors of additive categories with translation.

We define
	\[
			R \pi_{\alg{U}, !}
		:=
			R \pi_{\alg{X}, \ast}
			j_{!}
		\colon
			D(\alg{U}_{\et})
		\to
			D(F^{\perar}_{\et}).
	\]
This of course depends not only on $\alg{U}$
but also on $\alg{X}$ and the embedding $j \colon U \into X$.
We have distinguished triangles
	\begin{gather} \label{0370}
				R \pi_{\alg{Z}, \ast} R i^{!} G
			\to
				R \pi_{\alg{X}, \ast} G
			\to
				R \pi_{\alg{U}, \ast} j^{\ast} G,
		\\ \label{0371}
				R \pi_{\alg{U}, !} j^{\ast} G
			\to
				R \pi_{\alg{X}, \ast} G
			\to
				R \pi_{\alg{Z}, \ast} i^{\ast} G
	\end{gather}
in $D(F^{\perar}_{\et})$ functorial in $G \in D(\alg{X}_{\et})$.
The functor $R \pi_{\alg{Z}, \ast} R i^{!}$ is the right derived functor of
$[\pi_{\alg{X}, \ast} \to \pi_{\alg{U}, \ast} j^{\ast}][-1]$.

We have natural morphisms
	\begin{equation} \label{0364}
				i^{\ast} G \tensor^{L} R i^{!} H
			\to
				R i^{!}(G \tensor^{L} H),
		\quad
				R j_{\ast} G' \tensor^{L} j_{!} H'
			\isomto
				j_{!}(G' \tensor^{L} H')
	\end{equation}
in $D(\alg{Z}_{\et})$, $D(\alg{X}_{\et})$, respectively,
functorial in $G, H \in D(\alg{X}_{\et})$,
$G', H' \in D(\alg{U}_{\et})$, respectively.
Applying \eqref{0327} for $R \pi_{\alg{Z}, \ast}$ and $R \pi_{\alg{X}, \ast}$,
we obtain canonical morphisms
	\begin{gather} \label{0368}
					R \pi_{\alg{Z}, \ast} i^{\ast} G
				\tensor^{L}
					R \pi_{\alg{Z}, \ast} R i^{!} H
			\to
				R \pi_{\alg{Z}, \ast} R i^{!}(G \tensor^{L} H),
		\\ \label{0369}
					R \pi_{\alg{U}, \ast} G'
				\tensor^{L}
					R \pi_{\alg{U}, !} H'
			\to
				R \pi_{\alg{U}, !}(G' \tensor^{L} H')
	\end{gather}
in $D(F^{\perar}_{\et})$.

We will see how the localization triangles \eqref{0370} and \eqref{0371}
interact with the cup product morphisms \eqref{0368} and \eqref{0369}.
This will be useful to relate duality statements for $U$
and duality statements for a smaller open set $U'$
(see the proof of Proposition \ref{0448}
and the proof of Proposition \ref{0112} in Section \ref{0306}).

\begin{Prop} \label{0367}
	Let $G, H \in D(\alg{X}_{\et})$.
	Consider the morphisms
		\begin{align*}
					i_{\ast} R i^{!} R \sheafhom_{\alg{X}_{\et}}(G, H)
			&	\to
					i_{\ast} R \sheafhom_{\alg{Z}_{\et}}(i^{\ast} G, R i^{!} H)
			\\
			&	\to
					R \sheafhom_{\alg{X}_{\et}}(i_{\ast} i^{\ast} G, i_{\ast} R i^{!} H)
			\\
			&	\to
					R \sheafhom_{\alg{X}_{\et}}(i_{\ast} i^{\ast} G, H),
		\end{align*}
		\begin{align*}
					R j_{\ast} j^{\ast} R \sheafhom_{\alg{X}_{\et}}(G, H)
			&	\cong
					R j_{\ast} R \sheafhom_{\alg{U}_{\et}}(j^{\ast} G, j^{\ast} H)
			\\
			&	\to
					R \sheafhom_{\alg{X}_{\et}}(j_{!} j^{\ast} G, j_{!} j^{\ast} H)
			\\
			&	\to
					R \sheafhom_{\alg{X}_{\et}}(j_{!} j^{\ast} G, H)
		\end{align*}
	in $D(\alg{X}_{\et})$ obtained by applying \eqref{0364}.
	They form a morphism of distinguished triangles
		\[
			\begin{CD}
					i_{\ast} R i^{!} R \sheafhom_{\alg{X}_{\et}}(G, H)
				@>>>
					R \sheafhom_{\alg{X}_{\et}}(G, H)
				@>>>
					R j_{\ast} j^{\ast} R \sheafhom_{\alg{X}_{\et}}(G, H)
				\\ @VVV @| @VVV \\
					R \sheafhom_{\alg{X}_{\et}}(i_{\ast} i^{\ast} G, H)
				@>>>
					R \sheafhom_{\alg{X}_{\et}}(G, H)
				@>>>
					R \sheafhom_{\alg{X}_{\et}}(j_{!} j^{\ast} G, H),
			\end{CD}
		\]
	where the rows are the triangles \eqref{0461}.
\end{Prop}

\begin{proof}
	This is formal.
\end{proof}

Assume that $i \colon Z \into X$ factorizes as closed immersions
$i'' \colon Z \into Z'$ and $i' \colon Z' \into X$.
Set $U' = X \setminus Z'$.
Let $j'' \colon U' \into U$ and $j' = j \compose j'' \colon U' \into X$ be the inclusions.
Set $W = U \cap Z'$.
Let $i'|_{U} \colon W \into U$ and $j|_{Z'} \colon W \into Z'$ be the inclusions.
Note that
	\[
			R \pi_{\alg{U}, !} (i'|_{U})_{\ast}
		=
			R \pi_{\alg{X}, \ast} j_{!} (i'|_{U})_{\ast}
		=
			R \pi_{\alg{X}, \ast} i'_{\ast} (j|_{Z'})_{!}
		=
			R \pi_{\alg{Z'}, \ast} (j|_{Z'})_{!}
		=
			R \pi_{\alg{W}, !},
	\]
where $R \pi_{\alg{W}, !}$ is defined
with respect to the embedding $j|_{Z'} \colon W \into Z'$.
Hence applying $R \pi_{\alg{U}, !}$ to the distinguished triangle
	\[
			j''_{!} j''^{\ast} G
		\to
			G
		\to
			(i'|_{U})_{\ast} (i'|_{U})^{\ast} G
	\]
yields a distinguished triangle
	\begin{equation} \label{0372}
			R \pi_{\alg{U}', !} j''^{\ast} G
		\to
			R \pi_{\alg{U}, !} G
		\to
			R \pi_{\alg{W}, !} (i'|_{U})^{\ast} G
	\end{equation}
in $D(F^{\perar}_{\et})$ functorial in $G \in D(\alg{U}_{\et})$.
Also we have a distinguished triangle
	\begin{equation} \label{0373}
			R \pi_{\alg{W}, \ast} R (i'|_{U})^{!} G
		\to
			R \pi_{\alg{U}, \ast} G
		\to
			R \pi_{\alg{U}', \ast} j''^{\ast} G.
	\end{equation}

Applying Proposition \ref{0367} to
$\alg{U}'_{\et} \stackrel{j''}{\to} \alg{U}_{\et} \stackrel{i'|_{U}}{\gets} \alg{W}$,
we obtain a morphism of distinguished triangles
	\[
		\begin{CD}
				(i'|_{U})_{\ast} R (i'|_{U})^{!} [G, H]_{U}
			@>>>
				[G, H]_{U}
			@>>>
				R j''_{\ast} j''^{\ast} [G, H]_{U}
			\\ @VVV @| @VVV \\
				[(i'|_{U})_{\ast} (i'|_{U})^{\ast} G, H]_{U}
			@>>>
				[G, H]_{U}
			@>>>
				[j''_{!} j''^{\ast} G, H]_{U}.
		\end{CD}
	\]
in $D(\alg{U}_{\et})$ functorial in $G, H \in D(\alg{U}_{\et})$,
where we abbreviated $R \sheafhom_{\alg{U}_{\et}}$ as $[\var, \var]_{U}$.
Applying $R \pi_{\alg{U}, \ast}$ and using \eqref{0369},
we obtain:

\begin{Prop} \label{0430}
	We have a morphism of distinguished triangles
		\begin{equation}
			\begin{CD}
					R \pi_{\alg{W}, \ast}
					R (i'|_{U})^{!} [G, H]_{U}
				@>>>
					R \pi_{\alg{U}, \ast}
					[G, H]_{U}
				@>>>
					R \pi_{\alg{U}', \ast}
					[j''^{\ast} G, j''^{\ast} H]_{U'}
				\\ @VVV @VVV @VVV \\
					[
						R \pi_{\alg{W}, !} (i'|_{U})^{\ast} G,
						R \pi_{\alg{U}, !} H
					]_{F}
				@>>>
					[
						R \pi_{\alg{U}, !} G,
						R \pi_{\alg{U}, !} H
					]_{F}
				@>>>
					[
						R \pi_{\alg{U}', !} j''^{\ast} G,
						R \pi_{\alg{U}, !} H
					]_{F}
			\end{CD}
		\end{equation}
	in $D(F^{\perar}_{\et})$,
	where we abbreviated $R \sheafhom_{F^{\perar}_{\et}}$ as $[\var, \var]_{F}$.
	Here the upper triangle is \eqref{0373} and the lower \eqref{0372}.
	The vertical morphisms are \eqref{0368} and \eqref{0369}
	followed by the natural morphisms
	$R \pi_{\alg{W}, !} R (i'|_{U})^{!} H \to R \pi_{\alg{U}, !} H$
	and $R \pi_{\alg{U}', !} j''^{\ast} H \to R \pi_{\alg{U}, !} H$.
\end{Prop}

We will bring the above constructions to $D(F^{\ind\rat}_{\pro\et})$.
Define
	\begin{gather*}
				R \alg{\Gamma}(\alg{U}, \var)
			:=
				\algebrize
				R \pi_{\alg{U}, \ast},
			\quad
				R \alg{\Gamma}_{c}(\alg{U}, \var)
			:=
				\algebrize
				R \pi_{\alg{U}, !}
			\colon
		\\
				D(\alg{U}_{\et})
			\to
				D(F^{\ind\rat}_{\pro\et}).
	\end{gather*}
Set $\alg{H}^{q} = H^{q} R \alg{\Gamma}$ and $\alg{H}_{c}^{q} = H^{q} R \alg{\Gamma}_{c}$.
Let $G, H \in D(\alg{U}_{\et})$ be such that
$R \pi_{\alg{U}, \ast} G$ and $R \pi_{\alg{U}, !} H$ are $h$-acyclic and $h$-compatible
(which is the case if $R \pi_{\alg{U}, \ast} G$ and $R \pi_{\alg{U}, !} H$ are
in $\genby{\mathcal{W}_{F}}_{F^{\perar}_{\et}}$
by Proposition \ref{0153}).
Then for any $F' \in F^{\perar}$, we have
	\[
				R \Gamma(F', R \alg{\Gamma}(\alg{U}, G))
			\cong
				R \Gamma(\alg{U}(F'), G),
		\quad
				R \Gamma(F', R \alg{\Gamma}_{c}(\alg{U}, H))
			\cong
				R \Gamma_{c}(\alg{U}(F'), H)
	\]
in $D(\Ab)$ functorial in such $F'$, $G$ and $H$,
where $R \Gamma_{c}(\alg{U}(F'), \var)$ is
defined as the composite of $R \Gamma(\alg{U}(F'), \var)$
and the extension-by-zero functor $D(\alg{U}(F')_{\et}) \to D(\alg{X}(F')_{\et})$.
The morphism \eqref{0369} induces a morphism
	\[
				R \alg{\Gamma}(\alg{U}, G)
			\tensor^{L}
				R \alg{\Gamma}_{c}(\alg{U}, H)
		\to
			R \alg{\Gamma}_{c}(\alg{U}, G \tensor^{L} H)
	\]
in $D(F^{\ind\rat}_{\pro\et})$ functorial in such $G$ and $H$
by \eqref{0453}.
For any $F' \in F^{\perar}$, applying the functor $R \Gamma(F', \var)$ yields a morphism
	\[
				R \Gamma(\alg{U}(F'), G)
			\tensor^{L}
				R \Gamma_{c}(\alg{U}(F'), H)
		\to
			R \Gamma_{c}(\alg{U}(F'), G \tensor^{L} H),
	\]
which agrees with the usual cup product morphism.

We will see how the above constructions behave under base change
with the associated equivariant structures.
For a field $F_{0} \in F^{\perar}$ with $\Sigma = \Aut(F_{0} / F)$,
we have the restrictions $\alg{X}_{F_{0}}$ and $\alg{U}_{F_{0}}$
of $\alg{X}$ and $\alg{U}$, respectively, to $F_{0}^{\perar}$.
They are $F_{0}^{\perar}$-schemes by Proposition \ref{0527}.
Hence we have the restrictions of the above constructions
	\begin{gather*}
				\pi_{\alg{U}_{F_{0}}}
			\colon
				\alg{U}_{F_{0}, \et}
			\to
				\Spec F^{\perar}_{0, \et},
		\\
				R \pi_{\alg{U}_{F_{0}}, !}
			\colon
				D(\alg{U}_{F_{0}, \et})
			\to
				D(F^{\perar}_{0, \et})
	\end{gather*}
and their $\Sigma$-equivariant versions (Section \ref{0518})
	\begin{gather*}
				\pi_{\alg{U}_{F_{0}}}^{\Sigma}
			\colon
				(\alg{U}_{F_{0}, \et})_{\Sigma}
			\to
				(\Spec F^{\perar}_{0, \et})_{\Sigma},
		\\
				R \pi_{\alg{U}_{F_{0}}, !}^{\Sigma}
			\colon
				D(\alg{U}_{F_{0}, \et})_{\Sigma}
			\to
				D(F^{\perar}_{0, \et})_{\Sigma}.
	\end{gather*}
Define
	\begin{gather*}
				R \alg{\Gamma}^{\Sigma}(\alg{U}_{F_{0}}, \var)
			:=
				\algebrize_{F_{0}}^{\Sigma}
				R \pi_{\alg{U}_{F_{0}}, \ast}^{\Sigma},
			\quad
				R \alg{\Gamma}_{c}^{\Sigma}(\alg{U}_{F_{0}}, \var)
			:=
				\algebrize_{F_{0}}^{\Sigma}
				R \pi_{\alg{U}_{F_{0}}, !}^{\Sigma}
			\colon
		\\
				D(\alg{U}_{F_{0}, \et})_{\Sigma}
			\to
				D(F^{\ind\rat}_{0, \pro\et})_{\Sigma}.
	\end{gather*}
Denote the restriction functor
$D(\alg{U}_{\et}) \to D(\alg{U}_{F_{0}, \et})$ by $(\var)|_{\alg{U}_{F_{0}}}$
and its $\Sigma$-equivariant version
$D(\alg{U}_{\et}) \to D(\alg{U}_{F_{0}, \et})_{\Sigma}$ by $(\var)|_{\alg{U}_{F_{0}}}^{\Sigma}$.

\begin{Prop} \label{0530}
	Let $G \in D(\alg{U}_{\et})$.
	Assume that $R \pi_{\alg{U}, \ast} G$ and $R \pi_{\alg{U}, !} G$ are
	$h$-compatible
	(for example, that they are objects of $\genby{\mathcal{W}_{F}}_{F^{\perar}_{\et}}$).
	Then
		\begin{equation} \label{0529}
			\begin{gathered}
						R \alg{\Gamma}(\alg{U}, G)|_{F_{0}}^{\Sigma}
					\cong
						R \alg{\Gamma}^{\Sigma}(
							\alg{U}_{F_{0}},
							G|_{\alg{U}_{F_{0}}}^{\Sigma}
						),
				\\
						R \alg{\Gamma}_{c}(\alg{U}, G)|_{F_{0}}^{\Sigma}
					\cong
						R \alg{\Gamma}_{c}^{\Sigma}(
							\alg{U}_{F_{0}},
							G|_{\alg{U}_{F_{0}}}^{\Sigma}
						)
			\end{gathered}
		\end{equation}
	in $D(F^{\ind\rat}_{0, \pro\et})_{\Sigma}$.
	In particular,
		\begin{gather*}
					\alg{H}^{q}(\alg{U}, G)|_{F_{0}}
				\cong
					\alg{H}^{q} \bigl(
						\alg{U}_{F_{0}},
						G|_{\alg{U}_{F_{0}}}
					\bigr),
			\\
					\alg{H}_{c}^{q}(\alg{U}, G)|_{F_{0}}
				\cong
					\alg{H}_{c}^{q} \bigl(
						\alg{U}_{F_{0}},
						G|_{\alg{U}_{F_{0}}}
					\bigr)
		\end{gather*}
	as $\Sigma$-equivariant sheaves on $\Spec F^{\ind\rat}_{0, \pro\et}$ for all $q$,
		\begin{gather*}
					R \Gamma^{\Sigma} \bigl(
						F_{0},
						R \alg{\Gamma}(\alg{U}, G)|_{F_{0}}^{\Sigma}
					\bigr)
				\cong
					R \Gamma^{\Sigma} \bigl(
						F_{0},
						R \alg{\Gamma}^{\Sigma}(
							\alg{U}_{F_{0}},
							G|_{\alg{U}_{F_{0}}}^{\Sigma}
						)
					\bigr),
			\\
					R \Gamma^{\Sigma} \bigl(
						F_{0},
						R \alg{\Gamma}_{c}(\alg{U}, G)|_{F_{0}}^{\Sigma}
					\bigr)
				\cong
					R \Gamma^{\Sigma} \bigl(
						F_{0},
						R \alg{\Gamma}_{c}^{\Sigma}(
							\alg{U}_{F_{0}},
							G|_{\alg{U}_{F_{0}}}^{\Sigma}
						)
					\bigr)
		\end{gather*}
	in $D(\Mod{\Sigma})$ and
		\begin{gather*}
					H^{q} \bigl(
						F_{0},
						R \alg{\Gamma}(\alg{U}, G)
					\bigr)
				\cong
					H^{q} \bigl(
						F_{0},
						R \alg{\Gamma}(
							\alg{U}_{F_{0}},
							G|_{\alg{U}_{F_{0}}}
						)
					\bigr),
			\\
					H^{q} \bigl(
						F_{0},
						R \alg{\Gamma}_{c}(\alg{U}, G)
					\bigr)
				\cong
					H^{q} \bigl(
						F_{0},
						R \alg{\Gamma}_{c}(
							\alg{U}_{F_{0}},
							G|_{\alg{U}_{F_{0}}}
						)
					\bigr)
		\end{gather*}
	as $\Sigma$-modules for all $q$.
	
	If moreover the isomorphic objects \eqref{0529} are bounded below
	(for example, if $R \pi_{\alg{U}, \ast} G$ and $R \pi_{\alg{U}, !} G$ are
	objects of $\genby{\mathcal{W}_{F}}_{F^{\perar}_{\et}}$),
	then we have an isomorphism between the spectral sequence
		\[
				E_{2}^{i j}
			=
				H^{i} \bigl(
					F_{0},
					\alg{H}^{j}(\alg{U}, G)
				\bigr)
			\Longrightarrow
				H^{i + j} \bigl(
					F_{0},
					R \alg{\Gamma}(\alg{U}, G)
				\bigr)
		\]
	of $\Sigma$-modules and the spectral sequence
		\[
				E_{2}^{i j}
			=
				H^{i} \bigl(
					F_{0},
					\alg{H}^{j}(
						\alg{U}_{F_{0}},
						G|_{\alg{U}_{F_{0}}}
					)
				\bigr)
			\Longrightarrow
				H^{i + j} \bigl(
					F_{0},
					R \alg{\Gamma}(
						\alg{U}_{F_{0}},
						G|_{\alg{U}_{F_{0}}}
					)
				\bigr)
		\]
	of $\Sigma$-modules compatible with the $E_{\infty}$-terms,
	and also we have an isomorphism between the spectral sequence
		\[
				E_{2}^{i j}
			=
				H^{i} \bigl(
					F_{0},
					\alg{H}_{c}^{j}(\alg{U}, G)
				\bigr)
			\Longrightarrow
				H^{i + j} \bigl(
					F_{0},
					R \alg{\Gamma}_{c}(\alg{U}, G)
				\bigr)
		\]
	of $\Sigma$-modules and the spectral sequence
		\[
				E_{2}^{i j}
			=
				H^{i} \bigl(
					F_{0},
					\alg{H}_{c}^{j}(
						\alg{U}_{F_{0}},
						G|_{\alg{U}_{F_{0}}}
					)
				\bigr)
			\Longrightarrow
				H^{i + j} \bigl(
					F_{0},
					R \alg{\Gamma}_{c}(
						\alg{U}_{F_{0}},
						G|_{\alg{U}_{F_{0}}}
					)
				\bigr)
		\]
	of $\Sigma$-modules compatible with the $E_{\infty}$-terms.
\end{Prop}

\begin{proof}
	This follows from Proposition \ref{0528}.
\end{proof}


\subsection{The case of finite boundary}
\label{0436}

We continue the notation of Section \ref{0381}.
Assume moreover the following two conditions:
\begin{enumerate}
	\item
		$\alg{Z}(F')$ is a finite set of closed points of $\alg{X}(F')$ for all $F'$, and
	\item
		$\alg{X}(F') \to \alg{X}(F'')$ is affine for all morphisms $F' \to F''$ in $F^{\perar}$.
\end{enumerate}
We only need this case in this paper.
We will give more explicit descriptions of the constructions of the previous subsection.

For $x \in Z$, set $\alg{X}_{x}^{h} = \Spec \alg{O}_{X, x}^{h}$,
and $\alg{U}_{x}^{h} = \alg{X}_{x}^{h} \times_{\alg{X}} \alg{U}$
(which are well-defined by the second condition above).
The inclusion $x \into X$ defines an $F^{\perar}$-scheme $\alg{x}$
and morphisms $i_{x} \colon \alg{x} \to \alg{X}$
and $i_{x}^{h} \colon \alg{x} \to \alg{X}_{x}^{h}$.
We have a commutative diagram
	\[
		\begin{CD}
				\alg{U}_{x, \et}^{h}
			@> j_{x}^{h} >>
				\alg{X}_{x, \et}^{h}
			@< i_{x}^{h} <<
				\alg{x}_{\et}
			\\
			@V \pi_{\alg{U}_{x}^{h} / \alg{U}} VV
			@VV \pi_{\alg{X}_{x}^{h} / \alg{X}} V
			@|
			\\
				\alg{U}_{\et}
			@>> j >
				\alg{X}_{\et}
			@<< i_{x} <
				\alg{x}_{\et}
		\end{CD}
	\]
of morphisms of sites,
with the composite morphism
$\pi_{\alg{U}_{x}^{h} / \alg{X}} \colon \alg{U}_{x, \et}^{h} \to \alg{X}_{\et}$.
By Proposition \ref{0313}, we have
	\[
			R \pi_{\alg{X}_{x}, \ast}
		\cong
			R \pi_{\alg{x}, \ast}
			i_{x}^{h, \ast}.
	\]
Define
	\begin{equation} \label{0433}
			R \pi_{\alg{X}_{x}, !}
		:=
			R \pi_{\alg{x}, \ast}
			R i_{x}^{h, !}
		\colon
			D(\alg{X}_{x, \et})
		\to
			D(F^{\perar}_{\et}).
	\end{equation}
Then \eqref{0370} gives a distinguished triangle
	\begin{equation} \label{0434}
			R \pi_{\alg{X}_{x}, !} G
		\to
			R \pi_{\alg{X}_{x}, \ast} G
		\to
			R \pi_{\alg{U}_{x}, \ast} j_{x}^{h, \ast} G
	\end{equation}
and \eqref{0368} gives a morphism
	\begin{equation} \label{0443}
				R \pi_{\alg{X}_{x}, !} G
			\tensor^{L}
				R \pi_{\alg{X}_{x}, \ast} H
		\to
			R \pi_{\alg{X}_{x}, !}(G \tensor^{L} H).
	\end{equation}
The functor $R i^{!} G$ decomposes into local factors:

\begin{Prop} \label{0404}
	For $G \in D^{+}(\alg{X}_{\et})$, the natural morphism
		\[
				i_{\ast} R i^{!} G
			\to
				\bigoplus_{x \in Z}
					R \pi_{\alg{X}_{x}^{h} / \alg{X}, \ast}
					i_{x, \ast}^{h}
					R i_{x}^{h, !}
					\pi_{\alg{X}_{x}^{h} / \alg{X}}^{\ast} G
		\]
	in $D(\alg{X}_{\et})$ is an isomorphism.
	In particular, we have
		\[
				R \pi_{\alg{Z}, \ast} R i^{!} G
			\isomto
				\bigoplus_{x \in Z}
					R \pi_{\alg{X}_{x}^{h}, !}
					\pi_{\alg{X}_{x}^{h} / \alg{X}}^{\ast} G
		\]
	in $D(F^{\perar}_{\et})$.
\end{Prop}

\begin{proof}
	We may assume $G \in \Ab(\alg{X}_{\et})$.
	By restricting to $\alg{X}(F')_{\et}$ for all $F' \in F^{\perar}$,
	this reduces to the usual excision isomorphism for \'etale cohomology
	\cite[Chapter III, Corollary 1.28]{Mil80}.
\end{proof}

Therefore \eqref{0370} can be written as a distinguished triangle
	\[
			\bigoplus_{x \in Z}
				R \pi_{\alg{X}_{x}^{h}, !}
				\pi_{\alg{X}_{x}^{h} / \alg{X}}^{\ast} G
		\to
				R \pi_{\alg{X}, \ast} G
			\to
				R \pi_{\alg{U}, \ast} j^{\ast} G
	\]
in $D(F^{\perar}_{\et})$ functorial in $G \in D^{+}(\alg{X}_{\et})$.

\begin{Prop}
	The functor $i^{\ast} \colon \Ab(\alg{X}_{\et}) \to \Ab(\alg{Z}_{\et})$
	sends acyclic sheaves to acyclic sheaves.
\end{Prop}

\begin{proof}
	For each $x \in Z$, the pullback of $i^{\ast}$ to $\Ab(\alg{x}_{\et})$ can be written as
	$i_{x}^{\ast} = i_{x}^{h, \ast} \pi_{\alg{X}_{x}^{h} / \alg{X}}^{\ast}$.
	The functor $\pi_{\alg{X}_{x}^{h} / \alg{X}}^{\ast}$ sends acyclics to acyclics
	by Proposition \ref{0314}
	and the functor $i_{x}^{h, \ast} \cong \pi_{\alg{X}_{x}^{h} / \alg{x}, \ast}$ sends acyclics to acyclics
	by Proposition \ref{0313}.
\end{proof}

Therefore $R \pi_{\alg{U}, !}$ restricted to $D^{+}(\alg{U}_{\et})$ is the right derived functor of
	\[
		[
				\pi_{\alg{U}, \ast}
			\to
				\pi_{\alg{Z}, \ast} i^{\ast} j_{\ast}
		][-1].
	\]
We have
	\[
			\pi_{\alg{Z}, \ast} i^{\ast}
		\cong
			\bigoplus_{x \in Z}
				\pi_{\alg{x}, \ast}
				i_{x}^{\ast}
		\cong
			\bigoplus_{x \in Z}
				\pi_{\alg{x}, \ast}
				i_{x}^{h, \ast}
				\pi_{\alg{X}_{x}^{h} / \alg{X}}^{\ast}
		\cong
			\bigoplus_{x \in Z}
				\pi_{\alg{X}_{x}^{h}, \ast}
				\pi_{\alg{X}_{x}^{h} / \alg{X}}^{\ast},
	\]
where the last isomorphism is by Proposition \ref{0313}.
Hence, on $\Ab(\alg{U}_{\et})$, we have
	\[
			\pi_{\alg{Z}, \ast} i^{\ast} j_{\ast}
		\cong
			\bigoplus_{x \in Z}
				\pi_{\alg{X}_{x}^{h}, \ast}
				\pi_{\alg{X}_{x}^{h} / \alg{X}}^{\ast}
				j_{\ast}
		\cong
			\bigoplus_{x \in Z}
				\pi_{\alg{X}_{x}^{h}, \ast}
				j_{x, \ast}^{h}
				\pi_{\alg{U}_{x}^{h} / \alg{U}}^{\ast}
		\cong
			\bigoplus_{x \in Z}
				\pi_{\alg{U}_{x}^{h}, \ast}
				\pi_{\alg{U}_{x}^{h} / \alg{U}}^{\ast}.
	\]
Therefore we have the following explicit description
of $R \pi_{\alg{U}, !}$:

\begin{Prop} \label{0378}
	The functor $R \pi_{\alg{U}, !}$ restricted to $D^{+}(\alg{U}_{\et})$ is
	the right derived functor of
		\begin{equation} \label{0374}
			\left[
					\pi_{\alg{U}, \ast}
				\to
					\bigoplus_{x \in Z}
						\pi_{\alg{U}_{x}^{h}, \ast}
						\pi_{\alg{U}_{x}^{h} / \alg{U}}^{\ast}
			\right][-1].
		\end{equation}
\end{Prop}

We thus have a distinguished triangle
	\begin{equation} \label{0380}
			R \pi_{\alg{U}, !} G
		\to
			R \pi_{\alg{U}, \ast} G
		\to
			\bigoplus_{x \in Z}
				R \pi_{\alg{U}_{x}^{h}, \ast}
				\pi_{\alg{U}_{x}^{h} / \alg{U}}^{\ast} G
	\end{equation}
in $D(F^{\perar}_{\et})$ functorial in $G \in D^{+}(\alg{U}_{\et})$.
The presentation \eqref{0374} gives a priori another cup product morphism,
but it is the same as the previous cup product morphism
(under the bounded below condition):

\begin{Prop} \label{0376}
	Let $\pi_{\alg{U}, !}'$ be the functor \eqref{0374}.
	For $G, H \in \Ch^{+}(\alg{U}_{\et})$, let
		\[
				\pi_{\alg{U}, \ast}
				\sheafhom_{\alg{U}_{\et}}(G, H)
			\to
				\sheafhom_{F^{\perar}_{\et}}(\pi_{\alg{U}, !}' G, \pi_{\alg{U}, !}' H)
		\]
	be the morphism in $\Ch(F^{\perar}_{\et})$
	defined by the functoriality of $\pi_{\alg{U}, !}'$.
	Then the induced morphism
		\begin{equation} \label{0375}
				R \pi_{\alg{U}, \ast}
				R \sheafhom_{\alg{U}_{\et}}(G, H)
			\to
				R \sheafhom_{F^{\perar}_{\et}}(R \pi_{\alg{U}, !}' G, R \pi_{\alg{U}, !}' H)
		\end{equation}
	in $D(F^{\perar}_{\et})$ for $G, H \in D^{+}(\alg{U}_{\et})$
	agrees with the morphism
		\[
				R \pi_{\alg{U}, \ast}
				R \sheafhom_{\alg{U}_{\et}}(G, H)
			\to
				R \sheafhom_{F^{\perar}_{\et}}(R \pi_{\alg{U}, !} G, R \pi_{\alg{U}, !} H)
		\]
	coming from \eqref{0369} via the above isomorphism
	$R \pi_{\alg{U}, !} \cong R \pi_{\alg{U}, !}'$.
\end{Prop}

\begin{proof}
	This is formal.
\end{proof}

\begin{Lem} \label{0136}
	Let $S$ be a site.
	Let $G, H \in D^{+}(S)$.
	Then $G \tensor^{L} H \in D^{+}(S)$.
\end{Lem}

\begin{proof}
	This is obvious if $G, H \in \Ab(S)$
	since $\Tor_{\ge 2}^{\Ab} = 0$ and hence $\Tor_{\ge 2}^{\Ab(S)} = 0$.
	This implies the general case
	by a two-variable version of \cite[Tag 07K8 (4c)]{Sta21}.
\end{proof}

\begin{Prop} \label{0379}
	Let $G, H \in D^{+}(\alg{U}_{\et})$.
	Then the morphism
		\[
					R \pi_{\alg{U}, \ast}' G'
				\tensor^{L}
					R \pi_{\alg{U}, !}' H'
			\to
				R \pi_{\alg{U}, !}'(G' \tensor^{L} H')
		\]
	defined by \eqref{0375} agrees with the morphism \eqref{0369}
	via the above isomorphism
	$R \pi_{\alg{U}, !} \cong R \pi_{\alg{U}, !}'$.
\end{Prop}

\begin{proof}
	This follows from Proposition \ref{0376} and Lemma \ref{0136}.
\end{proof}


\subsection{Fibered sites and cup product}
\label{0325}

We want a version of \eqref{0374} where henselian localizations are replaced by completions.
But this is not obvious since the pullback functor for
the morphism $\Hat{U}_{x, \et} \to U_{\et}$ from the punctured completion
does not send acyclic sheaves to acyclic sheaves
and hence we cannot naturally take derived functors.
We deal with this problem by considering the category of tuples
$(G, (G_{x})_{x \in Z}, (\varphi_{x})_{x \in Z})$,
where $G$ is a sheaf on $U_{\et}$, 
$G_{x}$ is a sheaf on $\Hat{U}_{x, \et}$
and $\varphi_{x}$ is a morphism $\pi_{\Hat{U}_{x} / U}^{\ast} G \to G_{x}$.
An injective resolution in this category consists of
an injective resolution for $G$ in $\Ab(U_{\et})$ and
an injective resolution for $G_{x}$ in $\Ab(\Hat{U}_{x, \et})$ for each $x$
with a compatibility between them,
so that there is no problem in taking derived functors.
Actually the category of such tuples is the category of sheaves on the ``total site''
of a certain fibered site.
We will develop this machinery,
especially its compact support and cup product formalism.

Let $I$ be a finite poset (viewed as a category).
Let $\{S_{i}\}_{i \in I} = (\{S_{i}\}_{i \in I}, \{f_{j i}\}_{i \le j})$ be
a fibered site (\cite[Expos\'e VI, \S 7.2.1]{AGV72b}).
This means that $S_{i}$ for each $i \in I$ is a site and
$f_{j i} \colon S_{j} \to S_{i}$ for each $j \ge i$ is a morphism of sites
such that $f_{j i} \compose f_{k j} \cong f_{k i}$ naturally
(where we more precisely assume this on the level of the underlying functors
$f_{k j}^{-1} \compose f_{j i}^{-1} \cong f_{k i}^{-1}$).
Assume for simplicity that
$S_{i}$ contains an initial object $\emptyset$,
the empty family $\{\,\}$ is a covering of $\emptyset$
(so that the sheafification of $\emptyset$ is the initial sheaf)
and the underlying functor $f_{j i}^{-1}$ of $f_{j i}$ sends $\emptyset$ to $\emptyset$.
(Note that we can always add an initial object to a site without changing its topos.)

We denote the total site of $\{S_{i}\}_{i \in I}$
(\cite[Expos\'e VI, \S D\'efinition 7.4.1]{AGV72b})
by $(S_{i})_{i \in I}$.
More explicitly, an object is a tuple $(\{X_{i}\}_{i \in I}, \{\varphi_{i j}\}_{i \le j})$,
where $X_{i} \in S_{i}$ and
$\varphi_{i j} \colon f_{j i}^{-1} X_{i} \to X_{j}$ is a morphism in $S_{j}$
satisfying the cocycle condition:
$\varphi_{j k} \compose (f_{k j}^{-1} \varphi_{i j}) = \varphi_{i k}$
for $i \le j \le k$.
A morphism from $(\{X_{i}'\}, \{\varphi_{i j}'\})$ to $(\{X_{i}\}, \{\varphi_{i j}\})$
is a set of morphisms $X_{i}' \to X_{i}$ over $i \in I$ that is compatible with
$\{\varphi_{i j}'\}$ and $\{\varphi_{i j}\}$.
A covering of an object $(\{X_{i}\}, \{\varphi_{i j}\})$ is a family of morphisms
$(\{X_{\lambda i}\}, \{\varphi_{\lambda i j}\}) \to (\{X_{i}\}, \{\varphi_{i j}\})$
indexed by $\lambda \in \Lambda$
such that the family $X_{\lambda i} \to X_{i}$ indexed by $\lambda$ is a covering for all $i$.
(In \cite[Expos\'e VI, \S D\'efinition 7.4.1]{AGV72b},
the total site is defined with a slightly different underlying category,
but the above site defines an equivalent topos,
which can be seen from the description below.)

An object of the topos $\Set(S_{i})_{i \in I} = \Set((S_{i})_{i \in I})$ consists of a tuple
$F = (\{F_{i}\}_{i \in I}, \{\varphi_{i j}\}_{i \le j})$,
where $F_{i} \in \Set(S_{i})$ and
$\varphi_{i j} \colon f_{j i}^{\ast} F_{i} \to F_{j}$ is a morphism in $\Set(S_{j})$
(or equivalently, $\varphi_{i j} \colon F_{i} \to f_{j i, \ast} F_{j}$ in $\Set(S_{i})$)
satisfying the cocycle condition: $\varphi_{j k} \compose (f_{k j}^{\ast} \varphi_{i j}) = \varphi_{i k}$
for $i \le j \le k$.
A morphism from $(\{F_{i}'\}, \{\varphi_{i j}'\})$ to $(\{F_{i}\}, \{\varphi_{i j}\})$
is a set of morphisms $F_{i}' \to F_{i}$ over $i \in I$ that is compatible with
$\{\varphi_{i j}'\}$ and $\{\varphi_{i j}\}$.
We call $F_{i}$ the $i$-th projection of $F$.

For each $i' \in I$ and $X_{i'} \in S_{i'}$,
consider the object $q_{i'}^{-1} X_{i'} := (\{X_{i}\}, \{\varphi_{i j}\})$
of $(S_{i})_{i \in I}$ defined by
$X_{i} =  f_{i i'}^{-1} X_{i'}$ for $i \ge i'$ and $X_{i} = \emptyset$ else
and $\varphi_{i j}$ is the natural isomorphism
$f_{j i}^{-1} f_{i i'}^{-1} X_{i'} \isomto f_{j i'}^{-1} X_{i'}$ if $i \ge i'$
and $\emptyset$ else.
Then this functor $q_{i'}^{-1}$ defines a premorphism of sites
$q_{i'} \colon (S_{i})_{i \in I} \to S_{i'}$.
We call $q_{i'}$ the $i'$-th projection premorphism.
For $(F_{i})_{i \in I} \in \Set(S_{i})_{i \in I}$,
we have $q_{i', \ast}(F_{i})_{i \in I} = F_{i'}$.
Hence for $j' \ge i'$, the $\varphi$ part of the data gives a natural transformation
	\begin{equation} \label{0336}
			\varphi_{i j}
		\colon
			q_{i', \ast}
		\to
			f_{j' i', \ast} q_{j', \ast}
		\colon
			\Set(S_{i})_{i \in I}
		\to
			\Set(S_{i'}).
	\end{equation}
For $F_{i'} \in \Set(S_{i'})$,
the sheaf $q_{i'}^{\ast \set} F_{i'} = (F_{i})_{i \in I}$ is given by
$F_{i} = f_{i i'}^{\ast} F_{i'}$ if $i \ge i'$ and $F_{i} = \emptyset$ else.
In particular, $q_{i'}$ is a morphism of sites if $i'$ is the minimum element of $I$.
(But it is not if $i'$ is not the minimum
since $q_{i'}^{\ast \set}$ does not preserve the final object then,
and the minimum element might not exist.)
For $G_{i'} \in \Ab(S_{i'})$,
the abelian sheaf pullback $q_{i'}^{\ast} G_{i'} = (G_{i})_{i \in I}$ is given by
$G_{i} = f_{i i'}^{\ast} G_{i'}$ if $i \ge i'$ and $G_{i} = 0$ else.
Hence $q_{i'}^{\ast}$ is exact for any $i'$,
so $q_{i', \ast} \colon \Ab(S_{i})_{i \in I} \to \Ab(S_{i'})$ sends
(K-)injectives to (K-)injectives and (K-)acyclics to (K-)acyclics.

Here we already used the description of $\Ab(S_{i})_{i \in I}$
as the (abelian) category of tuples
$(\{G_{i}\}_{i \in I}, \{\varphi_{i j}\}_{i \le j})$,
where $G_{i} \in \Ab(S_{i})$ and
$\varphi_{i j} \colon f_{j i}^{\ast} G_{i} \to G_{j}$ is a morphism in $\Ab(S_{j})$
satisfying the cocycle condition.
Similarly, the category of complexes $\Ch(S_{i})_{i \in I}$ is the category of tuples
$(G_{i} \in \Ch(S_{i}), f_{j i}^{\ast} G_{i} \to G_{j} \in \Ch(S_{j}))$.
We have a canonical morphism
$\varphi_{i j} \colon f_{j i}^{\ast} G_{i} \to G_{j}$ in $D(S_{j})$ by \eqref{0336}
or, equivalently, $\varphi_{i j} \colon G_{i} \to R f_{i j, \ast} G_{j}$ in $D(S_{i})$.
The cohomology sheaves of $G \in D(S_{i})_{i \in I}$ in terms of tuples is given as follows.

\begin{Prop} \label{0352}
	Let $G \in D(S_{i})_{i \in I}$, $G_{i}$ and $\varphi_{i j}$ be as above.
	Then for any $q \in \Z$, the sheaf $H^{q} G \in \Ab(S_{i})_{i \in I}$ is given by the tuple
	$(\{H^{q} G_{i}\}, \{H^{q} \varphi_{i j}\})$,
	where $H^{q} \varphi_{i j}$ is viewed as $f_{j i}^{\ast} H^{q} G_{i} \to H^{q} G_{j}$,
	or $H^{q} G_{i} \to f_{j i, \ast} H^{q} G_{j}$.
\end{Prop}

\begin{proof}
	Represent $G$ by a K-injective complex and then take the $q$-th cohomology.
\end{proof}

As a special case, take $I = \{\bullet \le \bullet\}$ to be
the two-element set with one non-trivial relation.
Then $\{S_{i}\}_{i \in I}$ corresponds to a single morphism of sites $f \colon T \to S$.
Denote the total site $(T \to S)$ by $S_{c}$
(where $c$ stands for ``compactified by $T$'').
Let $\pi_{S} \colon S \to C$ be a morphism of sites to another site $C$.
We have the projection premorphisms
$q_{S} \colon S_{c} \to S$ and $q_{T} \colon S_{c} \to T$.
Set $\pi_{T} = \pi_{S} \compose f \colon T \to C$,
$\Bar{\pi}_{S} = \pi_{S} \compose q_{S} \colon S_{c} \to C$
and $\Bar{\pi}_{T} = \pi_{T} \compose q_{T} \colon S_{c} \to C$.
From \eqref{0336}, we have a natural transformation
$\Bar{\pi}_{S, \ast} \to \Bar{\pi}_{T, \ast}$.
Define
	\begin{equation} \label{0386}
			\Bar{\pi}_{S, !}
		:=
			[\Bar{\pi}_{S, \ast} \to \Bar{\pi}_{T, \ast}][-1]
		\colon
			\Ch(S_{c})
		\to
			\Ch(C).
	\end{equation}
We have its right derived functor
	\begin{equation} \label{0334}
			R \Bar{\pi}_{S, !}
		\colon
			D(S_{c})
		\to
			D(C).
	\end{equation}
We have a distinguished triangle
	\[
			R \Bar{\pi}_{S, !} G
		\to
			R \Bar{\pi}_{S, \ast} G
		\to
			R \Bar{\pi}_{T, \ast} G
	\]
in $D(C)$ functorial in $G \in D(S_{c})$.
The functoriality of $\Bar{\pi}_{S, !}$ gives a canonical morphism
	\[
			\Bar{\pi}_{S, \ast} \sheafhom_{S_{c}}(G, H)
		\to
			\sheafhom_{C}(\Bar{\pi}_{S, !} G, \Bar{\pi}_{S, !} H)
	\]
in $\Ch(C)$ functorial in $G, H \in \Ch(S_{c})$.
By Proposition \ref{0326}, this morphism induces a morphism
	\[
			R \Bar{\pi}_{S, \ast} R \sheafhom_{S_{c}}(G, H)
		\to
			R \sheafhom_{C}(R \Bar{\pi}_{S, !} G, R \Bar{\pi}_{S, !} H)
	\]
in $D(C)$ functorial in $G, H \in D(S_{c})$.
Hence we have a canonical morphism
	\begin{equation} \label{0328}
			R \Bar{\pi}_{S, \ast} G \tensor^{L} R \Bar{\pi}_{S, !} H
		\to
			R \Bar{\pi}_{S, !}(G \tensor^{L} H)
	\end{equation}
in $D(C)$ functorial in $G, H \in D(S_{c})$.

We will see how the above is functorial in $T \to S$.
Let
	\[
		\begin{CD}
				T'
			@> g_{T} >>
				T
			\\ @V f' VV @VV f V \\
				S'
			@> g_{S} >>
				S
		\end{CD}
	\]
be a commutative diagram of morphisms of sites.
Denote the total site $(T' \to S')$ by $S'_{c}$.
The diagram induces a morphism of sites $S_{c}' \to S_{c}$,
which we denote by $g$.
Set $\pi_{S'} = \pi_{S} \compose g_{S} \colon S' \to C$.
We can do the above constructions for
$T' \stackrel{f'}{\to} S' \stackrel{\pi_{S'}}{\to} C$,
resulting $q_{S'} \colon S_{c}' \to S'$, $q_{T'} \colon S_{c}' \to T'$,
$\Bar{\pi}_{S'}, \Bar{\pi}_{T'} \colon S_{c}' \to C$ and
$\Bar{\pi}_{S', !} \colon \Ch(S_{c}') \to \Ch(C)$.
We have $\Bar{\pi}_{S, \ast} g_{\ast} \cong \Bar{\pi}_{S', \ast}$,
so $R \Bar{\pi}_{S, \ast} R g_{\ast} \cong R \Bar{\pi}_{S', \ast}$,
and we have a natural transformation
	\begin{equation} \label{0467}
			R \Bar{\pi}_{S, \ast}
		\to
			R \Bar{\pi}_{S', \ast} g^{\ast}.
	\end{equation}
Also $\Bar{\pi}_{S, !} g_{\ast} \cong \Bar{\pi}_{S', !}$,
so $R \Bar{\pi}_{S, !} R g_{\ast} \cong R \Bar{\pi}_{S', !}$,
and we have a natural transformation
	\begin{equation} \label{0468}
			R \Bar{\pi}_{S, !}
		\to
			R \Bar{\pi}_{S', !} g^{\ast}.
	\end{equation}
The functors $R g_{\ast}$ and $g^{\ast}$ are described as follows.

\begin{Prop} \label{0345} \mbox{}
	\begin{enumerate}
		\item
			Let $G' \in D(S_{c}')$.
			Set $G'_{S} = q_{S', \ast} G'$, $G'_{T} = q_{T', \ast} G'$
			and $\varphi' \colon G'_{S} \to R f'_{\ast} G'_{T}$ the induced morphism.
			Then for any $q \in \Z$, the sheaf $R^{q} g_{\ast} G'$ is given by the triple
				\[
					(R^{q} g_{S, \ast} G'_{S},
					R^{q} g_{T, \ast} G'_{T},
					R^{q} g_{S, \ast} \varphi'),
				\]
			where we view the morphism $\varphi'$ as $f'^{\ast} G'_{S} \to G'_{T}$,
			the morphism $R g_{S, \ast} \varphi'$ as
			$f^{\ast} R g_{S, \ast} G'_{S} \to R g_{T, \ast} G'_{T}$
			and hence the morphism $R^{q} g_{S, \ast} \varphi'$ as
			$R^{q} g_{S, \ast} G'_{S} \to f_{\ast} R^{q} g_{T, \ast} G'_{T}$.
		\item
			Let $G \in D(S_{c})$.
			Set $G_{S} = q_{S, \ast} G$, $G_{T} = q_{T, \ast} G$
			and $\varphi \colon G_{S} \to R f_{\ast} G_{T}$ the induced morphism.
			Then for any $q \in \Z$, the sheaf $H^{q} g^{\ast} G$ is given by the triple
				\[
					(g_{S}^{\ast} H^{q} G_{S},
					g_{T}^{\ast} H^{q} G_{T},
					g_{S}^{\ast} H^{q} \varphi),
				\]
			where we view the morphism $\varphi$ as $f^{\ast} G_{S} \to G_{T}$,
			the morphism $g_{T}^{\ast} H^{q} \varphi$ as
			$f'^{\ast} g_{S}^{\ast} H^{q} G_{S} \to g_{T}^{\ast} H^{q} G_{T}$
			and hence the morphism $g_{S}^{\ast} H^{q} \varphi$ as
			$g_{S}^{\ast} H^{q} G_{S} \to f'_{\ast} g_{T}^{\ast} H^{q} G_{T}$.
	\end{enumerate}
\end{Prop}

\begin{proof}
	This follows from Proposition \ref{0352}.
\end{proof}

The cup product is functorial:

\begin{Prop} \label{0329}
	Let $G', H' \in D(S_{c}')$.
	Then the diagram
		\[
			\begin{CD}
						R \Bar{\pi}_{S, \ast} R g_{\ast} G'
					\tensor^{L}
						R \Bar{\pi}_{S, !} R g_{\ast} H'
				@>>>
					R \Bar{\pi}_{S, !} R g_{\ast}(G' \tensor^{L} H')
				\\ @| @| \\
						R \Bar{\pi}_{S', \ast} G'
					\tensor^{L}
						R \Bar{\pi}_{S', !} H'
				@>>>
					R \Bar{\pi}_{S', \ast}(G' \tensor^{L} H')
			\end{CD}
		\]
	in $D(C)$ is commutative,
	where the upper horizontal morphism is the composite of
	\eqref{0328} for $T \to S \to C$ and \eqref{0327} for $g$
	and the lower horizontal morphism is \eqref{0328} for $T' \to S' \to C$.
\end{Prop}

\begin{proof}
	It is enough to show that the similarly defined diagram
		\[
			\begin{CD}
					R \Bar{\pi}_{S, \ast} R g_{\ast}
					R \sheafhom_{S_{c}'}(G', H')
				@>>>
					R \sheafhom_{C}(
						R \Bar{\pi}_{S, !} R g_{\ast} G',
						R \Bar{\pi}_{S, !} R g_{\ast} H'
					)
				\\ @| @| \\
					R \Bar{\pi}_{S', \ast}
					R \sheafhom_{S_{c}'}(G', H')
				@>>>
					R \sheafhom_{C}(
						R \Bar{\pi}_{S', !} G',
						R \Bar{\pi}_{S', !} H'
					)
			\end{CD}
		\]
	is commutative.
	This reduces to the commutativity of the corresponding underived diagram,
	which is easy.
\end{proof}

\begin{Prop} \label{0332}
	Let $G, H \in D(S_{c})$.
	Then the digram
		\[
			\begin{CD}
						R \Bar{\pi}_{S, \ast} G
					\tensor^{L}
						R \Bar{\pi}_{S, !} H
				@>>>
					R \Bar{\pi}_{S, !}(G \tensor^{L} H)
				\\ @VVV @VVV \\
						R \Bar{\pi}_{S', \ast} g^{\ast} G
					\tensor^{L}
						R \Bar{\pi}_{S', !} g^{\ast} H
				@>>>
					R \Bar{\pi}_{S', !} g^{\ast}(G \tensor^{L} H)
			\end{CD}
		\]
	is commutative,
	where the upper horizontal morphism is \eqref{0328} for $T \to S \to C$
	and the lower horizontal morphism is the composite of
	\eqref{0328} for $T' \to S' \to C$ and the commutativity of $g^{\ast}$ with $\tensor^{L}$.
\end{Prop}

\begin{proof}
	This follows from Proposition \ref{0329}.
\end{proof}

Assume that $T$ can be written as a disjoint union
$T_{1} \sqcup \dots \sqcup T_{n}$ of sites $T_{1}, \dots, T_{n}$.
Write $f \colon T \to S$ as a disjoint union of $f_{i} \colon T_{i} \to S$.
Set $\pi_{T_{i}} = \pi_{S} \compose f_{i} \colon T_{i} \to C$.
Let $q_{T_{i}} \colon S_{c} \to T_{i}$ be the premorphism of sites
defined by the inclusion functor on the underlying categories.
Set $\Bar{\pi}_{T_{i}} = \pi_{T_{i}} \compose q_{T_{i}} \colon S_{c} \to C$.
Then $\Bar{\pi}_{T, \ast} \cong \bigoplus_{i} \Bar{\pi}_{T_{i}, \ast}$, and hence
	\[
			\Bar{\pi}_{S, !}
		\cong
			\left[
				\Bar{\pi}_{S, \ast} \to \bigoplus_{i} \Bar{\pi}_{T_{i}, \ast}
			\right][-1].
	\]
Therefore we have a distinguished triangle
	\[
			R \Bar{\pi}_{S, !} G
		\to
			R \Bar{\pi}_{S, \ast} G
		\to
			\bigoplus_{i} \Bar{\pi}_{T_{i}, \ast} G
	\]
in $D(C)$ functorial in $G \in D(S_{c})$.

Now consider the setting of Section \ref{0436},
so we have morphisms
	\[
			\bigsqcup_{x \in Z}
				\alg{U}_{x, \et}^{h}
		\to
			\alg{U}_{\et}
		\to
			\Spec F^{\perar}_{\et}.
	\]
Denote the total site
$(\bigsqcup_{x \in Z} \alg{U}_{x, \et}^{h} \to \alg{U}_{\et})$
by $\alg{U}_{c, \et}$.
We will see how the formalism of this subsection is compatible with
that of the previous subsection.

We have a natural projection morphism $q_{\alg{U}} \colon \alg{U}_{c, \et} \to \alg{U}_{\et}$.
The morphisms $\pi_{\alg{U}}$ and $\pi_{\alg{U}_{x}^{h}}$ induce morphisms
	\[
			\Bar{\pi}_{\alg{U}},
			\Bar{\pi}_{\alg{U}_{x}^{h}}
		\colon
			\Spec \alg{U}_{c, \et}
		\to
			\Spec F^{\perar}_{\et}.
	\]
We have a functor
	\[
			\Bar{\pi}_{\alg{U}, !}
		=
			\left[
					\Bar{\pi}_{\alg{U}, \ast}
				\to
					\bigoplus_{x \in Z}
						\Bar{\pi}_{\alg{U}_{x}^{h}, \ast}
			\right][-1]
		\colon
			\Ch(\alg{U}_{c, \et})
		\to
			\Ch(F^{\perar}_{\et}).
	\]
We have a morphism
	\begin{equation} \label{0377}
				R \Bar{\pi}_{\alg{U}, \ast} G
			\tensor^{L}
				R \Bar{\pi}_{\alg{U}, !} H
		\to
			R \Bar{\pi}_{\alg{U}, !}(G \tensor^{L} H)
	\end{equation}
in $D(F^{\perar}_{\tau})$ functorial in $G, H \in D(\alg{U}_{c, \et})$.

\begin{Prop} \label{0397}
	Let $G, H \in D^{+}(\alg{U}_{\et})$.
	Then
		\[
					R \Bar{\pi}_{\alg{U}, \ast} q_{\alg{U}}^{\ast} G
				\cong
					R \pi_{\alg{U}, \ast} G,
			\quad
					R \Bar{\pi}_{\alg{U}, !} q_{\alg{U}}^{\ast} G
				\cong
					R \pi_{\alg{U}, !} G.
		\]
	Via these isomorphisms, the morphism
		\[
					R \Bar{\pi}_{\alg{U}, \ast} q_{\alg{U}}^{\ast} G
				\tensor^{L}
					R \Bar{\pi}_{\alg{U}, !} q_{\alg{U}}^{\ast} H
			\to
				R \Bar{\pi}_{\alg{U}, !} q_{\alg{U}}^{\ast}(G \tensor^{L} H)
		\]
	defined by \eqref{0377} agrees with the morphism \eqref{0369}.
\end{Prop}

\begin{proof}
	This follows from Propositions \ref{0378} and \ref{0379}.
\end{proof}


\subsection{More fibered sites}
\label{0322}

We consider the situation of the beginning of Section \ref{0325},
so we have a fibered site over a finite poset $I$.
We take $I$ to be the commutative diagram
	\[
		\begin{CD}
			@.
				\bullet
			@<<<
				\bullet
			@.
			\\ @. @AAA @AAA @. \\
				\bullet
			@<<<
				\bullet
			@<<<
				\bullet
			@<<<
				\bullet.
		\end{CD}
	\]
Let
	\begin{equation} \label{0394}
		\begin{CD}
			@.
				T_{3}
			@> f_{32} >>
				T_{2}
			@.
			\\ @. @V f_{31} VV @VV f_{2} V @. \\
				T
			@>> f_{01} >
				T_{1}
			@>> f_{1} >
				S
			@>> \pi_{S} >
				C,
		\end{CD}
	\end{equation}
be our fibered site.
(One might want to denote $T$ by $T_{0}$ for notational consistency.)
Let $\pi_{T} \colon T \to C$, $\pi_{T_{i}} \colon T_{i} \to C$,
$f \colon T \to S$, $f_{i} \colon T_{i} \to S$ be the appropriate composite morphisms.
Let
	\begin{gather*}
			S_{c} = (T \to S),
		\\
			T_{\bullet} = (T_{1} \gets T_{3} \to T_{2}),
		\\
			T_{\bullet, c} = (T \to T_{1} \gets T_{3} \to T_{2})
	\end{gather*}
be the total sites.
The above commutative diagram induces morphisms
$f_{\bullet} \colon T_{\bullet} \to S$ and $\Bar{f}_{\bullet} \colon T_{\bullet, c} \to S_{c}$.
Let $f_{0 \bullet} \colon T \to T_{\bullet}$ be the morphism defined by the functor sending
$(X_{1}, X_{2}, X_{3}, f_{31}^{-1} X_{1} \to X_{3} \gets f_{32}^{-1} X_{2})$
to $f_{01}^{-1} X_{1}$.
Then $T_{\bullet, c} \cong (T \to T_{\bullet})$.

\begin{Ex}
	A toy example is the \'etale sites of the spectra of the $F$-algebra homomorphisms
		\[
			\begin{CD}
				@.
					F((u))
				@<<<
					F[u^{\pm}]
				@.
				\\ @. @AAA @AAA @. \\
					F((t))
				@<<<
					F[[t, u]] / (u t)
				@<<<
					F[[t]][u] / (u t)
				@<<<
					F.
			\end{CD}
		\]
	This diagram is aimed at describing the cohomology of $\Spec F[[t]][u] / (u t)$
	with support on the closed subscheme $\Spec F[u]$ (whose complement is $\Spec F((t))$)
	in terms of the simpler pieces $\Spec F[u^{\pm}]$ (which is global and regular) and
	$\Spec F[[t, u]] / (u t)$ (which is local)
	with gluing data $\Spec F((u))$.
	Actual examples we will see in Sections \ref{0297} and \ref{0303}
	will be more involved $p$-adic versions of this diagram.
\end{Ex}

We want to describe $\pi_{S, \ast}$ and
$\Bar{\pi}_{S, !} = [\Bar{\pi}_{S, \ast} \to \Bar{\pi}_{T, \ast}][-1]$
in terms of $\pi_{T_{i}}$.
For $i = 1, 2, 3$, we denote the composite of the natural projection $T_{\bullet} \to T_{i}$
and $\pi_{T_{i}} \colon T_{i} \to C$ simply by $\pi_{T_{i}}$ unless confusion occurs.
Hence we have natural transformations
	\[
			\pi_{T_{1}, \ast}
		\stackrel{\varphi_{13}}{\to}
			\pi_{T_{3}, \ast}
		\stackrel{\varphi_{23}}{\gets}
			\pi_{T_{2}, \ast}
		\colon
			\Ch(T_{\bullet})
		\to
			\Ch(C)
	\]
Define
	\begin{equation} \label{0399}
			\pi_{T_{\bullet}, \ast}
		:=
			[\pi_{T_{1}, \ast} \oplus \pi_{T_{2}, \ast} \to \pi_{T_{3}, \ast}][-1]
		\colon
			\Ch(T_{\bullet})
		\to
			\Ch(C),
	\end{equation}
where the morphism in the mapping cone is $\varphi_{13} - \varphi_{23}$.
Note that this functor is not induced from any functor $\Ab(T_{\bullet}) \to \Ab(C)$;
it is genuinely on the level of complexes.
For $i = 1, 2$, the morphism $f_{i} \colon T_{i} \to S$ induces a natural transformation
	\begin{gather*}
				\varphi_{i}
			\colon
				\pi_{S, \ast} f_{\bullet, \ast}
			\to
				\pi_{T_{i}, \ast}
			\colon
				\Ch(T_{\bullet})
			\to
				\Ch(C),
			\quad
				\text{or}
		\\
				\varphi_{i}
			\colon
				\pi_{S, \ast}
			\to
				\pi_{T_{i}, \ast}  f_{\bullet}^{\ast}
			\colon
				\Ch(S)
			\to
				\Ch(C).
	\end{gather*}
The direct sum $(\varphi_{1}, \varphi_{2})$ defines a morphism
$\pi_{S, \ast} \to (\pi_{T_{1}, \ast} \oplus \pi_{T_{2}, \ast}) f_{\bullet}^{\ast}$,
whose composite with $\varphi_{13} - \varphi_{23}$ is zero.
Hence this defines a morphism
	\begin{equation} \label{0387}
			(\varphi_{1}, \varphi_{2})
		\colon
			\pi_{S, \ast}
		\to
			\pi_{T_{\bullet}, \ast} f_{\bullet}^{\ast}
		\colon
			\Ch(S)
		\to
			\Ch(C).
	\end{equation}
We define a cup product morphism for $R \pi_{T_{\bullet}, \ast}$:

\begin{Prop} \label{0385}
	For $G, H \in D(T_{\bullet})$,
	there exists a canonical morphism
		\[
				R \pi_{T_{\bullet}, \ast} G \tensor^{L} R \pi_{T_{\bullet}, \ast} H
			\to
				R \pi_{T_{\bullet}, \ast}(G \tensor^{L} H)
		\]
	in $D(C)$ functorial in $G, H$.
\end{Prop}

\begin{proof}
	First, for $G, H \in \Ch(T_{\bullet})$, we construct a canonical morphism
		\begin{equation} \label{0384}
				\pi_{T_{\bullet}, \ast} G \tensor \pi_{T_{\bullet}, \ast} H
			\to
				\pi_{T_{\bullet}, \ast}(G \tensor H)
		\end{equation}
	in $\Ch(C)$ functorial in $G, H$.
	The left-hand side is the total complex of the complex of double complexes
		\begin{equation} \label{0382}
			\begin{aligned}
							0
				&	\to
							(\pi_{T_{1}, \ast} G \oplus \pi_{T_{2}, \ast} G)
						\tensor
							(\pi_{T_{1}, \ast} H \oplus \pi_{T_{2}, \ast} H)
				\\
				&	\to
							(\pi_{T_{1}, \ast} G \oplus \pi_{T_{2}, \ast} G) \tensor \pi_{T_{3}, \ast} H
						\oplus
							\pi_{T_{3}, \ast} G \tensor (\pi_{T_{1}, \ast} H \oplus \pi_{T_{2}, \ast} H)
				\\
				&	\to
						\pi_{T_{3}, \ast} G \tensor \pi_{T_{3}, \ast} H
				\\
				&	\to
						0,
			\end{aligned}
		\end{equation}
	where these three non-zero terms are in horizontal (or external) degrees $0$, $1$ and $2$.
	The right-hand side is the total complex of the complex of double complexes
		\begin{equation} \label{0383}
				0
			\to
				\pi_{T_{1}, \ast}(G \tensor H) \oplus \pi_{T_{2}, \ast}(G \tensor H)
			\to
				\pi_{T_{3}, \ast}(G \tensor H)
			\to
				0,
		\end{equation}
	where these two non-zero terms are in horizontal (or external) degrees $0$ and $1$.
	We define a morphism from the degree zero term of \eqref{0382} to that of \eqref{0383} by
		\[
				(g_{1}, g_{2}) \tensor (h_{1}, h_{2})
			\mapsto
				(g_{1} \tensor h_{1}, g_{2} \tensor h_{2}).
		\]
	We define a morphism from the degree one term of \eqref{0382} to that of \eqref{0383} by
		\[
				\bigl(
					(g_{1}, g_{2}) \tensor h_{3},
					g_{3} \tensor (h_{1}, h_{2})
				\bigr)
			\mapsto
				\varphi_{1 3}(g_{1}) \tensor h_{3} + g_{3} \tensor \varphi_{2 3}(h_{2}).
		\]
	(The other choice is to map it to
		$
			\varphi_{2 3}(g_{2}) \tensor h_{3} + g_{3} \tensor \varphi_{1 3}(h_{1})
		$,
	but this choice yields a homotopically equivalent morphism
	thanks to the degree two term of \eqref{0382},
	so it will not make a difference on the derived category level.)
	These two morphisms form a morphism of complexes from \eqref{0382} to \eqref{0383}.
	Therefore they define the morphism \eqref{0384}.
	
	Now this induces a morphism
		\[
				\pi_{T_{\bullet}, \ast}
				\sheafhom_{T_{\bullet}}(G, H)
			\to
				\sheafhom_{C}(\pi_{T_{\bullet}, \ast} G, \pi_{T_{\bullet}, \ast} H).
		\]
	By Proposition \ref{0326}, this derives to a morphism
		\[
				R \pi_{T_{\bullet}, \ast}
				R \sheafhom_{T_{\bullet}}(G, H)
			\to
				R \sheafhom_{C}(R \pi_{T_{\bullet}, \ast} G, R \pi_{T_{\bullet}, \ast} H)
		\]
	in $D(C)$ functorial in $G, H \in D(T_{\bullet})$.
	From this, we obtain a desired morphism.
\end{proof}

Now we put a support data by $T$.
For $i = 1, 2, 3$, we denote the composite of the natural projection $T_{\bullet, c} \to T_{i}$
and $\pi_{T_{i}} \colon T_{i} \to C$ by $\Bar{\pi}_{T_{i}}$.
Similarly, we denote the composite
$T_{\bullet, c} \to T \to C$ by $\Bar{\pi}_{T}$.
Hence we have natural transformations
	\begin{gather*}
				\Bar{\pi}_{T, \ast}
			\stackrel{\varphi_{10}}{\gets}
				\Bar{\pi}_{T_{1}, \ast}
			\stackrel{\varphi_{13}}{\to}
				\Bar{\pi}_{T_{3}, \ast}
			\stackrel{\varphi_{23}}{\gets}
				\Bar{\pi}_{T_{2}, \ast}
			\colon
				\Ch(T_{\bullet, c})
			\to
				\Ch(C),
	\end{gather*}
Define a natural transformation of functors $\Ch(T_{\bullet, c}) \to \Ch(C)$
by the composite
	\[
			\varphi_{\bullet 0}
		\colon
			\Bar{\pi}_{T_{\bullet}, \ast}
		\to
			\Bar{\pi}_{T_{1}, \ast} \oplus \Bar{\pi}_{T_{2}, \ast}
		\to
			\Bar{\pi}_{T_{1}, \ast}
		\stackrel{\varphi_{10}}{\to}
			\Bar{\pi}_{T, \ast},
	\]
where the first morphism is the connecting morphism of mapping cone
and the second is the projection onto the first factor.
Define
	\begin{equation} \label{0400}
		\begin{gathered}
					\Bar{\pi}_{T_{\bullet}, !}
				:=
					[
							\Bar{\pi}_{T_{\bullet}, \ast}
						\to
							\Bar{\pi}_{T}, \ast
					][-1]
				\cong
					[
							\Bar{\pi}_{T_{1}, \ast} \oplus \Bar{\pi}_{T_{2}, \ast}
						\to
							\Bar{\pi}_{T_{3}, \ast} \oplus \Bar{\pi}_{T, \ast}
					][-1]
				\colon
			\\
					\Ch(T_{\bullet, c})
				\to
					\Ch(C),
		\end{gathered}
	\end{equation}
where the morphism in the first mapping cone is $\varphi_{\bullet 0}$
and the morphism in the second mapping cone is
$(t_{1}, t_{2}) \mapsto (\varphi_{13}(t_{1}) - \varphi_{23}(t_{2}), \varphi_{10}(t_{1}))$.
We have a commutative diagram
	\[
		\begin{CD}
				\Bar{\pi}_{S, \ast}
			@> \varphi >>
				\Bar{\pi}_{T, \ast}
			\\ @V (\varphi_{1}, \varphi_{2}) VV @| \\
				\Bar{\pi}_{T_{\bullet}, \ast} \Bar{f}_{\bullet}^{\ast}
			@>> \varphi_{\bullet 0} >
				\Bar{\pi}_{T, \ast} \Bar{f}_{\bullet}^{\ast}
		\end{CD}
	\]
of natural transformations between functors $\Ch(S_{c}) \to \Ch(C)$,
where $\Bar{\pi}_{T, \ast}$ on the upper and lower right corners mean
the essentially same functors defined on different domains
$\Ch(S_{c}) \to \Ch(C)$ and $\Ch(T_{\bullet, c}) \to \Ch(C)$, respectively.
This induces a morphism
	\begin{equation} \label{0398}
			(\varphi_{1}, \varphi_{2})
		\colon
			\Bar{\pi}_{S, !}
		\to
			\Bar{\pi}_{T_{\bullet}, !} \Bar{f}_{\bullet}^{\ast}
	\end{equation}
on the mapping fibers of the horizontal morphisms
and a morphism of distinguished triangles
	\begin{equation} \label{0409}
		\begin{CD}
				R \Bar{\pi}_{S, !} G
			@>>>
				R \Bar{\pi}_{S, \ast} G
			@>>>
				R \Bar{\pi}_{T, \ast} G
			\\ @VVV @VVV @| \\
				R \Bar{\pi}_{T_{\bullet}, !} \Bar{f}_{\bullet}^{\ast} G
			@>>>
				R \Bar{\pi}_{T_{\bullet}, \ast} \Bar{f}_{\bullet}^{\ast} G
			@>>>
				R \Bar{\pi}_{T, \ast} \Bar{f}_{\bullet}^{\ast} G
		\end{CD}
	\end{equation}
in $D(C)$ functorial in $G \in D(S_{c})$.
Similarly, we have a morphism of distinguished triangles
	\begin{equation} \label{0419}
		\begin{CD}
				R \Bar{\pi}_{S, !} G
			@>>>
				R \Bar{\pi}_{S, \ast} G
			@>>>
				R \Bar{\pi}_{T, \ast} G
			\\ @VVV @VVV @VVV \\
				R \Bar{\pi}_{T_{\bullet}, !} \Bar{f}_{\bullet}^{\ast} G
			@>>>
					R \Bar{\pi}_{T_{1}, \ast} \Bar{f}_{\bullet}^{\ast} G
				\oplus
					R \Bar{\pi}_{T_{2}, \ast} \Bar{f}_{\bullet}^{\ast} G
			@>>>
					R \Bar{\pi}_{T_{3}, \ast} \Bar{f}_{\bullet}^{\ast} G
				\oplus
					R \Bar{\pi}_{T, \ast} \Bar{f}_{\bullet}^{\ast} G.
		\end{CD}
	\end{equation}

For $G, H \in \Ch(T_{\bullet, c})$, the diagram
	\[
		\begin{CD}
					\Bar{\pi}_{T_{\bullet}, \ast} G
				\tensor
					\Bar{\pi}_{T_{\bullet}, \ast} H
			@>>>
				\Bar{\pi}_{T_{\bullet}, \ast}(G \tensor H)
			\\
			@V \id \tensor \varphi_{\bullet 0} VV
			@VV \varphi_{\bullet 0} V
			\\
					\Bar{\pi}_{T_{\bullet}, \ast} G
				\tensor
					\Bar{\pi}_{T, \ast} H
			@>>>
				\Bar{\pi}_{T, \ast}(G \tensor H)
		\end{CD}
	\]
is commutative, where the upper horizontal morphism is \eqref{0384}
and the lower one is $\varphi_{\bullet 0} \tensor \id$ followed by
the morphism \eqref{0327} for $\Bar{\pi}_{T, \ast}$.
Taking mapping fibers of the vertical morphisms, we obtain a morphism
	\[
				\Bar{\pi}_{T_{\bullet}, \ast} G
			\tensor
				\Bar{\pi}_{T_{\bullet}, !} H
		\to
			\Bar{\pi}_{T_{\bullet}, !}(G \tensor H)
	\]
in $\Ch(C)$ functorial in $G, H \in \Ch(T_{\bullet, c})$.
As in the final part of the proof of Proposition \ref{0385},
this induces a canonical morphism
	\begin{equation} \label{0389}
				R \Bar{\pi}_{T_{\bullet}, \ast} G
			\tensor^{L}
				R \Bar{\pi}_{T_{\bullet}, !} H
		\to
			R \Bar{\pi}_{T_{\bullet}, !}(G \tensor^{L} H)
	\end{equation}
in $D(C)$ functorial in $G, H \in D(T_{\bullet, c})$.

The cup products in the previous section and this section are compatible:

\begin{Prop} \label{0410}
	Let $G, H \in D(S_{c})$.
	Then the diagram
		\[
			\begin{CD}
						R \Bar{\pi}_{S, \ast} G
					\tensor^{L}
						R \Bar{\pi}_{S, !} H
				@>>>
					R \Bar{\pi}_{S, !}(G \tensor^{L} H)
				\\
				@V (\varphi_{1}, \varphi_{2}) VV
				@VV (\varphi_{1}, \varphi_{2}) V
				\\
						R \Bar{\pi}_{T_{\bullet}, \ast} \Bar{f}_{\bullet}^{\ast} G
					\tensor^{L}
						R \Bar{\pi}_{T_{\bullet}, !} \Bar{f}_{\bullet}^{\ast} H
				@>>>
					R \Bar{\pi}_{T_{\bullet}, !} \Bar{f}_{\bullet}^{\ast}(G \tensor^{L} H)
			\end{CD}
		\]
	in $D(C)$ is commutative,
	where the horizontal morphisms are \eqref{0328} and \eqref{0389}.
\end{Prop}

\begin{proof}
	It is enough to show that the diagram
		\[
			\begin{CD}
						\Bar{\pi}_{S, \ast} G
					\tensor
						\Bar{\pi}_{S, !} H
				@>>>
					\Bar{\pi}_{S, !}(G \tensor H)
				\\
				@V (\varphi_{1}, \varphi_{2}) VV
				@VV (\varphi_{1}, \varphi_{2}) V
				\\
						\Bar{\pi}_{T_{\bullet}, \ast} \Bar{f}_{\bullet}^{\ast} G
					\tensor
						\Bar{\pi}_{T_{\bullet}, !} \Bar{f}_{\bullet}^{\ast} H
				@>>>
					\Bar{\pi}_{T_{\bullet}, !} \Bar{f}_{\bullet}^{\ast}(G \tensor H)
			\end{CD}
		\]
	in $\Ch(C)$ is commutative for $G, H \in \Ch(S_{c})$.
	For this, it is enough to show that the diagram
		\[
			\begin{CD}
						\pi_{S, \ast} G
					\tensor
						\pi_{S, \ast} H
				@>>>
					\pi_{S, \ast}(G \tensor H)
				\\
				@V (\varphi_{1}, \varphi_{2}) VV
				@VV (\varphi_{1}, \varphi_{2}) V
				\\
						\pi_{T_{\bullet}, \ast} f_{\bullet}^{\ast} G
					\tensor
						\pi_{T_{\bullet}, \ast} f_{\bullet}^{\ast} H
				@>>>
					\pi_{T_{\bullet}, \ast} f_{\bullet}^{\ast}(G \tensor H)
			\end{CD}
		\]
	in $\Ch(C)$ is commutative for $G, H \in \Ch(S)$.
	For this, it is enough to show that the diagram
		\[
			\begin{CD}
						\pi_{S, \ast} G
					\tensor
						\pi_{S, \ast} H
				@>>>
					\pi_{S, \ast}(G \tensor H)
				\\
				@V \varphi_{i} VV
				@VV \varphi_{i} V
				\\
						\pi_{T_{i}, \ast} f_{i}^{\ast} G
					\tensor
						\pi_{T_{i}, \ast} f_{i}^{\ast} H
				@>>>
					\pi_{T_{i}, \ast} f_{i}^{\ast}(G \tensor H)
			\end{CD}
		\]
	in $\Ab(C)$ is commutative for $G, H \in \Ab(S)$ and $i = 1, 2$.
	But this is obvious.
\end{proof}

The cup product \eqref{0389} breaks down into pieces as follows.
We have distinguished triangles
	\begin{gather} \label{0390}
				R [\Bar{\pi}_{T_{2}, \ast} \to \Bar{\pi}_{T_{3}, \ast}][-1] G
			\to
				R \Bar{\pi}_{T_{\bullet}, \ast} G
			\to
				R \Bar{\pi}_{T_{1}, \ast} G,
		\\ \label{0391}
				R [\Bar{\pi}_{T_{1}, \ast} \to \Bar{\pi}_{T_{3}, \ast} \oplus \Bar{\pi}_{T, \ast}][-1] G
			\to
				R \Bar{\pi}_{T_{\bullet}, !} G
			\to
				R \Bar{\pi}_{T_{2}, \ast} G
	\end{gather}
in $D(C)$ functorial in $G \in D(T_{\bullet, c})$.
For $G, H \in D(T_{\bullet, c})$, we have morphisms
	\begin{gather} \label{0392}
		\begin{aligned}
			&			R \Bar{\pi}_{T_{2}, \ast} G
					\tensor^{L}
						R [\Bar{\pi}_{T_{2}, \ast} \to \Bar{\pi}_{T_{3}, \ast}][-1] H
			\\
			&	\to
					R [\Bar{\pi}_{T_{2}, \ast} \to \Bar{\pi}_{T_{3}, \ast}][-1](G \tensor^{L} H)
			\\
			&	\to
					R \Bar{\pi}_{T_{\bullet}, !}(G \tensor^{L} H),
		\end{aligned}
		\\ \label{0393}
		\begin{aligned}
			&			R \Bar{\pi}_{T_{1}, \ast} G
					\tensor^{L}
						R [\Bar{\pi}_{T_{1}, \ast} \to \Bar{\pi}_{T_{3}, \ast} \oplus \Bar{\pi}_{T, \ast}][-1] H
			\\
			&	\to
					R [\Bar{\pi}_{T_{1}, \ast} \to \Bar{\pi}_{T_{3}, \ast} \oplus \Bar{\pi}_{T, \ast}][-1](G \tensor^{L} H)
			\\
			&	\to
					R \Bar{\pi}_{T_{\bullet}, !}(G \tensor^{L} H).
		\end{aligned}
	\end{gather}

\begin{Prop} \label{0427}
	Let $G, H \in D(T_{\bullet, c})$.
	Denote
		\[
				(\var)^{\vee}
			=
				R \sheafhom_{C}(\var, R \Bar{\pi}_{T_{\bullet}, !}(G \tensor^{L} H)).
		\]
	Then the morphisms \eqref{0392}, \eqref{0389} and \eqref{0393} form
	a morphism of distinguished triangles
		\[
			\begin{CD}
					R [\Bar{\pi}_{T_{2}, \ast} \to \Bar{\pi}_{T_{3}, \ast}][-1] G
				@>>>
					R \Bar{\pi}_{T_{\bullet}, \ast} G
				@>>>
					R \Bar{\pi}_{T_{1}, \ast} G
				\\ @VVV @VVV @VVV \\
					(R \Bar{\pi}_{T_{2}, \ast} H)^{\vee}
				@>>>
					(R \Bar{\pi}_{T_{\bullet}, !} H)^{\vee}
				@>>>
					(R [\Bar{\pi}_{T_{1}, \ast} \to \Bar{\pi}_{T_{3}, \ast} \oplus \Bar{\pi}_{T, \ast}][-1] H)^{\vee}
			\end{CD}
		\]
	in $D(C)$ between \eqref{0390} and the dual of \eqref{0391}.
\end{Prop}

\begin{proof}
	This reduces to the corresponding underived diagram in $K(C)$ for $G, H \in K(T_{\bullet, c})$,
	which is routine to check.
\end{proof}

Let
	\begin{equation} \label{0395}
		\begin{CD}
			@.
				T_{3}'
			@> f_{32}' >>
				T_{2}'
			@.
			\\ @. @V f_{31}' VV @VV f_{2}' V @. \\
				T'
			@>> f_{01}' >
				T_{1}'
			@>> f_{1}' >
				S'
			@>> \pi_{S'} >
				C,
		\end{CD}
	\end{equation}
be another fibered site over the same category $I$,
where we are assuming that the lower rightmost term is still the same $C$.
We can do the above constructions for this fibered site,
yielding functors $\Bar{\pi}_{T_{\bullet}', !}$ and so on.

Assume that we are given a morphism of fibered sites
from \eqref{0395} to \eqref{0394}
such that the morphism on the lower rightmost terms is the identity $C \to C$.
The data of such a morphism is given by morphisms of sites
$g \colon S' \to S$, $h_{i} \colon T_{i}' \to T_{i}$ for $i = 1, 2, 3$
and $h \colon T' \to T$ satisfying the obvious commutativity relations.
Let $\Bar{h}_{\bullet} \colon T_{\bullet, c}' \to T_{\bullet, c}$ be the naturally induced morphism.
We will encounter this situation in Section \ref{0303},
where the diagram \eqref{0394} uses henselizations
and the diagram \eqref{0395} uses completions.

We have morphisms
	\begin{equation} \label{0414}
				R \Bar{\pi}_{T_{\bullet}, \ast} G
			\to
				R \Bar{\pi}_{T_{\bullet}', \ast} \Bar{h}_{\bullet}^{\ast} G,
		\quad
				R \Bar{\pi}_{T_{\bullet}, !} G
			\to
				R \Bar{\pi}_{T_{\bullet}', !} \Bar{h}_{\bullet}^{\ast} G,
	\end{equation}
morphisms of distinguished triangles
	\begin{equation} \label{0416}
		\begin{CD}
				R \Bar{\pi}_{T_{\bullet}, !} G
			@>>>
				R \Bar{\pi}_{T_{\bullet}, \ast} G
			@>>>
				R \Bar{\pi}_{T, \ast} G
			\\ @VVV @VVV @VVV \\
				R \Bar{\pi}_{T_{\bullet}', !} \Bar{h}_{\bullet}^{\ast} G
			@>>>
				R \Bar{\pi}_{T_{\bullet}', \ast} \Bar{h}_{\bullet}^{\ast} G
			@>>>
				R \Bar{\pi}_{T', \ast} \Bar{h}_{\bullet}^{\ast} G,
		\end{CD}
	\end{equation}
	\begin{equation} \label{0420}
		\begin{CD}
				R \Bar{\pi}_{T_{\bullet}, !} G
			@>>>
					R \Bar{\pi}_{T_{1}, \ast} G
				\oplus
					R \Bar{\pi}_{T_{2}, \ast} G
			@>>>
					R \Bar{\pi}_{T_{3}, \ast} G
				\oplus
					R \Bar{\pi}_{T, \ast}G
			\\ @VVV @VVV @VVV \\
				R \Bar{\pi}_{T'_{\bullet}, !} \Bar{h}_{\bullet}^{\ast} G
			@>>>
					R \Bar{\pi}_{T'_{1}, \ast} \Bar{h}_{\bullet}^{\ast} G
				\oplus
					R \Bar{\pi}_{T'_{2}, \ast} \Bar{h}_{\bullet}^{\ast} G
			@>>>
					R \Bar{\pi}_{T'_{3}, \ast} \Bar{h}_{\bullet}^{\ast} G
				\oplus
					R \Bar{\pi}_{T', \ast} \Bar{h}_{\bullet}^{\ast} G,
		\end{CD}
	\end{equation}
	\[
		\begin{CD}
				R [\Bar{\pi}_{T_{2}, \ast} \to \Bar{\pi}_{T_{3}, \ast}][-1] G
			@>>>
				R \Bar{\pi}_{T_{\bullet}, \ast} G
			@>>>
				R \Bar{\pi}_{T_{1}, \ast} G
			\\ @VVV @VVV @VVV \\
				R [\Bar{\pi}_{T_{2}', \ast} \to \Bar{\pi}_{T_{3}', \ast}][-1] \Bar{h}_{\bullet}^{\ast} G
			@>>>
				R \Bar{\pi}_{T_{\bullet}', \ast} \Bar{h}_{\bullet}^{\ast} G
			@>>>
				R \Bar{\pi}_{T_{1}', \ast} \Bar{h}_{\bullet}^{\ast} G,
		\end{CD}
	\]
	\begin{equation} \label{0415}
		\begin{CD}
				R [\Bar{\pi}_{T_{1}, \ast} \to \Bar{\pi}_{T_{3}, \ast} \oplus \Bar{\pi}_{T, \ast}][-1] G
			@>>>
				R \Bar{\pi}_{T_{\bullet}, !} G
			@>>>
				R \Bar{\pi}_{T_{2}, \ast} G
			\\ @VVV @VVV @VVV \\
				R [\Bar{\pi}_{T_{1}', \ast} \to \Bar{\pi}_{T_{3}', \ast} \oplus \Bar{\pi}_{T, \ast}][-1] \Bar{h}_{\bullet}^{\ast} G
			@>>>
				R \Bar{\pi}_{T_{\bullet}', !} \Bar{h}_{\bullet}^{\ast} G
			@>>>
				R \Bar{\pi}_{T_{2}', \ast} \Bar{h}_{\bullet}^{\ast} G
		\end{CD}
	\end{equation}
and a commutative diagram
	\begin{equation} \label{0417}
		\begin{CD}
					R \Bar{\pi}_{T_{\bullet}, \ast} G
				\tensor^{L}
					R \Bar{\pi}_{T_{\bullet}, !} H
			@>>>
				R \Bar{\pi}_{T_{\bullet}, !}(G \tensor^{L} H)
			\\ @VVV @VVV \\
					R \Bar{\pi}_{T_{\bullet}', \ast} \Bar{h}_{\bullet}^{\ast} G
				\tensor^{L}
					R \Bar{\pi}_{T_{\bullet}', !} \Bar{h}_{\bullet}^{\ast} H
			@>>>
				R \Bar{\pi}_{T_{\bullet}', !} \Bar{h}_{\bullet}^{\ast}(G \tensor^{L} H)
		\end{CD}
	\end{equation}
in $D(C)$ for $G, H \in D(T_{\bullet, c})$.


\section{Setup for two-dimensional local fields}
\label{0169}

Recall that $F$ is a fixed perfect field of characteristic $p > 0$ throughout the paper.
Let $k$ be a complete discrete valuation field with residue field $F$.
Let $K$ be a henselian discrete valuation field of characteristic zero
with residue field $k$.
The field $k$ has either characteristic $p$ and $0$.
Let $\Hat{K}$ be the completion of $K$.
In this section, we fix some notation for two-dimensional local fields,
which will be used in both Sections \ref{0174} (where $k$ has characteristic $p$)
and \ref{0190} (where $k$ has characteristic zero).


\subsection{Relative \'etale sites for two-dimensional local fields}
\label{0037}

The ring of integers $\Order_{k}$ of $k$ has a canonical structure
as a complete $W(F)$-algebra,
where $W$ denotes the ring of $p$-typical Witt vectors of infinite length.
For a perfect artinian $F$-algebra $F'$,
define a $k$-algebra $\alg{k}(F')$ by
	\begin{equation} \label{0450}
			\alg{k}(F')
		=
			(W(F') \ctensor_{W(F)} \Order_{k}) \tensor_{\Order_{k}} k.
	\end{equation}
The functor $F' \mapsto \alg{k}(F')$ is an $F^{\perar}$-algebra.
If $F'$ is a field,
then $\alg{k}(F')$ is a complete discrete valuation field with residue field $F'$.
There is a notion of lifting of $\alg{k}$ to $\Order_{K}$:

\begin{Def} \label{0038}
	A \emph{lifting system} for $\Order_{K}$ consists of:
	\begin{enumerate}
		\item \label{0170}
			a functor $\alg{O}_{K}$ from the category of perfect artinian $F$-algebras
			to the category of $\Order_{K}$-algebras, and
		\item \label{0171}
			an isomorphism $\alg{O}_{K}(F') \tensor_{\Order_{K}} k \cong \alg{k}(F')$
			of $k$-algebras for any $F' \in F^{\perar}$ functorial in $F'$,
	\end{enumerate}
	satisfying:
	\begin{enumerate}
		\item
			$\alg{O}_{K}(F')$ is flat over $\Order_{K}$,
		\item
			$\alg{O}_{K}(F')$ is $\ideal{p}_{K} \alg{O}_{K}(F')$-adically separated and henselian, and
		\item
			$\alg{O}_{K}(F') \to \alg{O}_{K}(F'')$ is finite
			if $F' \to F''$ is \'etale.
	\end{enumerate}
\end{Def}

For a field $F' \in F^{\perar}$,
the ring $\alg{O}_{K}(F')$ is a henselian discrete valuation ring
with maximal ideal generated by a prime element of $\Order_{K}$
and with residue field $F'$.

If $k$ has characteristic $p$,
there is the following canonical such data:
For a perfect field extension $F'$ of $F$,
if we write $k \cong F((t))$, then $\alg{k}(F') \cong F'((t))$.
In particular, $\alg{k}(F')$ is relatively perfect over $k$.
The same is true if $F'$ is a perfect artinian $F$-algebra.
Define $\alg{O}_{K}(F')$ to be the Kato canonical lifting of $\alg{k}(F')$ over $\Hat{\Order}_{K}$.
Then this $\alg{O}_{K}(F')$ with the natural reduction map gives a lifting system.

Even if we start from a complete ring $A$ as in Introduction,
compact support cohomology requires non-complete henselian local fields,
which is why we need to treat more general lifting systems than the canonical one.

If $k$ has characteristic zero,
a lifting system (non-canonically) exists:
by fixing a continuous isomorphism $\Hat{\Order}_{K} \cong k[[\varpi_{K}]]$,
one may take $\alg{O}_{K}(F') = \alg{k}(F')[[\varpi_{K}]]$ with the obvious isomorphism
$\alg{O}_{K}(F') \tensor_{\Order_{K}} k \cong \alg{k}(F')$.

In general, fix a choice of a lifting system for $\Order_{K}$.
Set $\alg{K}(F') = \alg{O}_{K}(F') \tensor_{\Order_{K}} K$.
We have natural morphisms $\alg{K} \gets \alg{O}_{K} \to \alg{k}$
of $F^{\perar}$-algebras.
Applying them to the constructions in Section \ref{0029},
we obtain a commutative diagram of morphisms of sites
	\[
		\begin{CD}
				\Spec \alg{K}_{\et}
			@> j >>
				\Spec \alg{O}_{K, \et}
			@< i <<
				\Spec \alg{k}_{\et}
			\\
			@V \pi_{\alg{K}} VV
			@V \pi_{\alg{O}_{K}} VV
			@V \pi_{\alg{k}} VV
			\\
				\Spec F^{\perar}_{\et}
			@=
				\Spec F^{\perar}_{\et}
			@=
				\Spec F^{\perar}_{\et}.
		\end{CD}
	\]
With respect to $i$ and $j$,
we have functors $j_{!} \colon \Ab(\alg{K}_{\et}) \to \Ab(\alg{O}_{K, \et})$
and $i^{!} \colon \Ab(\alg{O}_{K, \et}) \to \Ab(\alg{k}_{\et})$
as in Section \ref{0381}
and a functor
	\[
			R \pi_{\alg{O}_{K}, !}
		=
			R \pi_{\alg{k}, \ast}
			R i^{!}
		\colon
			D(\alg{O}_{K, \et})
		\to
			D(F^{\perar}_{\et})
	\]
by \eqref{0433}.
We have a distinguished triangle
	\[
			R \pi_{\alg{O}_{K}, !} G
		\to
			R \pi_{\alg{O}_{K}, \ast} G
		\to
			R \pi_{\alg{K}, \ast} j^{\ast} G
	\]
in $D(F^{\perar}_{\et})$ functorial in $G \in D(\alg{O}_{K, \et})$ by \eqref{0434}.


\subsection{Invariance under completion}
\label{0172}

For any $F' \in F^{\perar}$,
let $\Hat{\alg{O}}_{K}(F')$ be the $\ideal{p}_{K} \alg{O}_{K}(F')$-adic completion of $\alg{O}_{K}(F')$.
Then the functor $\Hat{\alg{O}}_{K}$ is a lifting system for $\Order_{K}$.
Let $\Hat{\alg{K}}(F') = \Hat{\alg{O}}_{K}(F') \tensor_{\Order_{K}} K$.
As in Section \ref{0037}, we have a commutative diagram of morphisms of sites
	\[
		\begin{CD}
				\Spec \Hat{\alg{K}}_{\et}
			@> \Hat{\jmath} >>
				\Spec \Hat{\alg{O}}_{K, \et}
			@< \Hat{\imath} <<
				\Spec \alg{k}_{\et}
			\\
			@V \pi_{\Hat{\alg{K}}} VV
			@V \pi_{\Hat{\alg{O}}_{K}} VV
			@V \pi_{\alg{k}} VV
			\\
				\Spec F^{\perar}_{\et}
			@=
				\Spec F^{\perar}_{\et}
			@=
				\Spec F^{\perar}_{\et}.
		\end{CD}
	\]
We have natural morphisms $\alg{O}_{K} \to \Hat{\alg{O}}_{K}$
and $\alg{K} \to \Hat{\alg{K}}$ of $F^{\perar}$-algebras.
They define morphisms of sites
	\begin{gather*}
				\pi_{\Hat{\alg{O}}_{K} / \alg{O}_{K}}
			\colon
				\Spec \Hat{\alg{O}}_{K, \et}
			\to
				\Spec \alg{O}_{K, \et}
		\\
				\pi_{\Hat{\alg{K}} / \alg{K}}
			\colon
				\Spec \Hat{\alg{K}}_{\et}
			\to
				\Spec \alg{K}_{\et}
	\end{gather*}
by Proposition \ref{0036}.
Hence we have morphisms
	\begin{equation} \label{0039}
		\begin{gathered}
					R \pi_{\alg{O}_{K}, \ast}
				\to
					R \pi_{\Hat{\alg{O}}_{K}, \ast}
					\pi_{\Hat{\alg{O}}_{K} / \alg{O}_{K}}^{\ast},
			\quad
					R \pi_{\alg{K}, \ast}
				\to
					R \pi_{\Hat{\alg{K}}, \ast}
					\pi_{\Hat{\alg{K}} / \alg{K}}^{\ast},
			\\
					R \pi_{\alg{O}_{K}, !}
				\to
					R \pi_{\Hat{\alg{O}}_{K}, !}
					\pi_{\Hat{\alg{O}}_{K} / \alg{O}_{K}}^{\ast}
		\end{gathered}
	\end{equation}
such that the diagram
	\begin{equation} \label{0435}
		\begin{CD}
				R \pi_{\alg{O}_{K}, !}
			@>>>
				R \pi_{\alg{O}_{K}, \ast}
			@>>>
				R \pi_{\alg{K}, \ast} j^{\ast}
			\\ @VVV @VVV @VVV \\
				R \pi_{\Hat{\alg{O}}_{K}, !}
				\pi_{\Hat{\alg{O}}_{K} / \alg{O}_{K}}^{\ast}
			@>>>
				R \pi_{\Hat{\alg{O}}_{K}, \ast}
				\pi_{\Hat{\alg{O}}_{K} / \alg{O}_{K}}^{\ast}
			@>>>
				R \pi_{\Hat{\alg{K}}, \ast}
				\pi_{\Hat{\alg{K}} / \alg{K}}^{\ast} j^{\ast}
		\end{CD}
	\end{equation}
is a morphism of distinguished triangles,
where we used
	$
			\pi_{\Hat{\alg{K}} / \alg{K}}^{\ast}
			j^{\ast}
		\cong
			\Hat{\jmath}^{\ast}
			\pi_{\Hat{\alg{O}}_{K} / \alg{O}_{K}}^{\ast}
	$
for the lower triangle.

\begin{Prop} \label{0173}
	The morphisms \eqref{0039} are isomorphisms
	and hence the diagram \eqref{0435} is an isomorphism of distinguished triangles.
\end{Prop}

\begin{proof}
	Using Proposition \ref{0031},
	the statement reduces to the well-known isomorphisms
		\begin{gather*}
					R \Gamma(\Order_{K, \et}, \var)
				\isomto
					R \Gamma(
						\Order_{\Hat{K}, \et},
						\pi_{\Order_{\Hat{K}} / \Order_{K}}^{\ast}(\var)
					),
			\\
					R \Gamma(K_{\et}, \var)
				\isomto
					R \Gamma(
						\Hat{K}_{\et},
						\pi_{\Hat{K} / K}^{\ast}(\var)
					).
		\end{gather*}
\end{proof}


\section{Two-dimensional local fields of mixed characteristic}
\label{0174}

Let $k$ be a complete discrete valuation field of characteristic $p$ with residue field $F$.
Let $K$ be a henselian discrete valuation field of characteristic zero
with residue field $k$.
In this section, we formulate and prove a duality for $K$ and $\Order_{K}$
with ind-pro-algebraic group structures on cohomology.

Let $\Frob \colon \alg{k} \to \alg{k}$ be the natural transformation
given by the $p$-th power map on $\alg{k}$.
Set $\wp = \Frob - 1$.
For $n \ge 1$, let $\alg{k}^{p^{n}}$ and $\alg{k}^{\times p^{n}}$ be
the (presheaf) images of $\Frob^{n}$
on $\alg{k}$ and $\alg{k}^{\times}$, respectively.
Let $\Omega_{\alg{k}}^{1}$ be the functor $F' \mapsto \Omega_{\alg{k}(F')}^{1}$.
The Cartier operator on $\Omega_{\alg{k}(F')}^{1}$ defines a natural transformation
$C \colon \Omega_{\alg{k}}^{1} \to \Omega_{\alg{k}}^{1}$.
All these are in $\Ab(F^{\perar}_{\et})$ and hence in $\Ab(F^{\perar}_{\zar})$.


\subsection{Review of known results}
\label{0040}

The main arithmetic inputs are Kato's calculations
of Galois cohomology and his duality pairings in terms of symbols and differential forms.
These calculations are classical, and we recall them here.

Recall our notation $\xi_{n}(F) = W_{n}(F) / \wp W_{n}(F)$,
where $\wp = \Frob - 1$,
and $\xi(F) = \xi_{1}(F)$.

\begin{Prop}[{\cite[\S 5, Theorem 1]{Kat79}}] \label{0041}
	Let $n \ge 1$ and $r \in \Z$.
	Then we have $H^{q}(K, \Lambda_{n}(r)) = 0$ for $q \ge 4$.
	The boundary maps give morphisms
		\begin{equation} \label{0437}
				H^{2}(K, \Lambda_{n}(2))
			\to
				H^{1}(k, \Lambda_{n}(1))
			\to
				H^{0}(F, \Lambda_{n})
			=
				\Lambda_{n}
		\end{equation}
	and isomorphisms
		\begin{equation} \label{0455}
				H^{3}(K, \Lambda_{n}(2))
			\isomto
				H^{2}(k, \Lambda_{n}(1))
			\isomto
				H^{1}(F, \Lambda_{n})
			\cong
				\xi_{n}(F).
		\end{equation}
\end{Prop}

Here $H^{q}(k, \Lambda_{n}(1))$ for positive characteristic $k$
means flat cohomology,
which is isomorphic to $H^{q - 1}(k, \nu_{n}(1))$.
The morphism \eqref{0455} is compatible with the morphism \eqref{0437}
in the following sense (see \cite[\S 5, Proof of Theorem 1, Step 4]{Kat79}):
Let $K_{\closure{F}}$ be the unramified extension of $K$
corresponding to an algebraic closure $\closure{F}$ of $F$.
Then the composite
	\begin{equation} \label{0491}
			H^{3}(K, \Lambda_{n}(2))
		\cong
			H^{1} \bigl(
				\Gal(\closure{F} / F),
				H^{2}(K_{\closure{F}}, \Lambda_{n}(2))
			\bigr)
		\isomto
			H^{1}(F, \Lambda_{n})
		\cong
			\xi_{n}(F)
	\end{equation}
of the natural morphisms and \eqref{0437} for $K_{\closure{F}}$ is \eqref{0455}.

In particular, for any integers $q, q', r, r'$ with $r + r' = 2$ and $q + q' = 3$,
we have canonical pairings
	\begin{equation} \label{0042}
			H^{q}(K, \Lambda_{n}(r))
		\times
			H^{q'}(K, \Lambda_{n}(r'))
		\to
			H^{3}(K, \Lambda_{n}(2))
		\cong
			\xi_{n}(F),
	\end{equation}
	\[
			R \Gamma(K, \Lambda_{n}(r))
		\otimes^{L}
			R \Gamma(K, \Lambda_{n}(r'))
		\to
			R \Gamma(K, \Lambda_{n}(2))
		\to
			\xi_{n}(F)[-3].
	\]

Let $e_{K}$ be the absolute ramification index of $K$
and set $f_{K} = p e_{K} / (p - 1)$.
For each $r \ge 0$, we have the symbol map
$(K^{\times})^{\tensor r} \to H^{r}(K, \Lambda(r))$,
$(x_{1}, \dots, x_{r}) \mapsto \{x_{1}, \dots, x_{r}\}$
from the tensor power of $K^{\times}$.

Assume that $K$ contains a primitive $p$-th root of unity $\zeta_{p}$
and fix a prime element $\varpi_{K}$ of $K$.
Identify $\Lambda(r)$ with $\Lambda$ for any $r$ using $\zeta_{p}$.
For $m \ge 0$, let $U_{K}^{(m)} = 1 + \ideal{p}_{K}^{m}$ if $m \ge 1$
and $U_{K}^{(m)} = K^{\times}$ if $m = 0$.
Define $U^{m} H^{1}(K, \Lambda) \subset H^{1}(K, \Lambda)$ to be the image of $U_{K}^{(m)}$
via the isomorphism $H^{1}(K, \Lambda) \cong K^{\times} / K^{\times p}$
(where $K^{\times p} = (K^{\times})^{p}$ means the subgroup of $p$-th powers in $K^{\times}$).
Define $U^{m} H^{2}(K, \Lambda) \subset H^{2}(K, \Lambda)$ to be the subgroup generated by
elements of the form $\{x, y\}$ with $x \in U_{K}^{(m)}$ and $y \in K^{\times}$.
Set
	\[
			\gr^{m} H^{q}(K, \Lambda)
		=
			U^{m} H^{q}(K, \Lambda) / U^{m + 1} H^{q}(K, \Lambda)
	\]
for $q = 1, 2$.
For $0 < m < f_{K}$ and $p \mid m$,
let $b_{m}$ be an element of $K$ such that $v_{K}(b_{m}) = m / p$.
For $0 < m < f_{K}$ and $p \nmid m$,
let $c_{m}$ be an element of $K$ such that $v_{K}(c_{m}) = m$.

\begin{Prop}[{\cite[\S 2, Proposition 1]{Kat79}}] \label{0043}
	Let $m \ge 0$.
	\begin{enumerate}
		\item \label{0175}
			If $m = 0$, then
				\begin{align*}
							k^{\times} / k^{\times p}
					&	\isomto
							\gr^{m} H^{2}(K, \Lambda),
					\\
							x
					&	\mapsto
							\{\Tilde{x}, \varpi_{K}\},
				\end{align*}
			where $\Tilde{x}$ denotes any lift of $x$ to $\Order_{K}$.
		\item \label{0176}
			If $0 < m < f_{K}$ and $p \mid m$, then
				\begin{align*}
							k / k^{p}
					&	\isomto
							\gr^{m} H^{2}(K, \Lambda),
					\\
							x
					&	\mapsto
							\{1 + \Tilde{x} b_{m}^{p}, \varpi_{K}\}.
				\end{align*}
		\item \label{0177}
			If $0 < m < f_{K}$ and $p \nmid m$, then
				\begin{align*}
							\Omega_{k}^{1}
					&	\isomto
							\gr^{m} H^{2}(K, \Lambda),
					\\
							x \dlog(y)
					&	\mapsto
							\{1 + \Tilde{x} c_{m}, \Tilde{y}\},
				\end{align*}
			where $x \in k$ and $y \in k^{\times}$.
		\item \label{0178}
			If $m = f_{K}$, then
				\begin{align*}
								\xi(F)
							\oplus
								\xi(k)
					&	\isomto
							\gr^{m} H^{2}(K, \Lambda),
					\\
							(x, y)
					&	\mapsto
								\{1 + \Tilde{x} (\zeta_{p} - 1)^{p}, \Tilde{\varpi}_{k}\}
							+
								\{1 + \Tilde{y} (\zeta_{p} - 1)^{p}, \varpi_{K}\},
				\end{align*}
			where $\varpi_{k}$ is any prime element of $k$.
		\item \label{0179}
			If $m > f_{K}$, then $\gr^{m} H^{2}(K, \Lambda) = 0$.
	\end{enumerate}
\end{Prop}

Here we used the fact $\xi(F) \isomto \Omega_{k}^{1} / (C - 1) \Omega_{k}^{1}$
via $x \mapsto \Tilde{x} \dlog(\varpi_{k})$.

\begin{Prop}[{\cite[\S 6]{Kat79}}] \label{0044}
	Let $m \ge 0$.
	\begin{enumerate}
		\item \label{0180}
			If $m = 0$, then
				\begin{align*}
							\Lambda \oplus k^{\times} / k^{\times p}
					&	\isomto
							\gr^{m} H^{1}(K, \Lambda),
					\\
							(i, x)
					&	\mapsto
							\{\varpi_{K}^{i} \Tilde{x}\}.
				\end{align*}
		\item \label{0181}
			If $0 < m < f_{K}$ and $p \mid m$, then
				\begin{align*}
							k / k^{p}
					&	\isomto
							\gr^{m} H^{1}(K, \Lambda),
					\\
							x
					&	\mapsto
							\left\{
								1 + \Tilde{x} \left(
									\frac{\zeta_{p} - 1}{b_{f_{K} - m}}
								\right)^{p}
							\right\}.
				\end{align*}
		\item \label{0182}
			If $0 < m < f_{K}$ and $p \nmid m$, then
				\begin{align*}
							k
					&	\isomto
							\gr^{m} H^{1}(K, \Lambda),
					\\
							x
					&	\mapsto
							\left\{
								1 + \Tilde{x} \frac{(\zeta_{p} - 1)^{p}}{c_{f_{K} - m}}
							\right\}.
				\end{align*}
		\item \label{0183}
			If $m = f_{K}$, then
				\begin{align*}
							\xi(k)
					&	\isomto
							\gr^{m} H^{1}(K, \Lambda),
					\\
							x
					&	\mapsto
							\{1 + \Tilde{x} (\zeta_{p} - 1)^{p}\}.
				\end{align*}
		\item \label{0184}
			If $m > f_{K}$, then $\gr^{m} H^{1}(K, \Lambda) = 0$.
	\end{enumerate}
\end{Prop}

The restriction of the pairing \eqref{0042} to
	$
			U^{m} H^{2}(K, \Lambda)
		\times
			U^{m'} H^{1}(K, \Lambda)
	$
is zero if $m + m' > f_{K}$.
Hence for $m + m' = f_{K}$, it induces a pairing
	\begin{equation} \label{0045}
				\gr^{m} H^{2}(K, \Lambda)
			\times
				\gr^{m'} H^{1}(K, \Lambda)
		\to
			\xi(F).
	\end{equation}
Let $\Res \colon \Omega_{k}^{1} \to F$ be the residue map.

\begin{Prop}[{\cite[\S 6]{Kat79}}] \label{0046}
	The pairing \eqref{0045},
	via the isomorphisms in
	Propositions \ref{0043}
	and \ref{0044},
	can be translated as follows.
	\begin{enumerate}
		\item \label{0185}
			If $m = 0$, then the pairing becomes
				\begin{align*}
							k^{\times} / k^{\times p} \times \xi(k)
					&	\to
							\xi(F),
					\\
							(x, y)
					&	\mapsto
							\Res(y \dlog(x)).
				\end{align*}
		\item \label{0186}
			If $0 < m < f_{K}$ and $p \mid m$, then the pairing becomes
				\begin{align*}
							k / k^{p} \times k / k^{p}
					&	\to
							\xi(F),
					\\
							(x, y)
					&	\mapsto
							\Res(x d y).
				\end{align*}
		\item \label{0187}
			If $0 < m < f_{K}$ and $p \nmid m$, then the pairing becomes
				\begin{align*}
							\Omega_{k}^{1} \times k
					&	\to
							\xi(F),
					\\
							(\omega, y)
					&	\mapsto
							- m \Res(y \omega).
				\end{align*}
		\item \label{0188}
			If $m = f_{K}$, then the pairing becomes the sum of the pairing
				\begin{align*}
							\xi(F) \times \Lambda
					&	\to
							\xi(F),
					\\
							(x, i)
					&	\mapsto
							i x
				\end{align*}
			and the pairing
				\begin{align*}
							\xi(k) \times k^{\times} / k^{\times p}
					&	\to
							\xi(F),
					\\
							(x, y)
					&	\mapsto
							- \Res(x \dlog y).
				\end{align*}
	\end{enumerate}
\end{Prop}

We give here some auxiliary results for $k$:

\begin{Prop} \label{0047}
	Denote the functor $\Hom_{\Ab}(\var, \xi(F))$ by $(\var)^{\vee}$.
	We have a morphism between exact sequences
		\[
			\begin{CD}
					0
				@>>>
					k / k^{p}
				@> d >>
					\Omega_{k}^{1}
				@> C >>
					\Omega_{k}^{1}
				@>>>
					0
				\\
				@. @VVV @VVV @VVV @.
				\\
					0
				@>>>
					(k / k^{p})^{\vee}
				@>> \mathrm{can}^{\vee} >
					k^{\vee}
				@>> \Frob^{\vee} >
					k^{\vee}
				@>>>
					0,
			\end{CD}
		\]
	where the left (resp.\ middle and right) vertical map is given by the pairing
	$(x, y) \mapsto \Res(y dx)$ (resp.\ $(\omega, y) \mapsto \Res(y \omega)$).
	We also have a morphism between exact sequences
		\[
			\begin{CD}
					0
				@>>>
					k^{\times} / k^{\times p}
				@> \dlog >>
					\Omega_{k}^{1}
				@> C - 1 >>
					\Omega_{k}^{1}
				@>>>
					\xi_{1}(F)
				@>>>
					0
				\\
				@. @VVV @VVV @VVV @VVV @.
				\\
					0
				@>>>
					\xi(k)^{\vee}
				@>> \mathrm{can}^{\vee} >
					k^{\vee}
				@>> \wp^{\vee} >
					k^{\vee}
				@>>>
					\Lambda^{\vee}
				@>>>
					0
			\end{CD}
		\]
	where the left (resp.\ two middle, resp.\ right) vertical map is given by the pairing
	$(x, y) \mapsto \Res(y \dlog(x))$
	(resp.\ $(\omega, y) \mapsto \Res(y \omega)$,
	resp.\ $(x, i) \mapsto i x$).
\end{Prop}

\begin{proof}
	The exactness of the rows is standard.
	The commutativity of the squares follows from the facts that
	$C(x^{p} \omega) = x C \omega$,
	$\Res(C \omega) = \Res(\omega)^{1 / p}$ in $F$
	and $\Res(\omega)^{1 / p} = \Res(\omega)$ in $\xi(F)$.
\end{proof}

\begin{Prop} \label{0451}
	Let $n \ge 1$.
	We have
		\[
				R^{q} \pi_{\alg{k}, \ast} \Lambda
			\cong
				\begin{cases}
						\Lambda
					&	\text{if }
						q = 0,
					\\
						\alg{k} / \wp \alg{k}
					&	\text{if }
						q = 1,
					\\
						0
					&	\text{otherwise},
				\end{cases}
		\]
		\[
				R^{q} \pi_{\alg{k}, \ast} \nu_{n}(1)
			\cong
				\begin{cases}
						\alg{k}^{\times} / \alg{k}^{\times p^{n}}
					&	\text{if }
						q = 0,
					\\
						0
					&	\text{otherwise}.
				\end{cases}
		\]
\end{Prop}

\begin{proof}
	The first isomorphism follows from the exact sequence
		\[
				0
			\to
				\Lambda
			\to
				\Ga
			\stackrel{\wp}{\to}
				\Ga
			\to
				0
		\]
	in $\Ab(\alg{k}_{\et})$ and $R^{q} \pi_{\alg{k}, \ast} \Ga = 0$ for $q \ge 1$.
	For the second, using the exact sequence
		\[
				0
			\to
				\Gm
			\stackrel{p^{n}}{\to}
				\Gm
			\to
				\nu_{n}(1)
			\to
				0
		\]
	in $\Ab(\alg{k}_{\et})$, it suffices to show $R^{q} \pi_{\alg{k}, \ast} \Gm = 0$ for $q \ge 1$.
	This is proved in \cite[Proof of Proposition 9.1]{Suz21}.
\end{proof}


\subsection{Duality statement}

Let $\alg{O}_{K}$ be a lifting system for $\Order_{K}$.
Set $\alg{K}(F') = \alg{O}_{K}(F') \tensor_{\Order_{K}} K$ for $F' \in F^{\perar}$.
The cohomological dimension of $R \pi_{\alg{K}, \ast} \Lambda_{n}(r)$ is $2$
and there is a trace morphism:

\begin{Prop} \label{0500}
	For any $q \ge 3$, $r \in \Z$ and $n \ge 1$, we have
	$R^{q} \pi_{\alg{K}, \ast} \Lambda_{n}(r) = 0$.
	There exists a unique morphism
		\begin{equation} \label{0358}
				R^{2} \pi_{\alg{K}, \ast} \Lambda_{n}(2)
			\to
				\Lambda_{n}
		\end{equation}
	in $\Ab(F^{\perar}_{\et})$ such that for any field $F' \in F^{\perar}$,
	the composite
		\[
				H^{2}(\alg{K}(F'), \Lambda_{n}(2))
			\to
				\Gamma \bigl(
					F',
					R^{2} \pi_{\alg{K}, \ast} \Lambda_{n}(2)
				\bigr)
			\to
				\Lambda_{n}
		\]
	with the natural map is the map \eqref{0437} for the field $\alg{K}(F')$.
\end{Prop}

\begin{proof}
	Since $R^{q} \pi_{\alg{K}, \ast} \Lambda_{n}(r) \in \Ab(F^{\perar}_{\et})$ is
	the \'etale sheafification of the presheaf
		\[
				F'
			\mapsto
				H^{q}(\alg{K}(F'), \Lambda_{n}(r)),
		\]
	this follows from Proposition \ref{0041}.
\end{proof}

Hence we have canonical morphisms
	\[
			R \pi_{\alg{K}, \ast} \Lambda_{n}(2)
		\to
			\Lambda[-2]
		\to
			\Lambda_{\infty}[-2]
	\]
in $D(F^{\perar}_{\et})$.
Here is the duality for $K$ over $\Spec F^{\perar}_{\et}$:

\begin{Prop} \label{0055}
	Let $n \ge 1$ and $r, r' \in \Z$ with $r + r' = 2$.
	\begin{enumerate}
		\item
			We have
			$R \pi_{\alg{K}, \ast} \Lambda_{n}(r) \in \genby{\mathcal{W}_{F}}_{F^{\perar}_{\et}}$.
		\item
			The composite morphism
				\begin{equation} \label{0439}
							R \pi_{\alg{K}, \ast} \Lambda_{n}(r)
						\tensor^{L}
							R \pi_{\alg{K}, \ast} \Lambda_{n}(r')
					\to
						R \pi_{\alg{K}, \ast} \Lambda_{n}(2)
					\to
						\Lambda_{\infty}[-2]
				\end{equation}
			of \eqref{0327} and \eqref{0358} is a perfect pairing in $D(F^{\perar}_{\et})$.
	\end{enumerate}
\end{Prop}

We will prove this below.
After the proof, we will pass to $\Spec F^{\ind\rat}_{\pro\et}$
in Section \ref{0052}.


\subsection{Duality in the perfect artinian Zariski topology}
\label{0048}

We work with the Zariski version $\Spec F^{\perar}_{\zar}$
to write down the duality pairing
and prove a Zariski version of the duality.

Assume that $K$ contains a primitive $p$-th root of unity $\zeta_{p}$.
We identify $\Lambda(r)$ for any $r$ with $\Lambda$ by this choice of $\zeta_{p}$.
We fix a prime element $\varpi_{K}$ of $K$
and a prime element $\varpi_{k}$ of $k$.
Let
	\[
				\mathcal{E}_{K}
			=
				R \pi_{\alg{K}, \ast} \Lambda
			\in
				D(F^{\perar}_{\et}),
		\quad
				\mathcal{F}_{K}
			=
				R \varepsilon_{\ast} \mathcal{E}_{K}
			\in
				D(F^{\perar}_{\zar}).
	\]
We have $H^{0} \mathcal{F}_{K} \cong \Lambda$.
Also, $\Gamma(F', H^{q} \mathcal{F}_{K}) = H^{q}(\alg{K}(F'), \Lambda)$
for any $q$ and $F' \in F^{\perar}$.
Hence $H^{q} \mathcal{F} = 0$ for $q \ge 4$ and we have a canonical isomorphism
	\[
			H^{3} \mathcal{F}_{K} \isomto \xi
	\]
in $\Ab(F^{\perar}_{\zar})$ by Proposition \ref{0041}.
By what we saw in \eqref{0491},
this morphism agrees with the composite
	\[
			H^{3} \mathcal{F}_{K}
		\to
			R^{1} \varepsilon_{\ast} H^{2} \mathcal{E}_{K}
		\to
			R^{1} \varepsilon_{\ast} \Lambda
		=
			\xi,
	\]
where the second morphism is \eqref{0358}.

\begin{Prop} \label{0050}
	The composite
		\[
				\mathcal{F}_{K} \tensor^{L} \mathcal{F}_{K}
			\to
				\mathcal{F}_{K}
			\to
				\xi_{\infty}[-3]
		\]
	obtained by applying $R \varepsilon_{\ast}$ to \eqref{0439}
	is a perfect pairing in $D(F^{\perar}_{\zar})$.
\end{Prop}

We will prove this below.
For $m \ge 0$ and $q = 1, 2$, define
$U^{m} H^{q} \mathcal{F}_{K} \in \Ab(F^{\perar}_{\zar})$
to be the functor
$F' \mapsto U^{m} H^{q}(\alg{K}(F'), \Lambda)$,
and $\gr^{m} H^{q} \mathcal{F}_{K} \in \Ab(F^{\perar}_{\zar})$ its graded pieces.
For $m \ge 0$, Propositions \ref{0043} and \ref{0044} give isomorphisms
	\begin{equation} \label{0440}
			\gr^{m} H^{2} \mathcal{F}_{K}
		\cong
			\begin{cases}
						\alg{k}^{\times} / \alg{k}^{\times p}
				&	\text{if }
						m = 0,
				\\
						\alg{k} / \alg{k}^{p}
				&	\text{if }
						0 < m < f_{K},\, p \mid m,
				\\
						\Omega_{\alg{k}}^{1}
				&	\text{if }
						0 < m < f_{K},\, p \nmid m,
				\\
						\xi \oplus \alg{k} / \wp \alg{k}
				&	\text{if }
						m = f_{K},
				\\
						0
				&	\text{if }
						m > f_{K}
			\end{cases}
	\end{equation}
and
	\begin{equation} \label{0441}
			\gr^{m} H^{1} \mathcal{F}_{K}
		\cong
			\begin{cases}
						\Lambda \oplus \alg{k}^{\times} / \alg{k}^{\times p}
				&	\text{if }
						m = 0,
				\\
						\alg{k} / \alg{k}^{p}
				&	\text{if }
						0 < m < f_{K},\, p \mid m,
				\\
						\alg{k}
				&	\text{if }
						0 < m < f_{K},\, p \nmid m,
				\\
						\alg{k} / \wp \alg{k}
				&	\text{if }
						m = f_{K},
				\\
						0
				&	\text{if }
						m > f_{K},
			\end{cases}
	\end{equation}
where the image $\wp \alg{k}$ of $\wp$ and the quotients are taken in $\Ab(F^{\perar}_{\zar})$.
We have a pairing
	\[
				H^{2} \mathcal{F}_{K}
			\times
				H^{1} \mathcal{F}_{K}
		\to
			H^{3} \mathcal{F}_{K}
		\to
			\xi_{\infty}
	\]
in $\Ab(F^{\perar}_{\zar})$.
For $m, m' \ge 0$ with $m + m' > f_{K}$, this is zero on
	$
			U^{m} H^{2} \mathcal{F}_{K}
		\times
			U^{m'} H^{1} \mathcal{F}_{K}
	$
and, for $m, m' \ge 0$ with $m + m' = f_{K}$, induces a pairing
	\begin{equation} \label{0049}
				\gr^{m} H^{2} \mathcal{F}_{K}
			\times
				\gr^{m'} H^{1} \mathcal{F}_{K}
		\to
			\xi_{\infty}.
	\end{equation}
By Proposition \ref{0046}
and via the above isomorphisms,
this pairing can be translated into the following pairings:
	\begin{align*}
				\alg{k}^{\times} / \alg{k}^{\times p} \times \alg{k} / \wp \alg{k}
		&	\to
				\xi_{\infty},
		\\
				(x, y)
		&	\mapsto
				\Res(y \dlog(x)).
	\end{align*}
for $m = 0$;
	\begin{align*}
				\alg{k} / \alg{k}^{p} \times \alg{k} / \alg{k}^{p}
		&	\to
				\xi_{\infty},
		\\
				(x, y)
		&	\mapsto
				\Res(x d y).
	\end{align*}
for $0 < m < f_{K}$ and $p \mid m$;
	\begin{align*}
				\Omega_{\alg{k}}^{1} \times \alg{k}
		&	\to
				\xi_{\infty},
		\\
				(\omega, y)
		&	\mapsto
				- m \Res(y \omega).
	\end{align*}
for $0 < m < f_{K}$ and $p \nmid m$;
the sum of
	\begin{align*}
				\xi \times \Lambda
		&	\to
				\xi_{\infty},
		\\
				(x, i)
		&	\mapsto
				i x
	\end{align*}
and
	\begin{align*}
				\alg{k} / \wp \alg{k} \times \alg{k}^{\times} / \alg{k}^{\times p}
		&	\to
				\xi_{\infty},
		\\
				(x, y)
		&	\mapsto
				- \Res(x \dlog y).
	\end{align*}
for $m = f_{K}$.

\begin{proof}[Proof of Proposition \ref{0050}]
	We need to show that the induced morphism
		\[
				H^{q} \mathcal{F}_{K} \tensor^{L} H^{q'} \mathcal{F}_{K}
			\to
				\xi_{\infty}
		\]
	for $q + q' = 3$ is a perfect pairing.
	For $q = 0$ and $3$, this follows from Proposition \ref{0028}.
	For $q = 1$ or $2$, it is enough to show that the induced morphism
		\[
				\gr^{m} H^{q} \mathcal{F}_{K} \tensor^{L} \gr^{m'} H^{q'} \mathcal{F}_{K}
			\to
				\xi_{\infty}
		\]
	for $m + m' = f_{K}$ is a perfect pairing.
	The pairing $\Omega_{k}^{1} \times k \to \xi_{1}(F)$,
	$(\omega, y) \mapsto \Res(y \omega)$,
	lifts to the perfect pairing
	$\Omega_{k}^{1} \times k \to F$
	of Tate vector spaces over $k$.
	Therefore Proposition \ref{0027}
	implies that the induced morphisms
		\[
				\Omega_{\alg{k}}^{1} \tensor^{L} \alg{k}
			\to
				\xi_{\infty}
		\]
	is a perfect pairing in $D(F^{\perar}_{\zar})$.
	By the above descriptions of \eqref{0049},
	this implies the case $0 < m < f_{K}$, $p \nmid m$.
	With Proposition \ref{0028},
	all other cases follow from
	Proposition \ref{0047}.
\end{proof}


\subsection{Proof of the duality}
\label{0052}

Now we return to $\Spec F^{\perar}_{\et}$ and prove the duality.

\begin{proof}[Proof of Proposition \ref{0055}]
	First assume that $\zeta_{p} \in K$ and $n = 1$.
	For any $q$, the \'etale sheafification of $H^{q} \mathcal{F}_{K}$ is $H^{q} \mathcal{E}_{K}$.
	Hence $H^{0} \mathcal{E}_{K} \cong \Lambda$ and $H^{q} \mathcal{E}_{K} = 0$ for $q \ge 3$
	(as $\xi$ sheafifies to zero).
	Also \eqref{0440} and \eqref{0441} show that for $q = 1, 2$,
	the sheaf $H^{q} \mathcal{E}_{K}$ has a filtration
	whose successive subquotients are isomorphic to either
		\[
			\alg{k}^{\times} / \alg{k}^{\times p}, \quad
			\alg{k} / \alg{k}^{p}, \quad
			\Omega_{\alg{k}}^{1}, \quad
			\alg{k} / \wp \alg{k}, \quad
			\Lambda, \quad
			\alg{k}.
		\]
	All these subquotients are isomorphic to either
	$\Ga^{\N}$, $\Ga^{\oplus \N}$, $\Lambda$ or their finite product.
	Hence they belong to $\mathcal{W}_{F}$.
	Therefore $H^{q} \mathcal{E}_{K} \in \mathcal{W}_{F}$.
	Now applying Proposition \ref{0026}
	for $G = G' = \mathcal{E}_{K}$ and using
	Proposition \ref{0050},
	we prove the proposition in this case.
	
	Since $p \nmid [K(\zeta_{p}) : K]$,
	by the usual argument using norm maps,
	we can see that the proposition for $K(\zeta_{p})$ implies the proposition for $K$
	by taking a direct summand when $n = 1$.
	(This is where the closure by direct summands
	for the definition of $\genby{\mathcal{W}_{F}}_{F^{\perar}_{\et}}$ is necessary.)
	This implies the case $n \ge 1$.
\end{proof}

The proof above is showing slightly stronger results about
$R \pi_{\alg{K}, \ast} \Lambda$ (when $\zeta_{p} \in K$):

\begin{Prop} \label{0502}
	Assume $\zeta_{p} \in K$.
	Then $R^{q} \pi_{\alg{K}, \ast} \Lambda \in \mathcal{W}_{F}$ for all $q$.
	The kernel of the morphism \eqref{0358} with $n = 1$ is a connected group in $\mathcal{W}_{F}$.
\end{Prop}

We now bring the duality to $\Spec F^{\ind\rat}_{\pro\et}$:

\begin{Thm}
	Let $n \ge 1$ and $r, r' \in \Z$ with $r + r' = 2$.
	\begin{enumerate}
		\item
			We have
			$R \alg{\Gamma}(\alg{K}, \Lambda_{n}(r)) \in D^{b}(\Ind \Pro \Alg_{u} / F)$.
		\item
			The morphism
				\[
							R \alg{\Gamma}(\alg{K}, \Lambda_{n}(r))
						\tensor^{L}
							R \alg{\Gamma}(\alg{K}, \Lambda_{n}(r'))
					\to
						\Lambda_{\infty}[-2]
				\]
			obtained by applying $\algebrize$ to the morphism \eqref{0439}
			is a perfect pairing in $D(F^{\ind\rat}_{\pro\et})$.
	\end{enumerate}
\end{Thm}

\begin{proof}
	This follows from Propositions \ref{0055}, \ref{0152}, \ref{0153} and \ref{0010}.
\end{proof}


\subsection{Duality for the ring of integers}
\label{0056}

We give a version of the above duality for $\Order_{K}$ instead of $K$.
The coefficient sheaves are now $p$-adic \'etale Tate twists.

For $n \ge 1$, we have a distinguished triangle
	\[
			\mathfrak{T}_{n}(2)
		\to
			R j_{\ast} \Lambda_{n}(2)
		\to
			i_{\ast} \nu_{n}(1)[-2]
	\]
in $D(\alg{O}_{K, \et})$.
Comparing this with the distinguished triangle
	\[
			i_{\ast} R i^{!} \mathfrak{T}_{n}(2)
		\to
			\mathfrak{T}_{n}(2)
		\to
			R j_{\ast} \Lambda_{n}(2),
	\]
we know that $R i^{!} \mathfrak{T}_{n}(2) \cong i_{\ast} \nu_{n}(1)[-3]$.
Hence
	\begin{equation} \label{0499}
			R \pi_{\alg{O}_{K}, !} \mathfrak{T}_{n}(2)
		\cong
			R \pi_{\alg{k}, \ast} \nu_{n}(1)[-3]
		\cong
			\alg{k}^{\times} / \alg{k}^{\times p^{n}}[-3]
	\end{equation}
in $D(F^{\perar}_{\et})$.
Composing this with the normalized valuation map to $\Lambda$, we obtain a canonical morphism
	\begin{equation} \label{0412}
			R \pi_{\alg{O}_{K}, !} \mathfrak{T}_{n}(2)
		\to
			\Lambda_{n}[-3]
	\end{equation}
such that the composite
	\[
			R \pi_{\alg{K}, \ast} \Lambda_{n}(2)
		\to
			R \pi_{\alg{O}_{K}, !} \mathfrak{T}_{n}(2)[1]
		\to
			\Lambda_{n}[-2]
	\]
with the natural morphism is the morphism \eqref{0358}.
For any $r, r' \in \Z$, we have a product structure
	\[
			\mathfrak{T}_{n}(r) \tensor^{L} \mathfrak{T}_{n}(r')
		\to
			\mathfrak{T}_{n}(r + r')
	\]
as in \cite[Proposition 4.2.6]{Sat07}
(which is
	\[
			j_{!} \Lambda_{n}(r) \tensor^{L} \mathfrak{T}_{n}(r')
		\cong
			j_{!} \Lambda_{n}(r + r')
		\to
			\mathfrak{T}_{n}(r + r')
	\]
if $r < 0$).
Applying \eqref{0443} to this, we obtain a canonical morphism
	\begin{equation} \label{0442}
				R \pi_{\alg{O}_{K}, \ast}
				\mathfrak{T}_{n}(r)
			\tensor^{L}
				R \pi_{\alg{O}_{K}, !}
				\mathfrak{T}_{n}(r')
		\to
			R \pi_{\alg{O}_{K}, !}
			\mathfrak{T}_{n}(r + r').
	\end{equation}
Here is a duality for $\Order_{K}$:

\begin{Prop} \label{0445}
	Let $n \ge 1$ and $r, r' \in \Z$ with $r + r' = 2$.
	\begin{enumerate}
		\item
			We have
				$
						R \pi_{\alg{O}_{K}, \ast} \mathfrak{T}_{n}(r),
						R \pi_{\alg{O}_{K}, !} \mathfrak{T}_{n}(r)
					\in
						\genby{\mathcal{W}_{F}}_{F^{\perar}_{\et}}
				$,
			which are concentrated in degrees $\le 2$, $\le 3$, respectively.
		\item
			The composite morphism
				\begin{equation} \label{0446}
							R \pi_{\alg{O}_{K}, \ast} \mathfrak{T}_{n}(r)
						\tensor^{L}
							R \pi_{\alg{O}_{K}, !} \mathfrak{T}_{n}(r')
					\to
						R \pi_{\alg{O}_{K}, !} \mathfrak{T}_{n}(2)
					\to
						\Lambda_{\infty}[-3]
				\end{equation}
			of \eqref{0442} and \eqref{0412} is a perfect pairing in $D(F^{\perar}_{\et})$.
	\end{enumerate}
\end{Prop}

We will prove this below.
Since this reduces to Proposition \ref{0055} if $r > 2$
and both $R \pi_{\alg{O}_{K}, \ast} \mathfrak{T}_{n}(r)$
and $R \pi_{\alg{O}_{K}, !} \mathfrak{T}_{n}(r')$ are zero if $r < 0$,
we may assume that $r$ (and hence also $r'$) is $0, 1$ or $2$.
We again work with the Zariski topology first and then \'etale sheafify.

First, assume that $\zeta_{p} \in K$ and $n = 1$.
Let
	\begin{gather*}
				\mathcal{E}_{\Order_{K}}^{r}
			=
				R \pi_{\alg{O}_{K}, \ast} \mathfrak{T}_{n}(r),
		\quad
				\mathcal{F}_{\Order_{K}}^{r}
			=
				R \varepsilon_{\ast}
				\mathcal{E}_{\Order_{K}}^{r}.
		\\
				\mathcal{E}_{\Order_{K}}'^{r}
			=
				R \pi_{\alg{O}_{K}, !} \mathfrak{T}_{n}(r),
		\quad
				\mathcal{F}_{\Order_{K}}'^{r}
			=
				R \varepsilon_{\ast}
				\mathcal{E}_{\Order_{K}}'^{r}.
	\end{gather*}
We have a distinguished triangle
	\[
			\mathcal{F}_{\Order_{K}}'^{r}
		\to
			\mathcal{F}_{\Order_{K}}^{r}
		\to
			\mathcal{F}_{K}.
	\]
in $D(F^{\perar}_{\zar})$.

\begin{Prop} \label{0448}
	The objects $\mathcal{F}_{\Order_{K}}^{r}$ and $\mathcal{F}_{\Order_{K}}'^{r}$
	are concentrated in degrees $\le 2$ and $\le 3$, respectively.
	The composite
		\[
					\mathcal{F}_{\Order_{K}}^{r}
				\tensor^{L}
					\mathcal{F}_{\Order_{K}}'^{r'}
			\to
				R \varepsilon_{\ast} \Lambda_{\infty}[-3]
			\to
				\xi_{\infty}[-4]
		\]
	induced by \eqref{0446} is a perfect pairing in $D(F^{\perar}_{\zar})$.
\end{Prop}

\begin{proof}
	Let $R \Psi \Lambda = i^{\ast} R j_{\ast} \Lambda \in D(\alg{k}_{\et})$ and
	$R^{q} \Psi \Lambda = H^{q} R \Psi \Lambda$.
	For $q = 1, 2$,
	the filtration $U^{m} H^{q}(K, \Lambda)$ of $H^{q}(K, \Lambda)$ defines 
	a filtration $U^{m} R^{q} \Psi \Lambda$
	of the sheaf $R^{q} \Psi \Lambda$
	by varying the residue field $k$
	and taking the \'etale sheafification in $k$
	(which is the Bloch-Kato filtration \cite[Section (1.2)]{BK86}).
	By taking the \'etale sheafification in $k$ in Propositions \ref{0043} and \ref{0044},
	we know that its graded pieces $\gr^{m} R^{q} \Psi \Lambda$ are given by
		\[
				\gr^{m} R^{2} \Psi \Lambda
			\cong
				\begin{cases}
						\nu(1)
					&	\text{if }
						m = 0,
					\\
						\Ga / \Ga^{p}
					&	\text{if }
						0 < m < f_{K},\, p \mid m,
					\\
						\Omega^{1}
					&	\text{if }
						0 < m < f_{K},\, p \nmid m,
					\\
						0
					&	\text{otherwise},
				\end{cases}
		\]
		\[
				\gr^{m} R^{1} \Psi \Lambda
			\cong
				\begin{cases}
						\Lambda \oplus \nu(1)
					&	\text{if }
						m = 0,
					\\
						\Ga / \Ga^{p}
					&	\text{if }
						0 < m < f_{K},\, p \mid m,
					\\
						\Ga
					&	\text{if }
						0 < m < f_{K},\, p \nmid m,
					\\
						0
					&	\text{otherwise}
				\end{cases}
		\]
	(see also \cite[Corollary (1.4.1)]{BK86}),
	and $R^{0} \Psi \Lambda \cong \Lambda$ and $R^{q} \Psi \Lambda = 0$ for $q \ge 3$.
	
	We have a distinguished triangle
		\[
				i^{\ast} \mathfrak{T}(r)
			\to
				\tau_{\le r} R \Psi \Lambda
			\to
				\nu(r - 1)[-r]
		\]
	in $D(\alg{k}_{\et})$.
	Let $U^{m} H^{q} i^{\ast} \mathfrak{T}(r) \subset H^{q} i^{\ast} \mathfrak{T}(r)$ be
	the inverse image of $U^{m} R^{q} \Psi \Lambda$
	and $\gr^{m} H^{q} i^{\ast} \mathfrak{T}(r)$ its graded pieces.
	Then the morphism $i^{\ast} \mathfrak{T}(r) \to R \Psi \Lambda$
	and the above descriptions of $\gr^{m} R^{q} \Psi \Lambda$ induce
	the following isomorphisms for $\gr^{m} H^{q} i^{\ast} \mathfrak{T}(r)$:
		\[
			\begin{array}{c|cc}
					\gr^{m} H^{q} i^{\ast} \mathfrak{T}(1)
				&
					q = 0
				&
					q = 1
				\\ \hline
					m = 0
				&
					\Lambda
				&
					\nu(1)
				\\
					\begin{array}{c}
						0 < m < f_{K},\\ p \mid m
					\end{array}
				&
					0
				&
					\Ga / \Ga^{p}
				\\
					\begin{array}{c}
						0 < m < f_{K},\\ p \nmid m
					\end{array}
				&
					0
				&
					\Ga
			\end{array}
		\]
		\[
			\begin{array}{c|ccc}
					\gr^{m} H^{q} i^{\ast} \mathfrak{T}(2)
				&
					q = 0
				&
					q = 1
				&
					q = 2
				\\ \hline
					m = 0
				&
					\Lambda
				&
					\Lambda \oplus \nu(1)
				&
					0
				\\
					\begin{array}{c}
						0 < m < f_{K},\\ p \mid m
					\end{array}
				&
					0
				&
					\Ga / \Ga^{p}
				&
					\Ga / \Ga^{p}
				\\
					\begin{array}{c}
						0 < m < f_{K},\\ p \nmid m
					\end{array}
				&
					0
				&
					\Ga
				&
					\Omega^{1}
			\end{array}
		\]
	Of course we have $i^{\ast} \mathfrak{T}(0) \cong \Lambda$.
	The graded pieces for values of $q$ and $m$ not mentioned in the above tables are all zero.
	
	Let $U^{m} H^{q} \mathcal{F}_{\Order_{K}}^{r}$ be the inverse image of
	$U^{m} H^{q} \mathcal{F}_{K}$ via the morphism $\mathcal{F}_{\Order_{K}}^{r} \to \mathcal{F}_{K}$.
	Let $\gr^{m} H^{q} \mathcal{F}_{\Order_{K}}^{r}$ be the graded pieces for this filtration.
	Then the above tables induce the following isomorphisms for $\gr^{m} H^{q} \mathcal{F}_{\Order_{K}}^{r}$:
		\[
			\begin{array}{c|cc}
					\gr^{m} H^{q} \mathcal{F}_{\Order_{K}}^{0}
				&
					q = 0
				&
					q = 1
				\\ \hline
					m = 0
				&
					\Lambda
				&
					0
				\\
					\begin{array}{c}
						0 < m < f_{K},\\ p \mid m
					\end{array}
				&
					0
				&
					0
				\\
					\begin{array}{c}
						0 < m < f_{K},\\ p \nmid m
					\end{array}
				&
					0
				&
					0
				\\
					m = f_{K}
				&
					0
				&
					\alg{k} / \wp \alg{k}
			\end{array}
		\]
		\[
			\begin{array}{c|ccc}
					\gr^{m} H^{q} \mathcal{F}_{\Order_{K}}^{1}
				&
					q = 0
				&
					q = 1
				&
					q = 2
				\\ \hline
					m = 0
				&
					\Lambda
				&
					\alg{k}^{\times} / \alg{k}^{\times p}
				&
					0
				\\
					\begin{array}{c}
						0 < m < f_{K},\\ p \mid m
					\end{array}
				&
					0
				&
					\alg{k} / \alg{k}^{p}
				&
					0
				\\
					\begin{array}{c}
						0 < m < f_{K},\\ p \nmid m
					\end{array}
				&
					0
				&
					\alg{k}
				&
					0
				\\
					m = f_{K}
				&
					0
				&
					\alg{k} / \wp \alg{k}
				&
					\xi
			\end{array}
		\]
		\[
			\begin{array}{c|ccc}
					\gr^{m} H^{q} \mathcal{F}_{\Order_{K}}^{2}
				&
					q = 0
				&
					q = 1
				&
					q = 2
				\\ \hline
					m = 0
				&
					\Lambda
				&
					\Lambda \oplus \alg{k}^{\times} / \alg{k}^{\times p}
				&
					0
				\\
					\begin{array}{c}
						0 < m < f_{K},\\ p \mid m
					\end{array}
				&
					0
				&
					\alg{k} / \alg{k}^{p}
				&
					\alg{k} / \alg{k}^{p}
				\\
					\begin{array}{c}
						0 < m < f_{K},\\ p \nmid m
					\end{array}
				&
					0
				&
					\alg{k}
				&
					\Omega_{\alg{k}}^{1}
				\\
					m = f_{K}
				&
					0
				&
					\alg{k} / \wp \alg{k}
				&
					\alg{k} / \wp \alg{k} \oplus \xi
			\end{array}
		\]
	Again the graded pieces for values of $q$ and $m$ not mentioned in the above tables are all zero.
	From these tables, it follows that the morphism
	$H^{q} \mathcal{F}_{\Order_{K}}^{r} \to H^{q} \mathcal{F}_{K}$
	is injective for all $q$,
	and the cokernel of this injection is $H^{q + 1} \mathcal{F}_{\Order_{K}}'^{r}$.
	Also $H^{q} \mathcal{F}_{\Order_{K}}^{r}$ and $H^{q} \mathcal{F}_{\Order_{K}}'^{r}$ are
	finite successive extensions of Tate vector groups, $\Lambda$ and $\xi$.
	Therefore
		\[
				\sheafext_{F^{\perar}_{\zar}}^{j}(
					H^{q} \mathcal{F}_{\Order_{K}}^{r},
					\xi_{\infty}
				)
			=
				\sheafext_{F^{\perar}_{\zar}}^{j}(
					H^{q} \mathcal{F}_{\Order_{K}}'^{r},
					\xi_{\infty}
				)
			=
				0
		\]
	for all $j \ge 1$ by Propositions \ref{0027} and \ref{0028}.
	Proposition \ref{0430} then gives a morphism of short exact sequences
		\[
			\begin{CD}
					0
				@>>>
					H^{q} \mathcal{F}_{\Order_{K}}^{r}
				@>>>
					H^{q} \mathcal{F}_{K}
				@>>>
					H^{q + 1} \mathcal{F}_{\Order_{K}}'^{r}
				@>>>
					0
				\\ @. @VVV @VVV @VVV @. \\
					0
				@>>>
					(H^{q' + 1} \mathcal{F}_{\Order_{K}}'^{r'})^{\vee}
				@>>>
					(H^{q'} \mathcal{F}_{K})^{\vee}
				@>>>
					(H^{q'} \mathcal{F}_{\Order_{K}}^{r'})^{\vee}
				@>>>
					0
			\end{CD}
		\]
	for $q + q' = 3$, where $(\var)^{\vee}$ denotes $\sheafhom_{F^{\perar}_{\zar}}(\var, \xi_{\infty})$.
	The middle vertical morphism is an isomorphism
	by Proposition \ref{0050}.
	From the above tables of $\gr^{m} H^{q} \mathcal{F}_{\Order_{K}}^{r}$,
	we can see that $H^{q} \mathcal{F}_{\Order_{K}}^{r}$ is the exact annihilator
	of $H^{q'} \mathcal{F}_{\Order_{K}}^{r'}$ in the pairing
	$H^{q} \mathcal{F}_{K} \times H^{q'} \mathcal{F}_{K} \to \xi_{\infty}$,
	namely that the left vertical morphism is an isomorphism.
	This implies the result.
\end{proof}

\begin{proof}[Proof of Proposition \ref{0445}]
	First assume that $\zeta_{p} \in K$ and $n = 1$.
	The \'etale sheafifications of
	$H^{q} \mathcal{F}_{\Order_{K}}^{r}$
	and $H^{q} \mathcal{F}_{\Order_{K}}'^{r}$ are
	$H^{q} \mathcal{E}_{\Order_{K}}^{r}$
	and $H^{q} \mathcal{E}_{\Order_{K}}'^{r}$, respectively.
	Hence the tables in the proof of Proposition \ref{0448} show that
	$H^{q} \mathcal{E}_{\Order_{K}}^{r}$
	and $H^{q} \mathcal{E}_{\Order_{K}}'^{r}$ are in $\mathcal{W}_{F}$.
	Thus
		$
				\mathcal{E}_{\Order_{K}}^{r},
				\mathcal{E}_{\Order_{K}}'^{r}
			\in
				\genby{\mathcal{W}_{F}}_{F^{\perar}_{\et}}
		$.
	Now Propositions \ref{0448} and \ref{0026} give the result in this case.
	
	Next assume that $n = 1$.
	Let $K' = K(\zeta_{p})$, with residue field $k'$
	We may assume that the residue field of $k'$ is $F$.
	For $F' \in F^{\perar}$, let $\alg{O}_{K'}(F') = \alg{O}_{K}(F') \tensor_{\Order_{K}} \Order_{K'}$.
	Define $\Spec \alg{O}_{K', \et}$, $\Spec \alg{K}'_{\et}$, $\Spec \alg{k}'_{\et}$ and so on similarly.
	Let $f \colon \Spec \alg{O}_{K', \et} \to \Spec \alg{O}_{K, \et}$ be the natural morphism.
	Then we have canonical morphisms $\mathfrak{T}(r) \to f_{\ast} \mathfrak{T}(r)$
	and $f_{\ast}\mathfrak{T}(r) \to \mathfrak{T}(r)$
	by \cite[Theorem 1.1.2]{Sat07}.
	The composite $\mathfrak{T}(r) \to f_{\ast} \mathfrak{T}(r) \to \mathfrak{T}(r)$
	is multiplication by $[K' : K]$ by the projection formula \cite[Corollary 7.2.4]{Sat07}
	and hence an isomorphism.
	Hence $R \pi_{\alg{O}_{K}, \ast} \mathfrak{T}(r)$ and $R \pi_{\alg{O}_{K}, !} \mathfrak{T}(r)$
	are direct summands of
	$R \pi_{\alg{O}_{K'}, \ast} \mathfrak{T}(r)$ and $R \pi_{\alg{O}_{K'}, !} \mathfrak{T}(r)$,
	respectively.
	Now the same projection formula implies that the statement for $\Order_{K'}$
	implies that for $\Order_{K}$.
	
	The statement for general $n$ follows from that for $n = 1$.
\end{proof}

\begin{Thm}
	Let $n \ge 1$ and $r, r' \in \Z$ with $r + r' = 2$.
	\begin{enumerate}
		\item
			We have
				$
						R \alg{\Gamma}(\alg{O}_{K}, \Lambda_{n}(r)),
						R \alg{\Gamma}_{c}(\alg{O}_{K}, \Lambda_{n}(r)),
					\in
						D^{b}(\Ind \Pro \Alg_{u} / F)
				$.
		\item
			The morphism
				\[
							R \alg{\Gamma}(\alg{O}_{K}, \Lambda_{n}(r))
						\tensor^{L}
							R \alg{\Gamma}_{c}(\alg{O}_{K}, \Lambda_{n}(r'))
					\to
						\Lambda_{\infty}[-3]
				\]
			obtained by applying $\algebrize$ to the morphism \eqref{0446}
			is a perfect pairing in $D(F^{\ind\rat}_{\pro\et})$.
	\end{enumerate}
\end{Thm}

\begin{proof}
	This follows from Propositions \ref{0445}, \ref{0152}, \ref{0153} and \ref{0010}.
\end{proof}


\section{Two-dimensional local fields with mixed characteristic residue field}
\label{0190}

Let $k$ be a complete discrete valuation field of characteristic zero with residue field $F$.
Let $K$ be a henselian discrete valuation field with residue field $k$.
We will give a duality for $\Order_{K}$.
We already have a duality for $k$ by \cite{Suz22Duality}.
Since $k$ has characteristic zero,
we can just apply Verdier duality to pass to $\Order_{K}$.
We still need to make sure the duality statement can be stated
in the manner parallel to the duality in the previous section,
so that these two duality statements can be applied to
the local fields of arbitrary height one primes of two-dimensional local rings
in a coherent manner.
Since the statement for $k$ in \cite{Suz22Duality} treats slightly more general coefficients than $\Lambda_{n}(r)$,
we can prove a slightly more general result than needed.


\subsection{Nearby cycle duality}
\label{0191}

Let $D^{b}_{c}(k_{\et}) \subset D^{b}(k_{\et})$ be the full subcategory of objects
with constructible cohomology sheaves.

\begin{Prop} \label{0192}
	Let $N \in D^{b}_{c}(k_{\et})$
	and set $M = R \sheafhom_{k_{\et}}(N, \Q / \Z(1))$.
	View them as objects of $D(\alg{k}_{\et})$ via pullback.
	\begin{enumerate}
		\item \label{0193}
			We have $R \pi_{\alg{k}, \ast} N \in D^{b}(\Alg / F)$.
		\item \label{0194}
			There exists a canonical morphism
				\[
						R \pi_{\alg{k}, \ast} \Q / \Z(1)
					\to
						\Q / \Z[-1].
				\]
		\item \label{0195}
			The composite morphism
				\[
						R \pi_{\alg{k}, \ast} N \tensor^{L} R \pi_{\alg{k}, \ast} M
					\to
						R \pi_{\alg{k}, \ast} \Q / \Z(1)
					\to
						\Q / \Z[-1].
				\]
			is a perfect pairing in $D(F^{\perar}_{\et})$.
	\end{enumerate}
\end{Prop}

\begin{proof}
	This follows from \cite[Theorem A]{Suz22Duality},
	with $\Spec F^{\rat}_{\et}$ replaced by $\Spec F^{\perar}_{\et}$
	in the same way as $\Spec F^{\rat}_{\et}$ replaced by $\Spec F^{\ind\rat}_{\et}$
	in \cite[Theorem 2.1.5]{Suz22Duality}.
\end{proof}

Note that if $N$ has order prime to $p$,
then $R \pi_{\alg{k}, \ast} N$ is a complex of finite \'etale group schemes over $F$.

Let $j_{0} \colon \Spec K_{\et} \to \Spec \Order_{K, \et}$ and
$i_{0} \colon \Spec k_{\et} \to \Spec \Order_{K, \et}$ be the natural morphisms.
Let $N \in D^{b}_{c}(\Order_{K, \et})$.
Then $i_{0}^{\ast} N, R i_{0}^{!} N \in D^{b}_{c}(k_{\et})$.
Set $M = R \sheafhom_{\Order_{K, \et}}(N, \Q / \Z(2))$.
We have a canonical morphism
$R i_{0}^{!} \Q / \Z(2) \cong \Q / \Z(1)[-2]$.
The composite morphism
	\[
			i_{0}^{\ast} N \tensor^{L} R i_{0}^{!} M
		\to
			R i_{0}^{!} \Q / \Z(2)
		\to
			\Q / \Z(1)[-2]
	\]
is a perfect pairing in $D(k_{\et})$ by Verdier duality
(see for example \cite[Expos\'e I, Theorem 5.1]{Ill77}).
Combining this with the above proposition, we get:

\begin{Prop} \label{0059}
	Let $N \in D^{b}_{c}(\Order_{K, \et})$
	and set $M = R \sheafhom_{\Order_{K, \et}}(N, \Q / \Z(2))$.
	\begin{enumerate}
		\item \label{0196}
			We have $R \pi_{\alg{k}, \ast} i_{0}^{\ast} N, R \pi_{\alg{k}, \ast} R i_{0}^{!} M \in D^{b}(\Alg / F)$.
		\item \label{0197}
			There exists a canonical morphism
				\[
						R \pi_{\alg{k}, \ast} R i_{0}^{!} \Q / \Z(2)
					\to
						\Q / \Z[-3].
				\]
		\item \label{0198}
			The composite morphism
				\[
							R \pi_{\alg{k}, \ast} i_{0}^{\ast} N
						\tensor^{L}
							R \pi_{\alg{k}, \ast} R i_{0}^{!} M
					\to
						R \pi_{\alg{k} \ast} R i_{0}^{!} \Q / \Z(2)
					\to
						\Q / \Z[-3]
				\]
			is a perfect pairing in $D(F^{\perar}_{\et})$.
	\end{enumerate}
\end{Prop}

This contains a duality statement for $D^{b}_{c}(K_{\et})$
as a special case where $N$ is the $R j_{0, \ast}$ of an object of $D^{b}_{c}(K_{\et})$.


\subsection{Duality with relative sites}
\label{0199}

We use the notation of Section \ref{0037}.
We fix a lifting system $\alg{O}_{K}$ for $K$.
We have a commutative diagram of morphisms of sites
	\[
		\begin{CD}
				\Spec \alg{K}_{\et}
			@> j >>
				\Spec \alg{O}_{K, \et}
			@< i <<
				\Spec \alg{k}_{\et}
			\\ @VV f_{K} V @VV f_{\Order_{K}} V @VV f_{k} V \\
				\Spec K_{\et}
			@> j_{0} >>
				\Spec \Order_{K, \et}
			@< i_{0} <<
				\Spec k_{\et},
		\end{CD}
	\]
where the vertical morphisms come from
Proposition \ref{0031}.
The functors $i^{\ast}$ and $i_{0}^{\ast}$
(resp.\ $R i^{!}$ and $R i_{0}^{!}$) are compatible with each other:

\begin{Prop} \label{0200}
	The natural morphism of distinguished triangles
		\[
			\begin{CD}
					f_{k}^{\ast} R i_{0}^{!}
				@>>>
					f_{k}^{\ast} i_{0}^{\ast}
				@>>>
					f_{k}^{\ast} i_{0}^{\ast} R j_{0, \ast} j_{0}^{\ast}
				\\ @VVV @VVV @VVV \\
					R i^{!} f_{\Order_{K}}^{\ast}
				@>>>
					i^{\ast} f_{\Order_{K}}^{\ast}
				@>>>
					i^{\ast} R j_{\ast} j^{\ast} f_{\Order_{K}}^{\ast}
			\end{CD}
		\]
	is an isomorphism of triangles.
\end{Prop}

\begin{proof}
	The middle vertical morphism is obviously an isomorphism.
	The right vertical morphism is given by
		\[
				f_{k}^{\ast} i_{0}^{\ast} R j_{0, \ast} j_{0}^{\ast}
			\cong
				i^{\ast} f_{\Order_{K}}^{\ast} R j_{0, \ast} j_{0}^{\ast}
			\to
				i^{\ast} R j_{\ast} f_{K}^{\ast} j_{0}^{\ast}
			\cong
				i^{\ast} R j_{\ast} j^{\ast} f_{\Order_{K}}^{\ast}.
		\]
	Hence it is enough to show that
	$f_{\Order_{K}}^{\ast} R j_{0, \ast} \to R j_{\ast} f_{K}^{\ast}$
	is an isomorphism.
	Let $F_{1} \in F^{\perar}$ be a field
	and set $K_{1} = \alg{K}(F_{1})$.
	Consider the natural morphisms
		\[
			\begin{CD}
					\Spec K_{1, \et}
				@> j_{1} >>
					\Spec \Order_{K_{1}, \et}
				\\ @VV f_{K_{1}/ K} V @VV f_{\Order_{K_{1}} / \Order_{K}} V \\
					\Spec K_{\et}
				@> j_{0} >>
					\Spec \Order_{K, \et}.
			\end{CD}
		\]
	By Proposition \ref{0031},
	it is enough to show that
	$f_{\Order_{K_{1}} / \Order_{K}}^{\ast} R j_{0, \ast} \to R j_{1, \ast} f_{K_{1} / K}^{\ast}$
	is an isomorphism for any $F_{1}$.
	This follows from the fact that the inertia groups of $K$ and $K_{1}$ are both isomorphic to $\Hat{\Z}$.
\end{proof}

\begin{Prop} \label{0201}
	For any $N \in D(\Order_{K, \et})$, we have
		\[
				R \pi_{\alg{k}, \ast} i_{0}^{\ast} N
			\cong
				R \pi_{\alg{k}, \ast} i^{\ast} N
			\cong
				R \pi_{\alg{O}_{K}, \ast} N,
			\quad
				R \pi_{\alg{k}, \ast} R i_{0}^{!} N
			\cong
				R \pi_{\alg{k}, \ast} R i^{!} N
			\cong
				R \pi_{\alg{O}_{K}, !} N.
		\]
	(where appropriate pullback functors are omitted).
\end{Prop}

\begin{proof}
	This follows from the previous proposition.
\end{proof}

In particular, we have canonical morphisms
	\begin{equation} \label{0510}
			R \pi_{\alg{K}, \ast} \Q / \Z(2)
		\to
			R \pi_{\alg{O}_{K}, !} \Q / \Z(2)[1]
		\to
			\Q / \Z[-2]
	\end{equation}
in $D(F^{\perar}_{\et})$.
Now we can state the duality in terms of
$R \pi_{\alg{O}_{K}, \ast}$ and $R \pi_{\alg{O}_{K}, !}$:

\begin{Prop} \label{0449}
	Let $N \in D^{b}_{c}(\Order_{K, \et})$
	and set $M = R \sheafhom_{\Order_{K, \et}}(N, \Q / \Z(2))$.
	View them as objects of $D(\alg{O}_{K, \et})$ via pullback.
	\begin{enumerate}
		\item
			We have
			$R \pi_{\alg{O}_{K}, \ast} N, R \pi_{\alg{O}_{K}, !} N \in D^{b}(\Alg / F)$.
		\item
			The composite morphism
				\begin{equation} \label{0459}
							R \pi_{\alg{O}_{K}, \ast} N
						\tensor^{L}
							R \pi_{\alg{O}_{K}, !} M
					\to
						R \pi_{\alg{O}_{K}, !} \Q / \Z(2)
					\to
						\Q / \Z[-3]
				\end{equation}
			is a perfect pairing in $D(F^{\perar}_{\et})$.
	\end{enumerate}
\end{Prop}

\begin{proof}
	This follows from previous propositions.
\end{proof}

\begin{Thm}
	Let $N \in D^{b}_{c}(\Order_{K, \et})$
	and set $M = R \sheafhom_{\Order_{K, \et}}(N, \Q / \Z(2))$.
	\begin{enumerate}
		\item
			We have
				$
						R \alg{\Gamma}(\alg{O}_{K}, N),
						R \alg{\Gamma}_{c}(\alg{O}_{K}, N),
					\in
						D^{b}(\Alg / F)
				$.
		\item
			The morphism
				\[
							R \alg{\Gamma}(\alg{O}_{K}, N)
						\tensor^{L}
							R \alg{\Gamma}_{c}(\alg{O}_{K}, M)
					\to
						\Q / \Z[-3]
				\]
			obtained by applying $\algebrize$ to the morphism \eqref{0459}
			is a perfect pairing in $D(F^{\ind\rat}_{\pro\et})$.
	\end{enumerate}
\end{Thm}

\begin{proof}
	This follows from Propositions \ref{0449}, \ref{0152}, \ref{0153} and \ref{0010}.
\end{proof}


\section{Preliminary calculations on regular two-dimensional local rings}
\label{0060}

Let $A$ be a local ring with residue field $F$.
As in \cite[\S 4]{Sai86}, assume all of the following:
\begin{enumerate}
	\item \label{0202}
		$A$ is regular, excellent, two-dimensional and henselian.
	\item \label{0203}
		The fraction field $K$ of $A$ has characteristic zero.
	\item \label{0205}
		$A$ contains a primitive $p$-th root of unity $\zeta_{p}$.
	\item \label{0206}
		Either of the following holds:
		\begin{enumerate}
			\item \label{0061}
				$(p)$ is divisible by exactly one prime ideal $\ideal{p}$
				and $A / \ideal{p}$ is regular, or
			\item \label{0062}
				$(p)$ is divisible by exactly two prime ideals
				$\ideal{p}_{\alpha}$ and $\ideal{p}_{\beta}$,
				and $\ideal{p}_{\alpha} + \ideal{p}_{\beta} = \ideal{m}$,
				and both $A / \ideal{p}_{\alpha}$ and $A / \ideal{p}_{\beta}$ are regular.
		\end{enumerate}
\end{enumerate}
This is the situation we encounter after the embedded resolution
of $(\Spec A, \Spec A / \sqrt{(p)})$.
In this section, we mostly recall the results of \cite[\S 4]{Sai86}.
Basically we describe $H^{q}(A[1 / p], \Lambda)$ through
embedding into $H^{q}(K, \Lambda)$ or $H^{q}(K_{\ideal{q}}^{h}, \Lambda)$,
where we have Merkurjev-Suslin's theorem \cite{MS82}.
This description will be used as a local theory in Sections \ref{0217} and \ref{0277}

Set $R = A[1 / p]$.
For any $q \ge 0$, let $K_{q}$ be the $q$-th Quillen K-group functor
and $\Bar{K}_{q}$ the cokernel of multiplication by $p$ on $K_{q}$.
We use the notation of Section \ref{0463}.
Let $P' \subset P$ be the subset consisting of primes not dividing $p$
(so it is $P \setminus \{\ideal{p}\}$ in Case \eqref{0061}
and $P \setminus \{\ideal{p}_{\alpha}, \ideal{p}_{\beta}\}$ in Case \eqref{0062}).
In Case \eqref{0062}, for $\nu = \alpha$ or $\beta$,
we write $K_{\nu}^{h} = K_{\ideal{p}_{\nu}}^{h}$ and
$A_{\nu}^{h} = A_{\ideal{p}_{\nu}}^{h}$.


\subsection{First calculations}
\label{0310}

\begin{Prop} \label{0063}
	The map
	$H^{q}(R, \Lambda) \to H^{q}(K, \Lambda)$
	is injective for all $q$.
	We have $H^{q}(R, \Lambda) = H^{q}(K, \Lambda) = 0$ for all $q \ge 4$.
	If $F$ is algebraically closed, then
	$H^{q}(R, \Lambda) = H^{q}(K, \Lambda) = 0$ for all $q \ge 3$.
\end{Prop}

\begin{proof}
	Consider the localization exact sequence (combined with purity)
		\begin{equation} \label{0462}
				\cdots
			\to
				H^{q}(R, \Lambda)
			\to
				H^{q}(K, \Lambda)
			\to
				\bigoplus_{\ideal{q} \in P'}
				H^{q - 1}(\kappa(\ideal{q}), \Lambda)
			\to
				H^{q + 1}(R, \Lambda)
			\to
				\cdots.
		\end{equation}
	We have
		\[
				H^{q}(\kappa(\ideal{q}), \Lambda)
			\cong
				\begin{cases}
						\Lambda
					&	\text{if }
						q = 0,
					\\
						\kappa(\ideal{q})^{\times} / \kappa(\ideal{q})^{\times p}
					&	\text{if }
						q = 1,
					\\
						F_{\ideal{q}} / \wp F_{\ideal{q}}
					&	\text{if }
						q = 2,
					\\
						0
					&	\text{if }
						q \ge 3.
				\end{cases}
		\]
	In the sequence \eqref{0462},
	the injectivity of the first map $H^{q}(R, \Lambda) \to H^{q}(K, \Lambda)$ for $q \le 2$
	(and the surjectivity of the second map for $q \le 1$) is obvious.
	For $q = 3$, the boundary map
		\[
				K_{2}(K)
			\to
				\bigoplus_{\ideal{q} \in P'}
					\kappa(\ideal{q})^{\times}
		\]
	is surjective (see \cite[Chapter V, (6.6.1)]{Wei13}, \cite[Lemma (1.16)]{Sai87}).
	Hence
		$
				H^{2}(K, \Lambda)
			\to
				\bigoplus_{\ideal{q} \in P'}
					\kappa(\ideal{q})^{\times} / \kappa(\ideal{q})^{\times p}
		$
	is surjective.
	The injectivity for $q = 3$ follows.
	For $q = 4$, let $A_{\closure{F}}$ be the henselian lift
	of the algebraic closure $\closure{F}$ to $A$
	and $K_{\closure{F}}$ its fraction field.
	Then $K_{\closure{F}}$ has $p$-cohomological dimension $2$
	by \cite[Theorem (5.1)]{Sai86}.
	Hence the map
	$H^{3}(K, \Lambda) \to \bigoplus_{\ideal{q} \in P'} H^{2}(\kappa(\ideal{q}), \Lambda)$
	can be identified with the natural map
		\[
				H^{1} \bigl(
					\Gal(\closure{F} / F),
					H^{2}(K_{\closure{F}}, \Lambda(2))
				\bigl)
			\to
				H^{1} \left(
					\Gal(\closure{F} / F),
					\bigoplus_{\closure{\ideal{q}} \in \closure{P}'}
							\kappa(\closure{\ideal{q}})^{\times}
						/
							\kappa(\closure{\ideal{q}})^{\times p}
				\right),
		\]
	where $\closure{P}'$ is the set of height one primes of $A_{\closure{F}}$ not dividing $p$.
	This is surjective already before taking $H^{1}(\Gal(\closure{F} / F), \var)$,
	and hence itself surjective.
	We have $H^{4}(K, \Lambda) = 0$ by the fact on the cohomological dimension cited above.
	These imply $H^{4}(R, \Lambda) = 0$.
	The statement for algebraically closed $F$ is implicit in the above.
\end{proof}

\begin{Prop} \label{0064}
	The \'etale Chern class map \cite[Chapter V, Example 11.10]{Wei13}
		\[
				\Bar{K}_{2}(R)
			\to
				H^{2}(R, \Lambda)
		\]
	is surjective.
	In the natural commutative digram
		\[
			\begin{CD}
					\Bar{K}_{2}(R)
				@>>>
					\Bar{K}_{2}(K)
				\\ @VVV @V \wr VV \\
					H^{2}(R, \Lambda)
				@>>>
					H^{2}(K, \Lambda),
			\end{CD}
		\]
	the group $H^{2}(R, \Lambda)$ is identified with the image of the map
	$\Bar{K}_{2}(R) \to \Bar{K}_{2}(K)$.
\end{Prop}

\begin{proof}
	The isomorphism
	$\Bar{K}_{2}(K) \isomto H^{2}(K, \Lambda)$
	is Merkurjev-Suslin.
	We have a localization exact sequence
		\[
				K_{2}(R)
			\to
				K_{2}(K)
			\to
				\bigoplus_{\ideal{q} \in P'}
					\kappa(\ideal{q})^{\times}
			\to
				0
		\]
	(see \cite[Chapter V, (6.6.1)]{Wei13}, \cite[Lemma (1.16)]{Sai87}).
	Hence we have an exact sequence
		\[
				\Bar{K}_{2}(R)
			\to
				\Bar{K}_{2}(K)
			\to
				\bigoplus_{\ideal{q} \in P'}
					\kappa(\ideal{q})^{\times} / \kappa(\ideal{q})^{\times p}
			\to
				0.
		\]
	Comparing this with the localization exact sequence
		\begin{equation} \label{0065}
				0
			\to
				H^{2}(R, \Lambda)
			\to
				H^{2}(K, \Lambda)
			\to
				\bigoplus_{\ideal{q} \in P'}
					\kappa(\ideal{q})^{\times} / \kappa(\ideal{q})^{\times p}
			\to
				0
		\end{equation}
	obtained in the previous proposition, we get the result.
\end{proof}

\begin{Prop} \label{0066}
	The group $H^{2}(R, \Lambda)$ is generated by
	Steinberg symbols $\{x, y\}$ with $x, y \in R^{\times}$.
\end{Prop}

\begin{proof}
	We have localization exact sequences
		\begin{gather*}
					K_{2}(A)
				\to
					K_{2}(K)
				\to
					\bigoplus_{\ideal{q} \in P}
						\kappa(\ideal{q})^{\times}
				\to
					\Z
				\to
					0,
			\\
					K_{2}(R)
				\to
					K_{2}(K)
				\to
					\bigoplus_{\ideal{q} \in P'}
						\kappa(\ideal{q})^{\times}
				\to
					0.
		\end{gather*}
	From these, since $\Z$ is torsion-free, we have exact sequences
		\begin{gather*}
					0
				\to
					\Im(K_{2}(A) \to \Bar{K}_{2}(K))
				\to
					\Bar{K}_{2}(K)
				\to
					\bigoplus_{\ideal{q} \in P}
						\kappa(\ideal{q})^{\times} / \kappa(\ideal{q})^{\times p}
				\to
					\Lambda
				\to
					0,
			\\
					0
				\to
					\Im(K_{2}(R) \to \Bar{K}_{2}(K))
				\to
					\Bar{K}_{2}(K)
				\to
					\bigoplus_{\ideal{q} \in P'}
						\kappa(\ideal{q})^{\times} / \kappa(\ideal{q})^{\times p}
				\to
					0.
		\end{gather*}
	The snake lemma then gives an exact sequence
		\[
				0
			\to
				\Im(K_{2}(A) \to \Bar{K}_{2}(K))
			\to
				\Im(K_{2}(R) \to \Bar{K}_{2}(K))
			\to
				\bigoplus_{\ideal{q} \in P \setminus P'}
					\kappa(\ideal{q})^{\times} / \kappa(\ideal{q})^{\times p}
			\to
				\Lambda
			\to
				0.
		\]
	As $A$ is local, $K_{2}(A)$ is generated by symbols (\cite[Chapter III, Theorem 5.10.5]{Wei13}).
	Consider the subgroup of $\Im(K_{2}(R) \to \Bar{K}_{2}(K))$ generated by
	symbols of the following form:
		\[
			\begin{cases}
						\{x, \varpi\}
					\quad \text{with} \quad x \in A^{\times}
				&	\text{in Case \eqref{0061}},
				\\
						 \{x, \varpi_{\alpha}\},
						 \{x, \varpi_{\beta}\},
						 \{\varpi_{\alpha}, \varpi_{\beta}\}
					\quad \text{with} \quad x \in A^{\times}
				&	\text{in Case \eqref{0062}},
			\end{cases}
		\]
	where $\varpi$ (resp.\ $\varpi_{\alpha}, \varpi_{\beta}$) is a generator of $\ideal{p}$
	(resp.\ $\ideal{p}_{\alpha}, \ideal{p}_{\beta}$).
	This subgroup surjects onto the kernel of
		$
				\bigoplus_{\ideal{q} \in P \setminus P'}
					\kappa(\ideal{q})^{\times} / \kappa(\ideal{q})^{\times p}
			\to
				\Lambda
		$.
	Thus $\Im(K_{2}(R) \to \Bar{K}_{2}(K))$ is generated by symbols.
\end{proof}

\begin{Prop} \label{0067}
	The natural map
		\[
				H^{q}(R, \Lambda)
			\to
				\bigoplus_{\ideal{q} \in P \setminus P'}
					H^{q}(K_{\ideal{q}}^{h}, \Lambda)
		\]
	is injective for all $q \le 2$.
\end{Prop}

\begin{proof}
	The injectivity is obvious for $q = 0$.
	For $q = 1$, an element of the kernel corresponds to a finite \'etale covering of $X$
	completely decomposed at the points of $P \setminus P'$.
	The purity of branch locus (\cite[Tag 0BMA]{Sta21}) shows that
	such a covering extends to a finite \'etale covering of $\Spec A$.
	A finite \'etale covering of $\Spec A$ comes from a finite extension of $F$.
	As it is completely decomposed at the points of $P \setminus P'$,
	it is trivial.
	
	For $q = 2$, first we have $H^{2}(R, \Lambda) \cong H^{2}(R, \Gm)[p]$
	since $\Pic(R) = 0$.
	We have a localization exact sequence
		\[
				H^{2}(X, \Gm)
			\to
				H^{2}(R, \Gm)
			\to
				\bigoplus_{\ideal{q} \in P \setminus P'}
					H^{3}_{\ideal{q}}(A_{\ideal{q}}^{h}, \Gm),
		\]
	where $H^{\ast}_{\ideal{q}}$ denotes cohomology with support on the closed point.
	We have
		\[
				H^{2}(X, \Gm)[p^{\infty}]
			=
				H^{2}(A, \Gm)[p^{\infty}]
			=
				H^{2}(F, \Gm)[p^{\infty}]
			=
				0
		\]
	by the regularity of $A$ and the purity for Brauer groups
	(\cite[Theorem (6.1), Corollary (6.2)]{Gro68}).
	The second map factors as
		\[
				H^{2}(R, \Gm)
			\to
				\bigoplus_{\ideal{q} \in P \setminus P'}
					H^{2}(K_{\ideal{q}}^{h}, \Gm)
			\to
				\bigoplus_{\ideal{q} \in P \setminus P'}
					H^{3}_{\ideal{q}}(A_{\nu}^{h}, \Gm).
		\]
	Thus the map is injective.
\end{proof}

The injectivity holds for $q \ge 3$ as well.
We will see this in the next two subsections.


\subsection{Structure of cohomology by symbols: good case}
\label{0068}

Assume \eqref{0061}
at the beginning of this section.
Set $B = A / \ideal{p}$ and $k = \kappa(\ideal{p})$.
Let $\varpi$ be a generator of $\ideal{p}$,
which we take as a uniformizer for $K_{\ideal{p}}^{h}$.
For $q \ge 0$,
we have a symbol map
	\[
			(R^{\times})^{\tensor q}
		\to
			H^{q}(R, \Lambda),
		\quad
			x_{1} \tensor \dots \tensor x_{q}
		\mapsto
			\{x_{1}, \dots, x_{q}\}.
	\]
For $m \ge 1$, define $U^{m} H^{q}(R, \Lambda)$
to be the subgroup of $H^{q}(R, \Lambda)$
generated by symbols $\{x_{1}, \dots, x_{q}\}$
with $x_{1} \in 1 + \ideal{p}^{m}$.
For $m = 0$, we set $U^{m} H^{q}(R, \Lambda) = H^{q}(R, \Lambda)$.
Define
	\[
			\gr^{m} H^{q}(R, \Lambda)
		=
				U^{m} H^{q}(R, \Lambda)
			/
				U^{m + 1} H^{q}(R, \Lambda).
	\]
The map $H^{q}(R, \Lambda) \to H^{q}(K_{\ideal{p}}^{h}, \Lambda)$
maps $U^{m} H^{q}(R, \Lambda)$ into the subgroup
$U^{m} H^{q}(K_{\ideal{p}}^{h}, \Lambda)$
recalled in Section \ref{0040}.
Let $e_{A} = v_{\ideal{p}}(p)$
and set $f_{A} = p e_{A} / (p - 1)$,
where $v_{\ideal{p}}$ denotes the normalized valuation of $A_{\ideal{p}}^{h}$.

\begin{Prop} \label{0069}
	Let $m \ge 0$.
	\begin{enumerate}
		\item \label{0207}
			If $m = 0$, then
				\begin{align*}
							B^{\times} / B^{\times p}
					&	\isomto
							\gr^{m} H^{2}(R, \Lambda),
					\\
							x
					&	\mapsto
							\{\Tilde{x}, \varpi\},
				\end{align*}
			where $\Tilde{x}$ denotes any lift of $x$ to $A$.
		\item \label{0208}
			If $0 < m < f_{A}$ and $p \mid m$, then
				\begin{align*}
							B / B^{p}
					&	\isomto
							\gr^{m} H^{2}(R, \Lambda),
					\\
							x
					&	\mapsto
							\{1 + \Tilde{x} \varpi^{m}, \varpi\}.
				\end{align*}
		\item \label{0209}
			If $0 < m < f_{A}$ and $p \nmid m$, then
				\begin{align*}
							\Omega_{B}^{1}
					&	\isomto
							\gr^{m} H^{2}(R, \Lambda),
					\\
							x \dlog(y)
					&	\mapsto
							\{1 + \Tilde{x} \varpi^{m}, \Tilde{y}\},
				\end{align*}
			where $x \in B$ and $y \in B^{\times}$.
		\item \label{0210}
			If $m = f_{A}$, then
				\begin{align*}
							\xi_{1}(F)
					&	\isomto
							\gr^{m} H^{2}(R, \Lambda),
					\\
							x
					&	\mapsto
							\{1 + \Tilde{x} (\zeta_{p} - 1)^{p}, \varpi\},
				\end{align*}
			where $\Tilde{x}$ denotes any lift of $x$ to $A$.
		\item \label{0211}
			If $m > f_{A}$, then $\gr^{m} H^{2}(R, \Lambda) = 0$.
	\end{enumerate}
\end{Prop}

\begin{proof}
	The maps are well-defined and surjective by classical calculations
	and Propositions \ref{0064}
	and \ref{0066}.
	They are injective by the injectivity result in
	Proposition \ref{0067}
	and by comparison with
	Proposition \ref{0043}.
\end{proof}

\begin{Prop} \label{0070}
	Let $m \ge 0$.
	\begin{enumerate}
		\item \label{0212}
			If $m = 0$, then
				\begin{align*}
							\Lambda \oplus B^{\times} / B^{\times p}
					&	\isomto
							\gr^{m} H^{1}(R, \Lambda),
					\\
							(i, x)
					&	\mapsto
							\{\varpi^{i} \Tilde{x}\}.
				\end{align*}
		\item \label{0213}
			If $0 < m < f_{A}$ and $p \mid m$, then
				\begin{align*}
							B / B^{p}
					&	\isomto
							\gr^{m} H^{1}(R, \Lambda),
					\\
							x
					&	\mapsto
							\left\{
									1
								+
									\Tilde{x}
									\frac{
										(\zeta_{p} - 1)^{p}
									}{
										\varpi^{f_{A} - m}
									}
							\right\}.
				\end{align*}
		\item \label{0214}
			If $0 < m < f_{A}$ and $p \nmid m$, then
				\begin{align*}
							B
					&	\isomto
							\gr^{m} H^{1}(R, \Lambda),
					\\
							x
					&	\mapsto
							\left\{
								1 + \Tilde{x} \frac{(\zeta_{p} - 1)^{p}}{\varpi^{f_{A} - m}}
							\right\}.
				\end{align*}
		\item \label{0215}
			If $m = f_{\ideal{p}}$, then
				\begin{align*}
							\xi_{1}(F)
					&	\isomto
							\gr^{m} H^{1}(R, \Lambda),
					\\
							x
					&	\mapsto
							\{1 + \Tilde{x} (\zeta_{p} - 1)^{p}\}.
				\end{align*}
		\item \label{0216}
			If $m > f_{A}$, then $\gr^{m} H^{1}(R, \Lambda) = 0$.
	\end{enumerate}
\end{Prop}

\begin{proof}
	Similar to (and easier than) the previous proposition.
\end{proof}

\begin{Prop} \label{0071}
	Let $q = 1$ or $2$ and $m \ge 0$.
	Then $H^{q}(R, \Lambda) \cap U^{m} H^{q}(K_{\ideal{p}}^{h}, \Lambda) = U^{m} H^{q}(R, \Lambda)$.
\end{Prop}

\begin{proof}
	By the previous propositions,
	the natural map
	$\gr^{m} H^{q}(R, \Lambda) \to \gr^{m} H^{q}(K_{\ideal{p}}^{h}, \Lambda)$
	is injective.
	This implies the result.
\end{proof}

\begin{Prop} \label{0503}
	Consider the maps
		\[
				H^{2}(R, \Lambda)
			\to
				H^{2}(K_{\ideal{p}}^{h}, \Lambda)
			\to
				\Lambda,
		\]
	where the first map is the natural one and the second \eqref{0437}.
	Their composite is zero.
\end{Prop}

\begin{proof}
	This follows from Propositions \ref{0069} \eqref{0207} and \ref{0043} \eqref{0175}.
\end{proof}

\begin{Prop} \label{0072}
	We have $H^{3}(R, \Lambda) = 0$.
\end{Prop}

\begin{proof}
	Using the same notation as the proof of
	Proposition \ref{0063},
	for $R_{\closure{F}} := R \tensor_{A} A_{\closure{F}}$,
	we have
		\[
				H^{3}(R, \Lambda)
			\cong
				H^{1}(\Gal \bigl(
					\closure{F} / F),
					H^{2}(R_{\closure{F}}, \Lambda)
				\bigr).
		\]
	We have
		\[
				\gr^{m} H^{2}(R_{\closure{F}}, \Lambda)
			\cong
				\begin{cases}
						B_{\closure{F}}^{\times} / (B_{\closure{F}}^{\times})^{p}
					&	\text{if }
						m = 0,
					\\
						B_{\closure{F}} / B_{\closure{F}}^{p}
					&	\text{if }
						0 < m < f_{A},\, p \mid m,
					\\
						\Omega_{B_{\closure{F}}}^{1}
					&	\text{if }
						0 < m < f_{A},\, p \nmid m,
					\\
						0
					&	\text{if }
						m \ge f_{A}
				\end{cases}
		\]
	by Proposition \ref{0069}.
	Applying $H^{1}(\Gal(\closure{F} / F), \var)$ to these graded pieces for $m > 0$ yields zero.
	For $m = 0$, we have
		\[
				H^{1} \bigl(
					\Gal(\closure{F} / F),
					B_{\closure{F}}^{\times} / (B_{\closure{F}}^{\times})^{p}
				\bigr)
			\cong
				H^{2}(B, \Gm)[p]
			\cong
				H^{2}(F, \Gm)[p]
			=
				0.
		\]
	The result then follows.
\end{proof}

In particular, the map
$H^{q}(R, \Lambda) \to H^{q}(K_{\ideal{p}}^{h}, \Lambda)$ is injective for all $q$.

Let $\Hat{A}$ be the completion of $A$.
We have corresponding rings $\Hat{R} = R \tensor_{A} \Hat{A}$,
 $\Hat{B} = \Hat{A} / \ideal{p} \Hat{A}$ and
$\Hat{k} = \Hat{\kappa}(\ideal{p})$.
Let $\Hat{K}_{\ideal{p}}$ be the complete local field of $\Hat{A}$ at $\ideal{p} \Hat{A}$.
Its residue field is $\Hat{k}$.
Even though $H^{q}(R, \Lambda)$ changes by replacing $A$ by $\Hat{A}$,
the quotient $H^{q}(K_{\ideal{p}}, \Lambda) / H^{q}(R, \Lambda)$ does not:

\begin{Prop} \label{0073}
	For any $q$, the natural map
		\[
				\frac{
					H^{q}(K_{\ideal{p}}, \Lambda)
				}{
					H^{q}(R, \Lambda)
				}
			\to
				\frac{
					H^{q}(\Hat{K}_{\ideal{p}}, \Lambda)
				}{
					H^{q}(\Hat{R}, \Lambda)
				}
		\]
	is an isomorphism.
\end{Prop}

\begin{proof}
	This is obvious for $q = 0$.
	For $q = 3$, this follows from
	Propositions \ref{0041}
	and \ref{0072}.
	
	Let $q = 1$ or $2$.
	By Proposition \ref{0071},
	the image of $U^{m} H^{q}(K_{\ideal{p}}, \Lambda)$
	in $H^{q}(K_{\ideal{p}}, \Lambda) / H^{q}(R, \Lambda)$ defines a finite filtration
	whose graded pieces are given by
		\[
			\frac{
				\gr^{m} H^{q}(K_{\ideal{p}}, \Lambda)
			}{
				\gr^{m} H^{q}(R, \Lambda)
			}.
		\]
	The same is true for
	$H^{q}(\Hat{K}_{\ideal{p}}, \Lambda) / H^{q}(\Hat{R}, \Lambda)$.
	We need to see that the natural map
		\[
				\frac{
					\gr^{m} H^{q}(K_{\ideal{p}}, \Lambda)
				}{
					\gr^{m} H^{q}(R, \Lambda)
				}
			\to
				\frac{
					\gr^{m} H^{q}(\Hat{K}_{\ideal{p}}, \Lambda)
				}{
					\gr^{m} H^{q}(\Hat{R}, \Lambda)
				}
		\]
	is an isomorphism.
	For this, by Propositions \ref{0069} and \ref{0070},
	it suffices to observe that the cokernels of the natural maps
		\[
				B^{\times} / B^{\times p} \into k^{\times} / k^{\times p},
			\quad
				B / B^{p} \into k / k^{p},
			\quad
				\Omega_{B}^{1} \into \Omega_{k}^{1},
			\quad
				B \into k
		\]
	do not change by replacing $B$ by $\Hat{B}$ and $k$ by $\Hat{k}$.
\end{proof}


\subsection{Structure of cohomology by symbols: bad case}
\label{0074}

Next, assume \eqref{0062}.
Below we generally follow the notation of \cite[Section 4]{Sai86}.
For $\nu = \alpha$ and $\beta$,
let $\varpi_{\nu}$ be a generator of $\ideal{p}_{\nu}$.
Set $\Bar{A}_{\nu} = A / \ideal{p}_{\nu}$ and
$k_{\nu} = \kappa(\ideal{p}_{\nu})$.
Let $e_{\nu}$ be the normalized valuation of $p$ in $K_{\nu}^{h}$
and set $f_{\nu} = p e_{\nu} / (p - 1)$.
Set $\nu' = \alpha$ if $\nu = \beta$ and $\nu' = \beta$ if $\nu = \alpha$.
For an integer $j \ge 1$, let $\Tilde{A}_{\nu}(j)$ be
the image of $\varpi_{\nu'}^{- j} \Bar{A}_{\nu}$ in
$k_{\nu} / \wp k_{\nu}$.
Let $\Tilde{U}_{k_{\nu}}^{(j)}$ be the image of
$1 + \varpi_{\nu'}^{j} \Bar{A}_{\nu}$ in
$k_{\nu}^{\times} / k_{\nu}^{\times p}$.

Define $H^{1}(R, \Lambda)_{\alpha}$ to be the kernel of the map
from $H^{1}(R, \Lambda)$ to $H^{1}(K_{\beta}^{h}, \Lambda)$.
Define $H^{1}(R, \Lambda)_{\beta}$ to be the quotient of
$H^{1}(R, \Lambda)$ by $H^{1}(R, \Lambda)_{\alpha}$.
Define $H^{2}(R, \Lambda)_{\beta}$ to be the kernel of the map
from $H^{2}(R, \Lambda)$ to $H^{2}(K_{\alpha}^{h}, \Lambda)$.
Define $H^{2}(R, \Lambda)_{\alpha}$ to be the quotient of
$H^{2}(R, \Lambda)$ by $H^{2}(R, \Lambda)_{\beta}$.
We have a commutative diagram with exact rows
	\[
		\begin{CD}
				0
			@>>>
				H^{q}(R, \Lambda)_{\nu}
			@>>>
				H^{q}(R, \Lambda)
			@>>>
				H^{q}(R, \Lambda)_{\nu'}
			@>>>
				0
			\\ @. @VVV @VVV @VVV @. \\
				0
			@>>>
				H^{q}(K_{\nu}^{h}, \Lambda)
			@>>>
					H^{q}(K_{\alpha}^{h}, \Lambda)
				\oplus
					H^{q}(K_{\beta}^{h}, \Lambda)
			@>>>
				H^{q}(K_{\nu'}^{h}, \Lambda)
			@>>>
				0,
		\end{CD}
	\]
where $q = 2$ (resp.\ $q = 1$)
if $\nu = \beta$ (resp.\ $\nu = \alpha$).

For $q = 1$ or $2$ and $\nu =\alpha$ or $\beta$,
the group $H^{q}(R, \Lambda)_{\nu}$ agrees with the groups $\Delta_{\nu}^{q}$
in the notation of \cite[(4.9)]{Sai86}
by Propositions \ref{0064}
and \ref{0066}.
For $m \ge 0$, define
	\begin{gather*}
				U^{m} H^{q}(R, \Lambda)_{\nu}
			=
				H^{q}(R, \Lambda)_{\nu} \cap U^{m} H^{q}(K_{\nu}^{h}, \Lambda),
		\\
				\gr^{m} H^{q}(R, \Lambda)_{\nu}
			=
				U^{m} H^{q}(R, \Lambda)_{\nu} / U^{m + 1} H^{q}(R, \Lambda)_{\nu}.
	\end{gather*}
We have an injection $\gr^{m} H^{q}(R, \Lambda)_{\nu} \into \gr^{m} H^{q}(K_{\nu}^{h}, \Lambda)$.

\begin{Prop}[{\cite[(4.11)]{Sai86}}] \label{0075}
	In Propositions \ref{0043}
	and \ref{0044},
	take $K = K_{\alpha}^{h}$, $\varpi = \varpi_{\alpha}$,
	$b_{m} = \varpi_{\alpha}^{m / p}$ and $c_{m} = \varpi_{\alpha}^{m} / \varpi_{\beta}$.
	Then, under the isomorphisms in
	Propositions \ref{0043}
	and \ref{0044},
	the image of $\gr^{m} H^{q}(R, \Lambda)_{\alpha} \into \gr^{m} H^{q}(K_{\alpha}^{h}, \Lambda)$
	is given as follows:
		\[
				\gr^{m} H^{2}(R, \Lambda)_{\alpha}
			\cong
				\begin{cases}
							k_{\alpha}^{\times} / k_{\alpha}^{\times p}
					&	\text{if }
							m = 0,
					\\
							\Bar{A}_{\alpha} / \Bar{A}_{\alpha}^{p}
					&	\text{if }
							0 < m < f_{\alpha},\ p \mid m,
					\\
							\Omega_{\Bar{A}_{\alpha}}^{1}
					&	\text{if }
							0 < m < f_{\alpha},\ p \nmid m,
					\\
							\xi(F) \oplus \Tilde{A}_{\alpha}(f_{\beta})
					&	\text{if }
							m = f_{\alpha},
				\end{cases}
		\]
		\[
				\gr^{m} H^{1}(R, \Lambda)_{\alpha}
			\cong
				\begin{cases}
							0 \oplus \Tilde{U}_{k_{\alpha}}^{(f_{\beta} + 1)}
					&	\text{if }
							m = 0,
					\\
							\Bar{A}_{\alpha} / \Bar{A}_{\alpha}^{p}
					&	\text{if }
							0 < m < f_{\alpha},\ p \mid m,
					\\
							\Bar{A}_{\alpha}
					&	\text{if }
							0 < m < f_{\alpha},\ p \nmid m,
					\\
							0
					&	\text{if }
							m = f_{\alpha},
				\end{cases}
		\]
\end{Prop}

\begin{Prop}[{\cite[(4.11)]{Sai86}}] \label{0076}
	In Propositions \ref{0043}
	and \ref{0044},
	take $K = K_{\beta}^{h}$, $\varpi = \varpi_{\beta}$,
	$b_{m} = \varpi_{\beta}^{m / p} \varpi_{\alpha}^{f_{\alpha} / p}$ and
	$c_{m} = \varpi_{\beta}^{m} \varpi_{\alpha}^{f_{\alpha}}$.
	Then, under the isomorphisms in
	Propositions \ref{0043}
	and \ref{0044},
	the image of $\gr^{m} H^{q}(R, \Lambda)_{\beta} \into \gr^{m} H^{q}(K_{\beta}^{h}, \Lambda)$
	is given as follows:
		\[
				\gr^{m} H^{2}(R, \Lambda)_{\beta}
			\cong
				\begin{cases}
							\Tilde{U}_{k_{\beta}}^{(f_{\alpha} + 1)}
					&	\text{if }
							m = 0,
					\\
							\Bar{A}_{\beta} / \Bar{A}_{\beta}^{p}
					&	\text{if }
							0 < m < f_{\beta},\ p \mid m,
					\\
							\Omega_{\Bar{A}_{\beta}}^{1}
					&	\text{if }
							0 < m < f_{\beta},\ p \nmid m,
					\\
							0
					&	\text{if }
							m = f_{\beta},
				\end{cases}
		\]
		\[
				\gr^{m} H^{1}(R, \Lambda)_{\beta}
			\cong
				\begin{cases}
							\Lambda \oplus k_{\beta}^{\times} / k_{\beta}^{\times p}
					&	\text{if }
							m = 0,
					\\
							\Bar{A}_{\beta} / \Bar{A}_{\beta}^{p}
					&	\text{if }
							0 < m < f_{\beta},\ p \mid m,
					\\
							\Bar{A}_{\beta}
					&	\text{if }
							0 < m < f_{\beta},\ p \nmid m,
					\\
							\Tilde{A}_{\beta}(f_{\alpha})
					&	\text{if }
							m = f_{\beta}.
				\end{cases}
		\]
\end{Prop}

\begin{Prop} \label{0504}
	Consider the maps
		\[
				H^{2}(R, \Lambda)
			\to
				\bigoplus_{\nu = \alpha, \beta}
					H^{2}(K_{\nu}^{h}, \Lambda)
			\to
				\Lambda^{2},
		\]
	where the first map is the natural one and the second \eqref{0437}.
	Their composite surjects onto the subgroup
	$(\Lambda^{2})_{0} \subset \Lambda^{2}$ of elements with zero sum.
\end{Prop}

\begin{proof}
	By Propositions \ref{0075}, \ref{0076} and \ref{0043} \eqref{0175},
	the image is a rank one subgroup.
	Hence it is enough to give an element of $H^{2}(R, \Lambda)$
	that maps to $(- 1, 1) \in \Lambda^{2}$.
	Such an element is given by the symbol $\{\varpi_{\alpha}, \varpi_{\beta}\}$.
\end{proof}

\begin{Prop} \label{0077}
	For $\nu = \alpha$ and $\beta$,
	the map $H^{3}(R, \Lambda) \to H^{3}(K_{\nu}^{h}, \Lambda)$
	is an isomorphism.
\end{Prop}

\begin{proof}
	It is enough to show this for $\nu = \alpha$.
	The map can be written as
		\[
				H^{1} \bigl(
					\Gal(\closure{F} / F),
					H^{2}(R_{\closure{F}}, \Lambda)
				\bigr)
			\to
				H^{1} \bigl(
					\Gal(\closure{F} / F),
					H^{2}((K_{\alpha}^{h})_{\closure{F}}, \Lambda)
				\bigr).
		\]
	By Propositions \ref{0075}
	and \ref{0076},
	the group $H^{2}(R_{\closure{F}}, \Lambda)$ has a filtration
	whose graded pieces are isomorphic to either
		\begin{equation} \label{0478}
			\Bar{A}_{\nu}^{\ur} / \Bar{A}_{\nu}^{\ur p},\;
			\Omega_{\Bar{A}_{\nu}^{\ur}}^{1},\;
			\Tilde{A}_{\nu}^{\ur}(m),\;
			\Tilde{U}_{k_{\nu}}^{\ur (m)},\;
			k_{\nu}^{\ur \times} / k_{\nu}^{\ur \times p}
		\end{equation}
	for some $\nu$ and $m$, where these are defined as the groups
		\[
			\Bar{A}_{\nu} / \Bar{A}_{\nu}^{p},\;
			\Omega_{\Bar{A}_{\nu}}^{1},\;
			\Tilde{A}_{\nu}(m),\;
			\Tilde{U}_{k_{\nu}}^{(m)},\;
			k_{\nu}^{\times} / k_{\nu}^{\times p}
		\]
	with $A$ replaced by $A_{\closure{F}}$.
	Some classical calculations show that
	$H^{1}(\Gal(\closure{F} / F), \var)$ of the first four groups in \eqref{0478} is zero.
	The map
		\[
				H^{2}(R_{\closure{F}}, \Lambda)
			\onto
					k_{\alpha}^{\ur \times}
				/
					(k_{\alpha}^{\ur \times})^{p}
		\]
	is given by the tame symbol at $\ideal{p}_{\alpha}$.
	This, with the valuation map at $\ideal{p}_{\alpha}$, induces an isomorphism
		\[
				H^{1} \bigl(
					\Gal(\closure{F} / F),
					H^{2}(R_{\closure{F}}, \Lambda)
				\bigr)
			\isomto
				H^{1} \bigl(
					\Gal(\closure{F} / F),
					\Lambda
				\bigr)
			\cong
				F / \wp F.
		\]
	Similarly, we have
		\[
				H^{1} \bigl(
					\Gal(\closure{F} / F),
					H^{2}((K_{\alpha}^{h})_{\closure{F}}, \Lambda)
				\bigr)
			\isomto
				F / \wp F.
		\]
	As these isomorphisms are compatible,
	the result follows.
\end{proof}

In particular, the map
$H^{q}(R, \Lambda) \to \bigoplus_{\nu = \alpha, \beta} H^{q}(K_{\nu}^{h}, \Lambda)$
is injective for all $q$.

Let $\Hat{A}$ be the completion of $A$ and set $\Hat{R} = R \tensor_{A} \Hat{A}$.
Let $\Hat{K}_{\nu}$ be the complete local field of $\Hat{A}$ at $\ideal{p}_{\nu} \Hat{A}$.

\begin{Prop} \label{0078}
	For any $q$, the natural map
		\[
				\frac{
					\bigoplus_{\nu = \alpha, \beta}
						H^{q}(K_{\nu}^{h}, \Lambda)
				}{
					H^{q}(R, \Lambda)
				}
			\to
				\frac{
					\bigoplus_{\nu = \alpha, \beta}
						H^{q}(\Hat{K}_{\nu}, \Lambda)
				}{
					H^{q}(\Hat{R}, \Lambda)
				}
		\]
	is an isomorphism.
\end{Prop}

\begin{proof}
	This is obvious for $q = 0$.
	For $q = 3$, this follows from
	Propositions \ref{0041}
	and \ref{0077}.
	
	Let $q = 1$ or $2$.
	The image of $\bigoplus_{\nu = \alpha, \beta} U^{m} H^{q}(K_{\nu}^{h}, \Lambda)$
	in
		$
			(
				\bigoplus_{\nu = \alpha, \beta}
					H^{q}(K_{\nu}^{h}, \Lambda)
			) / H^{q}(R, \Lambda)
		$
	defines a finite filtration whose graded pieces are given by
	$\gr^{m} H^{q}(K_{\alpha}^{h}, \Lambda) / \gr^{m} H^{q}(R, \Lambda)_{\alpha}$
	and $\gr^{m} H^{q}(K_{\beta}^{h}, \Lambda) / \gr^{m} H^{q}(R, \Lambda)_{\beta}$.
	Hence it is enough to see that for $\nu = \alpha, \beta$ and $m \ge 1$, the groups
		\[
				k_{\nu} / \Bar{A}_{\nu},
			\quad
				k_{\nu} / (k_{\nu}^{p} + \Bar{A}_{\nu}),
			\quad
				\Omega_{k_{\nu}}^{1} / \Omega_{\Bar{A}_{\nu}}^{1},
			\quad
				\xi(k_{\nu}) / \Tilde{A}_{\nu}(m),
			\quad
				(k_{\nu}^{\times} / k_{\nu}^{\times p}) / \Tilde{U}_{k_{\nu}}^{(m)}
		\]
	depend only on the completion $\Hat{A}$.
	But this is obvious.
\end{proof}


\section{Tubular neighborhoods of affine curves}
\label{0217}

In this section,
we first give a duality for cohomology of smooth affine curves over $F$
with coefficients in vector bundles, $\Lambda$ and $\nu(1)$
(Section \ref{0333}).
This is more or less the classical coherent duality,
but we need to formulate it
in the style of fibered sites of Section \ref{0325}.
With this duality formulated in this style,
we give a duality for cohomology of $p$-adic tubular neighborhoods of smooth affine curves
for the rest of the section.
We build relative and fibered sites associated with tubular neighborhoods
and define nearby cycle functors in this setting
in Section \ref{0340}.
Then we introduce filtrations on nearby cycles by symbols
and prove a duality for the zero-th graded piece (over the relative \'etale site of the curve)
in Section \ref{0227}.
For the other graded pieces, in Section \ref{0250},
we work over the relative Nisnevich site of the curve to prove the duality.
In Section \ref{0482},
we combine these two pieces of the duality,
yielding the desired duality for cohomology of tubular neighborhoods.

\subsection{Relative sites for affine curves and duality}
\label{0333}

Let $B$ be a smooth geometrically connected $F$-algebra of dimension $1$
with function field $k$.
Set $V = \Spec B$.
Let $Y$ be the smooth compactification of $V$.
Set $T = Y \setminus V$.
For each $x \in T$,
let $\Hat{\Order}_{k_{x}}$ be the completed local ring of $Y$ at $x$,
with residue field $F_{x}$ and fraction field $\Hat{k}_{x}$.

For $F' \in F^{\perar}$, set
	\begin{gather*}
				\alg{B}(F')
			=
				B \tensor_{F} F',
		\\
				\Hat{\alg{O}}_{k_{x}}(F')
			=
				\Hat{\Order}_{k_{x}} \ctensor_{F} F',
		\\
				\Hat{\alg{k}}_{x}(F')
			=
				\Hat{\alg{O}}_{k_{x}}(F') \tensor_{\Hat{\Order}_{k_{x}}} \Hat{k}_{x},
	\end{gather*}
Note that $\Hat{\tensor}$ in the second line is over $F$ and not over $F_{x}$.
We have another functor
	\[
			F_{x}'
		\mapsto
			(\Hat{\Order}_{k_{x}} \ctensor_{F_{x}} F_{x}') \tensor_{\Hat{\Order}_{k_{x}}} \Hat{k}_{x}
	\]
on $F_{x}^{\perar}$,
which is the functor \eqref{0450} and
whose Weil restriction $\Weil_{F_{x} / F}$ to $F^{\perar}$ is the above $\Hat{\alg{k}}_{x}$.
The functors $\alg{B}$, $\Hat{\alg{O}}_{k_{x}}$ and $\Hat{\alg{k}}_{x}$ are $F^{\perar}$-algebras,
and we have morphisms $\alg{B} \to \Hat{\alg{k}}_{x} \gets \Hat{\alg{O}}_{k_{x}}$ of $F^{\perar}$-algebras.
Hence, for $\tau = \et$ or $\nis$, we have sites
$\Spec \alg{B}_{\tau}$ and $\Spec \Hat{\alg{k}}_{x, \tau}$ and morphisms of sites
	\[
			\pi_{\Hat{\alg{k}}_{x}, \tau}
		\colon
			\Spec \Hat{\alg{k}}_{x, \tau}
		\stackrel{\pi_{\Hat{\alg{k}}_{x} / \alg{B}, \tau}}{\to}
			\Spec \alg{B}_{\tau}
		\stackrel{\pi_{\alg{B}, \tau}}{\to}
			\Spec F^{\perar}_{\tau}.
	\]
We apply the constructions in Section \ref{0325} to
the morphisms
	\[
			\bigsqcup_{x \in T}
				\Spec \Hat{\alg{k}}_{x, \tau}
		\to
			\Spec \alg{B}_{\tau}
		\to
			\Spec F^{\perar}_{\tau}.
	\]
Denote the total site
$(\bigsqcup_{x \in T} \Spec \Hat{\alg{k}}_{x, \tau} \to \Spec \alg{B}_{\tau})$
by $\Spec \alg{B}_{\Hat{c}, \tau}$.
(The hat for the subscript $c$ emphasizes
the complete local field $\Hat{k}_{x}$ rather than the henselian local field.)
The morphisms $\pi_{\alg{B}, \tau}$ and $\pi_{\Hat{\alg{k}}_{x}, \tau}$ induce morphisms
	\[
			\Bar{\pi}_{\alg{B}, \tau},
			\Bar{\pi}_{\Hat{\alg{k}}_{x}, \tau}
		\colon
			\Spec \alg{B}_{\Hat{c}, \tau}
		\to
			\Spec F^{\perar}_{\tau}.
	\]
We have a functor
	\[
			\Bar{\pi}_{\alg{B}, \tau, \Hat{!}}
		=
			\left[
					\Bar{\pi}_{\alg{B}, \tau, \ast}
				\to
					\bigoplus_{x \in T}
						\Bar{\pi}_{\Hat{\alg{k}}_{x}, \tau, \ast}
			\right][-1]
		\colon
			\Ch(\alg{B}_{\Hat{c}, \tau})
		\to
			\Ch(F^{\perar}_{\tau}).
	\]
We have a morphism
	\begin{equation} \label{0342}
				R \Bar{\pi}_{\alg{B}, \tau, \ast} G
			\tensor^{L}
				R \Bar{\pi}_{\alg{B}, \tau, \Hat{!}} H
		\to
			R \Bar{\pi}_{\alg{B}, \tau, \Hat{!}}(G \tensor^{L} H)
	\end{equation}
in $D(F^{\perar}_{\tau})$ functorial in $G, H \in D(\alg{B}_{\Hat{c}, \tau})$.
Let $\varepsilon \colon \Spec \alg{B}_{\Hat{c}, \et} \to \Spec \alg{B}_{\Hat{c}, \nis}$
be the morphism defined by the identity functor.

The sheaf $\Ga$ on $\Spec \alg{B}_{\et}$ and on $\Spec \Hat{\alg{k}}_{x, \et}$ naturally extends to
a sheaf on $\Spec \alg{B}_{\Hat{c}, \et}$:
it is the triple
	\[
		\left(
			\Ga,
			(\Ga)_{x \in T},
				\Ga
			\to
				\bigoplus_{x \in T}
					\pi_{\Hat{\alg{k}}_{x} / \alg{B}, \ast} \Ga
		\right).
	\]
We similarly have sheaves $\Gm$, $\Omega^{1}$, $\nu(1)$ and $M$ (a quasi-coherent sheaf on $B$).
We have exact sequences
	\begin{gather} \label{0484}
				0
			\to
				\nu(1)
			\to
				\Omega^{1}
			\stackrel{C - 1}{\to}
				\Omega^{1}
			\to
				0,
		\\ \notag
				0
			\to
				\Gm
			\stackrel{p}{\to}
				\Gm
			\to
				\nu(1)
			\to
				0
	\end{gather}
in $\Ab(\alg{B}_{\Hat{c}, \et})$, where $C$ is the Cartier operator.
Hence $R^{n} \varepsilon_{\ast} \nu(1) = 0$ for $n \ge 2$.
Define
	\[
			\xi(1)
		=
			R^{1} \varepsilon_{\ast} \nu(1).
	\]
We have an exact sequence
	\begin{equation} \label{0337}
			0
		\to
			\nu(1)
		\to
			\Omega^{1}
		\stackrel{C - 1}{\to}
			\Omega^{1}
		\to
			\xi(1)
		\to
			0
	\end{equation}
in $\Ab(\alg{B}_{\Hat{c}, \nis})$.
Define
	\[
			\Omega_{\alg{B}}^{1}
		=
			\pi_{\alg{B}, \ast} \Omega^{1},
		\quad
			\Omega_{\Hat{\alg{k}}_{x}}^{1}
		=
			\pi_{\Hat{\alg{k}}_{x}, \ast} \Omega^{1},
	\]
which are Tate vector groups associated with
$\Omega_{B}^{1}$, $\Omega_{k_{x}}^{1}$, respectively.
Now we calculate the sheaves
$R^{q} \pi_{\alg{B}, \ast} \nu(1)$ and $R^{q} \Bar{\pi}_{\alg{B}, \Hat{!}} \nu(1)$
and define a trace morphism (or the residue map) in the \'etale topology:

\begin{Prop} \label{0080}
	We have $R^{q} \pi_{\alg{B}, \ast} \nu(1) = R^{q} \pi_{\Hat{\alg{k}}_{x}, \ast} \nu(1) = 0$
	for $q \ge 1$.
	We have an isomorphism
		$
				\pi_{\Hat{\alg{k}}_{x}, \ast} \nu(1)
			\cong
				\Hat{\alg{k}}_{x}^{\times} / \Hat{\alg{k}}_{x}^{\times p}
		$
	and a commutative diagram with exact rows
		\[
			\begin{CD}
					0
				@>>>
					\Pic_{Y / F}^{0}[p]
				@>>>
					\pi_{\alg{B}, \ast} \nu(1)
				@>>>
					\left(
						\bigoplus_{x \in T}
							\Weil_{F_{x} / F} \Lambda
					\right)_{0}
				@>>>
					0
				\\ @. @VVV @VVV @VV \mathrm{incl} V @. \\
					0
				@>>>
					\bigoplus_{x \in T}
						\Hat{\alg{O}}_{k_{x}}^{\times} / \Hat{\alg{O}}_{k_{x}}^{\times p}
				@>> \mathrm{incl} >
					\bigoplus_{x \in T}
						\pi_{\Hat{\alg{k}}_{x}, \ast} \nu(1)
				@>>>
					\bigoplus_{x \in T}
						\Weil_{F_{x} / F} \Lambda
				@>>>
					0,
			\end{CD}
		\]
	where the left upper term $\Pic_{Y / F}^{0}[p]$ is
	(the perfection of) the part of the Jacobian of $Y$ killed by $p$,
	the right upper term $(\bigoplus_{x \in T} \Weil_{F_{x} / F} \Lambda)_{0}$ is the kernel
	of the sum of the norm maps
	$\bigoplus_{x \in T} \Weil_{F_{x} / F} \Lambda \onto \Lambda$,
	and the right lower horizontal morphism
		$
				\pi_{\Hat{\alg{k}}_{x}, \ast} \nu(1)
			\to
				\bigoplus_{x \in T}
					\Weil_{F_{x} / F} \Lambda
		$
	is the valuation map in each factor.
\end{Prop}

\begin{proof}
	The sheaf $R^{q} \pi_{\alg{B}, \ast} \nu(1)$ is
	the \'etale sheafification of the presheaf
	$F' \mapsto H^{q}(\alg{B}(F'), \nu(1))$.
	This presheaf as a functor in $F'$ commutes with filtered direct limits.
	Let $\alg{Y}(F') = Y \times_{F} F'$
	and $\alg{T}(F') = \alg{Y}(F') \setminus \Spec \alg{B}(F')$.
	Consider the localization distinguished triangles
		\[
				R \Gamma(\alg{Y}(F'), \nu(1))
			\to
				R \Gamma(\alg{B}(F'), \nu(1))
			\to
				\bigoplus_{x \in \alg{T}(F')}
					R \Gamma_{x}(\Order_{\alg{Y}(F'), x}^{h}, \nu(1))[1]
		\]
	If $F'$ is algebraically closed, then
	$H^{q}(\alg{Y}(F'), \nu(1))$
	is isomorphic to $\Pic(\alg{Y}(F'))[p]$ if $q = 0$;
	$\Lambda$ if $q = 1$; and zero otherwise.
	Hence the \'etale sheafifications of the presheaves
	$F' \mapsto H^{q}(\alg{Y}(F'), \nu(1))$ are $\Pic^{0}_{Y / F}[p]$ if $q = 0$;
	$\Lambda$ if $q = 1$; and zero otherwise.
	Similarly, the \'etale sheafification of
	$F' \mapsto H_{x}^{q + 1}(\Order_{\alg{Y}(F'), x}^{h}, \nu(1))$
	is $\bigoplus_{x \in T} \Weil_{F_{x} / F} \Lambda$ if $q = 0$
	and zero otherwise.
	Via these isomorphisms, the connecting morphism
	$\bigoplus_{x \in T} \Weil_{F_{x} / F} \Lambda \to \Lambda$
	is given by the sum of the norm maps.
	Hence we obtain the vanishing
	$R^{q} \pi_{\alg{B}, \ast} \nu(1) = 0$ for $q \ge 1$
	and the upper exact sequence of the statement.
	
	On the other hand,
	we have $R^{q} \pi_{\Hat{\alg{k}}_{x}, \ast} \nu(1) = 0$ for $q \ge 1$
	by applying $\Weil_{F_{x} / F}$ to the results of Proposition \ref{0451}.
	We have a commutative diagram
		\[
			\begin{CD}
					\Gamma(\alg{Y}(F'), \nu(1))
				@>>>
					\Gamma(\alg{B}(F'), \nu(1))
				\\ @VVV @VVV \\
					\Gamma(\Hat{\alg{O}}_{k_{x}}(F'), \nu(1))
				@>>>
					\Gamma(\Hat{\alg{k}}_{x}(F'), \nu(1))
			\end{CD}
		\]
	for any $F' \in F^{\perar}$.
	From this, we obtain the desired commutative diagram.
\end{proof}

\begin{Prop} \label{0356}
	We have $R^{q} \Bar{\pi}_{\alg{B}, \Hat{!}} \nu(1) = 0$ for $q \ne 1$.
	The distinguished triangle
		\[
				R \Bar{\pi}_{\alg{B}, \Hat{!}} \nu(1)
			\to
				R \Bar{\pi}_{\alg{B}, \Hat{!}} \Omega^{1}
		\stackrel{C - 1}{\to}
				R \Bar{\pi}_{\alg{B}, \Hat{!}} \Omega^{1}
		\]
	in $D(F^{\perar}_{\et})$ reduces to an exact sequence
		\[
				0
			\to
				R^{1} \Bar{\pi}_{\alg{B}, \Hat{!}} \nu(1)
			\to
				\frac{
					\bigoplus_{x \in T}
						\Omega_{\Hat{\alg{k}}_{x}}^{1}
				}{
					\Omega_{\alg{B}}^{1}
				}
			\stackrel{C - 1}{\to}
				\frac{
					\bigoplus_{x \in T}
						\Omega_{\Hat{\alg{k}}_{x}}^{1}
				}{
					\Omega_{\alg{B}}^{1}
				}
			\to
				0
		\]
	in $\Ab(F^{\perar}_{\et})$.
\end{Prop}

\begin{proof}
	It is enough to show that
	$C - 1$ is surjective on $\Omega_{\Hat{\alg{k}}_{x}}^{1}$.
	But this follows from
	$R^{1} \pi_{\Hat{\alg{k}}_{x}, \ast} \nu(1) = 0$.
\end{proof}

The sum of the residue maps $\Res \colon \Omega_{\Hat{\alg{k}}_{x}}^{1} \to \Ga$ over $x \in T$
is zero on $\Omega_{\alg{B}}^{1}$ by the residue theorem.
Hence we have a morphism between exact sequences
	\[
		\begin{CD}
				0
			@>>>
				R^{1} \Bar{\pi}_{\alg{B}, !} \nu(1)
			@>>>
				\frac{
					\bigoplus_{x \in T}
						\Omega_{\Hat{\alg{k}}_{x}}^{1}
				}{
					\Omega_{\alg{B}}^{1}
				}
			@> C - 1 >>
				\frac{
					\bigoplus_{x \in T}
						\Omega_{\Hat{\alg{k}}_{x}}^{1}
				}{
					\Omega_{\alg{B}}^{1}
				}
			@>>>
				0
			\\ @. @VVV @V \Res VV @V \Res VV \\
				0
			@>>>
				\Lambda
			@>>>
				\Ga
			@>> \Frob^{-1} - 1 >
				\Ga
			@>>>
				0.
		\end{CD}
	\]
With the inclusion $\Lambda \into \Lambda_{\infty}$, we have morphisms
	\begin{equation} \label{0350}
			R^{1} \Bar{\pi}_{\alg{B}, \Hat{!}} \nu(1)
		\to
			\Lambda_{\infty},
		\quad
			R \Bar{\pi}_{\alg{B}, \Hat{!}} \nu(1)
		\to
			\Lambda_{\infty}[-1]
	\end{equation}
in $\Ab(F^{\perar}_{\et})$, $D(F^{\perar}_{\et})$, respectively.
Here is a slightly different description of this morphism:

\begin{Prop} \label{0357}
	The sum of the residue maps
	$\bigoplus_{x \in T} \Bar{\pi}_{\Hat{\alg{k}}_{x}, \ast} \nu(1) \to \Lambda$
	annihilates the image of $\Bar{\pi}_{\alg{B}, \ast} \nu(1)$.
	The obtained morphism $R^{1} \Bar{\pi}_{\alg{B}, \Hat{!}} \nu(1) \to \Lambda_{\infty}$
	via the exact sequence
		\[
				\Bar{\pi}_{\alg{B}, \ast} \nu(1)
			\to
				\bigoplus_{x \in T} \Bar{\pi}_{\Hat{\alg{k}}_{x}, \ast} \nu(1)
			\to
				R^{1} \Bar{\pi}_{\alg{B}, \Hat{!}} \nu(1)
			\to
				0
		\]
	is the morphism \eqref{0350}.
\end{Prop}

\begin{proof}
	Obvious.
\end{proof}

Now we calculate
$R^{q} \pi_{\alg{B}, \nis, \ast} \xi(1)$ and $R^{q} \Bar{\pi}_{\alg{B}, \nis, \Hat{!}} \xi(1)$
and define a trace morphism in the Nisnevich topology.
Proposition \ref{0356} also shows that
$R^{q} \varepsilon_{\ast} R^{1} \Bar{\pi}_{\alg{B}, \Hat{!}} \nu(1) = 0$ for $q \ge 2$,
and we have an isomorphism and a morphism
	\begin{equation} \label{0338}
			R^{1} \varepsilon_{\ast} R^{1} \Bar{\pi}_{\alg{B}, \Hat{!}} \nu(1)
		\cong
			\Coker \left(
					C - 1
				\text{ on }
					\frac{
						\bigoplus_{x \in T}
							\Omega_{\Hat{\alg{k}}_{x}}^{1}
					}{
						\Omega_{\alg{B}}^{1}
					}
			\right)
		\stackrel{\Res}{\to}
			\xi
	\end{equation}
in $\Ab(F^{\perar}_{\zar})$ (remember $\xi = \Ga / (\Frob - 1) \Ga$).

\begin{Prop}
	For any $G \in \Ab(\alg{B}_{\Hat{c}, \nis})$ and $q \ge 2$, we have
		\[
				R^{q} \Bar{\pi}_{\alg{B}, \nis, \ast} G
			=
				R^{q - 1} \Bar{\pi}_{\Hat{\alg{k}}_{x}, \nis, \ast} G
			=
				R^{q} \Bar{\pi}_{\alg{B}, \nis, \Hat{!}} G
			=
				0.
		\]
\end{Prop}

\begin{proof}
	The statement $R^{q - 1} \Bar{\pi}_{\Hat{\alg{k}}_{x}, \nis, \ast} G = 0$ is trivial
	since fields have trivial Nisnevich cohomology.
	For $R^{q} \Bar{\pi}_{\alg{B}, \nis, \ast} G$,
	it is enough to see that $H^{q}(B_{\nis}, G) = 0$ for $q \ge 2$.
	But this follows from the localization sequence
	since $B$ is one-dimensional.
\end{proof}

In particular, the functor $R^{1} \Bar{\pi}_{\alg{B}, \nis, \Hat{!}}$ is right-exact.
The exact sequence \eqref{0337} and the residue map induce an isomorphism and a morphism
	\begin{equation} \label{0339}
			R^{1} \Bar{\pi}_{\alg{B}, \nis, \Hat{!}} \xi(1)
		\cong
			\Coker \left(
					C - 1
				\text{ on }
					\frac{
						\bigoplus_{x \in T}
							\Omega_{\Hat{\alg{k}}_{x}}^{1}
					}{
						\Omega_{\alg{B}}^{1}
					}
			\right)
		\stackrel{\Res}{\to}
			\xi
	\end{equation}
in $\Ab(F^{\perar}_{\zar})$.
With the inclusion $\xi \into \xi_{\infty}$, we have morphisms
	\begin{equation} \label{0354}
			R^{1} \Bar{\pi}_{\alg{B}, \nis, \Hat{!}} \xi(1)
		\to
			\xi_{\infty},
		\quad
			R \Bar{\pi}_{\alg{B}, \nis, \Hat{!}} \xi(1)
		\to
			\xi_{\infty}[-1]
	\end{equation}
in $\Ab(F^{\perar}_{\zar})$, $D(F^{\perar}_{\zar})$, respectively.
The two isomorphisms \eqref{0338} and \eqref{0339} are compatible with
the natural isomorphism between
$R^{1} \varepsilon_{\ast} R^{1} \Bar{\pi}_{\alg{B}, \Hat{!}} \nu(1)$
and $R^{1} \Bar{\pi}_{\alg{B}, \nis, \Hat{!}} \xi(1)$ by construction.

We have a representability result for the sheaves at hand:

\begin{Prop} \label{0082}
	Let $G \in \Ab(\alg{B}_{\Hat{c}, \et})$ be either
	$\Lambda$, $\nu(1)$ or a finite projective $B$-module.
	Then
		$
				R \Bar{\pi}_{\alg{B}, \ast} G,
				R \Bar{\pi}_{\alg{B}, \Hat{!}} G
			\in
				\genby{\mathcal{W}_{F}}_{F^{\perar}_{\et}}
		$.
	Moreover, $R^{q} \Bar{\pi}_{\alg{B}, \ast} G \in \mathcal{W}_{F}$ for all $q$.
\end{Prop}

\begin{proof}
	If $G = M$ is a finite projective $B$-module,
	then $R \Bar{\pi}_{\alg{B}, \ast} M = \pi_{\alg{B}, \ast} M$ is
	the underlying $F$-vector group $M$ of countable dimension.
	Hence it is in $\mathcal{W}_{F}$.
	For each $x \in T$,
	the object
		\[
				R \Bar{\pi}_{\Hat{\alg{k}}_{x}, \ast}(M \tensor_{B} \Hat{k}_{x})
			=
				\pi_{\Hat{\alg{k}}_{x}, \ast}(M \tensor_{B} \Hat{k}_{x})
		\]
	is the underlying Tate $F$-vector group $M \tensor_{B} \Hat{k}_{x} \in \mathcal{W}_{F}$.
	The object
		\[
				R \Bar{\pi}_{\alg{B}, \Hat{!}} M
			=
				\frac{
					\bigoplus_{x \in T}
						M \tensor_{B} \Hat{k}_{x}
				}{
					M
				}[-1],
		\]
	in degree $1$, is a profinite-dimensional $F$-vector group whose dual has countable dimension.
	Hence it is in $\mathcal{W}_{F}$.
	These imply
		$
				R \Bar{\pi}_{\alg{B}, \ast} G,
				R \Bar{\pi}_{\alg{B}, \Hat{!}} G
			\in
				\genby{\mathcal{W}_{F}}_{F^{\perar}_{\et}}
		$
	by the exact sequences
	$0 \to \Lambda \to \Ga \to \Ga \to 0$
	and $0 \to \nu(1) \to \Omega^{1} \to \Omega^{1} \to 0$.
	For $R^{q} \Bar{\pi}_{\alg{B}, \ast} \nu(1) \in \mathcal{W}_{F}$,
	one uses Proposition \ref{0080},
	and for $R^{q} \Bar{\pi}_{\alg{B}, \ast} \Lambda \in \mathcal{W}_{F}$,
	it is enough to note that $\alg{B} / \wp \alg{B}$ is the filtered union of
	the quasi-algebraic groups
		\[
			\frac{
				\Gamma(Y, \Order_{Y}(p n T)) \tensor_{F} \Ga
			}{
				\wp \bigl(
					\Gamma(Y, \Order_{Y}(n T)) \tensor_{F} \Ga
				\bigr)
			}
		\]
	over $n \ge 0$.
\end{proof}

Now we give duality results for $\Spec B$:

\begin{Prop} \label{0086}
	Let $M$ be a finite projective $B$-module.
	Set $N = \Hom_{B\mathrm{-module}}(M, \Omega_{B}^{1})$.
	View them as sheaves of abelian groups on $\Spec \alg{B}_{\Hat{c}, \nis}$.
	Consider the pairing $M \times N \to \Omega^{1} \onto \xi(1)$.
	Then the induced morphism
		\[
					R \Bar{\pi}_{\alg{B}, \nis, \ast} M
				\tensor^{L}
					R \Bar{\pi}_{\alg{B}, \nis, \Hat{!}} N
			\to
				R \Bar{\pi}_{\alg{B}, \nis, \Hat{!}} \xi(1)
			\to
				\xi_{\infty}[-1]
		\]
	in $D(F^{\perar}_{\zar})$ is a perfect pairing.
\end{Prop}

\begin{proof}
	We have
		\[
				R \Bar{\pi}_{\alg{B}, \nis, \ast} M
			\cong
				M \tensor_{B} \alg{B},
			\quad
				R \Bar{\pi}_{\alg{B}, \nis, \Hat{!}} N
			\cong
				\frac{
					\bigoplus_{x \in T}
						N \tensor_{B} \Hat{\alg{k}}_{x}
				}{
					N \tensor_{B} \alg{B}
				}[-1].
		\]
	The pairing (shifted by $1$) can alternatively be given by the natural morphisms
		\[
					M \tensor_{B} \alg{B}
				\times
					\frac{
						\bigoplus_{x \in T}
							N \tensor_{B} \Hat{\alg{k}}_{x}
					}{
						N \tensor_{B} \alg{B}
					}
			\to
				\frac{
					\bigoplus_{x \in T}
						\Omega_{\Hat{\alg{k}}_{x}}^{1}
				}{
					\Omega_{\alg{B}}^{1}
				}
			\stackrel{\Res}{\longrightarrow}
				\Ga
			\to
				\xi_{\infty}.
		\]
	The pairing
		\[
					M
				\times
					\frac{
						\bigoplus_{x \in T}
							N \tensor_{B} \Hat{k}_{x}
					}{
						N
					}
			\to
				\frac{
					\bigoplus_{x \in T}
						\Omega_{\Hat{k}_{x}}^{1}
				}{
					\Omega_{B}^{1}
				}
			\stackrel{\Res}{\longrightarrow}
				F
		\]
	is a perfect pairing of Tate vector spaces over $F$ by coherent duality.
	Hence the result follows from
	Proposition \ref{0027}.
\end{proof}

\begin{Prop} \label{0483}
	Let $M$ be a finite projective $B$-module.
	Set $N = \Hom_{B\mathrm{-module}}(M, \Omega_{B}^{1})$.
	View them as sheaves of abelian groups on $\Spec \alg{B}_{\Hat{c}, \et}$.
	Consider the pairing $M \times N \to \Omega^{1} \onto \nu(1)[1]$,
	where the last morphism is the connecting morphism for \eqref{0484}.
	Then the induced morphism
		\[
					R \Bar{\pi}_{\alg{B}, \ast} M
				\tensor^{L}
					R \Bar{\pi}_{\alg{B}, \Hat{!}} N
			\to
				R \Bar{\pi}_{\alg{B}, \Hat{!}} \nu(1)[1]
			\to
				\Lambda_{\infty}
		\]
	in $D(F^{\perar}_{\et})$ is a perfect pairing.
\end{Prop}

\begin{proof}
	This follows from Propositions \ref{0086}, \ref{0082} and \ref{0026}.
\end{proof}

\begin{Prop} \label{0087}
	Let $(G, H)$ be either pair of objects
	$(\Lambda, \nu(1))$ or $(\nu(1), \Lambda)$
	of $\Ab(\alg{B}_{\Hat{c}, \et})$.
	Consider the natural pairing $G \times H \to \nu(1)$.
	Then the induced morphism
		\[
					R \Bar{\pi}_{\alg{B}, \ast} G
				\tensor^{L}
					R \Bar{\pi}_{\alg{B}, \Hat{!}} H
			\to
				R \Bar{\pi}_{\alg{B}, \Hat{!}} \nu(1)
			\to
				\Lambda_{\infty}[-1]
		\]
	in $D(F^{\perar}_{\et})$ is a perfect pairing.
\end{Prop}

\begin{proof}
	Using the exact sequences
	$0 \to \Lambda \to \Ga \to \Ga \to 0$
	and $0 \to \nu(1) \to \Omega^{1} \to \Omega^{1} \to 0$,
	this reduces to Proposition \ref{0483}.
\end{proof}


\subsection{Relative sites for tubular neighborhoods}
\label{0340}

Let $A$ be a ring.
Assume all of the following:
\begin{enumerate}
	\item \label{0485}
		$A$ is a two-dimensional regular integral domain.
	\item \label{0486}
		$A$ contains a primitive $p$-th root of unity $\zeta_{p}$.
	\item \label{0487}
		The radical $I$ of the ideal $(p)$ of $A$ is principal.
	\item \label{0488}
		$B := A / I$ is a one-dimensional geometrically connected smooth algebra over $F$.
	\item \label{0489}
		The pair $(A, I)$ is complete.
\end{enumerate}
This means that $\Spec A$ is a $p$-adic tubular neighborhood of the curve $\Spec B$.
The conditions \eqref{0486} and \eqref{0487} are for simplicity.
We will apply the results of this section to two-dimensional local rings in the next section
after enough reduction steps where these conditions are satisfied.
In particular, the coefficient sheaf for the duality will be just $\Lambda$, not $\Lambda_{n}(r)$.
We do not try to give best general results for tubular neighborhoods of their own.

Note that $A$ is excellent by \cite[Proposition 2.3]{Gre82} and \cite[Main Theorem 1]{KS21}.
Let $\varpi$ be a generator of $I$.
Set $R = A[1 / p]$.
We use the same notation as Section \ref{0333}
applied to $B$.
For example, let $Y$ be the smooth compactification of $V := \Spec B$.

We build relative sites for $A$.
For any $F' \in F^{\perar}$,
the $B$-algebra $\alg{B}(F') = B \tensor_{F} F'$ is relatively perfect.
Define $\Hat{\alg{A}}(F')$ to be the Kato canonical lifting of $\alg{B}(F')$ over $A$.
If $F'$ has only one direct factor,
then it satisfies the same conditions above as $A$,
with $F$, $B$, $I$ replaced by $F'$, $\alg{B}(F')$ and $I \Hat{\alg{A}}(F')$, respectively.
Set $\Hat{\alg{R}}(F') = \Hat{\alg{A}}(F') \tensor_{A} R$.
For a point $x'$ of $\Spec \alg{B}(F')$ ($\subset \Spec \Hat{\alg{A}}(F')$),
let $\Hat{\alg{A}}(F')_{x'}^{h}$ and $\Hat{\alg{A}}(F')_{x'}^{sh}$ be
the henselian and strictly henselian local rings, respectively, of $\Hat{\alg{A}}(F')$ at $x'$.
They are two-dimensional (resp.\ one-dimensional) excellent regular local rings
if $x'$ is a closed (resp.\ generic) point of $\Spec B'$.
Set $\Hat{\alg{R}}(F')_{x'}^{h} = \Hat{\alg{A}}(F')_{x'}^{h} \tensor_{A} R$
and $\Hat{\alg{R}}(F')_{x'}^{sh} = \Hat{\alg{A}}(F')_{x'}^{sh} \tensor_{A} R$.

For any closed point $x \in Y$,
the complete local field $\Hat{k}_{x}$ of $Y$ at $x$ is relatively perfect over $B$.
Define $\Hat{\Order}_{K_{\eta_{x}}}$ to be the canonical lifting of $\Hat{k}_{x}$ over $A$.
It is a complete discrete valuation ring with prime element $\varpi$ and residue field $\Hat{k}_{x}$.
Let $\Hat{K}_{\eta_{x}}$ be the fraction field of $\Hat{\Order}_{K_{\eta_{x}}}$.
For any $F' \in F^{\perar}$,
set $\Hat{\alg{k}}_{x}(F') = \Hat{k}_{x} \Hat{\tensor}_{F} F'$ as in
Section \ref{0333}.
It is relatively perfect over $\Hat{k}_{x}$.
Define $\Hat{\alg{O}}_{K_{\eta_{x}}}(F')$ to be the canonical lifting of
$\Hat{\alg{k}}_{x}(F')$ over $\Hat{\Order}_{K_{\eta_{x}}}$.
There exists a unique $A$-algebra homomorphism $\Hat{\alg{A}}(F') \to \Hat{\alg{O}}_{K_{\eta_{x}}}(F')$
whose reduction $(\var) \tensor_{A} B$ is the natural map $\alg{B}(F') \to \Hat{\alg{k}}_{x}(F')$.
Set $\Hat{\alg{K}}_{\eta_{x}}(F') = \Hat{\alg{O}}_{K_{\eta_{x}}}(F') \tensor_{\Hat{\Order}_{K_{\eta_{x}}}} \Hat{K}_{\eta_{x}}$.
(The strange notation ``$\eta_{x}$'' is for consistency
when the results of this section are applied in Section \ref{0303}.)
For a point $x'$ of $\Spec \Hat{\alg{k}}_{x}(F')$ ($\subset \Spec \Hat{\alg{O}}_{K_{\eta_{x}}}(F')$),
let $\Hat{\alg{O}}_{K_{\eta_{x}}}(F')_{x'}^{h}$ and $\Hat{\alg{O}}_{K_{\eta_{x}}}(F')_{x'}^{sh}$ be
the henselian and strictly henselian local rings, respectively, of $\Hat{\alg{O}}_{K_{\eta_{x}}}(F')$ at $x'$.
They are discrete valuation rings
with residue fields $\Hat{\alg{k}}_{x}(F')_{x'}$ and $\Hat{\alg{k}}_{x}(F')_{x'}^{\sep}$, respectively.
(Note that $\Hat{\alg{k}}_{x}(F')$ is a finite product of complete discrete valuation fields.)
Set $\Hat{\alg{K}}_{\eta_{x}}(F')_{x'}^{h} = \Hat{\alg{O}}_{K_{\eta_{x}}}(F')_{x'}^{h} \tensor_{A} R$
and $\Hat{\alg{K}}_{\eta_{x}}(F')_{x'}^{sh} = \Hat{\alg{O}}_{K_{\eta_{x}}}(F')_{x'}^{sh} \tensor_{A} R$.

We have a commutative diagram
	\[
		\begin{CD}
				B
			@>>>
				\alg{B}(F')
			\\ @VVV @VVV \\
				\Hat{k}_{x}
			@>>>
				\Hat{\alg{k}}_{x}(F')
		\end{CD}
	\]
All the morphisms are relatively perfect.
The Kato canonical lifting of this diagram to $A$ is
	\[
		\begin{CD}
				A
			@>>>
				\Hat{\alg{A}}(F')
			\\ @VVV @VVV \\
				\Hat{\Order}_{K_{\eta_{x}}}
			@>>>
				\Hat{\alg{O}}_{K_{\eta_{x}}}(F'),
		\end{CD}
	\]
The base change $(\var) \tensor_{A} R$ of this latter diagram gives a commutative digram
	\[
		\begin{CD}
				R
			@>>>
				\Hat{\alg{R}}(F')
			\\ @VVV @VVV \\
				\Hat{K}_{\eta_{x}}
			@>>>
				\Hat{\alg{K}}_{\eta_{x}}(F').
		\end{CD}
	\]

The functors
$\alg{B}, \Hat{\alg{A}}, \Hat{\alg{R}}, \Hat{\alg{k}}_{x}, \Hat{\alg{O}}_{K_{\eta_{x}}}, \Hat{\alg{K}}_{\eta_{x}}$
are $F^{\perar}$-algebras.
The one $\Hat{\alg{O}}_{K_{\eta_{x}}}$ is the canonical lifting system for $\Hat{\Order}_{K_{\eta_{x}}}$.
We have a natural commutative diagram
	\[
		\begin{CD}
				\Hat{\alg{K}}_{\eta_{x}}
			@<<<
				\Hat{\alg{O}}_{K_{\eta_{x}}}
			@>>>
				\Hat{\alg{k}}_{x}
			\\ @AAA @AAA @AAA \\
				\Hat{\alg{R}}
			@<<<
				\Hat{\alg{A}}
			@>>>
				\alg{B}.
		\end{CD}
	\]
It defines a commutative diagram
	\[
		\begin{CD}
				\Spec \Hat{\alg{K}}_{\eta_{x}, \et}
			@> j_{x} >>
				\Spec \Hat{\alg{O}}_{K_{\eta_{x}}, \et}
			@< i_{x} <<
				\Spec \Hat{\alg{k}}_{x, \et}
			\\
			@V \pi_{\Hat{\alg{K}}_{\eta_{x}} / \Hat{\alg{R}}} VV
			@V \pi_{\Hat{\alg{O}}_{K_{\eta_{x}}} / \Hat{\alg{A}}} VV
			@VV \pi_{\Hat{\alg{k}}_{x} / \alg{B}} V
			\\
				\Spec \Hat{\alg{R}}_{\et}
			@>> j >
				\Spec \Hat{\alg{A}}_{\et}
			@<< i <
				\Spec \alg{B}_{\et}
		\end{CD}
	\]
of morphisms of sites.
Define $\Spec \Hat{\alg{R}}_{\Hat{c}, \et}$ and $\Spec \Hat{\alg{A}}_{\Hat{c}, \et}$ to be the total sites
	\[
			\left(
					\bigsqcup_{x \in T}
						\Spec \Hat{\alg{K}}_{\eta_{x}, \et}
				\to
					\Spec \Hat{\alg{R}}_{\et}
			\right)
		\quad \text{and} \quad
			\left(
					\bigsqcup_{x \in T}
						\Spec \Hat{\alg{O}}_{K_{\eta_{x}}, \et}
				\to
					\Spec \Hat{\alg{A}}_{\et}
			\right),
	\]
respectively (and $\Spec \alg{B}_{\Hat{c}, \et}$ as before).
The morphisms $\pi_{\Hat{\alg{R}}}$ and $\pi_{\Hat{\alg{K}}_{\eta_{x}}}$ induce morphisms
	\[
			\Bar{\pi}_{\Hat{\alg{R}}},
			\Bar{\pi}_{\Hat{\alg{K}}_{\eta_{x}}}
		\colon
			\Spec \Hat{\alg{R}}_{\Hat{c}, \et}
		\to
			\Spec F^{\perar}_{\et},
	\]
and the morphisms $\pi_{\Hat{\alg{A}}}$ and $\pi_{\Hat{\alg{O}}_{K_{\eta_{x}}}}$ induce morphisms
	\[
			\Bar{\pi}_{\Hat{\alg{A}}},
			\Bar{\pi}_{\Hat{\alg{O}}_{K_{\eta_{x}}}}
		\colon
			\Spec \Hat{\alg{A}}_{\Hat{c}, \et}
		\to
			\Spec F^{\perar}_{\et}.
	\]
We have functors
	\begin{gather*}
				\Bar{\pi}_{\Hat{\alg{R}}, \Hat{!}}
			=
				\left[
						\Bar{\pi}_{\Hat{\alg{R}}, \ast}
					\to
						\bigoplus_{x \in T}
							\Bar{\pi}_{\Hat{\alg{K}}_{\eta_{x}}, \ast}
				\right][-1]
			\colon
				\Ch(\Hat{\alg{R}}_{\Hat{c}, \et})
			\to
				\Ch(F^{\perar}_{\et}),
		\\
				\Bar{\pi}_{\Hat{\alg{A}}, \Hat{!}}
			=
				\left[
						\Bar{\pi}_{\Hat{\alg{A}}, \ast}
					\to
						\bigoplus_{x \in T}
							\Bar{\pi}_{\Hat{\alg{O}}_{K_{\eta_{x}}}, \ast}
				\right][-1]
			\colon
				\Ch(\Hat{\alg{A}}_{\Hat{c}, \et})
			\to
				\Ch(F^{\perar}_{\et}).
	\end{gather*}
We have a morphism
	\begin{equation} \label{0344}
				R \Bar{\pi}_{\Hat{\alg{R}}, \ast} G
			\tensor^{L}
				R \Bar{\pi}_{\Hat{\alg{R}}, \Hat{!}} H
		\to
			R \Bar{\pi}_{\Hat{\alg{R}}, \Hat{!}}(G \tensor^{L} H)
	\end{equation}
in $D(F^{\perar}_{\et})$ functorial in $G, H \in D(\Hat{\alg{R}}_{\Hat{c}, \et})$
and a morphism
	\[
				R \Bar{\pi}_{\Hat{\alg{A}}, \ast} G
			\tensor^{L}
				R \Bar{\pi}_{\Hat{\alg{A}}, \Hat{!}} H
		\to
			R \Bar{\pi}_{\Hat{\alg{A}}, \Hat{!}}(G \tensor^{L} H)
	\]
in $D(F^{\perar}_{\et})$ functorial in $G, H \in D(\Hat{\alg{A}}_{\Hat{c}, \et})$.

We will study the duality for
$R \Bar{\pi}_{\Hat{\alg{R}}, \ast}$ and $R \Bar{\pi}_{\Hat{\alg{R}}, \Hat{!}}$.
For this, we use the following nearby cycle constructions.
The pairs $(j, j_{x})_{x \in T}$ and $(i, i_{x})_{x \in T}$ of morphisms induce morphisms
	\[
			\Spec \Hat{\alg{R}}_{\Hat{c}, \et}
		\stackrel{\Bar{\jmath}}{\to}
			\Spec \Hat{\alg{A}}_{\Hat{c}, \et}
		\stackrel{\Bar{\imath}}{\gets}
			\Spec \alg{B}_{\Hat{c}, \et},
	\]
respectively.
Define
	\begin{gather*}
					\Psi
				=
					i^{\ast} j_{\ast}
			\colon
				\Ab(\Hat{\alg{R}}_{\et})
			\to
				\Ab(\alg{B}_{\et}),
		\\
					\Psi_{x}
				=
					i_{x}^{\ast} j_{x, \ast}
			\colon
				\Ab(\Hat{\alg{K}}_{\eta_{x}, \et})
			\to
				\Ab(\Hat{\alg{k}}_{x, \et}),
		\\
					\Bar{\Psi}
				=
					\Bar{\imath}^{\ast} \Bar{\jmath}_{\ast}
			\colon
				\Ab(\Hat{\alg{R}}_{\Hat{c}, \et})
			\to
				\Ab(\alg{B}_{\Hat{c}, \et}).
	\end{gather*}
Their derived functors are described as follows.

\begin{Prop} \label{0349}
	Let $\Bar{G} \in D(\Hat{\alg{R}}_{\Hat{c}, \et})$.
	Let $G$ and $G_{x}$ (where $x \in T$) be the natural projections
	to $\Spec \Hat{\alg{R}}_{\et}$ and $\Spec \Hat{\alg{K}}_{\eta_{x}, \et}$, respectively.
	Then for any $q \in \Z$,
	the sheaf $R^{q} \Bar{\Psi} \Bar{G} \in \Ab(\alg{B}_{\Hat{c}, \et})$ is given by the triple
		\[
			\left(
				R^{q} \Psi G,\,
				(R^{q} \Psi_{x} G)_{x \in T},\,
					R^{q} \Psi G
				\to
					\bigoplus_{x \in T}
						\pi_{\Hat{\alg{k}}_{x} / \alg{B}, \ast}
						R^{q} \Psi_{x} G_{x}
			\right).
		\]
\end{Prop}

\begin{proof}
	This follows from Proposition \ref{0345}.
\end{proof}

\begin{Prop} \label{0347}
	Let $G \in D^{+}(\Hat{\alg{R}}_{\et})$,
	$F' \in F^{\perar}$,
	$x' \in \Spec \alg{B}(F')$
	and $q \in \Z$.
	\begin{enumerate}
		\item
			The stalk of $R^{q} \Psi G \in \Ab(\alg{B}_{\et})$ at $x'$ is given by
			$H^{q}(\Hat{\alg{R}}(F')_{x'}^{sh}, G|_{\Hat{\alg{R}}(F'), \et})$.
		\item
			The stalk of $R^{q} \varepsilon_{\ast} R \Psi G \in \Ab(\alg{B}_{\nis})$ at $x'$ is given by
			$H^{q}(\Hat{\alg{R}}(F')_{x'}^{h}, G|_{\Hat{\alg{R}}(F'), \et})$.
	\end{enumerate}
	Here $G|_{\Hat{\alg{R}}(F'), \et} \in D^{+}(\Hat{\alg{R}}(F')_{\et})$ is
	pulled back to $D^{+}(\Hat{\alg{R}}(F')_{x', \et}^{h})$ and $\Hat{\alg{R}}(F')_{x', \et}^{sh})$
	and taken cohomology.
\end{Prop}

\begin{proof}
	Obvious.
\end{proof}

\begin{Prop} \label{0348}
	Let $G \in D^{+}(\Hat{\alg{K}}_{\eta_{x}, \et})$,
	$F' \in F^{\perar}$,
	$x' \in \Spec \Hat{\alg{k}}_{x}(F')$
	and $q \in \Z$.
	\begin{enumerate}
		\item
			The stalk of $R^{q} \Psi_{x} G \in \Ab(\Hat{\alg{k}}_{x, \et})$ at $x'$ is given by
			$H^{q}(\Hat{\alg{K}}_{\eta_{x}}(F')_{x'}^{sh}, G|_{\Hat{\alg{K}}_{\eta_{x}}(F'), \et})$.
		\item
			The stalk of $R^{q} \varepsilon_{\ast} R \Psi_{x} G \in \Ab(\Hat{\alg{k}}_{x, \nis})$ at $x'$ is given by
			$H^{q}(\Hat{\alg{K}}_{\eta_{x}}(F')_{x'}^{h}, G|_{\Hat{\alg{K}}_{\eta_{x}}(F'), \et})$.
	\end{enumerate}
	Here $G|_{\Hat{\alg{K}}_{\eta_{x}}(F'), \et} \in D^{+}(\Hat{\alg{K}}_{\eta_{x}}(F')_{\et})$ is
	pulled back to $D^{+}(\Hat{\alg{K}}_{\eta_{x}}(F')_{x', \et}^{h})$ and $\Hat{\alg{K}}_{\eta_{x}}(F')_{x', \et}^{sh})$
	and taken cohomology.
\end{Prop}

\begin{proof}
	Obvious.
\end{proof}

As $\Bar{\pi}_{\Hat{\alg{A}}} \compose i$ and $\Bar{\pi}_{\alg{B}}$
are equal as morphisms $\Spec \alg{B}_{\et} \to \Spec F^{\perar}_{\et}$,
we have natural morphisms
$R \Bar{\pi}_{\Hat{\alg{A}}, \ast} \to R \Bar{\pi}_{\alg{B}, \ast} i^{\ast}$
and $R \Bar{\pi}_{\Hat{\alg{A}}, \Hat{!}} \to R \Bar{\pi}_{\alg{B}, \Hat{!}} i^{\ast}$.
Hence we have two natural transformations
	\[
			R \Bar{\pi}_{\Hat{\alg{R}}, \ast} \to R \Bar{\pi}_{\alg{B}, \ast} R \Bar{\Psi},\,
			R \Bar{\pi}_{\Hat{\alg{R}}, \Hat{!}} \to R \Bar{\pi}_{\alg{B}, \Hat{!}} R \Bar{\Psi}
		\colon
			D(\Hat{\alg{R}}_{\Hat{c}, \et})
		\to
			D(F^{\perar}_{\et}).
	\]

\begin{Prop} \label{0355}
	The above two natural transformations are isomorphisms
	when restricted to $D_{\tor}^{+}(\Hat{\alg{R}}_{\Hat{c}, \et})$.
\end{Prop}

\begin{proof}
	By Proposition \ref{0341},
	we have $R \pi_{\Hat{\alg{A}}, \ast} \isomto R \pi_{\alg{B}, \ast} i^{\ast}$
	as $D_{\tor}^{+}(\Hat{\alg{A}}_{\et}) \to D(F^{\perar}_{\et})$.
	This implies that the first morphism is an isomorphism on $D_{\tor}^{+}(\Hat{\alg{R}}_{\Hat{c}, \et})$.
	By Proposition \ref{0313},
	we have $R \pi_{\Hat{\alg{O}}_{K_{\eta_{x}}}, \ast} \isomto R \pi_{\Hat{\alg{k}}_{x}, \ast} i_{x}^{\ast}$ for any $x \in T$
	as $D(\Hat{\alg{O}}_{K_{\eta_{x}}, \et}) \to D(F^{\perar}_{\et})$.
	This implies that the second morphism is also an isomorphism on $D_{\tor}^{+}(\Hat{\alg{R}}_{\Hat{c}, \et})$.
\end{proof}

For $G, H \in D(\Hat{\alg{R}}_{\Hat{c}, \et})$, we have a morphism
	\begin{equation} \label{0346}
			R \Bar{\Psi} G \tensor^{L} R \Bar{\Psi} H
		\to
			R \Bar{\Psi}(G \tensor^{L} H)
	\end{equation}
in $D(\alg{B}_{\Hat{c}, \et})$ functorial in $G, H$
by \eqref{0327} applied to $R \Bar{\jmath}_{\ast}$
and applying $\Bar{\imath}^{\ast}$.
Hence we have a morphism
	\begin{equation} \label{0343}
				R \Bar{\pi}_{\alg{B}, \ast} R \Bar{\Psi} G
			\tensor^{L}
				R \Bar{\pi}_{\alg{B}, \Hat{!}} R \Bar{\Psi} H
		\to
			R \Bar{\pi}_{\alg{B}, \Hat{!}} R \Bar{\Psi}(G \tensor^{L} H)
	\end{equation}
in $D(F^{\perar}_{\et})$ functorial in $G, H$
by \eqref{0342}.

\begin{Prop} \label{0361}
	The morphisms \eqref{0344} and \eqref{0343} are compatible,
	that is, the diagram
		\[
			\begin{CD}
						R \Bar{\pi}_{\Hat{\alg{R}}, \ast} G
					\tensor^{L}
						R \Bar{\pi}_{\Hat{\alg{R}}, \Hat{!}} H
				@>>>
					R \Bar{\pi}_{\Hat{\alg{R}}, \Hat{!}}(G \tensor^{L} H)
				\\ @VVV @VVV \\
						R \Bar{\pi}_{\alg{B}, \ast} R \Bar{\Psi} G
					\tensor^{L}
						R \Bar{\pi}_{\alg{B}, \Hat{!}} R \Bar{\Psi} H
				@>>>
					R \Bar{\pi}_{\alg{B}, \Hat{!}} R \Bar{\Psi}(G \tensor^{L} H)
			\end{CD}
		\]
	in $D(F^{\perar}_{\et})$ is commutative.
\end{Prop}

\begin{proof}
	This follows from Proposition \ref{0332}.
\end{proof}


\subsection{%
	\texorpdfstring{Duality for $\gr^{0}$}
	{Duality for gr 0}
}
\label{0227}

We continue the notation from the previous subsection.
Set $\Bar{\mathcal{E}} = R \Bar{\Psi} \Lambda \in D(\alg{B}_{\Hat{c}, \et})$,
$\mathcal{E} = R \Psi \Lambda \in D(\alg{B}_{\et})$ and
$\mathcal{E}_{x} = R \Psi_{x} \Lambda \in D(\Hat{\alg{k}}_{x, \et})$,
where $x \in T$.
For any $q \ge 0$, the sheaf $H^{q} \Bar{\mathcal{E}}$ is given by the triple
	\[
		\left(
			H^{q} \mathcal{E},\,
			(H^{q} \mathcal{E}_{x})_{x \in T},\,
				H^{q} \mathcal{E}
			\to
				\bigoplus_{x \in T}
					\pi_{\Hat{\alg{k}}_{x} / \alg{B}, \ast}
					H^{q} \mathcal{E}_{x}
		\right)
	\]
by Proposition \ref{0349}.
By \eqref{0346}, the multiplication map $\Lambda \times \Lambda \to \Lambda$ induces a morphism
	\begin{equation} \label{0093}
		\Bar{\mathcal{E}} \tensor^{L} \Bar{\mathcal{E}} \to \Bar{\mathcal{E}}.
	\end{equation}

\begin{Prop} \label{0228}
	The cohomology sheaves of
	$\mathcal{E}$, $\mathcal{E}_{x}$ with $x \in T$ and $\Bar{\mathcal{E}}$
	are all zero in degrees $\ne 0, 1, 2$.
\end{Prop}

\begin{proof}
	Let $F' \in F^{\perar}$.
	Let $x'$ be a closed point of $\Spec \alg{B}(F')$.
	The stalk of $H^{q} \mathcal{E}$ at $x'$ is
	$H^{q}(\Hat{\alg{R}}(F')_{x'}^{sh}, \Lambda)$
	by Proposition \ref{0347}.
	The ring $\Hat{\alg{A}}(F')_{x'}^{sh}$ is a two-dimensional local ring
	satisfying the conditions listed at the beginning of
	Section \ref{0060},
	particularly \eqref{0061}.
	Its residue field is algebraically closed.
	Hence $H^{q}(\Hat{\alg{R}}(F')_{x'}^{sh}, \Lambda)$ is zero for $q \ne 0, 1, 2$
	by Proposition \ref{0063}.
	Thus $H^{q} \mathcal{E} = 0$ for $q \ne 0, 1, 2$.
	
	The stalk of $H^{q} \mathcal{E}_{x}$ at $x'$ is
	$H^{q}(\Hat{\alg{K}}_{\eta_{x}}(F')_{x'}^{sh}, \Lambda)$
	by Proposition \ref{0348}.
	The ring $\Hat{\alg{K}}_{\eta_{x}}(F')_{x'}^{sh}$ is a henselian discrete valuation field
	with residue field $\Hat{\alg{k}}_{x}(F')_{x'}^{\sep}$.
	Hence $H^{q} \mathcal{E}_{x} = 0$ for $q \ne 0, 1, 2$
	by Proposition \ref{0041}.
	The statement for $\Bar{\mathcal{E}}$ then follows.
\end{proof}

We will introduce a filtration on $H^{q} \Bar{\mathcal{E}}$.
Consider the exact sequence
$0 \to \Lambda \to \Gm \stackrel{p}{\to} \Gm \to 0$
in $\Ab(\Hat{\alg{R}}_{\Hat{c}, \et})$.
It defines a morphism
	\[
			\{\var\}
		\colon
			\Bar{\Psi} \Gm
		\to
			R^{1} \Bar{\Psi} \Lambda
	\]
in $\Ab(\alg{B}_{\Hat{c}, \et})$.
For any $q \ge 0$, the cup product then defines a morphism
	\[
			\{\var, \dots, \var\}
		\colon
			(\Bar{\Psi} \Gm)^{\tensor q}
		\to
			R^{q} \Bar{\Psi} \Lambda
		=
			H^{q} \Bar{\mathcal{E}}
	\]
from the tensor power over $\Z$,
which we call the symbol map.
We can similarly define symbol maps
	\[
				(\Psi \Gm)^{\tensor q}
			\to
				R^{q} \Psi \Lambda
			=
				H^{q} \mathcal{E},
		\quad
				(\Psi_{x} \Gm)^{\tensor q}
			\to
				R^{q} \Psi_{x} \Lambda
			=
				H^{q} \mathcal{E}_{x}.
	\]
For any $F' \in F^{\perar}$ and any point $x' \in \alg{B}(F')_{x'}^{sh}$,
the morphism $(\Psi \Gm)^{\tensor q} \to R^{q} \Psi \Lambda$ on the stalk at $x'$
is given by the usual symbol map
	\[
			((\Hat{\alg{R}}(F')_{x'}^{sh})^{\times})^{\tensor q}
		\to
			H^{q}(\Hat{\alg{R}}(F')_{x'}^{sh}, \Lambda);
		\quad
			(x_{1}, \dots, x_{q})
		\mapsto
			\{x_{1}, \dots, x_{q}\}.
	\]
The morphism $(\Psi_{x} \Gm)^{\tensor q} \to R^{q} \Psi_{x} \Lambda$ on the stalk at $x'$
is similarly described.

For $m \ge 1$, note that
$1 + I^{m} \Bar{\imath}^{\ast} \Ga \subset \Bar{\imath}^{\ast} \Gm \subset \Bar{\Psi} \Gm$
in $\Ab(\alg{B}_{\Hat{c}, \et})$.
Let $U^{m} H^{q} \Bar{\mathcal{E}}$ be the subsheaf of $H^{q} \Bar{\mathcal{E}}$
generated by the image of
$\{1 + \Bar{\imath}^{\ast} \Ga, \Bar{\Psi} \Gm, \dots, \Bar{\Psi} \Gm\}$.
For $m = 0$, we set $U^{0} H^{q} \Bar{\mathcal{E}}$ to be the whole sheaf $H^{q} \Bar{\mathcal{E}}$.
Define
	\[
			\gr^{m} H^{q} \Bar{\mathcal{E}}
		=
				U^{m} H^{q} \Bar{\mathcal{E}}
			/
				U^{m + 1} H^{q} \Bar{\mathcal{E}}.
	\]
We similarly define sheaves
$U^{m} H^{q} \mathcal{E}$,
$\gr^{m} H^{q} \mathcal{E}$,
$U^{m} H^{q} \mathcal{E}_{x}$
and $\gr^{m} H^{q} \mathcal{E}_{x}$
for each $x \in T$.
The cup product
$H^{q} \Bar{\mathcal{E}} \times H^{q'} \Bar{\mathcal{E}} \to H^{q + q'} \Bar{\mathcal{E}}$
maps $U^{m} H^{q} \Bar{\mathcal{E}} \times U^{m'} H^{q'} \Bar{\mathcal{E}}$ to
$U^{m + m'} H^{q + q'} \Bar{\mathcal{E}}$
(see \cite[Lemma (4.1)]{BK86}).
A similar property holds for $\mathcal{E}$ and $\mathcal{E}_{x}$.
We have $H^{0} \Bar{\mathcal{E}} \cong \Lambda$.

Now we describe the graded pieces.
Let $e_{A}$ be the $\varpi$-adic valuation of $p$
and set $f_{A} = p e_{A} / (p - 1)$.
In what follows,
note the isomorphism $\dlog \colon \Gm / \Gm^{p} \isomto \nu(1)$.

\begin{Prop} \label{0229}
	Let $m \ge 0$.
	\begin{enumerate}
		\item \label{0230}
			If $m = 0$, then
				\begin{align*}
							\Gm / \Gm^{p}
					&	\isomto
							\gr^{m} H^{2} \Bar{\mathcal{E}},
					\\
							x
					&	\mapsto
							\{\Tilde{x}, \varpi\}.
				\end{align*}
		\item \label{0231}
			If $0 < m < f_{A}$ and $p \mid m$, then
				\begin{align*}
							\Ga / \Ga^{p}
					&	\isomto
							\gr^{m} H^{2} \Bar{\mathcal{E}},
					\\
							x
					&	\mapsto
							\{1 + \Tilde{x} \varpi^{m}, \varpi\}.
				\end{align*}
		\item \label{0232}
			If $0 < m < f_{A}$ and $p \nmid m$, then
				\begin{align*}
							\Omega^{1}
					&	\isomto
							\gr^{m} H^{2} \Bar{\mathcal{E}},
					\\
							x \dlog(y)
					&	\mapsto
							\{1 + \Tilde{x} \varpi^{m}, \Tilde{y}\},
				\end{align*}
			where $x \in \Ga$ and $y \in \Gm$.
		\item \label{0233}
			If $m \ge f_{A}$, then $\gr^{m} H^{2} \Bar{\mathcal{E}} = 0$.
	\end{enumerate}
\end{Prop}

\begin{proof}
	This follows from
	Propositions \ref{0069} and \ref{0043}.
\end{proof}

\begin{Prop} \label{0234}
	Let $m \ge 0$.
	\begin{enumerate}
		\item \label{0235}
			If $m = 0$, then
				\begin{align*}
							\Lambda \oplus \Gm / \Gm^{p}
					&	\isomto
							\gr^{m} H^{1} \Bar{\mathcal{E}},
					\\
							(i, x)
					&	\mapsto
							\{\varpi^{i} \Tilde{x}\}.
				\end{align*}
		\item \label{0236}
			If $0 < m < f_{A}$ and $p \mid m$, then
				\begin{align*}
							\Ga / \Ga^{p}
					&	\isomto
							\gr^{m} H^{1} \Bar{\mathcal{E}},
					\\
							x
					&	\mapsto
							\left\{
									1
								+
									\Tilde{x}
									\frac{
										(\zeta_{p} - 1)^{p}
									}{
										\varpi^{f_{A} - m}
									}
							\right\}.
				\end{align*}
		\item \label{0237}
			If $0 < m < f_{A}$ and $p \nmid m$, then
				\begin{align*}
							\Ga
					&	\isomto
							\gr^{m} H^{1} \Bar{\mathcal{E}},
					\\
							x
					&	\mapsto
							\left\{
								1 + \Tilde{x} \frac{(\zeta_{p} - 1)^{p}}{\varpi^{f_{A} - m}}
							\right\}.
				\end{align*}
		\item \label{0238}
			If $m \ge f_{A}$, then $\gr^{m} H^{1} \Bar{\mathcal{E}} = 0$.
	\end{enumerate}
\end{Prop}

\begin{proof}
	This follows from
	Propositions \ref{0070} and \ref{0044}.
\end{proof}

\begin{Prop} \label{0249}
	For any $q \in \Z$ and $m \ge 0$,
	the objects
		\begin{gather*}
				R \Bar{\pi}_{B, \ast} U^{m} H^{q} \Bar{\mathcal{E}}, \quad
				R \Bar{\pi}_{B, \ast} \gr^{m} H^{q} \Bar{\mathcal{E}}, \quad
				R \Bar{\pi}_{B, \Hat{!}} U^{m} H^{q} \Bar{\mathcal{E}}, \quad
				R \Bar{\pi}_{B, \Hat{!}} \gr^{m} H^{q} \Bar{\mathcal{E}},
			\\
				R \Bar{\pi}_{B, \ast} \Bar{\mathcal{E}}, \quad
				R \Bar{\pi}_{B, \Hat{!}} \Bar{\mathcal{E}}
		\end{gather*}
	are in $\genby{\mathcal{W}_{F}}_{F^{\perar}_{\et}}$.
	Moreover, $R^{q} \Bar{\pi}_{B, \ast} \Bar{\mathcal{E}} \in \mathcal{W}_{F}$ for all $q$.
\end{Prop}

\begin{proof}
	This follows from the previous propositions
	and Proposition \ref{0082}.
\end{proof}

By Propositions \ref{0228} and \ref{0229} \eqref{0230}, we have a canonical morphism
	\begin{equation} \label{0351}
			\Bar{\mathcal{E}}
		\to
			\nu(1)[-2]
	\end{equation}
in $D(\alg{B}_{\Hat{c}, \et})$.
With \eqref{0093}, we have canonical morphisms
	\begin{equation} \label{0095}
			\Bar{\mathcal{E}} \tensor^{L} \Bar{\mathcal{E}}
		\to
			\Bar{\mathcal{E}}
		\to
			\nu(1)[-2]
	\end{equation}
in $D(\alg{B}_{\Hat{c}, \et})$.

We now split this pairing into the $\gr^{0}$ part and the $U^{1}$ part.
For $i = 1, 2$, define $\tau_{\ge i}' \Bar{\mathcal{E}}$ to be the unique mapping cone of
the morphism $U^{1} H^{i} \Bar{\mathcal{E}}[-i] \to \tau_{\ge i} \Bar{\mathcal{E}}$.
Similarly, define $\tau_{\le i}' \Bar{\mathcal{E}}$ to be the unique mapping fiber of
the morphism $\tau_{\le i} \Bar{\mathcal{E}} \to \gr^{0} H^{i} \Bar{\mathcal{E}}[-i]$.

\begin{Prop} \label{0096}
	Set $(\var)^{\vee} = R \sheafhom_{\alg{B}_{\Hat{c}, \et}}(\var, \nu(1))$.
	Then there exists a unique way to extend
	the morphism \eqref{0095}
	to morphisms of distinguished triangles
		\[
			\begin{CD}
					H^{0} \Bar{\mathcal{E}}
				@>>>
					\Bar{\mathcal{E}}
				@>>>
					\tau_{\ge 1} \Bar{\mathcal{E}}
				\\ @VVV @VVV @VVV \\
					(\gr^{0} H^{2} \Bar{\mathcal{E}})^{\vee}
				@>>>
					\Bar{\mathcal{E}}^{\vee}[-2]
				@>>>
					(\tau_{\le 2}' \Bar{\mathcal{E}})^{\vee}[-2],
			\end{CD}
		\]
		\[
			\begin{CD}
					U^{1} H^{1} \Bar{\mathcal{E}}[-1]
				@>>>
					\tau_{\ge 1} \Bar{\mathcal{E}}
				@>>>
					\tau_{\ge 1}' \Bar{\mathcal{E}}
				\\ @VVV @VVV @VVV \\
					(U^{1} H^{2} \Bar{\mathcal{E}})^{\vee}
				@>>>
					(\tau_{\le 2}' \Bar{\mathcal{E}})^{\vee}[-2]
				@>>>
					(\tau_{\le 1} \Bar{\mathcal{E}})^{\vee}[-2],
			\end{CD}
		\]
		\[
			\begin{CD}
					\gr^{0} H^{1} \Bar{\mathcal{E}}[-1]
				@>>>
					\tau_{\ge 1}' \Bar{\mathcal{E}}
				@>>>
					H^{2} \Bar{\mathcal{E}}[-2]
				\\ @VVV @VVV @VVV \\
					(\gr^{0} H^{1} \Bar{\mathcal{E}})^{\vee}[-1]
				@>>>
					(\tau_{\le 1} \Bar{\mathcal{E}})^{\vee}[-2]
				@>>>
					(\tau_{\le 1}' \Bar{\mathcal{E}})^{\vee}[-2],
			\end{CD}
		\]
		\[
			\begin{CD}
					U^{1} H^{2} \Bar{\mathcal{E}}[-2]
				@>>>
					H^{2} \Bar{\mathcal{E}}[-2]
				@>>>
					\gr^{0} H^{2} \Bar{\mathcal{E}}[-2]
				\\ @VVV @VVV @VVV \\
					(U^{1} H^{1} \Bar{\mathcal{E}})^{\vee}[-1]
				@>>>
					(\tau_{\le 1}' \Bar{\mathcal{E}})^{\vee}[-2]
				@>>>
					(H^{0} \Bar{\mathcal{E}})^{\vee}[-2].
			\end{CD}
		\]
	The obtained pairings
	$H^{0} \Bar{\mathcal{E}} \times \gr^{0} H^{2} \Bar{\mathcal{E}} \to \nu(1)$,
	$\gr^{0} H^{1} \Bar{\mathcal{E}} \times \gr^{0} H^{1} \Bar{\mathcal{E}} \to \nu(1)$
	and $\gr^{0} H^{2} \Bar{\mathcal{E}} \times H^{0} \Bar{\mathcal{E}} \to \nu(1)$
	are induced by the usual cup products.
\end{Prop}

\begin{proof}
	To construct the first diagram,
	note that the composite
	$H^{0} \Bar{\mathcal{E}} \to (\tau_{\le 2}' \Bar{\mathcal{E}})^{\vee}[-2]$
	factors through $(U^{1} H^{2} \Bar{\mathcal{E}})^{\vee}$ by a degree reason.
	This corresponds to the cup product pairing
		$
					H^{0} \Bar{\mathcal{E}}
				\times
					U^{1} H^{2} \Bar{\mathcal{E}}
			\to
				\nu(1)
		$.
	This is zero since the morphism $H^{2} \Bar{\mathcal{E}} \to \nu(1)$ annihilates
	$U^{1} H^{2} \Bar{\mathcal{E}}$.
	Also, any morphism from $H^{0} \Bar{\mathcal{E}}$ to
	$(\tau_{\le 2}' \Bar{\mathcal{E}})^{\vee}[-3]$ is zero by a degree reason.
	Then we obtain the first diagram by a formal argument on triangulated categories.
	The rest of the diagrams can be obtained in a similar way,
	using instead the fact that the cup product pairings
		$
					U^{1} H^{1} \Bar{\mathcal{E}}
				\times
					H^{1} \Bar{\mathcal{E}}
			\to
				\nu(1)
		$,
		$
					H^{1} \Bar{\mathcal{E}}
				\times
					U^{1} H^{1} \Bar{\mathcal{E}}
			\to
				\nu(1)
		$,
		$
					U^{1} H^{2} \Bar{\mathcal{E}}
				\times
					H^{0} \Bar{\mathcal{E}}
			\to
				\nu(1)
		$
	are zero.
\end{proof}

By \eqref{0351} and \eqref{0350},
we have canonical morphisms
	\begin{equation} \label{0359}
			R \pi_{\alg{B}, \Hat{!}} \Bar{\mathcal{E}}
		\to
			R \pi_{\alg{B}, \Hat{!}} \nu(1)[-2]
		\to
			\Lambda_{\infty}[-3]
	\end{equation}
in $D(F^{\perar}_{\et})$.
With this and \eqref{0342} and \eqref{0095},
we obtain morphisms
	\[
				R \pi_{\alg{B}, \ast} \Bar{\mathcal{E}}
			\tensor^{L}
				R \pi_{\alg{B}, \Hat{!}} \Bar{\mathcal{E}}
		\to
			R \pi_{\alg{B}, \Hat{!}} \Bar{\mathcal{E}}
		\to
			\Lambda_{\infty}[-3]
	\]
in $D(F^{\perar}_{\et})$.
Here is the duality for the $\gr^{0}$ part:

\begin{Prop} \label{0098}
	Let $q \in \{0, 1, 2\}$ and set $q' = 2 - q$.
	Consider the cup product pairing
	$\gr^{0} H^{q} \Bar{\mathcal{E}} \times \gr^{0} H^{q'} \Bar{\mathcal{E}} \to \nu(1)$
	in $\Ab(\alg{B}_{\Hat{c}, \et})$.
	The induced morphism
		\[
					R \Bar{\pi}_{\alg{B}, \ast} \gr^{0} H^{q} \Bar{\mathcal{E}}
				\tensor^{L}
					R \Bar{\pi}_{\alg{B}, \Hat{!}} \gr^{0} H^{q'} \Bar{\mathcal{E}}
			\to
				R \Bar{\pi}_{\alg{B}, \Hat{!}} \nu(1)
			\to
				\Lambda_{\infty}[-1]
		\]
	in $D(F^{\perar}_{\et})$,
	where the last morphism is \eqref{0350},
	is a perfect pairing.
\end{Prop}

\begin{proof}
	After the identifications in Propositions \ref{0229} and \ref{0234}, the pairing
	$\gr^{0} H^{0} \Bar{\mathcal{E}} \times \gr^{0} H^{2} \Bar{\mathcal{E}} \to \nu(1)$
	is the multiplication map
	$\Lambda \times \nu(1) \to \nu(1)$;
	the pairing
	$\gr^{0} H^{2} \Bar{\mathcal{E}} \times \gr^{0} H^{0} \Bar{\mathcal{E}} \to \nu(1)$
	is the multiplication map
	$\nu(1) \times \Lambda \to \nu(1)$;
	and the pairing
	$\gr^{0} H^{1} \Bar{\mathcal{E}} \times \gr^{0} H^{1} \Bar{\mathcal{E}} \to \nu(1)$
	is the morphism
		\begin{align*}
			&
						\bigl(
							(i, a), (j, b)
						\bigr)
					\in
						(\Lambda \oplus \nu(1)) \times (\Lambda \oplus \nu(1))
			\\
			&	\quad
				\mapsto
						\{\varpi^{i} \Tilde{a}, \varpi^{j} \Tilde{b}\}
					\in
						\gr^{0} H^{2} \Bar{\mathcal{E}}
			\\
			&	\quad
				\leftrightarrow
						(-1)^{i j} a^{j} b^{-i}
					\in
						\nu(1),
		\end{align*}
	where we identified $\nu(1)$ with $\Gm / \Gm^{p}$.
	(Note that $H^{2} \Bar{\mathcal{E}}$ consists of
	not only $\Lambda \oplus \nu(1) \in \Ab(\alg{B}_{\et})$
	but also $\Lambda \oplus \nu(1) \in \Ab(\Hat{\alg{k}}_{x, \et})$ for each $x \in T$
	and a compatibility between them.)
	Hence the result follows from
	Proposition \ref{0087}.
\end{proof}


\subsection{%
	\texorpdfstring{Duality for $U^{1}$ in the Nisnevich topology}
	{Duality for U 1 in the Nisnevich topology}
}
\label{0250}

Set $\mathcal{F} = R \varepsilon_{\ast} \mathcal{E} \in D(\alg{B}_{\nis})$,
$\mathcal{F}_{x} = R \varepsilon_{\ast} \mathcal{E}_{x} \in D(\Hat{\alg{k}}_{x, \nis})$
and $\Bar{\mathcal{F}} = R \varepsilon_{\ast} \Bar{\mathcal{E}} \in D(\alg{B}_{\Hat{c}, \nis})$.
For any $q \in \Z$, the sheaf
$H^{q} \Bar{\mathcal{F}} \in \Ab(\alg{B}_{\Hat{c}, \nis})$
is given by the triple
	\[
		\left(
			H^{q} \mathcal{F},\,
			(H^{q} \mathcal{F}_{x})_{x \in T},\,
				H^{q} \mathcal{F}
			\to
				\bigoplus_{x \in T}
					\pi_{\Hat{\alg{k}}_{x} / \alg{B}, \nis, \ast}
					H^{q} \mathcal{F}_{x}
		\right)
	\]
by Proposition \ref{0345}.
The morphism \eqref{0093} induces a morphism
	\[
				\Bar{\mathcal{F}}
			\tensor^{L}
				\Bar{\mathcal{F}}
		\to
			\Bar{\mathcal{F}}.
	\]

The exact sequence
$0 \to \Lambda \to \Gm \stackrel{p}{\to} \Gm \to 0$
in $\Ab(\alg{R}_{\Hat{c}, \et})$ induces a morphism
	$
			R \Bar{\Psi} \Gm
		\to
			R \Bar{\Psi} \Lambda[1]
		=
			\Bar{\mathcal{E}}[1]
	$
in $D(\alg{B}_{\Hat{c}, \et})$, hence a morphism
	$
			R \varepsilon_{\ast} R \Bar{\Psi} \Gm
		\to
			R \varepsilon_{\ast} \Bar{\mathcal{E}}[1]
		=
			\Bar{\mathcal{F}}[1]
	$
in $D(\alg{B}_{\Hat{c}, \et})$, hence a morphism
	\[
			\{\var\}
		\colon
			\varepsilon_{\ast} \Bar{\Psi} \Gm
		\to
			H^{1} \Bar{\mathcal{F}},
	\]
hence a morphism
	\[
			\{\var, \dots, \var\}
		\colon
			(\varepsilon_{\ast} \Bar{\Psi} \Gm)^{\tensor q}
		\to
			H^{q} \Bar{\mathcal{F}}
	\]
for any $q \ge 0$.
Using this, for $m \ge 0$,
define a subsheaf $U^{m} H^{q} \Bar{\mathcal{F}}$ of $H^{q} \Bar{\mathcal{F}}$
in the same way we defined $U^{m} H^{q} \Bar{\mathcal{E}}$.
Define $\gr^{m} H^{q} \Bar{\mathcal{F}}$ similarly.
Note that since $R^{1} \varepsilon_{\ast} \Gm = 0$ by Hilbert's theorem 90,
we know that $\Gm / \Gm^{p} \cong \nu(1)$ also in $\Ab(\alg{B}_{\Hat{c}, \nis})$.

\begin{Prop} \label{0100}
	Let $m \ge 0$.
	\begin{enumerate}
		\item \label{0251}
			If $m = 0$, then
				\begin{align*}
							\Gm / \Gm^{p}
					&	\isomto
							\gr^{m} H^{2} \Bar{\mathcal{F}},
					\\
							x
					&	\mapsto
							\{\Tilde{x}, \varpi\}.
				\end{align*}
		\item \label{0252}
			If $0 < m < f_{A}$ and $p \mid m$, then
				\begin{align*}
							\Ga / \Ga^{p}
					&	\isomto
							\gr^{m} H^{2} \Bar{\mathcal{F}},
					\\
							x
					&	\mapsto
							\{1 + \Tilde{x} \varpi^{m}, \varpi\}.
				\end{align*}
		\item \label{0253}
			If $0 < m < f_{A}$ and $p \nmid m$, then
				\begin{align*}
							\Omega^{1}
					&	\isomto
							\gr^{m} H^{2} \Bar{\mathcal{F}},
					\\
							x \dlog(y)
					&	\mapsto
							\{1 + \Tilde{x} \varpi^{m}, \Tilde{y}\},
				\end{align*}
			where $x \in \Ga$ and $y \in \Gm$.
		\item \label{0254}
			If $m = f_{A}$, then
				\begin{align*}
							\xi(1) \oplus \xi
					&	\isomto
							\gr^{m} H^{2} \Bar{\mathcal{F}},
					\\
							(x \dlog y, z)
					&	\mapsto
								\{1 + \Tilde{x} (\zeta_{p} - 1)^{p}, \Tilde{y}\}
							+
								\{1 + \Tilde{z} (\zeta_{p} - 1)^{p}, \varpi\}.
				\end{align*}
		\item \label{0255}
			If $m > f_{A}$, then $\gr^{m} H^{2} \Bar{\mathcal{F}} = 0$.
	\end{enumerate}
\end{Prop}

\begin{proof}
	This follows from
	Propositions \ref{0069} and \ref{0043}.
\end{proof}

\begin{Prop} \label{0101}
	Let $m \ge 0$.
	\begin{enumerate}
		\item \label{0256}
			If $m = 0$, then
				\begin{align*}
							\Lambda \oplus \Gm / \Gm^{p}
					&	\isomto
							\gr^{m} H^{1} \Bar{\mathcal{F}},
					\\
							(i, x)
					&	\mapsto
							\{\varpi^{i} \Tilde{x}\}.
				\end{align*}
		\item \label{0257}
			If $0 < m < f_{A}$ and $p \mid m$, then
				\begin{align*}
							\Ga / \Ga^{p}
					&	\isomto
							\gr^{m} H^{1} \Bar{\mathcal{F}},
					\\
							x
					&	\mapsto
							\left\{
									1
								+
									\Tilde{x}
									\frac{
										(\zeta_{p} - 1)^{p}
									}{
										\varpi^{f_{A} - m}
									}
							\right\}.
				\end{align*}
		\item \label{0258}
			If $0 < m < f_{A}$ and $p \nmid m$, then
				\begin{align*}
							\Ga
					&	\isomto
							\gr^{m} H^{1} \Bar{\mathcal{F}},
					\\
							x
					&	\mapsto
							\left\{
								1 + \Tilde{x} \frac{(\zeta_{p} - 1)^{p}}{\varpi^{f_{A} - m}}
							\right\}.
				\end{align*}
		\item \label{0259}
			If $m = f_{A}$, then
				\begin{align*}
							\xi
					&	\isomto
							\gr^{m} H^{1} \Bar{\mathcal{F}},
					\\
							x
					&	\mapsto
							\left\{
								1 + \Tilde{x} (\zeta_{p} - 1)^{p}
							\right\}.
				\end{align*}
		\item \label{0260}
			If $m > f_{A}$, then $\gr^{m} H^{1} \Bar{\mathcal{F}} = 0$.
	\end{enumerate}
\end{Prop}

\begin{proof}
	This follows from
	Propositions \ref{0070} and \ref{0044}.
\end{proof}

We clarify the relation between $U^{m} H^{q} \Bar{\mathcal{E}}$ and $U^{m} H^{q} \Bar{\mathcal{F}}$.

\begin{Prop} \label{0104}
	Let $q \in \{1, 2\}$.
	Consider the spectral sequence
		\[
				E_{2}^{i j}
			=
				R^{i} \varepsilon_{\ast} H^{j} \Bar{\mathcal{E}}
			\Longrightarrow
				H^{i + j} \Bar{\mathcal{F}}
		\]
	and the induced exact sequence
		\[
				0
			\to
				R^{1} \varepsilon_{\ast} H^{q - 1} \Bar{\mathcal{E}}
			\to
				H^{q} \Bar{\mathcal{F}}
			\to
				\varepsilon_{\ast} H^{q} \Bar{\mathcal{E}}
			\to
				0.
		\]
	As subsheaves of $H^{q} \Bar{\mathcal{F}}$, we have
	$U^{f_{A}} H^{q} \Bar{\mathcal{F}} = R^{1} \varepsilon_{\ast} H^{q - 1} \Bar{\mathcal{E}}$.
\end{Prop}

\begin{proof}
	By the previous propositions,
	we know that $U^{f_{A}} H^{1} \Bar{\mathcal{F}} \cong \xi$
	and $U^{f_{A}} H^{2} \Bar{\mathcal{F}} \cong \xi(1) \oplus \xi$ \'etale-sheafify to zero
	and $H^{q} \Bar{\mathcal{F}} / U^{f_{A}} H^{q} \Bar{\mathcal{F}}$ is an \'etale sheaf.
	Such a subsheaf is necessarily $R^{1} \varepsilon_{\ast} H^{q - 1} \Bar{\mathcal{E}}$.
\end{proof}

The morphism \eqref{0095} induces morphisms
	\[
			\Bar{\mathcal{F}} \tensor^{L} \Bar{\mathcal{F}}
		\to
			R \varepsilon_{\ast} \nu(1)[-2]
		\to
			\xi(1)[-3].
	\]
For $q \in \Z$ and $q' = 3 - q$, this induces a pairing
	\begin{equation} \label{0353}
			H^{q} \Bar{\mathcal{F}} \times H^{q'} \Bar{\mathcal{F}}
		\to
			\xi(1).
	\end{equation}
Consider its restriction
	\begin{equation} \label{0105}
			U^{1} H^{q} \Bar{\mathcal{F}} \times U^{1} H^{q'} \Bar{\mathcal{F}}
		\to
			\xi(1).
	\end{equation}

\begin{Prop} \label{0272}
	Let $q \in \{1, 2\}$ and set $q' = 3 - q$.
	Consider the morphism
		\[
					\varepsilon_{\ast} U^{1} H^{q} \Bar{\mathcal{E}}
				\tensor^{L}
					\varepsilon_{\ast} U^{1} H^{q'} \Bar{\mathcal{E}}
			\to
				R \varepsilon_{\ast} \nu(1)[1]
			\to
				\xi(1)
		\]
	induced by the morphism
	$U^{1} H^{q} \Bar{\mathcal{E}} \tensor^{L} U^{1} H^{q'} \Bar{\mathcal{E}} \to \nu(1)[1]$
	in Proposition \ref{0096}.
	The induced pairing
		\begin{equation} \label{0515}
					U^{1} H^{q} \Bar{\mathcal{F}}
				\times
					U^{1} H^{q'} \Bar{\mathcal{F}}
			\to
					\varepsilon_{\ast} U^{1} H^{q} \Bar{\mathcal{E}}
				\times
					\varepsilon_{\ast} U^{1} H^{q} \Bar{\mathcal{E}}
			\to
				\xi(1)
		\end{equation}
	agrees with the pairing \eqref{0105}.
\end{Prop}

\begin{proof}
	We only explain this for $q = 2$.
	The other case is similar.
	Define $\tau_{\le 1}' \Bar{\mathcal{F}}$ to be the canonical mapping fiber of the morphism
	$\tau_{\le 1} \Bar{\mathcal{F}} \to H^{1} \Bar{\mathcal{F}} / U^{1} H^{1} \Bar{\mathcal{F}}[-1]$.
	Consider the commutative diagram
		\[
			\begin{CD}
					\Bar{\mathcal{E}}
				@>>>
					H^{2} \Bar{\mathcal{E}}[-2]
				\\
				@VVV @VVV
				\\
					R \sheafhom_{\alg{B}_{\Hat{c}, \et}}(
						\Bar{\mathcal{E}},
						\nu(1)
					)[-2]
				@>>>
					R \sheafhom_{\alg{B}_{\Hat{c}, \et}}(
						\tau_{\le 1}' \Bar{\mathcal{E}},
						\nu(1)
					)[-2]
			\end{CD}
		\]
	coming from the first three diagrams in
	Proposition \ref{0096}.
	Apply $H^{2} R \varepsilon_{\ast}$.
	Since $\Bar{\mathcal{E}} \cong \varepsilon^{\ast} \Bar{\mathcal{F}}$,
	we have a commutative diagram
		\[
			\begin{CD}
					H^{2} \Bar{\mathcal{F}}
				@>>>
					\dfrac{
						H^{2} \Bar{\mathcal{F}}
					}{
						U^{f_{A}} H^{2} \Bar{\mathcal{F}}
					}
				\\
				@VVV @VVV
				\\
					\sheafext_{\alg{B}_{\Hat{c}, \nis}}^{0}(
						\Bar{\mathcal{F}},
						R \varepsilon_{\ast} \nu(1)
					)
				@>>>
					\sheafext_{\alg{B}_{\Hat{c}, \nis}}^{0}(
						\tau_{\le 1}' \Bar{\mathcal{F}},
						R \varepsilon_{\ast} \nu(1)
					).
			\end{CD}
		\]
	Applying the morphism $R \varepsilon_{\ast} \nu(1) \to \xi(1)[-1]$
	to the lower row, we have a commutative diagram
		\[
			\begin{CD}
					H^{2} \Bar{\mathcal{F}}
				@>>>
					\dfrac{
						H^{2} \Bar{\mathcal{F}}
					}{
						U^{f_{A}} H^{2} \Bar{\mathcal{F}}
					}
				\\
				@VVV @VVV
				\\
					\sheafext_{\alg{B}_{\Hat{c}, \nis}}^{-1}(
						\Bar{\mathcal{F}},
						\xi(1)
					)
				@>>>
					\sheafhom_{\alg{B}_{\Hat{c}, \nis}}(
						U^{1} H^{1} \Bar{\mathcal{F}},
						\xi(1)
					).
			\end{CD}
		\]
	The pairing $H^{2} \Bar{\mathcal{F}} \times U^{1} H^{1} \Bar{\mathcal{F}} \to \xi(1)$
	coming from this diagram is \eqref{0105}.
	On the other hand, applying $H^{2}$ to the fourth diagram of
	Proposition \ref{0096},
	we have a commutative diagram
		\[
			\begin{CD}
					\dfrac{
						U^{1} H^{2} \Bar{\mathcal{F}}
					}{
						U^{f_{A}} H^{2} \Bar{\mathcal{F}}
					}
				@>>>
					\dfrac{
						H^{2} \Bar{\mathcal{F}}
					}{
						U^{f_{A}} H^{2} \Bar{\mathcal{F}}
					}
				\\ @VVV @VVV \\
					\sheafhom_{\alg{B}_{\Hat{c}, \nis}}(
						U^{1} H^{1} \Bar{\mathcal{F}},
						\xi(1)
					)
				@=
					\sheafhom_{\alg{B}_{\Hat{c}, \nis}}(
						U^{1} H^{1} \Bar{\mathcal{F}},
						\xi(1)
					).
			\end{CD}
		\]
	The pairing $U^{1} H^{2} \Bar{\mathcal{F}} \times U^{1} H^{1} \Bar{\mathcal{F}} \to \xi(1)$
	coming from the left vertical morphism is \eqref{0515}.
	Combining the above two diagrams, we get the result.
\end{proof}

Consider the projections
	\begin{gather} \label{0516}
				H^{q} \mathcal{F} \times H^{q'} \mathcal{F}
			\to
				\xi(1),
		\\ \label{0517}
				H^{q} \mathcal{F}_{x} \times H^{q'} \mathcal{F}_{x}
			\to
				\xi(1)
	\end{gather}
of the morphism \eqref{0353} to
$\Spec \alg{B}_{\et}$ and $\Spec \Hat{\alg{k}}_{x, \et}$ (where $x \in T$), respectively.
We describe these pairings.
For $F' \in F^{\perar}$,
the pairing \eqref{0516} on the stalk at a point $x' \in \Spec \alg{B}(F')$ is
	\begin{align*}
		&
					H^{q}(\Hat{\alg{R}}(F')_{x'}^{h}, \Lambda)
				\times
					H^{q'}(\Hat{\alg{R}}(F')_{x'}^{h}, \Lambda)
		\\
		&	\quad
			\to
					H^{3}(\Hat{\alg{R}}(F')_{x'}^{h}, \Lambda)
		\\
		&	\quad
			\to
				H^{1}(\alg{B}(F')_{x'}^{h}, \nu(1))
		\\
		&	\quad
			\cong
				\Gamma(\alg{B}(F')_{x'}^{h}, \xi(1)),
	\end{align*}
where $H^{3}(\Hat{\alg{R}}(F')_{x'}^{h}, \Lambda) = 0$ if $x'$ is a closed point
by Proposition \ref{0072}
and the second map is the Kato boundary map if $x'$ is a generic point.
The pairing \eqref{0517} on the stalk at a point $x' \in \Spec \Hat{\alg{k}}_{x}(F')$ is
	\begin{align*}
		&
					H^{q}(\Hat{\alg{K}}_{\eta_{x}}(F')_{x'}^{h}, \Lambda)
				\times
					H^{q'}(\Hat{\alg{K}}_{\eta_{x}}(F')_{x'}^{h}, \Lambda)
		\\
		&	\quad
			\to
					H^{3}(\Hat{\alg{K}}_{\eta_{x}}(F')_{x'}^{h}, \Lambda)
		\\
		&	\quad
			\to
				H^{1}(\Hat{\alg{k}}_{x}(F')_{x'}^{h}, \nu(1))
		\\
		&	\quad
			\cong
				\Gamma(\Hat{\alg{k}}_{x}(F')_{x'}^{h}, \xi(1)),
	\end{align*}
where the second map is the Kato boundary map.
Hence the restriction of \eqref{0353} to
$U^{m} H^{q} \Bar{\mathcal{F}} \times U^{m'} H^{q'} \Bar{\mathcal{F}}$
is zero if $m + m' > f_{A}$.
Hence it induces a pairing
	\begin{equation} \label{0106}
			\gr^{m} H^{q} \Bar{\mathcal{F}} \times \gr^{m'} H^{q'} \Bar{\mathcal{F}}
		\to
			\xi(1)
	\end{equation}
for pairs $(m, m')$ with $m + m' = f_{A}$ and $m, m' > 0$.
This pairing can be completely calculated as follows.

\begin{Prop} \label{0273}
	Via the isomorphisms in
	Propositions \ref{0100} and \ref{0101},
	the pairing \eqref{0106} becomes
	the following pairings:
		\begin{align*}
					\Ga / \Ga^{p} \times \Ga / \Ga^{p}
			&	\to
					\xi(1),
			\\
					(a, b)
			&	\mapsto
					a d b.
		\end{align*}
	for $q = 2$, $0 < m < f_{K}$ and $p \mid m$;
		\begin{align*}
					\Omega_{B}^{1} \times \Ga
			&	\to
					\xi(1),
			\\
					(\omega, b)
			&	\mapsto
					- m b \omega.
		\end{align*}
	for $q = 1$, $0 < m < f_{K}$ and $p \nmid m$.
\end{Prop}

\begin{proof}
	This is a local statement given in the proof of
	\cite[\S 6, (14)]{Kat79}.
\end{proof}

Hence we obtain a duality for these graded pieces:

\begin{Prop} \label{0274}
	Let $q \in \{1, 2\}$ and $0 < m < f_{A}$
	Set $q' = 3 - q$ and $m' = f_{A} - m$.
	Then the composite
		\[
					R \Bar{\pi}_{\alg{B}, \nis, \ast} \gr^{m} H^{q} \Bar{\mathcal{F}}
				\tensor^{L}
					R \Bar{\pi}_{\alg{B}, \nis, \Hat{!}} \gr^{m'} H^{q'} \Bar{\mathcal{F}}
			\to
				R \Bar{\pi}_{\alg{B}, \nis, \Hat{!}} \xi(1)
			\to
				\xi_{\infty}[-1]
		\]
	of the pairing \eqref{0106} and the morphism \eqref{0354} is a perfect pairing.
\end{Prop}

\begin{proof}
	This follows from Propositions \ref{0273} and \ref{0086}.
\end{proof}

\begin{Prop} \label{0275}
	Let $q \in \{1, 2\}$ and set $q' = 3 - q$.
	Then the composite
		\begin{equation} \label{0107}
					R \Bar{\pi}_{\alg{B}, \nis, \ast}
					\frac{
						U^{1} H^{q} \Bar{\mathcal{F}}
					}{
						U^{f_{A}} H^{q} \Bar{\mathcal{F}}
					}
				\tensor^{L}
					R \Bar{\pi}_{\alg{B}, \nis, \Hat{!}}
					\frac{
						U^{1} H^{q'} \Bar{\mathcal{F}}
					}{
						U^{f_{A}} H^{q'} \Bar{\mathcal{F}}
					}
			\to
				R \Bar{\pi}_{\alg{B}, \nis, \Hat{!}} \xi(1)
			\to
				\xi_{\infty}[-1]
		\end{equation}
	of the pairing \eqref{0105}
	and the morphism \eqref{0354} is a perfect pairing.
\end{Prop}

\begin{proof}
	This follows from the previous proposition.
\end{proof}

Now we bring this result in the Nisnevich/Zariski topology to the desired \'etale topology:

\begin{Prop} \label{0276}
	Let $q \in \{1, 2\}$ and set $q' = 3 - q$.
	Consider the morphisms
		\[
					U^{1} H^{q} \Bar{\mathcal{E}}
				\tensor^{L}
					U^{1} H^{q'} \Bar{\mathcal{E}}
			\to
				\nu(1)[1]
		\]
	in $D(\alg{B}_{\Hat{c}, \et})$ constructed in Proposition \ref{0096}.
	The induced morphism
		\begin{equation} \label{0108}
					R \Bar{\pi}_{\alg{B}, \ast} U^{1} H^{q} \Bar{\mathcal{E}}
				\tensor^{L}
					R \Bar{\pi}_{\alg{B}, \Hat{!}} U^{1} H^{q'} \Bar{\mathcal{E}}
			\to
				R \Bar{\pi}_{\alg{B}, \Hat{!}} \nu(1)[1]
			\to
				\Lambda_{\infty}
		\end{equation}
	in $D(F^{\perar}_{\et})$, where the last morphism is \eqref{0350},
	is a perfect pairing.
\end{Prop}

\begin{proof}
	We have
		\[
					R \varepsilon_{\ast}
					R \Bar{\pi}_{\alg{B}, \ast}
					U^{1} H^{q} \Bar{\mathcal{E}}
				\cong
					R \Bar{\pi}_{\alg{B}, \nis, \ast}
					\frac{
						U^{1} H^{q} \Bar{\mathcal{F}}
					}{
						U^{f_{A}} H^{q} \Bar{\mathcal{F}}
					}
		\]
	by Proposition \ref{0104}.
	By Proposition \ref{0272},
	the morphism \eqref{0107} can also be given by the composite
		\[
					R \varepsilon_{\ast}
					R \Bar{\pi}_{\alg{B}, \ast} U^{1} H^{q} \Bar{\mathcal{E}}
				\tensor^{L}
					R \varepsilon_{\ast}
					R \Bar{\pi}_{\alg{B}, \Hat{!}} U^{1} H^{q'} \Bar{\mathcal{E}}
			\to
				R \varepsilon_{\ast}
				\Lambda_{\infty}[1]
			\to
				\xi_{\infty},
		\]
	where the first morphism is $R \varepsilon_{\ast}$ of \eqref{0108}.
	Hence the result follows from the previous proposition and
	Proposition \ref{0026}.
\end{proof}


\subsection{Duality in full}
\label{0482}

We combine the duality for $\gr^{0}$ and $U^{1}$ into one:

\begin{Prop} \label{0363}
	The morphism
		\[
					R \Bar{\pi}_{\alg{B}, \ast} \Bar{\mathcal{E}}
				\tensor^{L}
					R \Bar{\pi}_{\alg{B}, \Hat{!}} \Bar{\mathcal{E}}
			\to
				R \Bar{\pi}_{\alg{B}, \Hat{!}} \nu(1)[-2]
			\to
				\Lambda_{\infty}[-3]
		\]
	in $D(F^{\perar}_{\et})$ induced by by \eqref{0095} and \eqref{0350}
	is a perfect pairing.
\end{Prop}

\begin{proof}
	This follows from Propositions \ref{0096}, \ref{0098} and \ref{0276}.
\end{proof}

We will interpret this result in terms of
$R \Bar{\pi}_{\Hat{\alg{R}}, \ast}$ and $R \Bar{\pi}_{\Hat{\alg{R}}, \Hat{!}}$.
We have canonical isomorphisms
	\[
				R \Bar{\pi}_{\Hat{\alg{R}}, \ast} \Lambda
			\isomto
				R \Bar{\pi}_{\alg{B}, \ast} \Bar{\mathcal{E}},
		\quad
				R \Bar{\pi}_{\Hat{\alg{R}}, \Hat{!}} \Lambda
			\isomto
				R \Bar{\pi}_{\alg{B}, \Hat{!}} \Bar{\mathcal{E}}
	\]
in $D(F^{\perar}_{\et})$ by Proposition \ref{0355}.

\begin{Prop} \label{0426}
	We have
		$
				R \Bar{\pi}_{\Hat{\alg{R}}, \ast} \Lambda,
				R \Bar{\pi}_{\Hat{\alg{R}}, \Hat{!}} \Lambda
			\in
				\genby{\mathcal{W}_{F}}_{F^{\perar}_{\et}}
		$.
	Moreover, $R^{q} \Bar{\pi}_{\Hat{\alg{R}}, \ast} \Lambda \in \mathcal{W}_{F}$
	for all $q$.
\end{Prop}

\begin{proof}
	This follows from Proposition \ref{0249}.
\end{proof}

\begin{Prop} \label{0421}
	We have $R^{q} \pi_{\Hat{\alg{R}}, \ast} \Lambda = 0$ for $q \ge 3$.
\end{Prop}

\begin{proof}
	This follows from Propositions \ref{0229} and \ref{0080}.
\end{proof}

For any $x \in T$, we have a morphism
	\[
			R^{2} \Bar{\pi}_{\Hat{\alg{K}}_{\eta_{x}}, \ast} \Lambda
		\to
			\Weil_{F_{x} / F} \Lambda
	\]
by the Weil restriction of \eqref{0358}.
Hence the sum of the norm maps gives morphisms
	\[
			\bigoplus_{x \in T}
				R^{2} \Bar{\pi}_{\Hat{\alg{K}}_{\eta_{x}}, \ast} \Lambda
		\to
			\bigoplus_{x \in T}
				\Weil_{F_{x} / F} \Lambda
		\to
			\Lambda.
	\]

\begin{Prop} \label{0362}
	The above composite
	$\bigoplus_{x \in T} R^{2} \Bar{\pi}_{\Hat{\alg{K}}_{\eta_{x}}, \ast} \Lambda \to \Lambda$
	annihilates the image of
	$R^{2} \Bar{\pi}_{\Hat{\alg{R}}, \ast} \Lambda$.
	The obtained morphism
		\begin{equation} \label{0360}
				R^{3} \Bar{\pi}_{\Hat{\alg{R}}, \Hat{!}} \Lambda
			\to
				\Lambda_{\infty}
		\end{equation}
	via the exact sequence
		\[
				R^{2} \Bar{\pi}_{\Hat{\alg{R}}, \ast} \Lambda
			\to
				\bigoplus_{x \in T}
					R^{2} \Bar{\pi}_{\Hat{\alg{K}}_{\eta_{x}}, \ast} \Lambda
			\to
				R^{3} \Bar{\pi}_{\Hat{\alg{R}}, \Hat{!}} \Lambda
			\to
				0
		\]
	agrees with the morphism \eqref{0359}.
\end{Prop}

\begin{proof}
	We have a commutative diagram
		\[
			\begin{CD}
					R^{2} \Bar{\pi}_{\Hat{\alg{R}}, \ast} \Lambda
				@>>>
					\bigoplus_{x \in T}
						R^{2} \Bar{\pi}_{\Hat{\alg{K}}_{\eta_{x}}, \ast} \Lambda
				@>>>
					R^{3} \Bar{\pi}_{\Hat{\alg{R}}, \Hat{!}} \Lambda
				@>>>
					0
				\\ @VVV @VVV @VVV @. \\
					\Bar{\pi}_{\alg{B}, \ast} \nu(1)
				@>>>
					\bigoplus_{x \in T} \Bar{\pi}_{\Hat{\alg{k}}_{x}, \ast} \nu(1)
				@>>>
					R^{1} \Bar{\pi}_{\alg{B}, \Hat{!}} \nu(1)
				@>>>
					0.
			\end{CD}
		\]
	Hence the result follows from Proposition \ref{0357}.
\end{proof}

\begin{Prop} \label{0505}
	The image of the composite
		\[
				R^{2} \Bar{\pi}_{\Hat{\alg{R}}, \ast} \Lambda
			\to
				\bigoplus_{x \in T}
					R^{2} \Bar{\pi}_{\Hat{\alg{K}}_{\eta_{x}}, \ast} \Lambda
			\onto
				\bigoplus_{x \in T}
					\Weil_{F_{x} / F} \Lambda
		\]
	of the natural morphism and the morphism \eqref{0358} is
	the kernel $(\bigoplus_{x \in T} \Weil_{F_{x} / F} \Lambda)_{0}$
	of the sum $\bigoplus_{x \in T} \Weil_{F_{x} / F} \Lambda \to \Lambda$ of the norm maps.
\end{Prop}

\begin{proof}
	This follows from Proposition \ref{0080}.
\end{proof}

In particular, we have a commutative diagram
	\[
		\begin{CD}
				@.
					R^{2} \Bar{\pi}_{\Hat{\alg{R}}, \ast} \Lambda
				@>>>
					\bigoplus_{x \in T}
						R^{2} \Bar{\pi}_{\Hat{\alg{K}}_{\eta_{x}}, \ast} \Lambda
				@>>>
					R^{3} \Bar{\pi}_{\Hat{\alg{R}}, \Hat{!}} \Lambda
				@>>>
					0
				\\ @. @VVV @VVV @VVV @. \\
					0
				@>>>
					\left(
						\bigoplus_{x \in T}
							\Weil_{F_{x} / F} \Lambda
					\right)_{0}
				@>>>
					\bigoplus_{x \in T} \Weil_{F_{x} / F} \Lambda
				@>>>
					\Lambda
				@>>>
					0
		\end{CD}
	\]
with exact rows and surjective vertical morphisms.

\begin{Prop} \label{0109}
	The composite morphism
		\[
					R \Bar{\pi}_{\Hat{\alg{R}}, \ast} \Lambda
				\tensor^{L}
					R \Bar{\pi}_{\Hat{\alg{R}}, \Hat{!}} \Lambda
			\to
				R \Bar{\pi}_{\Hat{\alg{R}}, \Hat{!}} \Lambda
			\to
				\Lambda_{\infty}[-3]
		\]
	of \eqref{0344} and \eqref{0360} is a perfect pairing.
\end{Prop}

\begin{proof}
	We have a commutative diagram
		\begin{equation} \label{0473}
			\begin{CD}
						R \Bar{\pi}_{\Hat{\alg{R}}, \ast} \Lambda
					\tensor^{L}
						R \Bar{\pi}_{\Hat{\alg{R}}, \Hat{!}} \Lambda
				@>>>
					R \Bar{\pi}_{\Hat{\alg{R}}, \Hat{!}} \Lambda
				@>>>
					\Lambda_{\infty}[-3]
				\\ @V \wr VV @V \wr VV @| \\
						R \Bar{\pi}_{\alg{B}, \ast} \Bar{\mathcal{E}}
					\tensor^{L}
						R \Bar{\pi}_{\alg{B}, \Hat{!}} \Bar{\mathcal{E}}
				@>>>
					R \Bar{\pi}_{\alg{B}, \Hat{!}} \Bar{\mathcal{E}}
				@>>>
					\Lambda_{\infty}[-3]
			\end{CD}
		\end{equation}
	by Propositions \ref{0361} and \ref{0362}.
	Hence the result follows from Proposition \ref{0363}.
\end{proof}

Define
	\[
				R \alg{\Gamma}(\Hat{\alg{R}}, \Lambda)
			=
				\algebrize
				R \Bar{\pi}_{\Hat{\alg{R}}, \ast} \Lambda,
		\quad
				R \alg{\Gamma}_{c}(\Hat{\alg{R}}, \Lambda)
			=
				\algebrize
				R \Bar{\pi}_{\Hat{\alg{R}}, \Hat{!}} \Lambda.
	\]

\begin{Thm} \mbox{}
	\begin{enumerate}
		\item
			We have
				$
						R \alg{\Gamma}(\Hat{\alg{R}}, \Lambda),
						R \alg{\Gamma}_{c}(\Hat{\alg{R}}, \Lambda),
					\in
						D^{b}(\Ind \Pro \Alg_{u} / F)
				$.
		\item
			The morphism
				\[
							R \alg{\Gamma}(\Hat{\alg{R}}, \Lambda)
						\tensor^{L}
							R \alg{\Gamma}_{c}(\Hat{\alg{R}}, \Lambda)
					\to
						\Lambda_{\infty}[-3]
				\]
			obtained by applying $\algebrize$ to the morphism \eqref{0473}
			is a perfect pairing in $D(F^{\ind\rat}_{\pro\et})$.
	\end{enumerate}
\end{Thm}

\begin{proof}
	This follows from Propositions \ref{0109}, \ref{0152}, \ref{0153} and \ref{0010}.
\end{proof}

\begin{Rmk}
	The duality theorem in this section should be closely related to
	the duality for relatively perfect nearby cycles (\cite{KS19}) in the following manner.
	Let $k_{0}$ ($\ni \zeta_{p}$) be a mixed characteristic complete discrete valuation field
	with prime element $\varpi$ and residue field $F$.
	Let $A_{0}$ be a smooth algebra over $\Order_{k_{0}}$ with relative dimension $1$
	and geometrically irreducible fibers.
	Then its $\varpi$-adic completion $A$ satisfies
	the conditions listed at the beginning of Section \ref{0340}.
	The duality theorem in \cite{KS19} gives a duality for
	the relatively perfect nearby cycle functor
	$R \Psi^{\mathrm{RP}} \colon D(R_{0, \Et}) \to D(B_{\mathrm{RPS}})$,
	where $R_{0} = A_{0}[1 / p]$ and $B = A / \varpi A$.
	This should be compatible with the nearby cycle functor
	$R \Psi \colon D(\Hat{\alg{R}}_{\et}) \to D(\alg{B}_{\et})$ of this paper
	in a suitable sense.
	
	This paper might be simplified and
	the results strengthened if we could just use the result of \cite{KS19}.
	This route does not seem possible at present, however,
	since our $A$ does not necessarily contain such a base $\Order_{k_{0}}$.
	Our theory is a fragment of a hypothetical theory of
	$p$-adic nearby cycles ``without a base''.
	It is hoped that there is such a theory applicable to the resolution of singularities
	$X \stackrel{j}{\to} \mathfrak{X} \stackrel{i}{\gets} Y$ of Section \ref{0297}.
\end{Rmk}


\section{Two-dimensional local rings}
\label{0277}

Let $A$ be an excellent henselian normal two-dimensional local ring of mixed characteristic $(0, p)$
with maximal ideal $\ideal{m}$ and residue field $F$.
We use the notation of Section \ref{0463}.
Let $\Hat{A}$ be the completion of $A$.

In Section \ref{0278}, we first formulate the duality for $A$ over $\Spec F^{\perar}_{\et}$.
Its proof finishes at Section \ref{0306}.
For the proof, the case where the embedded resolution of $(\Spec A, \Spec A / \sqrt{(p)})$ is already done
is treated in Section \ref{0113}
using the results of Section \ref{0060}.
In Section \ref{0297},
we take a resolution of $A$ and localize the cohomology theory
along the special fiber (or the reduced exceptional divisor) of the resolution.
This gives a fibered site of the type of Section \ref{0322}
consisting of relative \'etale sites of various henselian local rings and fields
and henselian neighborhoods of smooth affine curves over $F$.
In Section \ref{0303},
we take the completions of these local pieces.
The duality for these completed local pieces are proved
by the duality results of Sections \ref{0174}, \ref{0217} and \ref{0113}.
In advance to that, we show the completion invariance
of the duality statement for $A$ in an earlier Section \ref{0464}.
(It is necessary for $A$ to be complete for the duality to hold,
but we may replace the henselian local fields at height one primes
by complete local fields.)


\subsection{Setup}
\label{0278}

As explained in Introduction, we need to vary the residue field of $A$.
Even though there is a canonical way to do so (see Introduction or Definition \ref{0285} below),
we need a slight more flexibility:

\begin{Def} \label{0279}
	A \emph{lifting system} for $A$ consists of:
	\begin{enumerate}
		\item \label{0280}
			a functor $\alg{A}$ from $F^{\perar}$ to the category of $A$-algebras, and
		\item \label{0281}
			an isomorphism $\alg{A}(F') / \ideal{m} \alg{A}(F') \isomto F'$
			of $F$-algebras functorial in $F'$,
	\end{enumerate}
	satisfying the conditions that:
	\begin{enumerate}
		\item \label{0282}
			$\alg{A}$ commutes with finite products,
		\item \label{0283}
			$\alg{A}(F')$ is flat over $A$ for all $F' \in F^{\perar}$,
		\item \label{0284}
			if $F' \in F^{\perar}$ is a field,
			then $\alg{A}(F')$ is an excellent henselian normal two-dimensional local ring
			with maximal ideal $\ideal{m} \alg{A}(F')$, and
		\item \label{0490}
			if $F' \to F''$ is an \'etale morphism in $F^{\perar}$,
			then $\alg{A}(F') \to \alg{A}(F'')$ is finite.
	\end{enumerate}
\end{Def}

Condition \ref{0283} implies that $\alg{A}(F') / \ideal{m}^{n} \alg{A}(F')$
is the Kato canonical lifting of $F'$ over $A / \ideal{m}^{n}$ for any $n$.
In particular, $\alg{A}(F') \to \alg{A}(F'')$ is flat for all morphisms $F' \to F''$ in $F^{\perar}$.
In Condition \ref{0490}, $F' \to F''$ being \'etale implies it is finite \'etale,
and $\alg{A}(F') \to \alg{A}(F'')$ being finite implies it is finite \'etale.
If $F_{0} \in F^{\perar}$ is a field,
then the restriction of $\alg{A}$ to $F_{0}^{\perar}$ is a lifting system for $\alg{A}(F_{0})$.
Such a non-complete lifting shows up from resolutions of $A$ in Section \ref{0297}.

The Teichm\"uller map gives a canonical $W(F)$-algebra structure on $\Hat{A}$
(\cite[Chapter V, Section 4, Theorem 2.1]{DG70b}).

\begin{Def} \label{0285}
	The \emph{canonical lifting system} for $A$ is defined by
		\[
				\Hat{\alg{A}}(F')
			=
				W(F') \Hat{\tensor}_{W(F)} \Hat{A}
			=
				\invlim_{n}
					(W_{n}(F') \tensor_{W_{n}(F)} (A / \ideal{m}^{n}))
		\]
	for all $F' \in F^{\perar}$.
\end{Def}

This is indeed a lifting system for $A$.
If $F_{0} \in F^{\perar}$ is a field,
then the restriction of $\Hat{\alg{A}}$ to $F_{0}^{\perar}$
is the canonical lifting system for $\Hat{\alg{A}}(F_{0})$.
Below we fix a lifting system $\alg{A}$ for $A$.
For any $F' \in F^{\perar}$,
the completion of $\alg{A}(F')$ with respect to $\ideal{m} \alg{A}(F')$
is canonically isomorphic to $\Hat{\alg{A}}(F')$.
In particular, we have a canonical morphism
$\alg{A} \to \Hat{\alg{A}}$ of $F^{\perar}$-algebras.

The lifting system $\alg{A}$ is an $F^{\perar}$-algebra
(Definition \ref{0312}).
Hence the inclusion morphisms
	\[
			X
		\stackrel{j_{A}}{\into}
			\Spec A
		\stackrel{i_{A}}{\hookleftarrow}
			\Spec F
	\]
define an $F^{\perar}$-scheme $\alg{X}$ and sites and morphisms of sites
	\begin{gather*}
				\alg{X}_{\et}
			\stackrel{j_{\alg{A}}}{\to}
				\Spec \alg{A}_{\et}
			\stackrel{i_{\alg{A}}}{\gets}
				\Spec F^{\perar}_{\et},
		\\
				\pi_{\alg{A}}
			\colon
				\Spec \alg{A}_{\et}
			\to
				\Spec F^{\perar}_{\et},
			\quad
				\pi_{\alg{X}}
			\colon
				\alg{X}_{\et}
			\to
				\Spec F^{\perar}_{\et}.
	\end{gather*}
We have $i_{\alg{A}}^{\ast} \cong \pi_{\alg{A}, \ast}$
by Proposition \ref{0313}.
Hence
	\[
			R \Psi_{\alg{X}}
		:=
			i_{\alg{A}}^{\ast} R j_{\alg{A}, \ast}
		\cong
			R \pi_{\alg{X}, \ast}.
	\]

Let $S \subset P$ be a finite subset
and set $U_{S} = X \setminus S$.
The inclusion morphisms
	\[
			U_{S}
		\stackrel{\lambda_{S}}{\into}
			X
		\xleftarrow{\bigsqcup i_{\ideal{p}}}\joinrel\rhook
			\bigsqcup_{\ideal{p} \in S}
				\Spec \kappa(\ideal{p})
	\]
define an $F^{\perar}$-scheme $\alg{U}_{S}$, an $F^{\perar}$-algebra $\algfrak{\kappa}(\ideal{p})$
and sites and morphisms of sites
	\begin{gather*}
				\alg{U}_{S, \et}
			\stackrel{\lambda_{S}}{\to}
				\alg{X}_{\et}
			\stackrel{\bigsqcup i_{\ideal{p}}}{\longleftarrow}
				\bigsqcup_{\ideal{p} \in S}
					\Spec \algfrak{\kappa}(\ideal{p})_{\et},
		\\
				\pi_{\alg{U}_{S}}
			\colon
				\Spec \alg{U}_{S, \et}
			\to
				\Spec F^{\perar}_{\et},
			\quad
				\pi_{\algfrak{\kappa}(\ideal{p})}
			\colon
				\Spec \algfrak{\kappa}(\ideal{p})_{\et}
			\to
				\Spec F^{\perar}_{\et}.
	\end{gather*}

For each $\ideal{p} \in S$,
we have an $F^{\perar}$-algebra $\alg{A}_{\ideal{p}}^{h} = \alg{O}_{X, \ideal{p}}^{h}$ and
a site $\Spec \alg{A}_{\ideal{p}, \et}^{h} = \Spec \alg{O}_{X, \ideal{p}, \et}^{h}$
by the paragraphs after Proposition \ref{0311}.
Define
	\[
			\alg{K}_{\ideal{p}}^{h}(F')
		=
				\alg{A}_{\ideal{p}}^{h}(F')
			\tensor_{A_{\ideal{p}}^{h}}
				K_{\ideal{p}}^{h}.
	\]
The natural morphisms define morphisms of sites and
a commutative diagram of morphisms of sites
	\[
				\pi_{\alg{A}_{\ideal{p}}^{h}}
			\colon
				\Spec \alg{A}_{\ideal{p}, \et}^{h}
			\to
				\Spec F^{\perar}_{\et},
		\quad
				\pi_{\alg{K}_{\ideal{p}}^{h}}
			\colon
				\Spec \alg{K}_{\ideal{p}, \et}^{h}
			\to
				\Spec F^{\perar}_{\et},
	\]
	\[
		\begin{CD}
				\Spec \alg{K}_{\ideal{p}, \et}^{h}
			@> \lambda_{\ideal{p}}^{h} >>
				\Spec \alg{A}_{\ideal{p}, \et}^{h}
			@< i_{\ideal{p}}^{h} <<
				\Spec \algfrak{\kappa}(\ideal{p})_{\et}
			\\
			@V \pi_{\alg{K}_{\ideal{p}}^{h} / \alg{U}_{S}} VV
			@VV \pi_{\alg{A}_{\ideal{p}}^{h} / \alg{X}} V
			@|
			\\
				\alg{U}_{S, \et}
			@>> \lambda_{S} >
				\alg{X}_{\et}
			@<< i_{\ideal{p}} <
				\Spec \algfrak{\kappa}(\ideal{p})_{\et}.
		\end{CD}
	\]

We apply the constructions in Section \ref{0381} to $\lambda_{S} \colon \alg{U}_{S, \et} \to \alg{X}_{\et}$.
This defines a functor
	\[
			R \pi_{\alg{U}_{S}, !}
		\colon
			D(\alg{U}_{S, \et})
		\stackrel{\lambda_{S, !}}{\longrightarrow}
			D(\alg{X}_{\et})
		\stackrel{R \pi_{\alg{X}, \ast}}{\longrightarrow}
			D(F^{\perar}_{\et}).
	\]
For any $F' \in F^{\perar}$, define
	\[
			R \Gamma_{c}(\alg{U}_{S}(F'), \var)
		=
			R \Gamma(F', R \pi_{\alg{U}_{S}, !}(\var))
		\colon
			D(\alg{U}_{S, \et})
		\to
			D(\Ab).
	\]
For any $q \in \Z$ and $G \in D(\alg{U}_{S, \et})$,
the sheaf $R^{q} \pi_{\alg{U}_{S}, !} G$ is the \'etale sheafification of the presheaf
	\[
			F'
		\mapsto
			H_{c}^{q}(\alg{U}_{S}(F'), G).
	\]
Therefore for any perfect field $F'$ over $F$ with algebraic closure $\closure{F'}$,
we have
	\[
			(R^{q} \pi_{\alg{U}_{S}, !} G)(F')
		=
			\left(
				\dirlim_{F''}
					H_{c}^{q}(\alg{U}_{S}(F''), G)
			\right)^{\Gal(\closure{F'} / F')},
	\]
where $F''$ runs through finite subextensions of $\closure{F'} / F'$.

For $G, H \in D(\alg{U}_{S, \et})$,
we have a canonical morphism
	\begin{equation} \label{0396}
			R \pi_{\alg{U}_{S}, !} G \tensor^{L} R \pi_{\alg{U}_{S}, \ast} H
		\to
			R \pi_{\alg{U}_{S}, !}(G \tensor^{L} H)
	\end{equation}
in $D(F^{\perar}_{\et})$ functorial in $G$ and $H$
by \eqref{0369}.
For $G \in D^{+}(\alg{U}_{S, \et})$, we have a canonical distinguished triangle
	\begin{equation} \label{0110}
			R \pi_{\alg{U}_{S}, !} G
		\to
			R \pi_{\alg{U}_{S}, \ast} G
		\to
			\bigoplus_{\ideal{p} \in S}
				R \pi_{\alg{K}_{\ideal{p}}^{h}, \ast}
				\pi_{\alg{K}_{\ideal{p}}^{h} / \alg{U}_{S}}^{\ast} G
	\end{equation}
by \eqref{0380}.
Applying $R \Gamma(F, \var)$ yields a distinguished triangle
	\[
			R \Gamma_{c}(U_{S}, G)
		\to
			R \Gamma(U_{S}, G)
		\to
			\bigoplus_{\ideal{p} \in S}
				R \Gamma(K_{\ideal{p}}^{h}, G),
	\]
where the pullback of $G$ to $K_{\ideal{p}}^{h}$ is denoted
by the same symbol $G$ by abuse of notation.
Further localization gives a distinguished triangle
	\[
			R \Gamma_{c}(U_{S}, G)
		\to
			R \Gamma(K, G)
		\to
				\bigoplus_{\ideal{p} \in S}
					R \Gamma(K_{\ideal{p}}^{h}, G)
			\oplus
				\bigoplus_{\ideal{p} \in P \setminus S}
					R \Gamma_{\ideal{p}}(A_{\ideal{p}}^{h}, G)[1],
	\]
where $K$ is the fraction field of $A$ and
$R \Gamma_{\ideal{p}}$ denotes the cohomology with closed support.
Here is a trace morphism in this setting:

\begin{Prop} \label{0111}
	Let $n \ge 1$.
	We have $R^{q} \pi_{\alg{X}, \ast} \mathfrak{T}_{n}(2) = 0$ for $q \ge 4$.
	There exists a canonical morphism
		\begin{equation} \label{0498}
				R^{3} \pi_{\alg{X}, \ast} \mathfrak{T}_{n}(2)
			\to
				\Lambda_{n}.
		\end{equation}
\end{Prop}

\begin{proof}
	Let $F'$ be a perfect field extension of $F$.
	Let $K_{F'}$ be the fraction field of $\alg{A}(F')$.
	Let $P_{F'}$ be the set of height one prime ideals of $\alg{A}(F')$.
	We have a distinguished triangle
		\[
				R \Gamma(\alg{X}(F'), \mathfrak{T}_{n}(2))
			\to
				R \Gamma(K_{F'}, \Lambda_{n}(2))
			\to
				\bigoplus_{\ideal{p}' \in P_{F'}}
					R \Gamma_{\ideal{p}'}(\alg{A}(F')_{\ideal{p}'}^{h}, \mathfrak{T}_{n}(2))[1].
		\]
	We have
		\[
				R \Gamma_{\ideal{p}'}(\alg{A}(F')_{\ideal{p}'}^{h}, \mathfrak{T}_{n}(2))[1]
			\cong
				\begin{cases}
						R \Gamma(\kappa(\ideal{p}'), \nu_{n}(1))[-2]
					&	\text{if }
						\ideal{p}' \mid p,
					\\
						R \Gamma(\kappa(\ideal{p}'), \Lambda_{n}(1))[-1]
					&	\text{if }
						\ideal{p}' \nmid p,
				\end{cases}
		\]
	whose cohomologies are:
	$\kappa(\ideal{p}')^{\times} / \kappa(\ideal{p}')^{\times p^{n}}$
	in degree $2$;
	$H^{1}(F'_{\ideal{p}'}, \Lambda_{n})$ in degree $3$;
	and zero in degrees $\ge 4$.
	Notice that this cohomology in degree $3$ sheafifies (in $F'$) to zero.
	Let $K_{F'}^{sh}$ be the fraction field of the strict henselization of $\alg{A}(F')$.
	It has cohomological dimension $2$ by \cite[Theorem (5.1)]{Sai86}.
	It follows that
		\[
				H^{q}(K_{F'}, \Lambda_{n}(2))
			\cong
				\begin{cases}
						K_{2}(K_{F'}) / p^{n} K_{2}(K_{F'})
					&	\text{if }
						q = 2,
					\\
						H^{1} \bigl(
							\Gal(\closure{F'} / F'),
							K_{2}(K_{F'}^{sh}) / p^{n} K_{2}(K_{F'}^{sh})
						\bigr)
					&	\text{if }
						q = 3,
					\\
						0
					&	\text{if }
						q \ge 4.
				\end{cases}
		\]
	Notice again that this cohomology in degree $3$ sheafifies to zero.
	It follows that
	$R^{q} \pi_{\alg{X}, \ast} \mathfrak{T}_{n}(2) = 0$ for $q \ge 4$.
	It also yields an exact sequence
		\begin{align*}
			&		K_{2}(K_{F'}) / p^{n} K_{2}(K_{F'})
				\to
					\bigoplus_{\ideal{p}' \in P_{F'}}
						\kappa(\ideal{p}')^{\times} / \kappa(\ideal{p}')^{\times p^{n}}
			\\
			&	\quad
				\to
					H^{3}(\alg{X}(F'), \mathfrak{T}_{n}(2))
				\to
					H^{1} \bigl(
						\Gal(\closure{F'} / F'),
						K_{2}(K_{F'}^{sh}) / p^{n} K_{2}(K_{F'}^{sh})
					\bigr).
		\end{align*}
	The first map is given by the tame symbol maps.
	The localization sequence in K-theory
	(\cite[Chapter V, Proposition 9.2]{Wei13},
	\cite[(0.3)]{Sai87}) shows that the sum of the normalized valuation maps
		\[
				\bigoplus_{\ideal{p}' \in P_{F'}}
					\kappa(\ideal{p}')^{\times} / \kappa(\ideal{p}')^{\times p^{n}}
			\to
				\Lambda_{n}
		\]
	is zero on the image of $K_{2}(K_{F'}) / p^{n} K_{2}(K_{F'})$.
	This defines a map from the kernel of
		\[
				H^{3}(\alg{X}(F'), \mathfrak{T}_{n}(2))
			\to
				H^{1} \bigl(
					\Gal(\closure{F'} / F'),
					K_{2}(K_{F'}^{sh}) / p^{n} K_{2}(K_{F'}^{sh})
				\bigr)
		\]
	to $\Lambda_{n}$ functorial in $F'$.
	After \'etale sheafification in $F'$,
	this defines the desired morphism \eqref{0498}.
\end{proof}

The summary of the above construction of the trace morphism \eqref{0498} is that
if $F$ is algebraically closed, then we have an exact sequence
	\[
			K_{2}(K) / p^{n} K_{2}(K)
		\to
			\bigoplus_{\ideal{p} \in P}
				\kappa(\ideal{p})^{\times} / \kappa(\ideal{p})^{\times p^{n}}
		\to
			H^{3}(X, \mathfrak{T}_{n}(2))
		\to
			0
	\]
and the map $H^{3}(X, \mathfrak{T}_{n}(2)) \to \Lambda_{n}$ is induced by the sum of the normalized valuation maps
on $\kappa(\ideal{p})^{\times}$.
By varying $F$ and taking Galois actions into account, this also characterizes \eqref{0498}.

\begin{Prop} \mbox{} \label{0422}
	\begin{enumerate}
		\item \label{0511}
			We have
				$
						R^{q} \pi_{\alg{U}_{S}, \ast} \mathfrak{T}_{n}(2)
					=
						R^{q} \pi_{\alg{U}_{S}, !} \mathfrak{T}_{n}(2)
					=
						0
				$
			for $q \ge 4$.
		\item \label{0512}
			The composite
				\[
						\bigoplus_{\ideal{p} \in S}
							R^{2} \pi_{\alg{K}_{\ideal{p}}^{h}, \ast} \Lambda_{n}(2)
					\to
						R^{3} \pi_{\alg{U}, !} \mathfrak{T}_{n}(2)
					\to
						R^{3} \pi_{\alg{X}, \ast} \mathfrak{T}_{n}(2)
					\to
						\Lambda_{n}
				\]
			is equal to the composite
				\[
						\bigoplus_{\ideal{p} \in S}
							R^{2} \pi_{\alg{K}_{\ideal{p}}^{h}, \ast} \Lambda_{n}(2)
					\to
						\bigoplus_{\ideal{p} \in S}
							\Weil_{F_{\ideal{p}} / F} \Lambda_{n}
					\to
						\Lambda_{n}
				\]
			of the morphisms \eqref{0358} (for factors with $\ideal{p} \mid p$)
			and \eqref{0510} (for factors with $\ideal{p} \nmid p$)
			and the sum of the norm maps.
	\end{enumerate}
\end{Prop}

\begin{proof}
	\eqref{0511}
	This follows from the distinguished triangles
		\begin{gather*}
					\bigoplus_{\ideal{p} \in S}
						R \pi_{\alg{A}_{\ideal{p}}^{h}, !} \mathfrak{T}_{n}(2)
				\to
					R \pi_{\alg{X}, \ast} \mathfrak{T}_{n}(2)
				\to
					R \pi_{\alg{U}_{S}, \ast} \mathfrak{T}_{n}(2),
			\\
					R \pi_{\alg{U}_{S}, !} \mathfrak{T}_{n}(2)
				\to
					R \pi_{\alg{X}, \ast} \mathfrak{T}_{n}(2)
				\to
					\bigoplus_{\ideal{p} \in S}
						R \pi_{\alg{A}_{\ideal{p}}^{h}, \ast} \mathfrak{T}_{n}(2)
		\end{gather*}
	and Proposition \ref{0445}.
	
	\eqref{0512}
	This follows from the construction.
\end{proof}

As a consequence, we have canonical morphisms
	\begin{equation} \label{0418}
			R \pi_{\alg{U}_{S}, !} \mathfrak{T}_{n}(2)
		\to
			R \pi_{\alg{X}, \ast} \mathfrak{T}_{n}(2)
		\to
			\Lambda_{n}[-3]
		\to
			\Lambda_{\infty}[-3].
	\end{equation}
Now we can state the duality for $A$ over $\Spec F^{\perar}_{\et}$:

\begin{Prop} \label{0112}
	Assume that $A$ is complete and
	take $\alg{A} = \Hat{\alg{A}}$ to be the canonical lifting system.
	Let $n \ge 1$ and $r, r' \in \Z$ with $r + r' = 2$.
	\begin{enumerate}
		\item \label{0286}
			The objects $R \pi_{\alg{U}_{S}, \ast} \mathfrak{T}_{n}(r)$ and
			$R \pi_{\alg{U}_{S}, !} \mathfrak{T}_{n}(r)$
			belong to $\genby{\mathcal{W}_{F}}_{F^{\perar}_{\et}}$.
		\item \label{0287}
			The morphism
				\[
							R \pi_{\alg{U}_{S}, \ast} \mathfrak{T}_{n}(r)
						\tensor^{L}
							R \pi_{\alg{U}_{S}, !} \mathfrak{T}_{n}(r')
					\to
						R \pi_{\alg{U}_{S}, !} \mathfrak{T}_{n}(2)
					\to
						\Lambda_{\infty}[-3]
				\]
			in $D(F^{\perar}_{\et})$ is a perfect pairing.
	\end{enumerate}
\end{Prop}

We will prove the proposition below.
The proof for the case where $n = 1$, $\zeta_{p} \in A$ and $U_{S} = \Spec A[1 / p]$ will finish
at the end of \ref{0303}.
The general case will be proved at the beginning of Section \ref{0306}.


\subsection{Invariance under completion}
\label{0464}

We continue the notation of Section \ref{0278}.
In particular, $A$ is not necessarily complete
and $\alg{A}$ is not necessarily $\Hat{\alg{A}}$.
We first translate the duality setup to total sites of fibered sites.

Consider the total site
	\[
			\alg{U}_{S, c, \et}
		=
			\left(
					\bigsqcup_{\ideal{p} \in S}
						\Spec \alg{K}_{\ideal{p}, \et}^{h}
				\to
					\alg{U}_{S, \et}
			\right)
	\]
and the functors
	\[
			\Bar{\pi}_{\alg{U}_{S}, \ast},
			\Bar{\pi}_{\alg{U}_{S}, !}
		\colon
			\Ch(\alg{U}_{S, c, \et})
		\to
			\Ch(F^{\perar}_{\et})
	\]
defined at the end of Section \ref{0325},
where $\Bar{\pi}_{\alg{U}_{S}}$ factors as
	\[
			\alg{U}_{S, c, \et}
		\to
			\alg{U}_{S, \et}
		\stackrel{\pi_{\alg{U}_{S}}}{\to}
			\Spec F^{\perar}_{\et}
	\]
and we have
	\[
			\Bar{\pi}_{\alg{U}_{S}, !}
		=
			\left[
					\Bar{\pi}_{\alg{U}_{S}, \ast}
				\to
					\bigoplus_{\ideal{p} \in S}
						\Bar{\pi}_{\alg{K}_{\ideal{p}}^{h}, \ast}
			\right][-1].
	\]
For any $n \ge 1$ and $r, r' \in \Z$, we have isomorphisms
	\[
				R \pi_{\alg{U}_{S}, \ast} \mathfrak{T}_{n}(r)
			\cong
				R \Bar{\pi}_{\alg{U}_{S}, \ast} \mathfrak{T}_{n}(r),
		\quad
				R \pi_{\alg{U}_{S}, !} \mathfrak{T}_{n}(r)
			\cong
				R \Bar{\pi}_{\alg{U}_{S}, !} \mathfrak{T}_{n}(r)
	\]
by Proposition \ref{0397}.
The morphism
	\[
				R \pi_{\alg{U}_{S}, \ast} \mathfrak{T}_{n}(r)
			\tensor^{L}
				R \pi_{\alg{U}_{S}, !} \mathfrak{T}_{n}(r')
		\to
			R \pi_{\alg{U}_{S}, !} \mathfrak{T}_{n}(r + r')
	\]
coming from \eqref{0396} can be identified with the morphism
	\[
				R \Bar{\pi}_{\alg{U}_{S}, \ast} \mathfrak{T}_{n}(r)
			\tensor^{L}
				R \Bar{\pi}_{\alg{U}_{S}, !} \mathfrak{T}_{n}(r')
		\to
			R \Bar{\pi}_{\alg{U}_{S}, !} \mathfrak{T}_{n}(r + r')
	\]
coming from \eqref{0377},
by Proposition \ref{0397}.

Next we provide a completion version of this setup.
Set $\Hat{U}_{S} = U_{S} \times_{\Spec A} \Spec \Hat{A}$.
For $F' \in F^{\perar}$,
let $\Hat{\alg{U}}_{S}(F') = U_{S} \times_{\Spec A} \Spec \Hat{\alg{A}}(F')$.
Then we have an $F^{\perar}$-scheme $\Hat{\alg{U}}_{S}$,
a site $\Hat{\alg{U}}_{S, \et}$ and a morphism of sites
$\pi_{\Hat{\alg{U}}_{S}} \colon \Spec \Hat{\alg{U}}_{S, \et} \to \Spec F^{\perar}_{\et}$
by applying the same procedure as $U$.
For $\ideal{p} \in S$,
let $\Hat{A}_{\ideal{p}}$ be the complete local ring of $\Hat{A}$ at $\ideal{p} \Hat{A}$
(which is different from the completion of $A_{\ideal{p}}^{h}$
since the residue field of this $\Hat{A}_{\ideal{p}}$ is the completion of $\kappa(\ideal{p})$).
Let $\Hat{K}_{\ideal{p}}$ be the fraction field of $\Hat{A}_{\ideal{p}}$.
For $F' \in F^{\perar}$,
let $\Hat{\alg{A}}_{\ideal{p}}(F')$ be the completion of
$\Hat{A}_{\ideal{p}} \tensor_{\Hat{A}} \Hat{\alg{A}}(F')$ with respect to the ideal
$\ideal{p} \Hat{A}_{\ideal{p}} \tensor_{\Hat{A}} \Hat{\alg{A}}(F')$.
It is a lifting system for $\Hat{A}_{\ideal{p}}$
in the sense of Definition \ref{0038}.
Set
	$
			\Hat{\alg{K}}_{\ideal{p}}(F')
		=
				\Hat{\alg{A}}_{\ideal{p}}(F')
			\tensor_{\Hat{A}_{\ideal{p}}}
				\Hat{K}_{\ideal{p}}
	$.
We have a natural morphism $\Spec \Hat{\alg{K}}_{\ideal{p}} \to \Hat{\alg{U}}_{S}$
of $F^{\perar}$-schemes.
Hence we have morphisms of sites
	\[
			\bigsqcup_{\ideal{p} \in S}
				\Spec \Hat{\alg{K}}_{\ideal{p}, \et}
		\to
			\Hat{\alg{U}}_{S, \et}
		\to
			\Spec F^{\perar}_{\et}.
	\]
Applying the methods of Section \ref{0325}, we obtain the total site
	\[
			\Hat{\alg{U}}_{S, \Hat{c}, \et}
		=
			\left(
					\bigsqcup_{\ideal{p} \in S}
						\Spec \Hat{\alg{K}}_{\ideal{p}, \et}
				\to
					\Hat{\alg{U}}_{S, \et}
			\right)
	\]
and the functors
	\[
			\Bar{\pi}_{\Hat{\alg{U}}_{S}, \ast},
			\Bar{\pi}_{\Hat{\alg{U}}_{S}, \Hat{!}}
		\colon
			\Ch(\Hat{\alg{U}}_{S, \Hat{c}, \et})
		\to
			\Ch(F^{\perar}_{\et}),
	\]
where $\Bar{\pi}_{\Hat{\alg{U}}_{S}}$ factors as
	\[
			\Hat{\alg{U}}_{S, \Hat{c}, \et}
		\to
			\Hat{\alg{U}}_{S, \et}
		\stackrel{\pi_{\Hat{\alg{U}}_{S}}}{\to}
			\Spec F^{\perar}_{\et}
	\]
and we have
	\[
			\Bar{\pi}_{\Hat{\alg{U}}_{S}, \Hat{!}}
		=
			\left[
					\Bar{\pi}_{\Hat{\alg{U}}_{S}, \ast}
				\to
					\bigoplus_{\ideal{p} \in S}
						\Bar{\pi}_{\Hat{\alg{K}}_{\ideal{p}}, \ast}
			\right][-1].
	\]

Now we compare the two versions.
We have a commutative diagram
	\[
		\begin{CD}
				\bigsqcup_{\ideal{p} \in S}
					\Spec \Hat{\alg{K}}_{\ideal{p}, \et}
			@> \bigsqcup g_{\ideal{p}} >>
				\bigsqcup_{\ideal{p} \in S}
					\Spec \alg{K}_{\ideal{p}, \et}^{h}
			\\ @VVV @VVV \\
				\Hat{\alg{U}}_{S, \Hat{c}, \et}
			@>> g >
				\alg{U}_{S, c, \et}
		\end{CD}
	\]
of morphisms of sites.
We have morphism $g^{\ast} \mathfrak{T}_{n}(r) \to \mathfrak{T}_{n}(r)$
and $g_{\ideal{p}}^{\ast} \mathfrak{T}_{n}(r) \to \mathfrak{T}_{n}(r)$.
Hence \eqref{0467} and \eqref{0468} give morphisms
	\begin{gather} \label{0469}
				R \Bar{\pi}_{\alg{U}_{S}, \ast} \mathfrak{T}_{n}(r)
			\to
				R \Bar{\pi}_{\Hat{\alg{U}}_{S}, \ast} \mathfrak{T}_{n}(r),
		\\ \label{0470}
				R \Bar{\pi}_{\alg{U}_{S}, !} \mathfrak{T}_{n}(r)
			\to
				R \Bar{\pi}_{\Hat{\alg{U}}_{S}, \Hat{!}} \mathfrak{T}_{n}(r).
	\end{gather}
By Proposition \ref{0332}, these morphisms fit in a commutative diagram
	\begin{equation} \label{0465}
		\begin{CD}
					R \Bar{\pi}_{\alg{U}_{S}, \ast} \mathfrak{T}_{n}(r)
				\tensor^{L}
					R \Bar{\pi}_{\alg{U}_{S}, !} \mathfrak{T}_{n}(r')
			@>>>
				R \Bar{\pi}_{\alg{U}_{S}, !} \mathfrak{T}_{n}(r + r')
			\\ @VVV @VVV \\
					R \Bar{\pi}_{\Hat{\alg{U}}_{S}, \ast} \mathfrak{T}_{n}(r)
				\tensor^{L}
					R \Bar{\pi}_{\Hat{\alg{U}}_{S}, \Hat{!}} \mathfrak{T}_{n}(r')
			@>>>
				R \Bar{\pi}_{\Hat{\alg{U}}_{S}, \Hat{!}} \mathfrak{T}_{n}(r + r').
		\end{CD}
	\end{equation}

\begin{Prop} \label{0466}
	Assume that $A$ is complete and $\alg{A} = \Hat{\alg{A}}$.
	Then the morphisms \eqref{0469} and \eqref{0470} are isomorphisms.
	In particular, the vertical morphisms in \eqref{0465} are isomorphisms.
\end{Prop}

\begin{proof}
	The first morphism is obviously an isomorphism.
	For the second isomorphism, it is enough to show that
	for all $\ideal{p} \in S$, the morphism
		\[
				R \pi_{\alg{K}_{\ideal{p}}^{h}, \ast} \Lambda_{n}(r)
			\to
				R \pi_{\Hat{\alg{K}}_{\ideal{p}}, \ast} \Lambda_{n}(r)
		\]
	is an isomorphism.
	But this is Proposition \ref{0173}.
\end{proof}

In particular, under the assumptions of the proposition, we have a canonical morphism
	\[
			R \Bar{\pi}_{\Hat{\alg{U}}_{S}, \Hat{!}} \mathfrak{T}_{n}(2)
		\to
			\Lambda_{n}[-3]
	\]
by \eqref{0418}.

\begin{Prop} \label{0471}
	Assume that $A$ is complete and $\alg{A} = \Hat{\alg{A}}$.
	Let $r + r' = 2$.
	Then the morphism in Proposition \ref{0112} is a perfect pairing
	if and only if the morphism
		\[
					R \Bar{\pi}_{\Hat{\alg{U}}_{S}, \ast} \mathfrak{T}_{n}(r)
				\tensor^{L}
					R \Bar{\pi}_{\Hat{\alg{U}}_{S}, \Hat{!}} \mathfrak{T}_{n}(r')
			\to
				R \Bar{\pi}_{\Hat{\alg{U}}_{S}, \Hat{!}} \mathfrak{T}_{n}(2)
			\to
				\Lambda_{n}[-3]
		\]
	above is a perfect pairing.
\end{Prop}

\begin{proof}
	This follows from Proposition \ref{0466}.
\end{proof}


\subsection{Enough resolved case}
\label{0113}

Assume that $A$ satisfies the conditions listed at the beginning of
Section \ref{0060}.
We use the notation of Section \ref{0060},
particularly Section \ref{0074}.
Let $S \subset P$ be the subset consisting of primes dividing $p$.
Set $R = A[1 / p]$ and $n = 1$ and let $U = U_{S} = \Spec R$.
We prove Proposition \ref{0112}
in this case.
We only treat the case \eqref{0062}
as the case \eqref{0061} is similar and easier.
Therefore we assume \eqref{0062}.
For the moment, $\alg{A}$ is a general lifting system and not necessarily $\Hat{\alg{A}}$.

For $\nu = \alpha$ or $\beta$, set
	\begin{gather*}
				\mathcal{E}_{R}
			=
				R \pi_{\alg{R}, \ast} \Lambda,
			\quad
				\mathcal{E}_{R}'
			=
				R \pi_{\alg{R}, !} \Lambda,
			\quad
				\mathcal{E}_{K_{\nu}^{h}}
			=
				R \pi_{\alg{K}_{\nu}^{h}, \ast} \Lambda,
		\\
				\mathcal{F}_{R}
			=
				R \varepsilon_{\ast} \mathcal{E}_{R},
			\quad
				\mathcal{F}_{R}'
			=
				R \varepsilon_{\ast} \mathcal{E}_{R}',
			\quad
				\mathcal{F}_{K_{\nu}^{h}}
			=
				R \varepsilon_{\ast} \mathcal{E}_{K_{\nu}^{h}}.
	\end{gather*}
The distinguished triangle \eqref{0110} gives distinguished triangles
\begin{gather}
		\notag
				\mathcal{E}_{R}
			\to
				\bigoplus_{\nu = \alpha, \beta}
					\mathcal{E}_{K_{\nu}^{h}}
			\to
				\mathcal{E}_{R}'[1],
		\\ \label{0456}
				\mathcal{F}_{R}
			\to
				\bigoplus_{\nu = \alpha, \beta}
					\mathcal{F}_{K_{\nu}^{h}}
			\to
				\mathcal{F}_{R}'[1]
	\end{gather}
in $D(F^{\perar}_{\et})$, $D(F^{\perar}_{\zar})$, respectively.

\begin{Prop} \label{0288}
	The distinguished triangle \eqref{0456} induces an exact sequence
		\[
				0
			\to
				H^{q} \mathcal{F}_{R}
			\to
				\bigoplus_{\nu = \alpha, \beta}
					H^{q} \mathcal{F}_{K_{\nu}^{h}}
			\to
				H^{q + 1} \mathcal{F}_{R}'
			\to
				0
		\]
	for all $q$.
	We have
		\[
				H^{q} \mathcal{F}_{R}
			=
				H^{q + 1} \mathcal{F}_{R}'
			=
				0
		\]
	for $q \ge 4$.
\end{Prop}

\begin{proof}
	This follows from Propositions
	\ref{0063},
	\ref{0067},
	and \ref{0077}.
\end{proof}

Define $(H^{1} \mathcal{F}_{R})_{\alpha}$
to be the kernel of
the morphism from $H^{1} \mathcal{F}_{R}$
to $H^{1} \mathcal{F}_{K_{\beta}^{h}}$.
Define $(H^{1} \mathcal{R})_{\beta}$
to be the quotient of $H^{1} \mathcal{F}_{R}$
by $(H^{1} \mathcal{F}_{R})_{\alpha}$.
Define $(H^{2} \mathcal{F}_{R})_{\beta}$
to be the kernel of
the morphism from $H^{2} \mathcal{F}_{R}$
to $H^{2} \mathcal{F}_{K_{\alpha}^{h}}$.
Define $(H^{2} \mathcal{F}_{R})_{\alpha}$
to be the quotient of $H^{2} \mathcal{F}_{R}$
by $(H^{2} \mathcal{F}_{R})_{\beta}$.

For $q = 1, 2$ and $\nu = \alpha$ or $\beta$,
define $(H^{q} \mathcal{F}_{R}')_{\nu}$ to be the cokernel of the morphism
	\[
			(H^{q} \mathcal{F}_{R})_{\nu}
		\to
			H^{q} \mathcal{F}_{K_{\nu}^{h}}
	\]
For $m \ge 0$
define $U^{m} (H^{q} \mathcal{F}_{R})_{\nu}$
to be the subsheaf of $(H^{q} \mathcal{F}_{R})_{\nu}$ given by
	\[
			F'
		\mapsto
			U^{m} H^{q}(\alg{R}(F'), \Lambda)_{\nu},
	\]
where $U^{m} H^{q}(\alg{R}(F'), \Lambda)_{\nu}$ is as defined in
Section \ref{0068}.
Define $U^{m} (H^{q} \mathcal{F}_{R}')_{\nu}$ to be the image of
$U^{m} H^{q} \mathcal{F}_{K_{\nu}^{h}}$
in $(H^{q} \mathcal{F}_{R}')_{\nu}$.

For $\nu = \alpha$ or $\beta$ and $F' \in F^{\perar}$,
let $\alg{k}_{\nu}(F') = \alg{A}(F') \tensor_{A} k_{\nu}$
and $\Bar{\alg{A}}_{\nu}(F') = \alg{A}(F') \tensor_{A} \Bar{A}_{\nu}$.
Let $\Omega_{\Bar{\alg{A}}_{\nu}}^{1} = \Bar{\alg{A}}_{\nu} \tensor_{\Bar{A}_{\nu}} \Omega_{\Bar{A}_{\nu}}^{1}$.
Set $\nu' = \alpha$ if $\nu = \beta$ and $\nu' = \beta$ if $\nu = \alpha$.
For $j \ge 1$, let $\Tilde{\alg{U}}_{k_{\nu}}^{(j)}$ be the image of
$1 + \varpi_{\nu'}^{j} \Bar{\alg{A}}_{\nu}$ in
$\alg{k}_{\nu}^{\times} / \alg{k}_{\nu}^{\times p}$.
Let $\Tilde{\alg{A}}_{\alpha}(j)$ be the image of
$\varpi_{\nu'}^{-j} \Bar{\alg{A}}_{\nu}$ in $\alg{k}_{\nu} / \wp \alg{k}_{\nu}$.

The isomorphisms in Propositions
\ref{0075} and \ref{0076}
are functorial in perfect field extensions of $F$.
Therefore we have the following:

\begin{Prop} \label{0114}
	The image of
	$\gr^{m} (H^{q} \mathcal{F}_{R})_{\nu} \into \gr^{m} H^{q} \mathcal{F}_{K_{\nu}^{h}}$
	is given as follows:
		\[
				\gr^{m} (H^{2} \mathcal{F}_{R})_{\alpha}
			\cong
				\begin{cases}
							\alg{k}_{\alpha}^{\times} / \alg{k}_{\alpha}^{\times p}
					&	\text{if }
							m = 0,
					\\
							\Bar{\alg{A}}_{\alpha} / \Bar{\alg{A}}_{\alpha}^{p}
					&	\text{if }
							0 < m < f_{\alpha},\ p \mid m,
					\\
							\alg{\Omega}_{\Bar{A}_{\alpha}}^{1}
					&	\text{if }
							0 < m < f_{\alpha},\ p \nmid m,
					\\
							\xi \oplus \Tilde{\alg{A}}_{\alpha}(f_{\beta})
					&	\text{if }
							m = f_{\alpha},
				\end{cases}
		\]
		\[
				\gr^{m} (H^{1} \mathcal{F}_{R})_{\alpha}
			\cong
				\begin{cases}
							0 \oplus \Tilde{\alg{U}}_{k_{\alpha}}^{(f_{\beta} + 1)}
					&	\text{if }
							m = 0,
					\\
							\Bar{\alg{A}}_{\alpha} / \Bar{\alg{A}}_{\alpha}^{p}
					&	\text{if }
							0 < m < f_{\alpha},\ p \mid m,
					\\
							\Bar{\alg{A}}_{\alpha}
					&	\text{if }
							0 < m < f_{\alpha},\ p \nmid m,
					\\
							0
					&	\text{if }
							m = f_{\alpha},
				\end{cases}
		\]
		\[
				\gr^{m} (H^{2} \mathcal{F}_{R})_{\beta}
			\cong
				\begin{cases}
							\Tilde{\alg{U}}_{k_{\beta}}^{(f_{\alpha} + 1)}
					&	\text{if }
							m = 0,
					\\
							\Bar{\alg{A}}_{\beta} / \Bar{\alg{A}}_{\beta}^{p}
					&	\text{if }
							0 < m < f_{\beta},\ p \mid m,
					\\
							\alg{\Omega}_{\Bar{A}_{\beta}}^{1}
					&	\text{if }
							0 < m < f_{\beta},\ p \nmid m,
					\\
							0
					&	\text{if }
							m = f_{\beta},
				\end{cases}
		\]
		\[
				\gr^{m} (H^{1} \mathcal{F}_{R})_{\beta}
			\cong
				\begin{cases}
							\Lambda \oplus \alg{k}_{\beta}^{\times} / \alg{k}_{\beta}^{\times p}
					&	\text{if }
							m = 0,
					\\
							\Bar{\alg{A}}_{\beta} / \Bar{\alg{A}}_{\beta}^{p}
					&	\text{if }
							0 < m < f_{\beta},\ p \mid m,
					\\
							\Bar{\alg{A}}_{\beta}
					&	\text{if }
							0 < m < f_{\beta},\ p \nmid m,
					\\
							\Tilde{\alg{A}}_{\beta}(f_{\alpha})
					&	\text{if }
							m = f_{\beta}.
				\end{cases}
		\]
\end{Prop}

Consider the morphism
	\[
				\mathcal{F}_{R}
			\tensor^{L}
				\mathcal{F}_{R}'[1]
		\to
			R \varepsilon_{\ast} \Lambda[-2]
		\to
			\xi[-3].
	\]
For $q + q' = 3$, this gives a morphism
	\begin{equation} \label{0115}
				H^{q} \mathcal{F}_{R}
			\tensor^{L}
				H^{q' + 1} \mathcal{F}_{R}'
		\to
			\xi.
	\end{equation}

\begin{Prop} \label{0289}
	The morphism \eqref{0115}
	is a perfect pairing if $q = 0$ or $3$.
\end{Prop}

\begin{proof}
	By Proposition \ref{0077},
	the pairing can be identified with the multiplication pairing
	$\Lambda \times \xi \to \xi$ or $\xi \times \Lambda \to \xi$.
	Hence the result follows from
	Proposition \ref{0028}.
\end{proof}

\begin{Prop} \label{0290}
	Assume $A$ is complete and $\alg{A} = \Hat{\alg{A}}$.
	The morphism \eqref{0115} is a perfect pairing if $q = 1$ or $2$.
\end{Prop}

\begin{proof}
	First note that $\alg{k}_{\nu}$ for $\nu = \alpha$ or $\beta$ agrees with
	the $F^{\perar}$-algebra \eqref{0450} for $k_{\nu}$
	since $A$ (and hence $k_{\nu}$) is complete and $\alg{A} = \Hat{\alg{A}}$.
	
	For $m \ge 0$, $q = 1$ or $2$ and $\nu = \alpha$ or $\beta$,
	we have an exact sequence
		\begin{equation} \label{0116}
				0
			\to
				\gr^{m}(H^{q} \mathcal{F}_{R})_{\nu}
			\to
				\gr^{m} H^{q} \mathcal{F}_{K_{\nu}^{h}}
			\to
				\gr^{m}(H^{q + 1} \mathcal{F}_{R}')_{\nu}
			\to
				0.
		\end{equation}
	By Section \ref{0048}
	and Proposition \ref{0114},
	we know that all these sheaves are finite direct sums of copies of
	$\Lambda$, $\xi$ and Tate vector groups.
	Hence the values of the functor
	$\sheafext_{F^{\perar}_{\zar}}^{i}(\var, \xi)$
	on these sheaves are zero for $i \ge 1$
	by Propositions \ref{0027}
	and \ref{0028}.
	Denote the functor $\sheafhom_{F^{\perar}_{\zar}}(\var, \xi)$ by $(\var)^{\vee}$.
	We thus have an exact sequence
		\[
				0
			\to
				\bigl(
					\gr^{m}(H^{q + 1} \mathcal{F}_{R}')_{\nu}
				\bigr)^{\vee}
			\to
				\bigl(
					\gr^{m} H^{q} \mathcal{F}_{K_{\nu}^{h}}
				\bigr)^{\vee}
			\to
				\bigl(
					\gr^{m}(H^{q} \mathcal{F}_{R})_{\nu}
				\bigr)^{\vee}
			\to
				0.
		\]
	Let $m' = f_{\nu} - m$.
	Compare this sequence with the exact sequence
		\begin{equation} \label{0117}
				0
			\to
				\gr^{m'}(H^{q'} \mathcal{F}_{R})_{\nu}
			\to
				\gr^{m'} H^{q'} \mathcal{F}_{K_{\nu}^{h}}
			\to
				\gr^{m'}(H^{q' + 1} \mathcal{F}_{R}')_{\nu}
			\to
				0.
		\end{equation}
	Proposition \ref{0050}
	gives an isomorphism between the middle terms.
	
	It is enough to show that this isomorphism between the middle terms extends
	to an isomorphism of the sequences.
	By Section \ref{0048}
	and Proposition \ref{0114},
	Write the sequence \eqref{0116}
	as $0 \to C \to D \to E \to 0$ and
	the sequence \eqref{0117}
	as $0 \to C' \to D' \to E' \to 0$,
	so that we have a pairing $D \times D' \to \xi$ inducing an isomorphism
	$D' \isomto D^{\vee}$.
	It is enough to show that
	the kernel of the composite $D' \isomto D^{\vee} \onto C^{\vee}$ is $C'$.
	Let $\nu' = \alpha$ if $\nu = \beta$ and $\nu' = \beta$ if $\nu = \alpha$.
	There are three cases to consider.
	\begin{enumerate}
		\item \label{0291}
			$C \into D$ is
			$\Tilde{\alg{U}}_{k_{\nu}}^{(f_{\nu'} + 1)} \into \alg{k}_{\nu}^{\times} / \alg{k}_{\nu}^{\times p}$;
			$C' \into D'$ is
			$\Tilde{\alg{A}}_{\nu}(f_{\nu'}) \into \alg{k}_{\nu} / \wp \alg{k}_{\nu}$;
			and $D \times D' \to \xi$ is $(x, y) \mapsto \Res(y \dlog x)$.
		\item \label{0292}
			$C \into D$ is
			$\Bar{\alg{A}}_{\nu} \into \alg{k}_{\nu}$;
			$C' \into D'$ is
			$\Omega_{\Bar{\alg{A}}_{\nu}}^{1} \into \Omega_{\alg{k}_{\nu}}^{1}$;
			and $D \times D' \to \xi$ is $(x, \omega) \mapsto \Res(x \omega)$.
		\item \label{0293}
			Both $C \into D$ and $C' \into D'$ are
			$\Bar{\alg{A}}_{\nu} / \Bar{\alg{A}}_{\nu}^{p} \into \alg{k}_{\nu} / \alg{k}_{\nu}^{p}$;
			and $D \times D' \to \xi$ is $(x, y) \mapsto \Res(x dy)$.
	\end{enumerate}
	In all these cases, it is straightforward to check by explicit calculations that
	the kernel of the composite $D' \isomto D^{\vee} \onto C^{\vee}$ is indeed $C'$.
\end{proof}

\begin{Prop} \label{0294}
	Assume $A$ is complete and $\alg{A} = \Hat{\alg{A}}$.
	Then the statement of Proposition \ref{0112} is true
	(under the assumptions of this subsection).
	The object $R^{q} \pi_{\alg{R}, \ast} \Lambda$ is in $\mathcal{W}_{F}$ for all $q$
	and zero for $q \ge 3$.
	The image of the composite
		\[
				R^{2} \pi_{\alg{R}, \ast} \Lambda
			\to
				\bigoplus_{\ideal{q} \in P \setminus P'}
					R^{2} \pi_{\alg{K}_{\ideal{q}}^{h}, \ast} \Lambda
			\onto
				\bigoplus_{\ideal{q} \in P \setminus P'}
					\Lambda
		\]
	of the natural morphism and the morphism \eqref{0358} is
	the subgroup $(\bigoplus_{\ideal{q} \in P \setminus P'} \Lambda)_{0}$ consisting of elements of zero sum.
\end{Prop}

\begin{proof}
	This follows from the previous propositions,
	Propositions \ref{0026}, \ref{0503} and \ref{0504}.
\end{proof}

We record here the completion invariance of $R \Bar{\pi}_{\alg{U}_{S}, !}$
without assuming the completeness of $A$.
(In this case, the completion invariance of the other one
$R \Bar{\pi}_{\alg{U}_{S}, \ast}$ does not hold.)
This is a result of explicit calculations of the graded pieces.
It will be used in Section \ref{0303}.

\begin{Prop} \label{0295}
	The morphism \eqref{0470} is an isomorphism (under the assumption of this subsection).
\end{Prop}

\begin{proof}
	This follows from
	Proposition \ref{0078}
	(and Proposition \ref{0073}
	when in Case \eqref{0061}).
\end{proof}


\subsection{Localization over resolution of singularities}
\label{0297}

We return to general $A$, $\Hat{A}$ and $S$.
We recall some of the notation in \cite[\S 1]{Sai86}
with some slight additions and modifications.
Let $\pi_{\mathfrak{X} / A} \colon \mathfrak{X} \to \Spec A$ be a resolution of singularities
such that $Y \cup \closure{S} \subset \mathfrak{X}$ is
supported on a strict normal crossing divisor,
where $Y$ is the reduced part of $\mathfrak{X} \times_{A} F$ (\cite[Tag 0BIC]{Sta21}).
Let $Y_{0}$ (resp.\ $Y_{1}$) be the set of closed (resp.\ generic) points of $Y$.
For $\eta \in Y_{1}$, let $\eta_{0} = \closure{\{\eta\}} \cap Y_{0}$
and $F_{\eta}$ the constant field of $\closure{\{\eta\}}$.
For $x \in Y_{0}$, let $A_{x}^{h}$ be the henselian local ring of $\mathfrak{X}$ at $x$
with maximal ideal $\ideal{m}_{A_{x}^{h}}$ and residue field $F_{x}$
and let $Y_{1}^{x}$ be the set of height one primes of $A_{x}^{h}$
lying over some element of $Y_{1}$ (via the morphism $\Spec A_{x}^{h} \to \mathfrak{X}$).
For $\eta \in Y_{1}$ and $x \in \eta_{0}$,
there is a unique $\eta_{x} \in Y_{1}^{x}$ lying over $\eta$.
Let $K_{\eta_{x}}^{h}$ be the henselian local field of $A_{x}^{h}$ at $\eta_{x}$,
$\Order_{K_{\eta_{x}}}^{h}$ its ring of integers
and $\kappa(\eta_{x})^{h}$ its residue field.
For each $x \in Y_{0}$,
let $B_{x}^{h} = \Order(\Spec A_{x}^{h} \times_{\mathfrak{X}} Y)$,
$R_{x}^{h} = \Order(\Spec A_{x}^{h} \times_{\mathfrak{X}} X)$
and $R_{x, S}^{h} = \Order(\Spec A_{x}^{h} \times_{\mathfrak{X}} U_{S})$.

We will take henselian neighborhoods of small enough affine opens
of irreducible components of $Y$.
For each $\eta \in Y_{1}$,
choose an affine open neighborhood $W_{\eta} \subset \mathfrak{X}$ of $\eta$
small enough so that:
\begin{itemize}
	\item \label{0298}
		$W_{\eta} \cap \closure{\{\eta'\}} = \emptyset$
		for any $\eta' \in Y_{1} \setminus \{\eta\}$,
	\item \label{0299}
		$W_{\eta}$ does not contain the specialization of any element of $S$ in $Y$ and
	\item \label{0300}
		$W_{\eta} \cap Y \subset W_{\eta}$ is a principal divisor.
\end{itemize}
Set $T = Y \setminus \bigcup_{\eta \in Y_{1}} (W_{\eta} \cap Y)$
and $B_{\eta, T} = \Order(W_{\eta} \cap Y)$.
For each $\eta \in Y_{1}$,
define $A_{\eta, T}^{h}$ to be the the henselization of the pair
$(\Order(W_{\eta}), \ideal{m} \Order(W_{\eta}))$.
Write $\Spec R_{\eta, T}^{h} = \Spec A_{\eta, T}^{h} \setminus \Spec B_{\eta, T}$.
For any $\eta \in Y_{1}$ and $x \in \eta_{0}$,
we have a canonical $A$-algebra homomorphism
$A_{\eta, T}^{h} \to \Order_{K_{\eta_{x}}}^{h}$
inducing the natural inclusion map
$B_{\eta, T} \into \kappa(\eta_{x})^{h}$ on the quotients.

Now we make the above rings into $F^{\perar}$-algebras.
For any $F' \in F^{\perar}$,
define $\algfrak{X}(F') = \mathfrak{X} \times_{A} \alg{A}(F')$
and $\alg{Y}(F') = Y \times_{F} F'$.
For $x \in Y$, define $\alg{A}_{x}^{h}(F')$ to be the henselization of the pair
$(A_{x}^{h} \tensor_{A} \alg{A}(F'), \ideal{m}_{A_{x}^{h}} \tensor_{A} \alg{A}(F'))$,
and set $\alg{R}_{x}^{h}(F') = \alg{A}_{x}^{h}(F') \tensor_{A_{x}^{h}} R_{x}^{h}$,
$\alg{R}_{x, S}^{h}(F') = \alg{A}_{x}^{h}(F') \tensor_{A_{x}^{h}} R_{x, S}^{h}$,
$\alg{B}_{x}^{h}(F') = \alg{A}_{x}^{h}(F') \tensor_{A_{x}^{h}} B_{x}^{h}$
and $\algfrak{\kappa}(x)^{h}(F') = \alg{B}_{x}^{h}(F') \tensor_{B_{x}^{h}} \kappa(x)^{h}$.
For $\eta \in Y_{1}$,
define $\alg{A}_{\eta, T}^{h}(F')$ to be the henselization of the pair
$(A_{\eta, T}^{h} \tensor_{A} \alg{A}(F'), \ideal{m} A_{\eta, T}^{h} \tensor_{A} \alg{A}(F'))$,
and set $\alg{R}_{\eta, T}^{h}(F') = \alg{A}_{\eta, T}^{h}(F') \tensor_{A_{\eta, T}^{h}} R_{\eta, T}^{h}$
and $\alg{B}_{\eta, T}(F') = \alg{A}_{\eta, T}^{h}(F') \tensor_{A_{\eta, T}^{h}} B_{\eta, T}$.
For $x \in T$ and $\eta_{x} \in Y_{1}^{x}$,
let $\alg{O}_{K_{\eta_{x}}}^{h}(F')$ be the henselization of the pair
$(\Order_{K_{\eta_{x}}}^{h} \tensor_{A} \alg{A}(F'), \ideal{m} \Order_{K_{\eta_{x}}}^{h} \tensor_{A} \alg{A}(F'))$.
Let $\alg{K}_{\eta_{x}}^{h}(F') = \alg{O}_{K_{\eta_{x}}}^{h}(F') \tensor_{\Order_{K_{\eta_{x}}}^{h}} K_{\eta_{x}}^{h}$
and $\algfrak{\kappa}(\eta_{x})^{h}(F') = \alg{O}_{K_{\eta_{x}}}^{h}(F') \tensor_{\Order_{K_{\eta_{x}}}^{h}} \kappa(\eta_{x})^{h}$.

We have a commutative diagram
	\begin{equation} \label{0403}
		\begin{CD}
			@.
				\displaystyle
				\bigsqcup_{
					\substack{
						x \in T \\
						\eta_{x} \in Y_{1}^{x}
					}
				}
					\Spec \alg{K}_{\eta_{x}, \et}^{h}
			@>>>
				\displaystyle
				\bigsqcup_{\eta \in Y_{1}}
					\Spec \alg{R}_{\eta, T, \et}^{h}
			@.
			\\ @. @VVV @VVV @. \\
				\displaystyle
				\bigsqcup_{\ideal{p} \in S}
					\Spec \alg{K}_{\ideal{p}, \et}^{h}
			@>>>
				\displaystyle
				\bigsqcup_{x \in T}
					\Spec \alg{R}_{x, S, \et}^{h}
			@>>>
				\alg{U}_{S, \et}
			@>>>
				\Spec F^{\perar}_{\et}
		\end{CD}
	\end{equation}
of morphisms of sites.
Let
	\begin{gather} \label{0401}
				\alg{U}_{T, \et}
			=
				\left(
						\bigsqcup_{\eta \in Y_{1}}
							\Spec \alg{R}_{\eta, T, \et}^{h}
					\gets
						\bigsqcup_{
							\substack{
								x \in T \\
								\eta_{x} \in Y_{1}^{x}
							}
						}
							\Spec \alg{K}_{\eta_{x}, \et}^{h}
					\to
						\bigsqcup_{x \in T}
							\Spec \alg{R}_{x, S, \et}^{h}
				\right),
		\\ \notag
				\alg{U}_{T, c, \et}
			=
				\left(
						\bigsqcup_{\ideal{p} \in S}
							\Spec \alg{K}_{\ideal{p}, \et}^{h}
					\to
						\alg{U}_{T, \et}
				\right)
	\end{gather}
be the total sites.
Then we are in the situation of Section \ref{0322}.
Let $\pi_{\alg{U}_{T} / \alg{U}_{S}} \colon \alg{U}_{T, \et} \to \alg{U}_{S, \et}$ and
$\Bar{\pi}_{\alg{U}_{T} / \alg{U}_{S}} \colon \alg{U}_{T, c, \et} \to \alg{U}_{S, c, \et}$
be the natural morphisms.
Let
	\begin{equation} \label{0402}
		\begin{gathered}
					\pi_{\alg{U}_{T}, \ast}
				=
					\left[
								\bigoplus_{x \in T}
									\pi_{\alg{R}_{x, S}^{h}, \ast}
							\oplus
								\bigoplus_{\eta \in Y_{1}}
									\pi_{\alg{R}_{\eta, T}^{h}, \ast}
						\to
							\bigoplus_{
								\substack{
									x \in T \\
									\eta_{x} \in Y_{1}^{x}
								}
							}
								\pi_{\alg{K}_{\eta_{x}}^{h}, \ast}
					\right][-1]
				\colon
			\\
					\Ch(\alg{U}_{T, \et})
				\to
					\Ch(F^{\perar}_{\et}),
		\end{gathered}
	\end{equation}
	\begin{gather*}
		\begin{aligned}
					\Bar{\pi}_{\alg{U}_{T}, !}
			&	=
					\left[
							\Bar{\pi}_{\alg{U}_{T}, \ast}
						\to
							\bigoplus_{\ideal{p} \in S}
								\Bar{\pi}_{\alg{K}_{\ideal{p}}^{h}, \ast}
					\right][-1]
			\\
			&	\cong
					\left[
								\bigoplus_{x \in T}
									\Bar{\pi}_{\alg{R}_{x, S}^{h}, \ast}
							\oplus
								\bigoplus_{\eta \in Y_{1}}
									\Bar{\pi}_{\alg{R}_{\eta, T}^{h}, \ast}
						\to
								\bigoplus_{
									\substack{
										x \in T \\
										\eta_{x} \in Y_{1}^{x}
									}
								}
									\Bar{\pi}_{\alg{K}_{\eta_{x}}^{h}, \ast}
							\oplus
								\bigoplus_{\ideal{p} \in S}
									\Bar{\pi}_{\alg{K}_{\ideal{p}}^{h}, \ast}
					\right][-1]
				\colon
		\end{aligned}
		\\
				\Ch(\alg{U}_{T, c, \et})
			\to
				\Ch(F^{\perar}_{\et})
	\end{gather*}
be the functors as defined in \eqref{0399} and \eqref{0400},
where the functors pulled back from $D(\alg{U}_{T, \et})$ to $D(\alg{U}_{T, c, \et})$
are denoted by putting overlines.

We compare these functors with the previously defined functors
$\pi_{\alg{U}_{S}, \ast}$ and $\pi_{\alg{U}_{S}, !}$.
By \eqref{0387} and \eqref{0398}, we have natural transformations
	\begin{gather*}
					\pi_{\alg{U}_{S}, \ast}
				\to
					\pi_{\alg{U}_{T}, \ast} \pi_{\alg{U}_{T} / \alg{U}_{S}}^{\ast}
			\colon
				\Ch(\alg{U}_{S, \et})
			\to
				\Ch(F^{\perar}_{\et}),
		\\
					\Bar{\pi}_{\alg{U}_{S}, !}
				\to
					\Bar{\pi}_{\alg{U}_{T}, !} \Bar{\pi}_{\alg{U}_{T} / \alg{U}_{S}}^{\ast}.
			\colon
				\Ch(\alg{U}_{S, c, \et})
			\to
				\Ch(F^{\perar}_{c, \et}).
	\end{gather*}

\begin{Prop} \label{0411}
	Let $G \in D_{\tor}^{+}(\alg{U}_{\et})$.
	Then the above morphisms applied to $G$,
		\begin{gather} \label{0407}
					R \pi_{\alg{U}_{S}, \ast} G
				\to
					R \pi_{\alg{U}_{T}, \ast} \pi_{\alg{U}_{T} / \alg{U}_{S}}^{\ast} G,
			\\ \label{0408}
					R \Bar{\pi}_{\alg{U}_{S}, !} G
				\to
					R \Bar{\pi}_{\alg{U}_{T}, !} \Bar{\pi}_{\alg{U}_{T} / \alg{U}_{S}}^{\ast} G,
		\end{gather}
	in $D(F^{\perar}_{\et})$ are isomorphisms.
\end{Prop}

\begin{proof}
	We have commutative diagrams of morphisms of sites
		\[
			\begin{CD}
					\displaystyle
					\bigsqcup_{
						\substack{
							x \in T \\
							\eta_{x} \in Y_{1}^{x}
						}
					}
						\Spec \alg{K}_{\eta_{x}, \et}^{h}
				@>>>
					\displaystyle
					\bigsqcup_{\eta \in Y_{1}}
						\Spec \alg{R}_{\eta, T, \et}^{h}
				\\ @VVV @VVV \\
					\displaystyle
					\bigsqcup_{x \in T}
						\Spec \alg{R}_{x, \et}^{h}
				@>>>
					\alg{X}_{\et},
			\end{CD}
		\]
		\[
			\begin{CD}
					\displaystyle
					\bigsqcup_{
						\substack{
							x \in T \\
							\eta_{x} \in Y_{1}^{x}
						}
					}
						\Spec \alg{O}_{K_{\eta_{x}}, \et}^{h}
				@>>>
					\displaystyle
					\bigsqcup_{\eta \in Y_{1}}
						\Spec \alg{A}_{\eta, T, \et}^{h}
				\\ @VVV @VVV \\
					\displaystyle
					\bigsqcup_{x \in T}
						\Spec \alg{A}_{x, \et}^{h}
				@>>>
					\algfrak{X}_{\et},
			\end{CD}
		\]
		\[
			\begin{CD}
					\displaystyle
					\bigsqcup_{
						\substack{
							x \in T \\
							\eta_{x} \in Y_{1}^{x}
						}
					}
						\Spec \algfrak{\kappa}(\eta_{x})_{\et}^{h}
				@>>>
					\displaystyle
					\bigsqcup_{\eta \in Y_{1}}
						\Spec \alg{B}_{\eta, T, \et}
				\\ @VVV @VVV \\
					\displaystyle
					\bigsqcup_{x \in T}
						\Spec \alg{B}_{x, \et}^{h}
				@>>>
					\alg{Y}_{\et}.
			\end{CD}
		\]
	In the definitions of $\alg{U}_{T, \et}$ and $\pi_{\alg{U}_{T}, \ast}$
	in \eqref{0401} and \eqref{0402},
	using these squares instead of the square in \eqref{0403},
	we define sites and functors
		\begin{gather*}
				\alg{X}_{T, \et}, \algfrak{X}_{T, \et}, \alg{Y}_{T, \et},
			\\
					\pi_{\alg{X}_{T}, \ast}
				\colon
					\Ch(\alg{X}_{T, \et})
				\to
					\Ch(F^{\perar}_{\et}),
			\\
					\pi_{\algfrak{X}_{T}, \ast}
				\colon
					\Ch(\algfrak{X}_{T, \et})
				\to
					\Ch(F^{\perar}_{\et}),
			\\
					\pi_{\alg{Y}_{T}, \ast}
				\colon
					\Ch(\alg{Y}_{T, \et})
				\to
					\Ch(F^{\perar}_{\et}).
		\end{gather*}
	We have a commutative diagram of morphisms of sites
		\[
			\begin{CD}
					\alg{U}_{T, \et}
				@> \lambda_{T} >>
					\alg{X}_{T, \et}
				@> j_{\algfrak{X}_{T}} >>
					\algfrak{X}_{T, \et}
				@< i_{\algfrak{X}_{T}} <<
					\alg{Y}_{T, \et}
				\\
				@V \pi_{\alg{U}_{T} / \alg{U}_{S}} VV
				@V \pi_{\alg{X}_{T} / \alg{X}} VV
				@V \pi_{\algfrak{X}_{T} / \algfrak{X}} VV
				@V \pi_{\alg{Y}_{T} / \alg{Y}} VV
				\\
					\alg{U}_{S, \et}
				@>> \lambda_{S} >
					\alg{X}_{\et}
				@>> j_{\algfrak{X}} >
					\algfrak{X}_{\et}
				@<< i_{\algfrak{X}} <
					\alg{Y}_{\et}
			\end{CD}
		\]
	Let
		\begin{gather*}
					\Spec \alg{R}_{x, \et}^{h}
				\stackrel{j_{x}}{\to}
					\Spec \alg{A}_{x, \et}^{h}
				\stackrel{i_{x}}{\gets}
					\Spec \alg{B}_{x, \et}^{h},
			\\
					\Spec \alg{R}_{\eta, T, \et}^{h}
				\stackrel{j_{\eta}}{\to}
					\Spec \alg{A}_{\eta, T, \et}^{h}
				\stackrel{i_{\eta}}{\gets}
					\Spec \alg{B}_{\eta, T, \et}^{h},
			\\
					\Spec \alg{K}_{\eta_{x}, \et}^{h}
				\stackrel{j_{\eta_{x}}}{\to}
					\Spec \Spec \alg{O}_{K_{\eta_{x}}, \et}^{h}
				\stackrel{i_{\eta_{x}}}{\gets}
					\Spec \algfrak{\kappa}(x)_{\et}^{h}
		\end{gather*}
	be the components of $j_{\algfrak{X}_{T}}$ and $i_{\algfrak{X}_{T}}$.
	
	Then for any $G \in D^{+}(\alg{Y}_{\et})$,
	the natural morphism
		\begin{equation} \label{0405}
				R \pi_{\alg{Y}, \ast} G
			\to
				R \pi_{\alg{Y}_{T}, \ast} \pi_{\alg{Y}_{T} / \alg{Y}}^{\ast} G
		\end{equation}
	in $D(F^{\perar}_{\et})$ is an isomorphism by excision
	(see the proof of Proposition \ref{0404}).
	
	Consider the commutative diagram
		\[
			\begin{CD}
					\algfrak{X}_{\et}
				@< i_{\algfrak{X}} <<
					\alg{Y}_{\et}
				\\
				@V \pi_{\algfrak{X} / \alg{A}} VV
				@V \pi_{\alg{Y}} VV
				\\
					\Spec \alg{A}_{\et}
				@< i_{\alg{A}} <<
					\Spec F^{\perar}_{\et}.
			\end{CD}
		\]
	Since $\algfrak{X}(F') \to \Spec \alg{A}(F')$ is proper for any $F' \in F^{\perar}$,
	we have
		\[
				R \pi_{\algfrak{X}, \ast} G
			\isomto
				R \pi_{\alg{Y}, \ast} i_{\algfrak{X}}^{\ast} G
		\]
	for any $G \in D_{\tor}^{+}(\algfrak{X}_{\et})$
	by the proper base change theorem.
	
	For any $G_{x} \in D(\alg{A}_{x, \et}^{h})$,
	$G_{\eta} \in D_{\tor}^{+}(\alg{A}_{\eta, T, \et}^{h})$ and
	$G_{\eta_{x}} \in D(\alg{O}_{K_{\eta_{x}}, \et}^{h})$,
	the natural morphisms
		\begin{gather*}
					R \pi_{\alg{A}_{x}^{h}, \ast} G_{x}
				\to
					R \pi_{\alg{B}_{x}^{h}, \ast} i_{x}^{\ast} G_{x},
			\\
					R \pi_{\alg{A}_{\eta, T}^{h}, \ast} G_{\eta}
				\to
					R \pi_{\alg{B}_{\eta, T}^{h}, \ast} i_{\eta}^{\ast} G_{\eta},
			\\
					R \pi_{\alg{O}_{K_{\eta_{x}}}^{h}, \ast} G_{\eta_{x}}
				\to
					R \pi_{\algfrak{\kappa}(\eta_{x})^{h}, \ast} i_{\eta_{x}}^{\ast} G_{\eta_{x}}
		\end{gather*}
	are isomorphisms by Propositions \ref{0313}, \ref{0341} and \ref{0313}, respectively.
	
	Therefore for any $G \in D_{\tor}^{+}(\algfrak{X}_{\et})$,
	the isomorphism \eqref{0405} applied to $i_{\algfrak{X}}^{\ast} G$
	shows that the morphism
		\[
				R \pi_{\algfrak{X}, \ast} G
			\to
				R \pi_{\algfrak{X}_{T}, \ast} \pi_{\algfrak{X}_{T} / \algfrak{X}}^{\ast} G
		\]
	is an isomorphism.
	For any $G \in D_{\tor}^{+}(\alg{U}_{S, \et})$,
	applying this to $R j_{\mathfrak{X}, \ast} R \lambda_{S, \ast} G$,
	we obtain the desired isomorphism \eqref{0407}.
	The isomorphism \eqref{0408} follows from this by the diagram \eqref{0409}.
\end{proof}

For any $G, H \in D(\alg{U}_{T, c, \et})$, we have a canonical morphism
	\[
				R \Bar{\pi}_{\alg{U}_{T}, !} G
			\tensor^{L}
				R \Bar{\pi}_{\alg{U}_{T}, \ast} H
		\to
			R \Bar{\pi}_{\alg{U}_{T}, !}(G \tensor^{L} H)
	\]
by \eqref{0389}.
This morphism and the morphism \eqref{0396} fit in a commutative diagram
	\[
		\begin{CD}
					R \Bar{\pi}_{\alg{U}, !} G
				\tensor^{L}
					R \Bar{\pi}_{\alg{U}, \ast} H
			@>>>
				R \Bar{\pi}_{\alg{U}_{S}, !}(G \tensor^{L} H)
			\\ @VVV @VVV \\
					R \Bar{\pi}_{\alg{U}_{T}, !} \Bar{\pi}_{\alg{U}_{T} / \alg{U}_{S}}^{\ast} G
				\tensor^{L}
					R \Bar{\pi}_{\alg{U}_{T}, \ast} \Bar{\pi}_{\alg{U}_{T} / \alg{U}_{S}}^{\ast} H
			@>>>
				R \Bar{\pi}_{\alg{U}_{T}, !} \Bar{\pi}_{\alg{U}_{T} / \alg{U}_{S}}^{\ast}(G \tensor^{L} H)
		\end{CD}
	\]
by Proposition \ref{0410}.
With Proposition \ref{0411}, we have
	\[
		\begin{CD}
					R \Bar{\pi}_{\alg{U}_{S}, !} \mathfrak{T}_{n}(r)
				\tensor^{L}
					R \Bar{\pi}_{\alg{U}_{S}, \ast} \mathfrak{T}_{n}(r')
			@>>>
				R \Bar{\pi}_{\alg{U}_{S}, !} \mathfrak{T}_{n}(r + r')
			\\ @| @| \\
					R \Bar{\pi}_{\alg{U}_{T}, !} \mathfrak{T}_{n}(r)
				\tensor^{L}
					R \Bar{\pi}_{\alg{U}_{T}, \ast} \mathfrak{T}_{n}(r')
			@>>>
				R \Bar{\pi}_{\alg{U}_{T}, !} \mathfrak{T}_{n}(r + r')
		\end{CD}
	\]
for $n \ge 1$ and $r, r' \in \Z$.
The morphism \eqref{0418} gives a canonical morphism
	\[
			R \Bar{\pi}_{\alg{U}_{S}, \ast} \mathfrak{T}_{n}(2)
		\cong
			R \Bar{\pi}_{\alg{U}_{T}, !} \mathfrak{T}_{n}(2)
		\to
			\Lambda[-3].
	\]
The digram \eqref{0419} gives a morphism of distinguished triangles
	\begin{equation} \label{0424}
		\begin{CD}
				R \Bar{\pi}_{\alg{U}_{S}, !} \mathfrak{T}_{n}(r)
			@>>>
				R \Bar{\pi}_{\alg{U}_{S}, \ast} \mathfrak{T}_{n}(r)
			@>>>
				\displaystyle
				\bigoplus_{\ideal{p} \in S}
					R \Bar{\pi}_{\alg{K}_{\ideal{p}}^{h}, \ast} \mathfrak{T}_{n}(r)
			\\ @| @VVV @VVV \\
				R \Bar{\pi}_{\alg{U}_{T}, !} \mathfrak{T}_{n}(r)
			@>>>
				\begin{array}{c}
						\bigoplus_{x \in T}
							R \Bar{\pi}_{\alg{R}_{x, S}^{h}, \ast} \mathfrak{T}_{n}(r)
					\\
					\oplus
						\bigoplus_{\eta \in Y_{1}}
							R \Bar{\pi}_{\alg{R}_{\eta, T}^{h}, \ast} \mathfrak{T}_{n}(r)
				\end{array}
			@>>>
				\begin{array}{c}
						\bigoplus_{
							\substack{
								x \in T \\
								\eta_{x} \in Y_{1}^{x}
							}
						}
							R \Bar{\pi}_{\alg{K}_{\eta_{x}}^{h}, \ast} \mathfrak{T}_{n}(r)
					\\
					\oplus
						\bigoplus_{\ideal{p} \in S}
							R \Bar{\pi}_{\alg{K}_{\ideal{p}}^{h}, \ast} \mathfrak{T}_{n}(r).
				\end{array}
		\end{CD}
	\end{equation}


\subsection{Completions}
\label{0303}

We continue the assumption and notation from the previous subsection.
We make the completion version of the constructions of the previous section
and compare these two versions.

For each $\eta \in Y_{1}$,
let $\Hat{A}_{\eta, T}$ be the $\ideal{m} A_{\eta, T}^{h}$-adic completion of $A_{\eta, T}^{h}$
and set $\Hat{R}_{\eta, T} = \Hat{A}_{\eta, T} \tensor_{A_{\eta, T}^{h}} R_{\eta, T}^{h}$.
For each $x \in Y_{0}$,
let $\Hat{A}_{x}$ be the completion of the local ring $A_{x}^{h}$
and set $\Hat{B}_{x} = \Hat{A}_{x} \tensor_{A_{x}^{h}} B_{x}^{h}$
and $\Hat{R}_{x, S} = \Hat{A}_{x} \tensor_{A_{x}^{h}} R_{x, S}^{h}$.
For each $\eta \in Y_{1}$ and $x \in \eta_{0}$,
let $\Hat{\kappa}(\eta_{x})$ be the completion of the local ring $\kappa(\eta_{x})^{h}$.
Let $\Hat{\Order}_{K_{\eta_{x}}}$ be the canonical lifting of $\Hat{\kappa}(\eta_{x})$
to $\Order_{K_{\eta_{x}}}^{h}$
and set $\Hat{K}_{\eta_{x}} = \Hat{\Order}_{K_{\eta_{x}}}[1 / p]$.
For $F' \in F^{\perar}$,
define $\Hat{\alg{A}}_{x}(F')$, $\Hat{\alg{R}}_{\eta, T}(F')$ and so on
using completions instead of henselizations.
We have morphisms of $F^{\perar}$-algebras
$\Hat{\alg{R}}_{T, \eta} \to \Hat{\alg{K}}_{\eta_{x}}$
for $x \in T$ and $\eta = \eta_{x} \in Y_{1}^{x}$
and $\Hat{\alg{R}}_{x, S} \to \Hat{\alg{K}}_{\eta_{x}} \times \Hat{\alg{K}}_{\ideal{p}}$
for $x \in T$ and $\ideal{p} \in S^{x}$.

We have a commutative diagram of morphisms of sites
	\[
		\begin{CD}
			@.
				\displaystyle
				\bigsqcup_{
					\substack{
						x \in T \\
						\eta_{x} \in Y_{1}^{x}
					}
				}
					\Spec \Hat{\alg{K}}_{\eta_{x}, \et}
			@>>>
				\displaystyle
				\bigsqcup_{\eta \in Y_{1}}
					\Spec \Hat{\alg{R}}_{\eta, T, \et}
			@.
			\\ @. @VVV @VVV @. \\
				\displaystyle
				\bigsqcup_{\ideal{p} \in S}
					\Spec \Hat{\alg{K}}_{\ideal{p}, \et}
			@>>>
				\displaystyle
				\bigsqcup_{x \in T}
					\Spec \Hat{\alg{R}}_{x, S, \et}
			@>>>
				\Hat{\alg{U}}_{S, \et}
			@>>>
				\Spec F^{\perar}_{\et},
		\end{CD}
	\]
and a morphism of fibered sites from this diagram to \eqref{0403}.
Let
	\begin{gather*}
				\Hat{\alg{U}}_{\Hat{T}, \et}
			=
				\left(
						\bigsqcup_{\eta \in Y_{1}}
							\Spec \Hat{\alg{R}}_{\eta, T, \et}
					\gets
						\bigsqcup_{
							\substack{
								x \in T \\
								\eta_{x} \in Y_{1}^{x}
							}
						}
							\Spec \Hat{\alg{K}}_{\eta_{x}, \et}
					\to
						\bigsqcup_{x \in T}
							\Spec \Hat{\alg{R}}_{x, S, \et}
				\right),
		\\ \notag
				\Hat{\alg{U}}_{\Hat{T}, \Hat{c}, \et}
			=
				\left(
						\bigsqcup_{\ideal{p} \in S}
							\Spec \Hat{\alg{K}}_{\ideal{p}, \et}
					\to
						\Hat{\alg{U}}_{\Hat{T}, \et}
				\right)
	\end{gather*}
be the total sites.
Then we are in the situation of Section \ref{0322}.
Let
	\begin{gather*}
				\pi_{\Hat{\alg{U}}_{\Hat{T}} / \alg{U}_{T}}
			\colon
				\Hat{\alg{U}}_{\Hat{T}, \et}
			\to
				\alg{U}_{T, \et},
		\\
				\Bar{\pi}_{\Hat{\alg{U}}_{\Hat{T}} / \alg{U}_{T}}
			\colon
				\Hat{\alg{U}}_{\Hat{T}, \Hat{c}, \et}
			\to
				\alg{U}_{T, c, \et}
	\end{gather*}
be the natural morphisms.
Let
	\[
		\begin{gathered}
					\pi_{\Hat{\alg{U}}_{\Hat{T}}, \ast}
				=
					\left[
								\bigoplus_{x \in T}
									\pi_{\Hat{\alg{R}}_{x, S}, \ast}
							\oplus
								\bigoplus_{\eta \in Y_{1}}
									\pi_{\Hat{\alg{R}}_{\eta, T}, \ast}
						\to
							\bigoplus_{
								\substack{
									x \in T \\
									\eta_{x} \in Y_{1}^{x}
								}
							}
								\pi_{\Hat{\alg{K}}_{\eta_{x}}, \ast}
					\right][-1]
				\colon
			\\
					\Ch(\Hat{\alg{U}}_{\Hat{T}, \et})
				\to
					\Ch(F^{\perar}_{\et}),
		\end{gathered}
	\]
	\begin{gather*}
		\begin{aligned}
					\Bar{\pi}_{\Hat{\alg{U}}_{\Hat{T}}, \Hat{!}}
			&	=
					\left[
							\Bar{\pi}_{\Hat{\alg{U}}_{\Hat{T}}, \ast}
						\to
							\bigoplus_{\ideal{p} \in S}
								\Bar{\pi}_{\Hat{\alg{K}}_{\ideal{p}}, \ast}
					\right][-1]
			\\
			&	\cong
					\left[
								\bigoplus_{x \in T}
									\Bar{\pi}_{\Hat{\alg{R}}_{x, S}, \ast}
							\oplus
								\bigoplus_{\eta \in Y_{1}}
									\Bar{\pi}_{\Hat{\alg{R}}_{\eta, T}, \ast}
						\to
								\bigoplus_{
									\substack{
										x \in T \\
										\eta_{x} \in Y_{1}^{x}
									}
								}
									\Bar{\pi}_{\Hat{\alg{K}}_{\eta_{x}}, \ast}
							\oplus
								\bigoplus_{\ideal{p} \in S}
									\Bar{\pi}_{\Hat{\alg{K}}_{\ideal{p}}, \ast}
					\right][-1]
				\colon
		\end{aligned}
		\\
				\Ch(\Hat{\alg{U}}_{\Hat{T}, \Hat{c}, \et})
			\to
				\Ch(F^{\perar}_{\et})
	\end{gather*}
be the functors as defined in \eqref{0399} and \eqref{0400},
where the functors pulled back from $D(\Hat{\alg{U}}_{\Hat{T}, \et})$ to $D(\Hat{\alg{U}}_{\Hat{T}, \Hat{c}, \et})$
are denoted by putting overlines.

For the rest of this subsection,
we assume that $A$ is complete, $\alg{A} = \Hat{\alg{A}}$,
$\zeta_{p} \in A$ and $U_{S} = \Spec A[1 / p]$,
in order to use the results of Sections \ref{0217} and \ref{0113}.

\begin{Prop} \label{0304}
	The morphisms
		\[
				R \Bar{\pi}_{\alg{U}_{T}, \ast} \Lambda
			\to
				R \Bar{\pi}_{\Hat{\alg{U}}_{\Hat{T}}, \ast} \Lambda,
			\quad
				R \Bar{\pi}_{\alg{U}_{T}, !} \Lambda
			\to
				R \Bar{\pi}_{\Hat{\alg{U}}_{\Hat{T}}, \Hat{!}} \Lambda
		\]
	in $D(F^{\perar}_{\et})$ given by \eqref{0414} are isomorphisms.
\end{Prop}

\begin{proof}
	We have morphisms of distinguished triangles
		\[
			\begin{CD}
					R \Bar{\pi}_{\alg{U}_{T}, !} \Lambda
				@>>>
					R \Bar{\pi}_{\alg{U}_{T}, \ast} \Lambda
				@>>>
					\bigoplus_{\ideal{p} \in S}
						R \Bar{\pi}_{\alg{K}_{\ideal{p}}^{h}, \ast} \Lambda
				\\ @VVV @VVV @VVV \\
					R \Bar{\pi}_{\Hat{\alg{U}}_{\Hat{T}}, \Hat{!}} \Lambda
				@>>>
					R \Bar{\pi}_{\Hat{\alg{U}}_{\Hat{T}}, \ast} \Lambda
				@>>>
					\bigoplus_{\ideal{p} \in S}
						R \Bar{\pi}_{\Hat{\alg{K}}_{\ideal{p}}, \ast} \Lambda,
			\end{CD}
		\]
		\[
			\begin{CD}
					\bigoplus_{x \in T}
						R \Bar{\pi}_{\alg{R}_{x, S}^{h}, !} \Lambda
				@>>>
					R \Bar{\pi}_{\alg{U}_{T}, !}  \Lambda
				@>>>
					\bigoplus_{\eta \in Y_{1}}
						R \Bar{\pi}_{\alg{R}_{\eta, T}^{h}, \ast}  \Lambda
				\\ @VVV @VVV @VVV \\
					\bigoplus_{x \in T}
						R \Bar{\pi}_{\Hat{\alg{R}}_{x, S}, \Hat{!}} \Lambda
				@>>>
					R \Bar{\pi}_{\Hat{\alg{U}}_{\Hat{T}}, \Hat{!}} \Lambda
				@>>>
					\bigoplus_{\eta \in Y_{1}}
						R \Bar{\pi}_{\Hat{\alg{R}}_{\eta, T}, \ast} \Lambda
			\end{CD}
		\]
	by \eqref{0416} and \eqref{0415}.
	The morphism
		\[
				R \Bar{\pi}_{\alg{K}_{\ideal{p}}^{h}, \ast} \Lambda
			\to
				R \Bar{\pi}_{\Hat{\alg{K}}_{\ideal{p}}, \ast} \Lambda
		\]
	is an isomorphism by Proposition \ref{0173}.
	For any $F' \in F^{\perar}$, the morphism
		\[
				R \Gamma(\alg{R}_{\eta, T}^{h}(F'), \Lambda)
			\to
				R \Gamma(\Hat{\alg{R}}_{\eta, T}(F'), \Lambda)
		\]
	is an isomorphism by Fujiwara-Gabber's formal base change theorem
	\cite[Corollary 6.6.4]{Fuj95}, \cite[Expos\'e XX, \S 4.4]{ILO14},
	\cite[Corollary 1.18 (2)]{BM21}.
	This implies that the morphism
		\[
				R \Bar{\pi}_{\alg{R}_{\eta, T}^{h}, \ast}  \Lambda
			\to
				R \Bar{\pi}_{\Hat{\alg{R}}_{\eta, T}, \ast} \Lambda
		\]
	is an isomorphism.
	
	Finally we show that the morphism
		\[
				R \Bar{\pi}_{\alg{R}_{x, S}^{h}, !} \Lambda
			\to
				R \Bar{\pi}_{\Hat{\alg{R}}_{x, S}, \Hat{!}} \Lambda
		\]
	is an isomorphism.
	Since $S$ is the set of all primes above $p$,
	we have $R_{x, S}^{h} = A_{x}^{h}[1 / p]$ and $\Hat{R}_{x, S} = \Hat{A}_{x}[1 / p]$.
	The rings $A_{x}^{h}$ and $\Hat{A}_{x}$ satisfy the conditions
	listed at the beginning of Section \ref{0060}.
	Hence we may apply Proposition \ref{0295},
	proving that the above morphism is indeed an isomorphism.
\end{proof}

Note that the morphisms
	\[
				R \Bar{\pi}_{\alg{R}_{\eta, T}^{h}, !}  \Lambda
			\to
				R \Bar{\pi}_{\Hat{\alg{R}}_{\eta, T}, \Hat{!}} \Lambda,
		\quad
				R \Bar{\pi}_{\alg{R}_{x, S}^{h}, \ast} \Lambda
			\to
				R \Bar{\pi}_{\Hat{\alg{R}}_{x, S}, \ast} \Lambda
	\]
are not isomorphisms.

By this proposition, \eqref{0417} gives a commutative diagram
	\[
		\begin{CD}
					R \Bar{\pi}_{\alg{U}_{T}, \ast} \Lambda
				\tensor^{L}
					R \Bar{\pi}_{\alg{U}_{T}, !} \Lambda
			@>>>
				R \Bar{\pi}_{\alg{U}_{T}, !} \Lambda
			\\ @| @| \\
					R \Bar{\pi}_{\Hat{\alg{U}}_{\Hat{T}}, \ast} \Lambda
				\tensor^{L}
					R \Bar{\pi}_{\Hat{\alg{U}}_{\Hat{T}}, \Hat{!}} \Lambda
			@>>>
				R \Bar{\pi}_{\Hat{\alg{U}}_{\Hat{T}}, \Hat{!}} \Lambda
		\end{CD}
	\]
in $D(F^{\perar}_{\et})$.
The diagram \eqref{0420} gives a morphism of distinguished triangles
	\begin{equation} \label{0425}
		\begin{CD}
				R \Bar{\pi}_{\alg{U}_{T}, !} \Lambda
			@>>>
					\displaystyle
					\bigoplus_{x \in T}
						R \Bar{\pi}_{\alg{R}_{x, S}^{h}, \ast} \Lambda
				\oplus
					\displaystyle
					\bigoplus_{\eta \in Y_{1}}
						R \Bar{\pi}_{\alg{R}_{\eta, T}^{h}, \ast} \Lambda
			@>>>
					\displaystyle
					\bigoplus_{
						\substack{
							x \in T \\
							\eta_{x} \in Y_{1}^{x}
						}
					}
						R \Bar{\pi}_{\alg{K}_{\eta_{x}}^{h}, \ast} \Lambda
				\oplus
					\displaystyle
					\bigoplus_{\ideal{p} \in S}
						R \Bar{\pi}_{\alg{K}_{\ideal{p}}^{h}, \ast} \Lambda
			\\ @| @VVV @VVV \\
				R \Bar{\pi}_{\Hat{\alg{U}}_{\Hat{T}}, \Hat{!}} \Lambda
			@>>>
					\displaystyle
					\bigoplus_{x \in T}
						R \Bar{\pi}_{\Hat{\alg{R}}_{x, S}, \ast} \Lambda
				\oplus
					\displaystyle
					\bigoplus_{\eta \in Y_{1}}
						R \Bar{\pi}_{\Hat{\alg{R}}_{\eta, T}, \ast} \Lambda
			@>>>
					\displaystyle
					\bigoplus_{
						\substack{
							x \in T \\
							\eta_{x} \in Y_{1}^{x}
						}
					}
						R \Bar{\pi}_{\Hat{\alg{K}}_{\eta_{x}}, \ast} \Lambda
				\oplus
					\displaystyle
					\bigoplus_{\ideal{p} \in S}
						R \Bar{\pi}_{\Hat{\alg{K}}_{\ideal{p}}, \ast} \Lambda.
		\end{CD}
	\end{equation}

\begin{Prop} \label{0428}
	The objects
		\[
			\bigoplus_{x \in T}
				R \Bar{\pi}_{\Hat{\alg{R}}_{x, S}, \ast} \Lambda,\;
			\bigoplus_{\eta \in Y_{1}}
				R \Bar{\pi}_{\Hat{\alg{R}}_{\eta, T}, \ast} \Lambda,\;
			\bigoplus_{
				\substack{
					x \in T \\
					\eta_{x} \in Y_{1}^{x}
				}
			}
				R \Bar{\pi}_{\Hat{\alg{K}}_{\eta_{x}}, \ast} \Lambda,\;
			\bigoplus_{\ideal{p} \in S}
				R \Bar{\pi}_{\Hat{\alg{K}}_{\ideal{p}}, \ast} \Lambda
		\]
	are all concentrated in degrees $\le 2$ with cohomologies in $\mathcal{W}_{F}$.
	In particular,
		$
				R \Bar{\pi}_{\Hat{\alg{U}}_{\Hat{T}}, \ast} \Lambda,
				R \Bar{\pi}_{\Hat{\alg{U}}_{\Hat{T}}, \Hat{!}} \Lambda
			\in
				\genby{\mathcal{W}_{F}}_{F^{\perar}_{\et}}
		$.
\end{Prop}

\begin{proof}
	This follows from
	Propositions \ref{0294}, \ref{0421}, \ref{0426}, \ref{0500} and \ref{0502}.
\end{proof}

We will define a trace isomorphism for $R \Bar{\pi}_{\Hat{\alg{U}}_{\Hat{T}}, \Hat{!}}$
using the lower triangle of \eqref{0425} as follows.
This triangle induces a morphism
	\begin{equation} \label{0509}
				\bigoplus_{x \in T}
					\pi_{0} R^{2} \Bar{\pi}_{\Hat{\alg{R}}_{x, S}, \ast} \Lambda
			\oplus
				\bigoplus_{\eta \in Y_{1}}
					\pi_{0} R^{2} \Bar{\pi}_{\Hat{\alg{R}}_{\eta, T}, \ast} \Lambda
		\to
				\bigoplus_{
					\substack{
						x \in T \\
						\eta_{x} \in Y_{1}^{x}
					}
				}
					\pi_{0} R^{2} \Bar{\pi}_{\Hat{\alg{K}}_{\eta_{x}}, \ast} \Lambda
			\oplus
				\bigoplus_{\ideal{p} \in S}
					\pi_{0} R^{2} \Bar{\pi}_{\Hat{\alg{K}}_{\ideal{p}}, \ast} \Lambda.
	\end{equation}
For any $x \in T$, $\eta_{x} \in Y_{1}^{x}$ and $\ideal{p} \in S$,
we have isomorphisms
	\begin{equation} \label{0506}
			\pi_{0} R^{2} \Bar{\pi}_{\Hat{\alg{K}}_{\eta_{x}}, \ast} \Lambda
		\isomto
			\Weil_{F_{x} / F} \Lambda,
		\quad
			\pi_{0} R^{2} \Bar{\pi}_{\Hat{\alg{K}}_{\ideal{p}}, \ast} \Lambda
		\isomto
			\Weil_{F_{\ideal{p}} / F} \Lambda
	\end{equation}
by the morphism \eqref{0358} and Proposition \ref{0502}.
With these isomorphisms, we have
	\begin{equation} \label{0507}
			\pi_{0} R^{2} \Bar{\pi}_{\Hat{\alg{R}}_{x, S}, \ast} \Lambda
		\isomto
			\Weil_{F_{x} / F} \left(
				\left(
						\bigoplus_{\eta_{x} \in Y_{1}^{x}} \Lambda
					\oplus
						\bigoplus_{\ideal{p} \in S^{x}} \Lambda
				\right)_{0}
			\right)
	\end{equation}
by Proposition \ref{0294}.
Also for any $\eta \in Y_{1}$, we have a surjection
	\begin{equation} \label{0508}
			\pi_{0} R^{2} \Bar{\pi}_{\Hat{\alg{R}}_{\eta, T}, \ast} \Lambda
		\onto
			\Weil_{F_{\eta} / F} \left(
				\left(
					\bigoplus_{x \in T \cap \eta_{0}}
						\Weil_{F_{x} / F_{\eta}} \Lambda
				\right)_{0}
			\right)
	\end{equation}
and the part $\pi_{0} R^{2} \Bar{\pi}_{\Hat{\alg{R}}_{\eta, T}, \ast} \Lambda$
of the morphism \eqref{0509} factors through this surjection by Proposition \ref{0505}.

\begin{Prop} \label{0423}
	The morphisms \eqref{0509}, \eqref{0506}, \eqref{0507} and \eqref{0508} form an exact sequence
		\begin{align*}
			&
						\bigoplus_{x \in T}
							\Weil_{F_{x} / F} \left(
								\left(
										\bigoplus_{\eta_{x} \in Y_{1}^{x}} \Lambda
									\oplus
										\bigoplus_{\ideal{p} \in S^{x}} \Lambda
								\right)_{0}
						\right)
					\oplus
						\bigoplus_{\eta \in Y_{1}}
							\Weil_{F_{\eta} / F} \left(
								\left(
									\bigoplus_{x \in T \cap \eta_{0}}
										\Weil_{F_{x} / F_{\eta}} \Lambda
								\right)_{0}
						\right)
			\\
			&	\to
						\bigoplus_{
							\substack{
								x \in T \\
								\eta_{x} \in Y_{1}^{x}
							}
						}
							\Weil_{F_{x} / F} \Lambda
					\oplus
						\bigoplus_{\ideal{p} \in S}
							\Weil_{F_{\ideal{p}} / F} \Lambda
			\\
			&	\to
					\Lambda
				\to
					0,
		\end{align*}
	where the last morphism (to $\Lambda$) is the sum of the norm maps.
\end{Prop}

\begin{proof}
	We may replace $F$ by a large enough finite extension of $F$
	so that all the fields $F_{x}$, $F_{\eta}$ and $F_{\ideal{p}}$ in the statement are $F$.
	Then the statement is a simple combinatorics of the irreducible components of
	$\mathfrak{X} \setminus U_{S} = Y \cup \closure{S}$.
\end{proof}

This proposition and the lower triangle of \eqref{0425} defines a canonical morphism
	\begin{equation} \label{0513}
				R \Bar{\pi}_{\Hat{\alg{U}}_{\Hat{T}}, \Hat{!}} \Lambda
			\to
				R^{3} \Bar{\pi}_{\Hat{\alg{U}}_{\Hat{T}}, \Hat{!}} \Lambda[-3]
			\to
				\Lambda[-3].
	\end{equation}
This trace morphism and the trace morphism \eqref{0418} are compatible:

\begin{Prop}
	The morphism \eqref{0513} and the morphism
		\[
				R^{3} \Bar{\pi}_{\alg{U}_{S}, !} \Lambda
			\to
				\Lambda
		\]
	of \eqref{0418} are compatible under the isomorphisms
		\[
				R^{3} \Bar{\pi}_{\alg{U}_{S}, !} \Lambda
			\isomto
				R^{3} \Bar{\pi}_{\alg{U}_{T}, !} \Lambda
			\isomto
				R^{3} \Bar{\pi}_{\Hat{\alg{U}}_{\Hat{T}}, \Hat{!}} \Lambda.
		\]
\end{Prop}

\begin{proof}
	Denote $\eqref{0513}$ by $\Tr_{S}$ and \eqref{0418} by $\Tr_{T}$.
	By \eqref{0424} and \eqref{0425}, we have a commutative diagram
		\[
			\begin{CD}
					\displaystyle
					\bigoplus_{\ideal{p} \in S}
						R \Bar{\pi}_{\alg{K}_{\ideal{p}}^{h}, \ast} \Lambda
				@>>>
					R \Bar{\pi}_{\alg{U}_{S}, !} \Lambda[1]
				@> \Tr_{S} >>
					\Lambda[-2]
				\\ @VVV @| \\
					\displaystyle
						\bigoplus_{
							\substack{
								x \in T \\
								\eta_{x} \in Y_{1}^{x}
							}
						}
							R \Bar{\pi}_{\Hat{\alg{K}}_{\eta_{x}}, \ast} \Lambda
					\oplus
						\displaystyle
						\bigoplus_{\ideal{p} \in S}
							R \Bar{\pi}_{\Hat{\alg{K}}_{\ideal{p}}, \ast} \Lambda
				@>>>
					R \Bar{\pi}_{\Hat{\alg{U}}_{\Hat{T}}, \Hat{!}} \Lambda[1]
				@>> \Tr_{T} >
					\Lambda[-2]
			\end{CD}
		\]
	in $\genby{\mathcal{W}_{F}}_{F^{\perar}_{\et}}$.
	Applying $\algebrize$, we have a commutative diagram
		\begin{equation} \label{0514}
			\begin{CD}
					\displaystyle
					\bigoplus_{\ideal{p} \in S}
						\algebrize
						R \Bar{\pi}_{\alg{K}_{\ideal{p}}^{h}, \ast} \Lambda
				@>>>
					\algebrize
					R \Bar{\pi}_{\alg{U}_{S}, !} \Lambda[1]
				@> \Tr_{S} >>
					\Lambda[-2]
				\\ @VVV @| \\
					\displaystyle
						\bigoplus_{
							\substack{
								x \in T \\
								\eta_{x} \in Y_{1}^{x}
							}
						}
							\algebrize
							R \Bar{\pi}_{\Hat{\alg{K}}_{\eta_{x}}, \ast} \Lambda
					\oplus
						\displaystyle
						\bigoplus_{\ideal{p} \in S}
							\algebrize
							R \Bar{\pi}_{\Hat{\alg{K}}_{\ideal{p}}, \ast} \Lambda
				@>>>
					\algebrize
					R \Bar{\pi}_{\Hat{\alg{U}}_{\Hat{T}}, \Hat{!}} \Lambda[1]
				@>> \Tr_{T} >
					\Lambda[-2]
			\end{CD}
		\end{equation}
	in $D^{b}(\Ind \Pro \Alg_{u} / F)$.
	By Proposition \ref{0152},
	it is enough to show that $\Tr_{S} = \Tr_{T}$ in this diagram.
	By Propositions \ref{0428} and \ref{0501}, the object
	$\algebrize R \Bar{\pi}_{\Hat{\alg{U}}_{\Hat{T}}, \Hat{!}} \Lambda[1]$
	is concentrated in degrees $\le 2$.
	By Proposition \ref{0423},
	the $\pi_{0}$ of its $H^{2}$ is $\Lambda$.
	Hence $\pi_{0} H^{2}$ of the diagram \eqref{0514} can be written as
		\[
			\begin{CD}
					\displaystyle
					\bigoplus_{\ideal{p} \in S}
						\Weil_{F_{\ideal{p}} / F} \Lambda
				@>>>
					\Lambda
				@> \Tr_{S} >>
					\Lambda
				\\ @VVV @| \\
					\displaystyle
						\bigoplus_{
							\substack{
								x \in T \\
								\eta_{x} \in Y_{1}^{x}
							}
						}
							\Weil_{F_{x} / F} \Lambda
					\oplus
						\displaystyle
						\bigoplus_{\ideal{p} \in S}
							\Weil_{F_{\ideal{p}} / F} \Lambda
				@>>>
					\Lambda
				@>> \Tr_{T} = \id >
					\Lambda
			\end{CD}
		\]
	The composite of the upper two horizontal morphisms is the sum of the norm maps
	by Proposition \ref{0422} \eqref{0512}.
	The composite of the lower two horizontal morphisms is the sum of the norm maps
	by the definition of \eqref{0513}.
	Since the first upper horizontal morphism
	$\bigoplus_{\ideal{p} \in S} \Weil_{F_{\ideal{p}} / F} \Lambda \to \Lambda$
	is also the sum of the norm maps and hence surjective,
	we know that $\Tr_{S} \colon \Lambda \to \Lambda$ is the identity map.
	Therefore $\Tr_{S} = \Tr_{T}$ in the diagram \eqref{0514}.
\end{proof}

Thus we have a commutative diagram
	\[
		\begin{CD}
					R \Bar{\pi}_{\alg{U}_{S}, \ast} \Lambda
				\tensor^{L}
					R \Bar{\pi}_{\alg{U}_{S}, !} \Lambda
			@>>>
				\Lambda_{\infty}[-3]
			\\ @| @| \\
					R \Bar{\pi}_{\Hat{\alg{U}}_{\Hat{T}}, \ast} \Lambda
				\tensor^{L}
					R \Bar{\pi}_{\Hat{\alg{U}}_{\Hat{T}}, \Hat{!}} \Lambda
			@>>>
				\Lambda_{\infty}[-3].
		\end{CD}
	\]

\begin{Prop} \label{0429}
	The above is a perfect pairing.
\end{Prop}

\begin{proof}
	Denote $(\var)^{\vee} = R \sheafhom_{F^{\perar}_{\et}}(\var, \Lambda_{\infty})$.
	By Proposition \ref{0427}, we have a morphism of distinguished triangles
		\[
			\begin{CD}
					\displaystyle
					\bigoplus_{\eta \in Y_{1}}
						R \Bar{\pi}_{\Hat{\alg{R}}_{\eta, T}, \Hat{!}} \Lambda
				@>>>
					R \Bar{\pi}_{\Hat{\alg{U}}_{\Hat{T}}, \ast} \Lambda
				@>>>
					\displaystyle
					\bigoplus_{x \in T}
						R \Bar{\pi}_{\Hat{\alg{R}}_{x, S}, \ast} \Lambda
				\\ @VVV @VVV @VVV \\
					\displaystyle
					\bigoplus_{\eta \in Y_{1}}
						(R \Bar{\pi}_{\Hat{\alg{R}}_{\eta, T}, \ast} \Lambda)^{\vee}[-3]
				@>>>
					(R \Bar{\pi}_{\Hat{\alg{U}}_{\Hat{T}}, \Hat{!}} \Lambda)^{\vee}[-3]
				@>>>
					\displaystyle
					\bigoplus_{x \in T}
						(R \Bar{\pi}_{\Hat{\alg{R}}_{x, S}, \Hat{!}} \Lambda)^{\vee}[-3]
			\end{CD}
		\]
	The left vertical morphism is an isomorphism by Propositions \ref{0109}
	(and Proposition \ref{0452}
	if the constant field of $B_{\eta, T}$ is not $F$ but a finite extension of $F$).
	The right vertical morphism is an isomorphism
	by Propositions \ref{0294} and \ref{0471}
	(and similarly Proposition \ref{0452}
	if the residue field of $B_{x}^{h}$ is not $F$ but a finite extension of $F$).
	Hence so is the middle.
\end{proof}

\begin{Prop} \label{0431}
	The statement of Proposition \ref{0112} is true
	if $n = 1$, $\zeta_{p} \in A$ and $U_{S} = \Spec A[1 / p]$.
\end{Prop}

\begin{proof}
	This follows from Propositions \ref{0428} and \ref{0429}.
\end{proof}


\subsection{Main theorem}
\label{0306}

Now we can finish the proof of Proposition \ref{0112}.

\begin{proof}[Proof of Proposition \ref{0112}]
	The case $n \ge 1$ is reduced to the case $n = 1$.
	Let $S' \subset P$ be a finite subset containing $S$.
	Let $U_{S'} = X \setminus S'$.
	Then the statement for $U_{S}$ is equivalent to the statement for $U_{S'}$.
	Indeed, we have a morphism of distinguished triangles
		\[
			\begin{CD}
					\bigoplus_{\ideal{p} \in S' \setminus S}
						R \pi_{\alg{A}_{\ideal{p}}^{h}, !} \mathfrak{T}(r)
				@>>>
					R \pi_{\alg{U}_{S}, \ast} \mathfrak{T}(r)
				@>>>
					R \pi_{\alg{U}_{S'}, \ast} \mathfrak{T}(r)
				\\ @VVV @VVV @VVV \\
					\bigoplus_{\ideal{p} \in S' \setminus S}
						(R \pi_{\alg{A}_{\ideal{p}}^{h}, \ast} \mathfrak{T}(r'))^{\vee}[-3]
				@>>>
					(R \pi_{\alg{U}_{S}, !} \mathfrak{T}(r'))^{\vee}[-3]
				@>>>
					(R \pi_{\alg{U}_{S'}, !} \mathfrak{T}(r'))^{\vee}[-3]
			\end{CD}
		\]
	in $D(F^{\perar}_{\et})$ by Proposition \ref{0430}.
	The left vertical morphism is an isomorphism
	of objects of $\genby{\mathcal{W}_{F}}_{F^{\perar}_{\et}}$
	by Propositions \ref{0445} and \ref{0449}.
	Therefore the invertibility of the other two vertical morphisms is equivalent to each other.
	
	Next, to prove the statement of the theorem,
	we may assume $\zeta_{p} \in A$.
	Indeed, let $A'$ be the normalization of $A$ in $K(\zeta_{p})$.
	Let $U'_{S'}$ be the inverse image of $U_{S}$ by the morphism $\Spec A' \to \Spec A$.
	We may assume that $S$ contains all primes above $p$ and $U'_{S'} \to U_{S}$ is \'etale by the previous step.
	Since $[K(\zeta_{p}) : K]$ is prime to $p$,
	a norm argument shows that the statement for $U'_{S'}$ implies the statement for $U_{S}$.
	
	Now we are reduced to the case $n = 1$, $\zeta_{p} \in A$ and $U_{S} = \Spec A[1 / p]$,
	which is proved in Proposition \ref{0431}.
\end{proof}

\begin{Thm} \label{0308}
	Assume that $A$ is complete and take $\alg{A}$ to be the canonical lifting system.
	Let $n \ge 1$. Let $r, r' \in \Z$ with $r + r' = 2$.
	\begin{enumerate}
		\item
			The objects $R \alg{\Gamma}(\alg{U}_{S}, \mathfrak{T}_{n}(r))$
			and $R \alg{\Gamma}_{c}(\alg{U}_{S}, \mathfrak{T}_{n}(r))$
			belong to $D^{b}(\Ind \Pro \Alg_{u} / F)$.
		\item
			The composite morphism
				\[
							R \alg{\Gamma}(\alg{U}_{S}, \mathfrak{T}_{n}(r))
						\tensor^{L}
							R \alg{\Gamma}_{c}(\alg{U}_{S}, \mathfrak{T}_{n}(r'))
					\to
						R \alg{\Gamma}_{c}(\alg{U}_{S}, \mathfrak{T}_{n}(2))
					\to
						\Lambda_{\infty}[-3]
				\]
			obtained by applying $\algebrize$ to the morphism in Proposition \ref{0112} \eqref{0287}
			is a perfect pairing in $D(F^{\ind\rat}_{\pro\et})$.
	\end{enumerate}
\end{Thm}

\begin{proof}
	This follows from Propositions \ref{0112}, \ref{0152}, \ref{0153} and \ref{0010}.
\end{proof}

This finishes the proof of Theorems \ref{0128} and \ref{0132}:
Statement \eqref{0129} of the first theorem is
Proposition \ref{0153};
Statement \eqref{0130} follows from Proposition \ref{0530};
and Statement \eqref{0131} is a consequence of the previous two statements.


\end{document}